\documentclass[11pt]{amsart}
\usepackage[utf8]{inputenc}
\usepackage[T1]{fontenc}

\usepackage{amsmath,amsthm,amsfonts,amssymb,mathrsfs,pb-diagram,amscd}
\usepackage{amsmath}
\usepackage{amsxtra}
\usepackage{amscd}
\usepackage{amsthm}
\usepackage{amsfonts}
\usepackage{amssymb}
\usepackage{eucal}
\usepackage{color}
\usepackage[enableskew,vcentermath]{youngtab}
\usepackage{bm}

\usepackage[pagebackref=false]{hyperref}
\DeclareMathOperator\Std{Std}
\DeclareMathOperator\res{res}
\DeclareMathOperator\sgn{sgn}
\allowdisplaybreaks[4]

\newtheorem{thm}{Theorem}[section]
\theoremstyle{plain}

\newtheorem{lem}[thm]{Lemma}

\newtheorem{prop}[thm]{Proposition}
\newtheorem{cor}[thm]{Corollary}

\theoremstyle{definition}
\newtheorem{defn}[thm]{Definition}
\newtheorem{example}[thm]{Example}

\theoremstyle{remark}
\newtheorem{rem}[thm]{Remark}
\newtheorem{conjecture}[thm]{Conjecture}

\definecolor{A}{rgb}{.75,1,.75}

\numberwithin{equation}{section}

\newcommand{\Z}{\mathbb Z}
\newcommand{\N}{\mathbb N}

\newcommand{\mHcn}{\mathcal{H}_{c}(n)} 
\newcommand{\mHfcn}{\mathcal{H}^f_{c}(n)} 
\newcommand{\mhcn}{\mathfrak{H}_{c}(n)} 
\newcommand{\mhgcn}{\mathfrak{H}^g_{c}(n)} 

\newcommand{\mb}{\mathtt{b}}
\newcommand{\mt}{\mathfrak{t}}
\newcommand{\ms}{\mathfrak{s}}
\newcommand{\mfku}{\mathfrak{u}}
\newcommand{\mfkv}{\mathfrak{v}}

\newcommand{\End}{\text{End}}

\newcommand{\supp}{\text{supp}}

\newcommand{\undla}{\underline{\lambda}}
\newcommand{\undQ}{\underline{Q}}

{\vskip-\lastskip\medskip
	\noindent
	{\em #1.}\enspace
}%
{\qed\par\medskip
}

\begin{document}

	\title[cyclotomic Hecke-Clifford algebras]{Seminormal bases of cyclotomic Hecke-Clifford algebras}
	
	\author{Shuo Li}\address{School of Mathematics and Statistics\\
		Beijing Institute of Technology\\
		Beijing, 100081, P.R. China}
	\email{shuoli1203@163.com}

    \author{Lei Shi}\address{Academy of Mathematics and Systems Science\\
    	Chinese Academy of Sciences, Beijing 100190\\
    	P.R.China}
    	\address{Max-Planck-Institut f\"ur Mathematik\\
	    Vivatsgasse 7, 53111 Bonn\\
	    Germany}
    \email{leishi202406@163.com}

	\begin{abstract}
	In this paper, we describe the actions of standard generators on certain bases of simple modules for semisimple cyclotomic Hecke-Clifford superalgebras. As applications, we explicitly construct a complete set of primitive idempotents and seminormal bases for these algebras.
	\end{abstract}
	\maketitle
	
	\setcounter{tocdepth}{1}
	\tableofcontents
	
	\section{Introduction}
    \label{pag:N}
	Let $\N:=\{1,2,\ldots\},$ $n\in \N$ and $\mathfrak{S}_n$ be the symmetric group of $n$ letters. In \cite{Sch}, to study the projective representation theory of $\mathfrak{S}_n$, Schur introduced a double cover $\widetilde{\mathfrak{S}}_n$ and showed that the study of the projective representation of $\mathfrak{S}_n$ is equivalent to the study of linear representation of $\widetilde{\mathfrak{S}}_n$ while the latter is equivalent to the linear representation theory of $\mathfrak{S}_n$ and $\mathbb{C} \mathfrak{S}_n^-$ (spin symmetric group algebra). The linear representation theory of $\mathfrak{S}_n$ as well as its $q$-analogue, Iwahori-Hecke algebra of type $A$ has various generalizations, including the representation theory of non-degenerate and degenerate cyclotomic Hecke algebras \cite{AK} and quiver Hecke algebras of type $A$ \cite{BK:GradedKL}, which has been extensively studied in the literature, see \cite{A3, BK:GradedKL,HM1,K3,Ma1} and references therein.
	
	The spin representation of $\mathfrak{S}_n$ or the linear representation of $\mathbb{C}\mathfrak{S}_n^-$ has attracted much attention in recent years, see \cite{BK:spin,ELL,K2,Na1,WW}. It is ``super-equivalent'' to the representation theory of the so-called Sergeev algebra $\mathcal{C}_n \rtimes \mathbb{C} \mathfrak{S}_n$ \cite{Na2}. The representation theory of Sergeev algebra as well as its $q$-analogue \cite{JN}, Hecke-Clifford algebra $\mathcal{H}(n)$ also has various generalizations, including the representation theory of non-degenerate and degenerate cyclotomic Hecke-Clifford algebras \cite{BK} and quiver Hecke superalgebras \cite{KKT,KKO1,KKO2}. In \cite{BK,T}, Brundan, Kleshchev, and Tsuchioka studied modular branching rules for non-degenerate and degenerate cyclotomic Hecke-Clifford algebras. They used induction and restriction functors to construct simple modules for these algebras. Unfortunately, in contrast to the usual cyclotomic Hecke algebra \cite{DJM}, there is no ``cellular theory''  for cyclotomic Hecke-Clifford algebras, due to the lack of ``cellular basis''. In \cite{EM,HM1,HM2}, several (graded) cellular bases were constructed for cyclotomic Hecke algebras and cyclotomic quiver Hecke algebras of type $A$ and $C$ using the semisimple deformation and seminormal bases. Hence, to give a ``cellular'' structure for cyclotomic Hecke-Clifford algebras, it's natural to consider the semisimple deformation and construct seminormal bases for cyclotomic Hecke-Clifford algebras. This is the motivation of this work.
	
	
	From now on, let $\mathbb{K}$ \label{pag:K} be an algebraically closed field of characteristic different from $2$ and
$\mathbb{K}^*:=\mathbb{K}\setminus\{0\}.$ In order to state our main results, we introduce some notations. For each $\bullet\in\{\mathsf{0},\,\mathsf{s},\,\mathsf{ss}\}$ and $\underline{Q}=(Q_1,\ldots,Q_m)$, we associate the polynomial $f= f^{(\bullet)}_{\underline{Q}}(X_1)$ in the definition of non-degenerate cyclotomic Hecke-Clifford superalgebra $\mHfcn$ with a set $ \mathscr{P}^{\bullet,m}_{n}$  which consists of mixing of partitions  and strict partitions. More precisely, for $m,n\geq 0$, let $\mathscr{P}^{\mathsf{0},m}_{n}:=\mathscr{P}^m_n$ be the set of all $m$-multipartitions of $n$  and $\mathscr{P}^\mathsf{s}_n$ be the set of strict partitions of $n$. We set
	$$
	\mathscr{P}^{\mathsf{s},m}_{n}:=
	\cup_{a=0}^{n}( \mathscr{P}^{\mathsf{s}}_a\times \mathscr{P}^{m}_{n-a}),\qquad \mathscr{P}^{\mathsf{ss}, m}_{n}:=
	\cup_{a+b+c=n}(\mathscr{P}^{\mathsf{s}}_a \times \mathscr{P}^{\mathsf{s}}_b\times \mathscr{P}^{m}_{c}).$$  In \cite{SW}, the second author and Wan gave a ``separate condition'' for $\mHfcn$ in terms of the residues of the boxes in the multipartitions. The ``separate condition'' for $\mHfcn$ is equivalent to the polynomial $P^{\bullet}_{n}(q^2,\undQ)\neq 0$ (see \cite[Proposition 5.13]{SW}), where the $P^{\bullet}_{n}(q^2,\undQ)$ can be regarded as a super-version of the classical Poincar\'e polynomial. Under this condition, the second author and Wan showed that cyclotomic Hecke-Clifford superalgebra $\mHfcn$ is semisimple and gave a complete set of simple modules
$\{\mathbb{D}(\undla)|~ \undla\in\mathscr{P}^{\bullet,m}_{n}\}$, generalizing \cite{JN,Na2}, which gave the construction of simple modules for both Hecke-Clifford algebra and Sergeev algebra. The construction is inspired by \cite{Wa}. However, the construction in \cite{SW} didn't give an explicit basis for $\mathbb{D}(\undla)$. The following Theorem \ref{main1}, which is the first main result of this paper, constructs an explicit basis for $\mathbb{D}(\undla)$ and describes actions of standard generators on the certain basis, where we refer the reader to Section \ref{basic-Non-dege} and Section \ref{Non-dege-simplemodule} for unexplained notations used here.
	
		\begin{thm}\label{main1}
The simple module $\mathbb{D}(\undla)$ has a $\mathbb{K}$-basis of the form$$\bigsqcup_{\mt\in \Std(\undla)}\Biggl\{C^{\beta_{\mt}}C^{\alpha_{\mt}}v_{\mt}\biggm|\begin{matrix}\beta_{\mt} \in \Z_2([n]\setminus \mathcal{D}_{\mt})  \\
			\alpha_{\mt}\in \Z_2(\mathcal{OD}_{\mt})
		\end{matrix}\Biggr\}.$$ Moreover, the actions of generators of $\mHfcn$ on the above basis are given explicitly in \eqref{X eigenvalues}, \eqref{Caction} and \eqref{Taction Non-dege}.

	\end{thm}
	
Theorem \ref{main1} makes it possible to do computations in $\mHfcn$. For example, to show two elements $a,b\in\mHfcn$ are equal, it is enough to show that $a,b$ act in the same way on each basis element of $\mathbb{D}(\undla)$, for all $\undla\in\mathscr{P}^{\bullet,m}_{n}$.
As an application of Theorem \ref{main1}, we can construct a complete set of (super) primitive idempotents and matrix units for $\mHfcn$. Let ${\rm Tri}_{\bar{0}}(\mathscr{P}^{\bullet,m}_{n})$ be the set consisting of triple ${\rm T}=(\mt, \alpha_{\mt}, \beta_{\mt})$, where $\mt$ is a standard tableaux, $\alpha_{\mt}$ is a sequence consisting of the ``diagonal part'' of  $\mt$ and $\beta_{\mt}$ is a sequence consisting of the ``non-diagonal part'' of  $\mt$. Note that in contrast to Hecke algebra \cite{Ma1}, the index set of primitive idempotents ${\rm Tri}_{\bar{0}}(\mathscr{P}^{\bullet,m}_{n})$ becomes a bit more complicated. Roughly speaking, for any ${\rm T}=(\mt, \alpha_{\mt}, \beta_{\mt})\in {\rm Tri}_{\bar{0}}(\undla),$ the first component together with the third component of ${\rm T}\in {\rm Tri}_{\bar{0}}(\undla)$ are uniquely determined by the eigenvalues while the second component appears to record different primitive idempotents of Clifford algebra. Then for each ${\rm T}\in{\rm Tri}_{\bar{0}}(\mathscr{P}^{\bullet,m}_{n}),$ we can define an explicit element ${F_{\rm T}}\in\mHfcn$ under separate condition. The following Theorem \ref{main2} is the second main result of this paper.

	\begin{thm}\label{main2}
	Suppose $P^{\bullet}_{n}(q^2,\undQ)\neq 0$. We have the following.
	
	(a) $\{F_{\rm T} \mid {\rm T}\in {\rm Tri}_{\bar{0}}(\mathscr{P}^{\bullet,m}_{n})\}$ is a complete set of (super) primitive orthogonal idempotents of $\mHfcn.$
	
	(b) For ${\rm T}=(\mt, \alpha_{\mt}, \beta_{\mt})\in {\rm Tri}_{\bar{0}}(\undla),\,
	{\rm S}=(\ms, \alpha_{\ms}', \beta_{\ms}')\in {\rm Tri}_{\bar{0}}(\underline{\mu}),$ then $\mathbb{D}_{\rm T}:=\mHfcn F_{\rm T}$ and $\mathbb{D}_{\rm S}:=\mHfcn F_{\rm S}$ belong to the same block if and only if $\undla=\underline{\mu}$.
	\end{thm}
	
We can also give a complete set of (super) primitive central idempotents of $\mHfcn$ and decide when two $\mHfcn$-supermodules $\mathbb{D}_{\rm T},\,\mathbb{D}_{\rm S}$ are evenly isomorphic, see Theorem \ref{primitive idempotents}. As a further application of Theorem \ref{main1} and Theorem \ref{main2}, we can construct seminormal bases for the entire algebra $\mHfcn$. Using cellular bases theory, Mathas \cite[\S 2,\S 3]{Ma2} constructed seminormal and dual seminormal bases for cyclotomic Hecke algebras and in \cite[\S 4]{Ma2}, he also gave another realization of his seminormal and dual seminormal bases by intertwining elements $\Phi_{\mt}$. Our construction is inspired by \cite[\S 4]{Ma2}. Actually, for each $\mathfrak{w}\in\Std(\undla)$, we can associate a set inside the block $B_{\undla}$, \begin{align}\label{Semi1}
	\left\{ f_{{\rm S},{\rm T}}^\mathfrak{w} \Biggm|
	{\rm S}=(\ms, \alpha_{\ms}', \beta_{\ms}')\in {\rm Tri}_{\bar{0}}(\undla),
	{\rm T}=(\mt, \alpha_{\mt}, \beta_{\mt})\in {\rm Tri}(\undla)
	\right\},
\end{align} where the elements ``factor through'' a common $\mathfrak{w}$ via intertwining elements, see \eqref{fst. typeM. nondege.}, \eqref{fst. typeQ. nondege.}. We can also define the set \begin{align}\label{Semi2}
\left\{ f_{{\rm S},{\rm T}} \Biggm|
{\rm S}=(\ms, \alpha_{\ms}', \beta_{\ms}')\in {\rm Tri}_{\bar{0}}(\undla),
{\rm T}=(\mt, \alpha_{\mt}, \beta_{\mt})\in {\rm Tri}(\undla)
\right\}
\end{align} using the reduced expression of $d(\ms,\mt)\in \mathfrak{S}_n$.
The following Theorem \ref{main3} is the third main result of this paper.
	\begin{thm}\label{main3}
	Suppose  $P^{\bullet}_{n}(q^2,\undQ)\neq 0$. We fix $\mathfrak{w}\in\Std(\undla)$. Then the above two sets \eqref{Semi1} and \eqref{Semi2} form two $\mathbb{K}$-bases of the block $B_{\undla}$ of $\mHfcn$.
	
	Moreover, for ${\rm S}=(\ms, \alpha_{\ms}', \beta_{\ms}')\in {\rm Tri}_{\bar{0}}(\undla),
	{\rm T}=(\mt, \alpha_{\mt}, \beta_{\mt})\in {\rm Tri}(\undla),$ we have \begin{equation*}
		f_{{\rm S},{\rm T}}
		=\frac{\mathtt{c}_{\ms,\mt}}{\mathtt{c}_{\ms,\mathfrak{w} }\mathtt{c}_{\mathfrak{w},\mt }} f_{{\rm S},{\rm T}}^\mathfrak{w}\in F_{\rm S}\mHfcn F_{\rm T}.
	\end{equation*} The multiplications of basis elements in \eqref{Semi1} are given as follows.
	
	(1) Supppose $d_{\undla}=0.$  Then for any
	${\rm S}=(\ms, \alpha_{\ms}', \beta_{\ms}'),
	{\rm T}=(\mt, \alpha_{\mt}, \beta_{\mt}),
	{\rm U}=(\mfku,\alpha_{\mfku}^{''},\beta_{\mfku}^{''}),
	{\rm V}=(\mfkv,\alpha_{\mfkv}^{'''},\beta_{\mfkv}^{'''})\in {\rm Tri}(\undla),$ we have
	\begin{align*}
		f_{{\rm S},{\rm T}}^\mathfrak{w} f_{{\rm U},{\rm V}}^\mathfrak{w}
		=\delta_{{\rm T},{\rm U}} \mathtt{c}_{\rm T}^\mathfrak{w} f_{{\rm S},{\rm V}}^\mathfrak{w}.
	\end{align*}

	(2) Suppose $d_{\undla}=1.$ Then for any $a,b\in \mathbb{Z}_2$ and
	\begin{align*}
		{\rm S}&=(\ms, \alpha_{\ms}', \beta_{\ms}')\in {\rm Tri}_{\bar{0}}(\undla), \quad
		{\rm T}_{a}=(\mt, \alpha_{\mt,a}, \beta_{\mt})\in {\rm Tri}_{a}(\undla),\nonumber\\
		{\rm U}&=(\mfku,\alpha_{\mfku}^{''},\beta_{\mfku}^{''})\in {\rm Tri}_{\bar{0}}(\undla), \quad
		{\rm V}_{b}=(\mfkv,{\alpha_{\mfkv,b}^{'''}},\beta_{\mfkv}^{'''})\in {\rm Tri}_{b}(\undla),\nonumber
	\end{align*} we have
	\begin{align*}
		f_{{\rm S},{\rm T}_{a}}^\mathfrak{w} f_{{\rm U},{\rm V}_{b}}^\mathfrak{w}
		=\delta_{{\rm T}_{\bar{0}},{\rm U}}(-1)^{\left(|\alpha_{\mt}|_{>d(\mt,\mt^{\undla})(i_t)}\right)}\mathtt{c}_{\rm T}^\mathfrak{w} f_{{\rm S},{\rm V}_{a+b}}^\mathfrak{w}.
	\end{align*}
	
\end{thm}

Both \eqref{Semi1} and \eqref{Semi2} are called {\bf Seminormal Bases} for block $B_{\undla}$. For $\mathfrak{w}=\mt^{\undla},\mt_{\undla}$, i.e.,  the maximal and minimal tableaux, two constructions in \eqref{Semi1} are analogues to ``seminormal and dual seminarmal'' bases in \cite{EM, Ma2,HM2}.  For generic cyclotomic Hecke-Clifford algebra $\mHfcn$, we always have $P^{\bullet}_{n}(q^2,\undQ)\neq 0$. Hence Theorem \ref{main2} and \ref{main3} give a complete set of (super) primitive idempotents and seminormal basis for generic cyclotomic Hecke-Clifford algebra $\mHfcn$. In particular, if we take the cyclotomic polynomial $f=x-1$, then $\mHfcn$ becomes Hecke-Clifford algebra $\mathcal{H}(n)$. Therefore, our main results give a complete set of (super) primitive idempotents and seminormal basis for the Hecke-Clifford algebra $\mathcal{H}(n)$. Furthermore, if we specialize $q=1$, this gives rise to the corresponding results for Sergeev algebra $\mathcal{C}_n \rtimes \mathbb{C} \mathfrak{S}_n$, which is weakly Morita equivalent to the spin group algebra $\mathbb{C} \mathfrak{S}_n^-$. To the best of our knowledge, these results are completely new.

We also compute the action of anti-involution $*$ in \cite{BK} on seminormal bases. Using Theorem \ref{main1}, actions of generators of $\mHfcn$ on seminormal bases are obtained, see Subsection \ref{action}. Using our seminormal bases theory, we study some subalgebras of $\mHfcn$ in Subsection \ref{some subalgebras}, which seems to be some analogues of the Gelfand-Zetalin subalgebra of cyclotomic Hecke algebra. The above story for degenerate case is similar and we omit all of the details.

Note that the main aim of \cite{EM, HM2} is to give (homogeneous) integral bases, which happen to be cellular. We shall follow a similar idea to develop the ``cellular basis'' theory for cyclotomic Hecke-Clifford algebra in future work.

Here is the layout of this paper. In Section \ref{preli}, we first recall some basics on general superalgebras and representation theory of Clifford algebra. We also recall the notion of Affine Hecke-Clifford algebra $\mHcn$ (Affine Sergeev algebra $\mhcn$), cyclotomic Hecke-Clifford algebra $\mHfcn$ (cyclotomic Sergeev algebra $\mhgcn$) as well as the associated combinatorics and the Separate Conditions in Subsection \ref{basic-Non-dege} (\ref{basic-dege}).
In Section \ref{simplemodules}, we give an explicit basis for each simple module of $\mathcal{A}_n$ and $\mathcal{P}_n$ and write down the actions of generators on the basis explicitly.
In Section \ref{Nondeg}, we use the results in Section \ref{simplemodules} to obtain our first main result Theorem \ref{main1}. Then we use Theorem \ref{main1} to give a complete set of primitive idempotents and matrix units for cyclotomic Hecke-Clifford algebra $\mHfcn$ under separate condition, i.e., the second main result Theorem \ref{main2} in Subsection \ref{primitiveidem}. In Subsection \ref{Seminormalbase}, we introduce the key element $\Phi_{\ms,\mt}$ and discuss some properties of $\Phi_{\ms,\mt}$. We then turn to the construction of \eqref{Semi1} and \eqref{Semi2}. We prove our third main result Theorem \ref{main3} in Theorem \ref{seminormal basis}. We also compute the actions of anti-involution $*$ and generating elements of $\mHfcn$ on seminormal bases in Subsection \ref{action}. Finally, we study some subalgebras of $\mHfcn$ in Subsection \ref{some subalgebras} using our main result Theorem \ref{main3}. In Section \ref{dege}, we develop a parallel story for cyclotomic Sergeev algebra $\mhgcn$.

After we uploaded the earlier version of this paper to arXiv,  the preprint \cite{KMS} appears almost simultaneously which contains some similar constructions for Sergeev algebra $\mathcal{C}_n \rtimes \mathbb{C} \mathfrak{S}_n$, such as the primitive idempotents and seminormal basis. We would like to point out that our construction of primitive idempotents for the Sergeev algebra aligns with the version presented in \cite{KMS} (see Remark \ref{KMS idempotent} for more details). Additionally, the seminormal basis of $\mathcal{C}_n \rtimes \mathbb{C} \mathfrak{S}_n$ in \cite{KMS} is over some Clifford algebra.

\bigskip
\centerline{\bf Acknowledgements}
\bigskip
The research is supported by the National Natural Science Foundation of China
(No. 12431002) and the Natural Science Foundation of Beijing Municipality (No. 1232017).  The second author is partially supported by the Postdoctoral Fellowship Program of CPSF under Grant Number GZB20250717. Both authors thank Jun Hu and Andrew Mathas for their helpful suggestions and feedback. The authors would like to express their gratitude to the referee for the thorough report, which greatly enhanced the quality of this paper.
\bigskip

	\section{Preliminary}\label{preli}

	\subsection{Some basics about superalgebras}
	We shall recall some basic notions of superalgebras, referring the
	reader to~\cite[\S 2-b]{BK}. Let us denote by
	$|v|\in\mathbb{Z}_2$ \label{pag:||} the parity of a homogeneous vector $v$ of a
	vector superspace. By a superalgebra, we mean a
	$\mathbb{Z}_2$-graded associative algebra. Let $\mathcal{A}$ be a
	superalgebra. An $\mathcal{A}$-module means a $\mathbb{Z}_2$-graded
	left $\mathcal{A}$-module.	A homomorphism $f:V\rightarrow W$ of
	$\mathcal{A}$-modules $V$ and $W$ means a linear map such that $
	f(av)=(-1)^{|f||a|}af(v).$  Note that this and other such
	expressions only make sense for homogeneous $a, f$ and the meaning
	for arbitrary elements is to be obtained by extending linearly from
	the homogeneous case. An non-zero element $e\in\mathcal{A}$ is called a super primitive idempotent if
	it is an idempotent with $|e|=\bar{0}$ and it cannot be decomposed as the sum of
	two nonzero orthogonal idempotents with parity $\bar{0}.$ Let $V$ be a finite dimensional
	$\mathcal{A}$-module. Let $\Pi
	V$ \label{pag:parity shift} be the same underlying vector space but with the opposite
	$\mathbb{Z}_2$-grading. The new action of $a\in\mathcal{A}$ on $v\in\Pi
	V$ is defined in terms of the old action by $a\cdot
	v:=(-1)^{|a|}av$. Note that the identity map on $V$ defines
	an isomorphism from $V$ to $\Pi V$.
	
	A superalgebra analog of Schur's Lemma states that the endomorphism
	algebra of a finite dimensional irreducible module over a
	superalgebra is either one dimensional or two dimensional. In the
	former case, we call the module of {\em type }\texttt{M} while in
	the latter case the module is called of {\em type }\texttt{Q}.

	\begin{example}\label{simple algebra}
		1). Let $V$ be a superspace with superdimension $(m,n)$ over field $F$, then $\mathcal{M}_{m,n}:={\text{End}}_{F}(V)$ is a simple superalgebra with simple module $V$ of {\em type }\texttt{M}. Then the set of super primitive idempotents of $\mathcal{M}_{m,n}$ is
		$\{E_{ii} \mid i,j=1,\ldots,m+n\}.$ One can see that there is an evenly $\mathcal{M}_{m,n}$-supermodule isomorphism $V \cong \mathcal{M}_{m,n}E_{ii}$ if $i\in \{1,\ldots,m\},$ and there is an evenly $\mathcal{M}_{m,n}$-supermodule isomorphism $\Pi V \cong \mathcal{M}_{m,n}E_{ii}$ if $i\in \{m+1,\ldots,m+n\}.$\\
		2). Let $V$ be a superspace with superdimension $(n,n)$ over field $F$. We define $\mathcal{Q}_n:=\Biggl\{\biggl(\begin{matrix} &A & B\\
			&-B& A
		\end{matrix}\biggr) \biggm| A,B\in  M_n\Biggr\}\subset \mathcal{M}_{n,n}$. Then the set of super primitive idempotents of $\mathcal{Q}_n$ is
		$\Biggl\{\biggl(\begin{matrix} &E_{ii} & 0\\
			&0& E_{ii}
		\end{matrix}\biggr) \biggm| i\in  \{1,\ldots,n\} \Biggr\}$
		and there is an evenly $\mathcal{Q}_n$-supermodule isomorphism $V \cong \mathcal{Q}_n \biggl(\begin{matrix} &E_{ii} & 0\\
			&0& E_{ii}
		\end{matrix}\biggr)$ for each $i=1,\ldots,n.$
		
	\end{example}

	Given two superalgebras $\mathcal{A}$ and $\mathcal{B}$, we view
	the tensor product of superspaces $\mathcal{A}\otimes\mathcal{B}$
	as a superalgebra with multiplication defined by
	$$
	(a\otimes b)(a'\otimes b')=(-1)^{|b||a'|}(aa')\otimes (bb')
	\qquad (a,a'\in\mathcal{A}, b,b'\in\mathcal{B}).
	$$
	Suppose $V$ is an $\mathcal{A}$-module and $W$ is a
	$\mathcal{B}$-module. Then $V\otimes W$ affords $A\otimes B$-module
	denoted by $V\boxtimes W$ via
	$$
	(a\otimes b)(v\otimes w)=(-1)^{|b||v|}av\otimes bw,~a\in A,
	b\in B, v\in V, w\in W.
	$$
	
	If $V$ is an irreducible $\mathcal{A}$-module and $W$ is an
	irreducible $\mathcal{B}$-module, $V\boxtimes W$ may not be
	irreducible. Indeed, we have the following standard lemma (cf.
	\cite[Lemma 12.2.13]{K1}).
	\begin{lem}\label{tensorsmod}
		Let $V$ be an irreducible $\mathcal{A}$-module and $W$ be an
		irreducible $\mathcal{B}$-module.
		\begin{enumerate}
			\item If both $V$ and $W$ are of type $\texttt{M}$, then
			$V\boxtimes W$ is an irreducible
			$\mathcal{A}\otimes\mathcal{B}$-module of type $\texttt{M}$.
			
			\item If one of $V$ or $W$ is of type $\texttt{M}$ and the other
			is of type $\texttt{Q}$, then $V\boxtimes W$ is an irreducible
			$\mathcal{A}\otimes\mathcal{B}$-module of type $\texttt{Q}$.
			
			\item If both $V$ and $W$ are of type $\texttt{Q}$, then
			$V\boxtimes W\cong X\oplus \Pi X$ for a type $\texttt{M}$
			irreducible $\mathcal{A}\otimes\mathcal{B}$-module $X$.
		\end{enumerate}
		Moreover, all irreducible $\mathcal{A}\otimes\mathcal{B}$-modules
		arise as constituents of $V\boxtimes W$ for some choice of
		irreducibles $V,W$.
	\end{lem}
	
	If $V$ is an irreducible $\mathcal{A}$-module and $W$ is an
	irreducible $\mathcal{B}$-module, denote by $V\circledast W$ \label{pag:irrtensor} an
	irreducible component of $V\boxtimes W$. Thus,
	$$
	V\boxtimes W=\left\{
	\begin{array}{ll}
		V\circledast W\oplus \Pi (V\circledast W), & \text{ if both } V \text{ and } W
		\text{ are of type }\texttt{Q}, \\
		V\circledast W, &\text{ otherwise}.
	\end{array}
	\right.
	$$
	\subsection{Clifford algebra $\mathcal{C}_n$}
	In this subsection, we shall recall the representation theory of Clifford superalgebra $\mathcal{C}_n$ which will be used in later sections.  Let $\mathcal{C}_n$ \label{pag:Clifford algebra} be the Clifford superalgebra generated by odd generators $C_1,\ldots,C_n,$ subject to the following relations
	$$C_i^2=1,C_iC_j=-C_jC_i, \quad 1\leq i\neq j\leq n.$$

Recall that $\mathbb{K}$ is an algebraically closed field of characteristic different from $2$. For any $a\in \mathbb{K}$, we fix a solution of the equation $x^2=a$ and denote it by $\sqrt{a}$. Now it's easy to check the following.
	\begin{lem}\label{trivial cilfford}
		(1). $\mathcal{C}_1$ is a simple superalgebra with the unique simple (super)module $\mathcal{C}_1$ of type $\texttt{Q}$.
		
		(2).$\mathcal{C}_2$ is a simple superalgebra with the unique simple (super)module of type $\texttt{M}$. Moreover, $$
		\biggl\{\frac{1+\sqrt{-1}C_1C_2}{2},\frac{1-\sqrt{-1}C_1C_2}{2}\biggr\}$$ forms a complete set of orthogonal primitive idempotents. These two idempotents are conjugate via $C_1$. Let $\gamma\in
		\biggl\{\frac{1+\sqrt{-1}C_1C_2}{2},\frac{1-\sqrt{-1}C_1C_2}{2}\biggr\}$, then the simple $\mathcal{C}_2$-module $\mathcal{C}_2\gamma$ has basis $\biggl\{\gamma, C_1\gamma\biggr\}$.
	\end{lem}
	
	Note that \begin{equation}\label{tensor cilfford}\mathcal{C}_n\cong\begin{cases} \underbrace{\mathcal{C}_2\otimes\cdots\otimes \mathcal{C}_2}_{(n-1)/2  \text{ times}}\otimes\mathcal{C}_1, &{\text{if $n$ is odd;}}\\
			\underbrace{\mathcal{C}_2\otimes\cdots\otimes \mathcal{C}_2}_{n/2   \text{ times}}, &{\text{if $n$ is even.}}\\
		\end{cases}
	\end{equation}

 Let $A$ be any algebra and $a_1,a_2,\ldots,a_p\in A$, we define the ordered product \label{pag:ordered product} as
 $$\overrightarrow{\prod_{1 \leq i\leq p}}a_i:=a_1 a_2 \ldots a_p.$$ Then we obtain the following.
	\begin{lem}\label{lem:clifford rep}
		
		Let $n\in \N=\{1,2,\ldots\}$.
		
		(1).  $\mathcal{C}_n$ is a simple superalgebra with the unique simple (super)module of type $\texttt{Q}$ if $n$ is odd, of type $\texttt{M}$ if $n$ is even. Let
$$I_n:=\begin{cases}
	\{1\}, &\text{if $n=1$;}\\
	\Biggl\{2^{-\lfloor n/2 \rfloor}\cdot\overrightarrow{\prod_{k=1,\cdots,{\lfloor n/2 \rfloor}}}(1+(-1)^{a_k} \sqrt{-1}C_{2k-1}C_{2k})\Biggm|a_k\in\Z_2,\,1\leq k\leq {\lfloor n/2 \rfloor} \Biggr\}, &\text{if $n>1$,}
	\end{cases}$$
 where $\lfloor n/2 \rfloor$ \label{pag:round down} denotes the greatest integer less than or equal to $n/2.$
Then the set $I_n$ forms a complete set of super primitive idempotents for $\mathcal{C}_n.$
		
		(2). Let $\gamma_1,\gamma_2\in I_n$. There is a unique monomial of the form $C^{b_1}_1C^{b_3}_3\cdots C^{b_{2{\lfloor n/2 \rfloor}-1}}_{2{\lfloor n/2 \rfloor}-1}$, where $b_1,b_3,\cdots, b_{2{\lfloor n/2 \rfloor}-1}\in \Z_2$ such that $\gamma_1,\gamma_2$ are conjugate via $C^{b_1}_1C^{b_3}_3\cdots C^{b_{2{\lfloor n/2 \rfloor}-1}}_{2{\lfloor n/2 \rfloor}-1}$. Moreover, if $n$ is odd, there is also a unique monomial of the form $C^{b_1}_1C^{b_3}_3\cdots C^{b_{2{\lfloor n/2 \rfloor}-1}}_{2{\lfloor n/2 \rfloor}-1}C_{2{\lceil n/2 \rceil}-1}$, where $b_1,b_3,\cdots, b_{2{\lfloor n/2 \rfloor}-1}\in \Z_2$ such that $\gamma_1,\gamma_2$ are conjugate via $C^{b_1}_1C^{b_3}_3\cdots C^{b_{2{\lfloor n/2 \rfloor}-1}}_{2{\lfloor n/2 \rfloor}-1}C_{2{\lceil n/2 \rceil}-1}$.
		
		(3). Let $\gamma\in I_n$, then the simple $\mathcal{C}_n$-supermodule $\mathcal{C}_n\gamma$ has a basis $$\Biggl\{C^{b_1}_1C^{b_3}_3\cdots C^{b_{2{\lceil n/2 \rceil}-1}}_{2{\lceil n/2 \rceil}-1}\gamma\Biggm|b_1,b_3,\cdots, b_{2{\lceil n/2 \rceil}-1}\in \Z_2\Biggr\},$$
where $\lceil n/2 \rceil$ \label{pag:round up} denotes the smallest integer greater than or equal to $n/2$.
	\end{lem}
	
	\begin{proof}
		(1). The first statement follows from Lemma \ref{tensorsmod}, \eqref{tensor cilfford} and Lemma \ref{trivial cilfford}. It is straightforward to check that $I_n$ is a set of orthogonal primitive idempotents and the sum of elements in $I_n$ is equal to the identity. On the other hand, $\mathcal{C}_n$ can be decomposed into $\lfloor\frac{n}{2}\rfloor$ simple $\mathcal{C}_n$-modules by Lemma \ref{tensorsmod}, \eqref{tensor cilfford} and Lemma \ref{trivial cilfford} again. This implies that $I_n$ forms a complete set of orthogonal primitive idempotents.
		
		(2) and (3) follows from \eqref{tensor cilfford} and Lemma \ref{trivial cilfford}.
	\end{proof}
	\subsection{Non-degenerate case}\label{basic-Non-dege}
	\subsubsection{Affine Hecke-Clifford algebra $\mHcn$}
	Let $q$ be an invertible element in the field $\mathbb{K}$ such that $q^2\neq \pm 1$. We set $\epsilon=q-q^{-1}.$ Then $\epsilon \in \mathbb{K}^*.$ The non-degenerate affine Hecke-Clifford algebra $\mHcn$ \label{pag:AHCA} is
	the superalgebra over $\mathbb{K}$ generated by even generators
	$T_1,\ldots,T_{n-1},X_1^{\pm 1},\ldots,X_n^{\pm 1}$ and odd generators
	$C_1,\ldots,C_n$ subject to the following relations
	\begin{align}
		T_i^2=\epsilon T_i +1,\quad T_iT_j =T_jT_i, &\quad
		T_iT_{i+1}T_i=T_{i+1}T_iT_{i+1}, \quad|i-j|>1,\label{Braid}\\
		X_iX_j&=X_jX_i, X_iX^{-1}_i=X^{-1}_iX_i=1 \quad 1\leq i,j\leq n, \label{Poly}\\
		C_i^2=1,C_iC_j&=-C_jC_i, \quad 1\leq i\neq j\leq n, \label{Clifford}\\
		T_iX_i&=X_{i+1}T_i-\epsilon(X_{i+1}+C_iC_{i+1}X_i),\label{PX1}\\
		T_iX_{i+1}&=X_iT_i+\epsilon(1+C_iC_{i+1})X_{i+1},\label{PX2}\\
		T_iX_j&=X_jT_i, \quad j\neq i, i+1, \label{PX3}\\
		T_iC_i=C_{i+1}T_i, T_iC_{i+1}&=C_iT_i-\epsilon(C_i-C_{i+1}),T_iC_j=C_jT_i,\quad j\neq i, i+1, \label{PC}\\
		X_iC_i=C_iX^{-1}_i, X_iC_j&=C_jX_i,\quad 1\leq i\neq j\leq n.
		\label{XC}
	\end{align}
	
	For $\alpha=(\alpha_1,\ldots,\alpha_n)\in\mathbb{Z}^n$ and
	$\beta=(\beta_1,\ldots,\beta_n)\in\mathbb{Z}_2^n$, we set
	$X^{\alpha}=X_1^{\alpha_1}\cdots X_n^{\alpha},$
	$C^{\beta}=C_1^{\beta_1}\cdots C_n^{\beta_n}$ and define
	$\supp(\beta):=\{1 \leq k \leq n:\beta_{k}=\bar{1}\},$ $|\beta|:=\Sigma_{i=1}^{n}\beta_i \in \mathbb{Z}_2.$ \label{pag:suppot and sum}
	Then we have the
	following.
	\begin{lem}\cite[Theorem 2.2]{BK}\label{lem:PBWNon-dege}
		The set $\{X^{\alpha}C^{\beta}T_w~|~ \alpha\in\mathbb{Z}^n,
		\beta\in\mathbb{Z}_2^n, w\in \mathfrak{S}_n\}$ forms a basis of $\mHcn$.
	\end{lem}
	
	Let $\mathcal{A}_n$ \label{pag:subalg An} be the subalgebra generated by even generators $X_1^{\pm 1},\ldots,X_n^{\pm 1}$ and odd generators $C_1,\ldots,C_n$. By Lemma~\ref{lem:PBWNon-dege}, $\mathcal{A}_n$ actually can be identified with the superalgebra generated by even generators $X^{\pm 1}_1,\ldots,X^{\pm 1}_n$ and odd generators $C_1,\ldots,C_n$ subject to relations \eqref{Poly}, \eqref{Clifford}, \eqref{XC}. Clifford algebra $\mathcal{C}_n$ can be identified with the subalgebra of $\mathcal{A}_n$ generated by  odd generators $C_1,\ldots,C_n$ subject to relations \eqref{Clifford}.

	\subsubsection{Intertwining elements for $\mHcn$}
\label{pag:example}
	Given $1\leq i<n$, one can define the intertwining element $\tilde{\Phi}_i$ \label{pag:BK intertwining element} in $\mHcn$ as follows: $$ z_i:= X_i+X^{-1}_i-X_{i+1}-X^{-1}_{i+1}= X^{-1}_i(X_iX_{i+1}-1)(X_iX^{-1}_{i+1}-1),$$
	$$\label{intertwinNon-dege}\tilde{\Phi}_i:=z^2_i T_i+\epsilon\frac{z^2_i}{X_i X^{-1}_{i+1}-1}-\epsilon\frac{z^2_i}{X_i X_{i+1}-1}C_i C_{i+1}.$$ These elements satisfy the following properties (cf. \cite[(3.7),Proposition 3.1]{JN} and \cite[(4.11)-(4.15)]{BK})
	\begin{align}
		\tilde{\Phi}^2_i&=z^2_i\bigl(z^2_i-\epsilon^2 (X^{-1}_i X^{-1}_{i+1}(X_i X_{i+1}-1)^2-X^{-1}_iX_{i+1}(X_i X^{-1}_{i+1}-1)^2)\bigr)\label{Sqinter},\\
		\tilde{\Phi}_i X^\pm_i&=X^\pm_{i+1}\tilde{\Phi}_i, \tilde{\Phi}_iX^\pm_{i+1}=X^\pm_i\tilde{\Phi}_i,
		\tilde{\Phi}_i X^\pm_l=X^\pm_l\tilde{\Phi}_i \label{Xinter},\\
		\tilde{\Phi}_i C_i&=C_{i+1}\tilde{\Phi}_i, \tilde{\Phi}_i C_{i+1}=C_i \tilde{\Phi}_i,
		\tilde{\Phi}_iC_l=C_l\tilde{\Phi}_i \label{Cinter},\\
		\tilde{\Phi}_j \tilde{\Phi}_i&=\tilde{\Phi}_i \tilde{\Phi}_j,
		\tilde{\Phi}_i\tilde{\Phi}_{i+1}\tilde{\Phi}_i=\tilde{\Phi}_{i+1}\tilde{\Phi}_i \tilde{\Phi}_{i+1}\label{Braidinter}
	\end{align}
	for all admissible $i,j,l$ with $l\neq i, i+1$ and $|j-i|>1$.
	Define $\tilde{\Phi}_{w}:=\tilde{\Phi}_{i_1}\cdots \tilde{\Phi}_{i_p}$ for
	$w\in \mathfrak{S}_n$ and $w=s_{i_1}\cdots s_{i_p}$ is a reduced expression.
Jones and Nazarov (\cite[(3.6)]{JN}) also introduced the intertwining elements of the following version, \label{pag:JN intertwining element}
	\begin{align}\label{universal-Phi}
\Phi_i:= z_{i}^{-2} \tilde{\Phi}_i&=T_{i}+\frac{\epsilon}{X_{i}X_{i+1}^{-1}-1}-\frac{\epsilon}{X_{i}X_{i+1}-1}C_{i}C_{i+1}\\
&\in \mathbb{K}(X_1,\ldots,X_n)\otimes_{\mathbb{K}[X_1^{\pm 1},\ldots, X_n^{\pm 1}]} \mHcn, \nonumber
    \end{align}
	for $i=1,\ldots,n-1.$
	
	For any $i=1,2,\ldots,n-1$ and $x,y \in \mathbb{K}^*$ satisfying $y\neq x^{\pm 1},$ let (\cite[(3.13)]{JN})
    \label{pag:Phi function}
	\begin{align}\label{Phi function}
		\Phi_i(x,y):=T_{i}+\frac{\epsilon}{x^{-1}y-1}-\frac{\epsilon}{xy-1}C_{i}C_{i+1} \in \mHcn.
	\end{align}
	
	\begin{lem}(\cite[Lemma 4.1]{JN}) These functions satisfy the following properties
		\begin{align}
			\Phi_i(x,y)\Phi_i(y,x)=1-\epsilon^2 \left( \frac{x^{-1}y}{(x^{-1}y-1)^2} +  \frac{xy}{(xy-1)^2} \right), \label{square1}\\
			\Phi_i(x,y)\Phi_j(z,w)=\Phi_j(z,w)\Phi_i(x,y), \quad \text{ if } |j-i|>1, \label{braidrel1} \\
			\Phi_i(x,y)\Phi_{i+1}(z,y)\Phi_i(z,x)=\Phi_{i+1}(z,x)\Phi_i(z,y)\Phi_{i+1}(x,y) \label{braidrel2}
		\end{align}
		for all possible  $i,j$ and $x,y,z,w.$ Furthermore, we also have the following
		\begin{align}
			\Phi_i(x,y)^2=-\epsilon \frac{x+y}{x-y} \Phi_i(x,y) + 1 &-\epsilon^2 \left( \frac{x^{-1}y}{(x^{-1}y-1)^2} +  \frac{xy}{(xy-1)^2} \right),\label{square2}\\
			C_i\Phi_i(x,y)=\Phi_i(x,y^{-1})C_{i+1},\quad C_{i+1}\Phi_i(x,y)&=\Phi_i(x^{-1},y)C_i, \quad C_j\Phi_i(x,y)=\Phi_i(x,y)C_j \label{Phi and C}
		\end{align}
		for $ j\neq i,i+1.$
	\end{lem}
	
	For any pair of $(x,y)\in (\mathbb{K}^*)^2$ and $y\neq x^{\pm1}$, we consider the following idempotency condition on $(x,y)$
	\begin{align}\label{invertible}
		\frac{x^{-1}y}{(x^{-1}y-1)^2}+\frac{xy}{(xy-1)^2}=\frac{1}{\epsilon^2}.
	\end{align}
	Note that the equation \eqref{invertible} holds for the pair $(x,y)$ if and only if it holds for one of these four pairs $(x^{\pm 1},y^{\pm 1}).$
\begin{cor}\label{Phi}(\cite[Corollary 4.2]{JN})
(a) Suppose that the pair $(x,y)$ satisfies (\ref{invertible}) and $y\neq x^{\pm1}.$ Then
$$\Phi_i(x,y)^2=-\epsilon \frac{x+y}{x-y} \Phi_i(x,y).$$
(b) Suppose that the pair $(x,y)$ does not satisfy (\ref{invertible}) and $y\neq x^{\pm1}.$ Then $\Phi_i(x,y)$ is invertible and
$$\Phi_i(x,y)^{-1}
=\left( 1-\epsilon^2 \left( \frac{x^{-1}y}{(x^{-1}y-1)^2}
       + \frac{xy}{(xy-1)^2} \right) \right)^{-1} \left( \Phi_i(x,y)+\epsilon \frac{x+y}{x-y}\right).$$
\end{cor}

For any invertible $\iota \in \mathbb{K}$, we define
\label{pag:q-function and b-function}
\begin{align}\label{substitution0}
	\mathtt{q}(\iota):=2\frac{q\iota+(q\iota)^{-1}}{q+q^{-1}}, \quad \mathtt{b}_{\pm}(\iota):=\frac{\mathtt{q}(\iota)}{2}\pm \sqrt{\frac{\mathtt{q}(\iota)^2}{4}-1}.
\end{align}
Clearly, $\mathtt{b}_{\pm}(\iota)$ is the solution of equation $x+x^{-1}=\mathtt{q}(\iota),$ thus $\mathtt{b}_{+}(\iota)\mathtt{b}_{-}(\iota)=1.$

According to \cite{JN}, via the substitution
\begin{align}\label{substitute}
x+x^{-1}=2\frac{qu+q^{-1}u^{-1}}{q+q^{-1}}=\mathtt{q}(u),\qquad\qquad y+y^{-1}=2\frac{qv+q^{-1}v^{-1}}{q+q^{-1}}=\mathtt{q}(v)
\end{align}
the condition \eqref{invertible}  is  equivalent to the condition which states that $u,v$ satisfy one of the following four equations
\begin{align}\label{invertible2}
v=q^2u,\quad v=q^{-2}u,\quad v=u^{-1},\quad v=q^{-4}u^{-1}.
\end{align}

	\subsubsection{Cyclotomic Hecke-Clifford algebra $\mHfcn$}
	
	To define the cyclotomic Hecke-Clifford algebra $\mHfcn$, \label{pag:CHCA} we fix $m\geq 0$ and $\underline{Q}=(Q_1,Q_2,\ldots,Q_m)\in(\mathbb{K}^*)^m$ \label{pag:Q-parameters} and take a $f=f(X_1)\in \mathbb{K}[X_1^\pm]$ satisfying \cite[(3.2)]{BK}. Since we are working over algebraically closed field $\mathbb{K}$,  it is noted in \cite{SW} that we only need to consider $f(X_1)\in \mathbb{K}[X_1^\pm]$ to be one of the following three forms:
 $$\begin{aligned}
 f=\begin{cases}
     f^{\mathsf{(0)}}_{\underline{Q}}=\prod_{i=1}^m \biggl(X_1+X^{-1}_1-\mathtt{q}(Q_i)\biggr), \\ 
      f^{\mathsf{(s)}}_{\underline{Q}}=(X_1-1)\prod_{i=1}^m \biggl(X_1+X^{-1}_1-\mathtt{q}(Q_i)\biggr), \\  
     f^{\mathsf{(ss)}}_{\underline{Q}} = (X_1-1)(X_1+1)\prod_{i=1}^m \biggl(X_1+X^{-1}_1-\mathtt{q}(Q_i)\biggr).
    \end{cases}
    \end{aligned}$$
In each case, the degree $r$ of the polynomial $f$ is $2m,\,2m+1,\,2m+2$ respectively.
	
	The non-degenerate cyclotomic Hecke-Clifford algebra $\mHfcn$ is defined as $$\mHfcn:=\mHcn/\mathcal{I}_f,
	$$ where $\mathcal{I}_f$ is the two sided ideal of $\mHcn$ generated by $f(X_1)$. The degree $r$ \label{pag:nondege level} of $f$ is called the level of $\mHfcn.$ We shall denote the images of $X^{\alpha}, C^{\beta}, T_w$ in the cyclotomic quotient $\mHfcn$ still by the same symbols. Then we have the following due to \cite{BK}.
	
	\begin{lem}\cite[Theorem 3.6]{BK}
		The set $\{X^{\alpha}C^{\beta}T_w~|~ \alpha\in\{0,1,\cdots,r-1\}^n,
		\beta\in\mathbb{Z}_2^n, w\in {\mathfrak{S}_n}\}$ forms a basis of $\mHfcn$.
	\end{lem}
	
		By \cite[(2.32)]{BK}, the algebra $\mHfcn$ admits an anti-involution $*$ \label{pag:nondege anti-involution} defined by
	$$T_i^*:=T_i+\epsilon C_iC_{i+1},\quad C_j^*:=C_j,\quad X_j^*:=X_j$$
	for all $i=1,\ldots, n-1,j=1,\ldots, n.$
	
	\subsubsection{Combinatorics}
    \label{pag:The types of combinatorics}
	The different choices of $f\in \{f^{\mathsf{(0)}}_{\underline{Q}},\,f^{\mathsf{(s)}}_{\underline{Q}},\,f^{\mathsf{(ss)}}_{\underline{Q}}\}$ corresponds to different combinatorics $\mathscr{P}^{\mathsf{0},m}_{n},\mathscr{P}^{\mathsf{s},m}_{n},\mathscr{P}^{\mathsf{ss},m}_{n}$ respectively in the representation theory of $\mHfcn$. Let's recall these combinatorics. For $n\in \N$, let $\mathscr{P}_n$ be the set of partitions of $n$ and denote by $\ell(\mu)$ the number of nonzero parts in the partition $\mu$ for each $\mu\in\mathscr{P}_n$. Let $\mathscr{P}^m_n$ be the set of all $m$-multipartitions of $n$ for $m\geq 0$, where we use convention that $\mathscr{P}^0_n=\{\emptyset\}$. Let $\mathscr{P}^\mathsf{s}_n$ be the set of strict partitions of $n$. Then for $m\geq 0$, set
	$$
	\mathscr{P}^{\mathsf{s},m}_{n}:=
	\cup_{a=0}^{n}( \mathscr{P}^{\mathsf{s}}_a\times \mathscr{P}^{m}_{n-a}),\qquad \mathscr{P}^{\mathsf{ss}, m}_{n}:=
	\cup_{a+b+c=n}(\mathscr{P}^{\mathsf{s}}_a \times \mathscr{P}^{\mathsf{s}}_b\times \mathscr{P}^{m}_{c}).$$
We will formally write  $\mathscr{P}^{\mathsf{0},m}_{n}=\mathscr{P}^m_n$.  In convention, for any \label{pag:multipartition} $\undla\in  \mathscr{P}^{\mathsf{0},m}_{n}$, we write  $\undla=(\lambda^{(1)},\cdots,\lambda^{(m)}),$ while for any $\undla\in  \mathscr{P}^{\mathsf{s},m}_{n}$, we write  $\undla=(\lambda^{(0)},\lambda^{(1)},\cdots,\lambda^{(m)})$, i.e. we shall put the strict partition in the $0$-th component. Moreover, for any $\undla\in  \mathscr{P}^{\mathsf{ss},m}_{n}$, we write  $\undla=(\lambda^{(0_-)},\lambda^{(0_+)},\lambda^{(1)},\cdots,\lambda^{(m)})$, i.e. we shall put two strict partitions in the $0_-$-th component and the $0_+$-th component.
	
	We will also identify the (strict) partition with the corresponding (shifted) young diagram.  For any  $\undla\in\mathscr{P}^{\bullet,m}_{n}$ with $\bullet\in\{\mathsf{0},\mathsf{s},\mathsf{ss}\}$ and $m\in \N$, the box in the $l$-th component with row $i$, column $j$ will be denoted by $(i,j,l)$  with $l\in\{1,2,\ldots,m\},$ or $l\in\{0,1,2,\ldots,m\}$ or $l\in\{0_-,0_+,1,2,\ldots,m\}$ in the case $\bullet=\mathsf{0},\mathsf{s},\mathsf{ss},$ respectively. We also use the notation $\alpha=(i,j,l)\in \undla$ if the diagram of $\undla$ has a box $\alpha$ on the $l$-th component of row $i$ and column $j$. We use $\Std(\undla)$ \label{pag:standard tableaux} to denote the set of standard tableaux of shape $\undla$. One can also regard each $\mathfrak{t}\in\Std(\undla)$ as a bijection $\mathfrak{t}:\undla\rightarrow \{1,2,\ldots, n\}$ satisfying $\mathfrak{t}((i,j,l))=k$ if the box occupied by $k$ is located in the $i$th row, $j$th column in the $l$-th component $\lambda^{(l)}$. We use $\mathfrak{t}^{\undla}$ (resp. $\mathfrak{t}_{\undla}$) to denote the standard tableaux obtained by inserting the symbols $1,2,\ldots,n$ consecutively by rows (resp. column) from the first (resp. last) component of $\undla$.
	\label{pag:diag of undlam}
	\begin{defn}(\cite[Definition 2.5]{SW})
		Let $\undla\in\mathscr{P}^{\bullet,m}_{n}$ with $\bullet\in\{\mathsf{0},\mathsf{s},\mathsf{ss}\}$.  We define
        $$
		\mathcal{D}_{\undla}:=\begin{cases} \emptyset,&\text{if $\undla\in\mathscr{P}^{\mathsf{0},m}_n$,}\\
			\{(a,a,0)|(a,a,0)\in \undla,\,a\in\N\}, &\text{if $\undla\in\mathscr{P}^{\mathsf{s},m}_{n}$,}\\
			\big\{(a,a,l)|(a,a,l)\in \undla,\,a\in\N, l\in\{0_-,0_+\}\big\}, &\text{if $\undla\in\mathscr{P}^{\mathsf{ss}, m}_{n}.$}
		\end{cases}
		$$
        For any $\mathfrak{t}\in\Std(\undla), $ we define
        \begin{align}\label{Dt}
        \mathcal{D}_{\mt}:=\{\mathfrak{t}((a,a,l))|(a,a,l)\in\mathcal{D}_{\undla}\}.
		\end{align}
	\end{defn}

	\begin{example} Let $\undla=(\lambda^{(0)},\lambda^{(1)})\in \mathscr{P}^{\mathsf{s},1}_{5}$, where via the identification with strict Young diagrams and Young diagrams:
		$$
		\lambda^{(0)}=\young(\,\, ,:\,),\qquad \lambda^{(1)}=\young(\,,\,).
		$$
		Then
		$$
		\mathfrak{t}^{\undla}=\Biggl(\young(12,:3),\quad \young(4,5)\Biggr).
		$$
		and  an example of standard tableau is as follows:
		$$
		\mathfrak{t}=\Biggl(\young(13,:5),\quad \young(2,4)\Biggr)\in \Std(\undla).
		$$ We have $$\mathcal{D}_{\undla}=\{(1,1,0),(2,2,0)\},\qquad \mathcal{D}_{\mathfrak{t}}=\{1,5\}.
		$$
	\end{example}
	
	Let $\mathfrak{S}_n$ be the symmetric group on ${1,2,\ldots,n}$ with basic transpositions $s_1,s_2,\ldots, s_{n-1}$.  And $\mathfrak{S}_n$ acts on the set of tableaux of shape $\underline{\lambda}$ in the natural way.
    \label{pag:adimssible}
	\begin{defn}(\cite[Definition 2.7]{SW})
		Let  $\undla\in\mathscr{P}^{\bullet,m}_{n}$ with $\bullet\in\{\mathsf{0},\mathsf{s},\mathsf{ss}\}$.  For any standard tableaux $\mathfrak{t}\in \Std(\undla)$, if $s_l \mathfrak{t}$ is still standard, we call the simple transposition $s_l$ is admissible with respect to $\mathfrak{t}$. We set
$$P(\undla):=\biggl\{\tau=s_{k_t}\ldots s_{k_1}\in\mathfrak{S}_n\biggm|~\begin{matrix}&s_{k_l} \,\text{is admissible with respect to }\\
			& s_{k_{l-1}}\ldots s_{k_1}\mathfrak{t}^{\undla}, \,\text{for $l=1,2,\cdots,t$}
		\end{matrix}\biggr\}.
		$$
	\end{defn}
	The following standard facts will be used in the sequel.
	
	
	\begin{lem}\label{admissible transposes}(\cite[Lemma 2.8]{SW})
		Let $\undla\in\mathscr{P}^{\bullet,m}_{n}$ with $\bullet\in\{\mathsf{0},\mathsf{s},\mathsf{ss}\}$. For any $\ms,\mt \in \Std(\undla),$ we denote by $d(\ms,\mt)\in \mathfrak{S}_n$ the unique element
		such that $\ms=d(\ms,\mt)\mt.$ Then we have
		$$s_{k_i}\text{ is admissible with respect to } s_{k_{i-1}}\ldots s_{k_1}\mathfrak{t},\quad i=1,2,\ldots,p$$
		for any reduced expression $d(\ms,\mt)=s_{k_p}\cdots s_{k_1}.$
	\end{lem}
	\begin{cor}\label{bij}(\cite[Corollary 2.9]{SW})
		Let $\undla\in\mathscr{P}^{\bullet,m}_{n}$ with $\bullet\in\{\mathsf{0},\mathsf{s},\mathsf{ss}\}$. There is a bijection
		$$\psi:\, \tau \mapsto \tau \mt^{\undla},\quad P(\undla)\rightarrow \Std(\undla).$$
	\end{cor}

	
We set $Q_0=Q_{0_+}=1, Q_{0_-}=-1$. Now we can define the residue.
\begin{defn}
Suppose  $\undla\in\mathscr{P}^{\bullet,m}_{n}$ with $\bullet\in\{\mathsf{0},\mathsf{s},\mathsf{ss}\}$ and $(i,j,l)\in \undla$, we define the residue of box $(i,j,l)$ with respect to the parameter $\undQ=(Q_1,Q_2,\ldots,Q_m)\in(\mathbb{K}^*)^m$ as follows:
\label{pag:nondeg residue}
\begin{equation}\label{eq:residue}
\res(i,j,l):=Q_lq^{2(j-i)}.
\end{equation}
If $\mathfrak{t}\in \Std(\undla)$ and $\mathfrak{t}(i,j,l)=k$, we set
\begin{align}
        \res_\mathfrak{t}(k)&:=Q_lq^{2(j-i)},\label{resNon-dege-1}\\
\res(\mathfrak{t})&:=(\res_\mathfrak{t}(1),\cdots,\res_\mathfrak{t}(n)),\label{resNon-dege-2}
       \end{align}
       then the $\mathtt{q}$-sequence of $\mt$ is
\begin{align}\label{resNon-dege-3}
       \mathtt{q}(\res(\mathfrak{t})):=(\mathtt{q}(\res_{\mathfrak{t}}(1)), \mathtt{q}(\res_{\mathfrak{t}}(2)),\ldots, \mathtt{q}(\res_{\mathfrak{t}}(n))).
       \end{align}
\end{defn}
	
	We can write down an explicit formula of $\mb_\pm(\res_\mathfrak{t}(k))$ using deformed quantum integers. We define the deformed quantum integers  as\begin{equation}\label{deformed quantum number}
		\left[ m \right]_{l,q^2}:=\frac{Q_lq^{2m}-Q_l^{-1}q^{-2m}}{q^{2}-q^{-2}},\quad m\in \mathbb{Z}.\end{equation}
	Then
	\begin{align}
		\mb_\pm(\res_\mathfrak{t}(k))&= \frac{Q_lq^{2(j-i)+1}+Q_l^{-1}q^{-2(j-i)-1}}{q+q^{-1}}
		\pm \sqrt{\left(\frac{Q_lq^{2(j-i)+1}+Q_l^{-1}q^{-2(j-i)-1}}{q+q^{-1}}\right)^2 -1} \nonumber\\
		&=[j-i+1]_{l,q^2}-[j-i]_{l,q^2}\pm\epsilon\sqrt{[j-i+1]_{l,q^2}[j-i]_{l,q^2}}\label{eigenvalues}
	\end{align}
	for each $k=\mt(i,j,l).$
	
	We remark here that when the level $r=1,$ $\mb_-(\res_\mathfrak{t}(k))$ coincides with $q_k$ in \cite[(4.6)]{JN}  i.e., suppose $k=\mt(i,j,0),$ then
	\begin{align}
		\mb_-(\res_\mathfrak{t}(k)) = [j-i+1]_{q^2}-[j-i]_{q^2}-\epsilon\sqrt{[j-i+1]_{q^2}[j-i]_{q^2}}.\nonumber
	\end{align}

	\subsubsection{Separate Condition}

	Let $[n]:=\{1,2,\ldots,n\}$. \label{pag:[n]} In the rest of this subsection, we recall the separate condition \cite[Definition 3.9]{SW} on the choice of the parameters $\underline{Q}$ and $f=f^{(\bullet)}_{\underline{Q}}$ with $\bullet\in\{\mathtt{0},\mathtt{s},\mathtt{ss}\}$, where $r=\deg(f)$.
\begin{defn}\label{defn:separate}\cite[Definition 3.9]{SW}
Let $\bullet\in\{\mathsf{0},\mathsf{s},\mathsf{ss}\}$ and $\undQ=(Q_1,\ldots,Q_m)\in(\mathbb{K}^*)^m$.  Assume $\undla\in\mathscr{P}^{\bullet,m}_{n}$. Then $(q,\undQ)$ is said to be {\em separate} with respect to $\undla$ if for any $\mathfrak{t}\in \undla$, the $\mathtt{q}$-sequence for $\mathfrak{t}$ defined via \eqref{resNon-dege-3} satisfy the following condition:
$$
\mathtt{q}(\res_{\mathfrak{t}}(k))\neq\mathtt{q}(\res_{\mathfrak{t}}(k+1)) \text{ for any } k=1,\ldots,n-1.
$$
\end{defn}

The separate condition holds for any $\underline{\mu} \in \mathscr{P}^{\bullet,m}_{n+1}$ can be reformulated via the condition $P^{(\bullet)}_{n}(q^2,\undQ)\neq 0$ (\cite[Proposition 3.11]{SW}), where $P^{(\bullet)}_{n}(q^2,\undQ)$ ($\bullet\in\{\mathsf{0},\mathsf{s},\mathsf{ss}\}$) \label{pag:nondege Pioncare poly} is an explicit polynomial on $(q,\undQ)$. Since we will not use the explicit expression of $P^{(\bullet)}_{n}(q^2,\undQ)$ in this paper, we skip the definition of $P^{(\bullet)}_{n}(q^2,\undQ)$ here.
\begin{prop}\label{separate formula}\cite[Proposition 3.11]{SW}
Let $n\geq 1,\,m\geq 0$,  $\undQ=(Q_1,\ldots,Q_m)\in(\mathbb{K}^*)^m$ and  $\bullet\in\{\mathsf{0},\mathsf{s},\mathsf{ss}\}$. Then $(q,\undQ)$ is separate with respect to $\underline{\mu}$ for any  $\underline{\mu}\in\mathscr{P}^{\bullet,m}_{n+1}$ if and only if $P_n^{(\bullet)}(q^2,\undQ)\neq 0$.
\end{prop}

The following is key for our various constructions in the rest of this paper.
\begin{lem}\label{important condition1}\cite{SW}
		Let $\undQ=(Q_1,\ldots,Q_m)\in(\mathbb{K}^*)^m$ and  $\bullet\in\{\mathsf{0},\mathsf{s},\mathsf{ss}\}$. Suppose $P_n^{(\bullet)}(q^2,\undQ)\neq 0$. Then
	\begin{enumerate}
		\item For any $\undla\in\mathscr{P}^{\bullet,m}_{n},$ $\mathfrak{t}\in\Std(\undla)$, we have  $\mathtt{b}_{\pm}(\res_{\mathfrak{t}}(k))\notin \{\pm 1\}$ for $k\notin \mathcal{D}_{\mathfrak{t}}$;
		\item For any $\undla,\underline{\mu} \in\mathscr{P}^{\bullet,m}_{n},$ $\mt\in\Std(\undla), \mt' \in \Std(\underline{\mu}),$
if $\mt\neq \mt',$ then we have $\mathtt{q}(\res(\mt))\neq \mathtt{q}(\res(\mt'))$;
		\item For any $\undla\in\mathscr{P}^{\bullet,m}_{n},$ $\mathfrak{t}\in\Std(\undla)$ and $k\in [n-1]$, the four pairs $(\mathtt{b}_{\pm}(\res_{\mathfrak{t}}(k)),\mathtt{b}_{\pm}(\res_{\mathfrak{t}}(k+1)))$ do not satisfy \eqref{invertible} if $k,k+1$ are not in the adjacent diagonals of $\mathfrak{t}$.
	\end{enumerate}
	\end{lem}
	
	\begin{proof}
		The statement (1) is \cite[Lemma 3.13(1)]{SW} and (2) is \cite[Lemma 3.15]{SW}. The statement (3) follows from \cite[Lemma 3.13(3)]{SW} and the fact that \eqref{invertible} is equivalent to \eqref{invertible2} via the substitution \eqref{substitute}. This completes our proof.
	\end{proof}

	\begin{thm}\cite[Theorem 4.10]{SW}
	Let $\undQ=(Q_1,Q_2,\ldots,Q_m)\in(\mathbb{K}^*)^m$. Assume $f=f^{(\bullet)}_{\undQ}$ and $P^{(\bullet)}_{n}(q^2,\undQ)\neq 0$ with $\bullet\in\{\mathtt{0},\mathtt{s},\mathtt{ss}\}$. Then $\mHfcn$ is a (split) semisimple algebra.
	\end{thm}
	The explicit construction of simple modules will be given in Section \ref{Nondeg}.
	\subsection{Degenerate case}\label{basic-dege}
		\subsubsection{Affine Sergeev algebra $\mhcn$}
    \label{pag:ASA}
	For $n\in\mathbb{N}$, the affine Sergeev (or degenerate Hecke-Clifford) algebra $\mhcn$ is
	the superalgebra generated by even generators
	$s_1,\ldots,s_{n-1},x_1,\ldots,x_n$ and odd generators
	$c_1,\ldots,c_n$ subject to the following relations
	\begin{align}
		s_i^2=1,\quad s_is_j =s_js_i, \quad
		s_is_{i+1}s_i&=s_{i+1}s_is_{i+1}, \quad|i-j|>1,\label{braid}\\
		x_ix_j&=x_jx_i, \quad 1\leq i,j\leq n, \label{poly}\\
		c_i^2=1,c_ic_j&=-c_jc_i, \quad 1\leq i\neq j\leq n, \label{clifford}\\
		s_ix_i&=x_{i+1}s_i-(1+c_ic_{i+1}),\label{px1}\\
		s_ix_j&=x_js_i, \quad j\neq i, i+1, \label{px2}\\
		s_ic_i=c_{i+1}s_i, s_ic_{i+1}&=c_is_i,s_ic_j=c_js_i,\quad j\neq i, i+1, \label{pc}\\
		x_ic_i=-c_ix_i, x_ic_j&=c_jx_i,\quad 1\leq i\neq j\leq n.
		\label{xc}
	\end{align}
	
	For $\alpha=(\alpha_1,\ldots,\alpha_n)\in\mathbb{Z}_+^n$ and
	$\beta=(\beta_1,\ldots,\beta_n)\in\mathbb{Z}_2^n$, set
	$x^{\alpha}:=x_1^{\alpha_1}\cdots x_n^{\alpha},$
	$c^{\beta}:=c_1^{\beta_1}\cdots c_n^{\beta_n}$ and recall that
	$\supp(\beta):=\{1 \leq k \leq n:\beta_{k}=\bar{1}\},$
	$|\beta|:=\Sigma_{i=1}^{n}\beta_i \in \mathbb{Z}_2.$
	Then we have the
	following.
	\begin{lem}\cite[Theorem 2.2]{BK}\label{lem:PBW}
		The set $\{x^{\alpha}c^{\beta}w~|~ \alpha\in\mathbb{Z}_+^n,
		\beta\in\mathbb{Z}_2^n, w\in \mathfrak{S}_n\}$ forms a basis of $\mhcn$.
	\end{lem}
	
	Let $\mathcal{P}_n$ \label{pag:subalg Pn} be the superalgebra generated by even generators $x_1,\ldots,x_n$ and odd generators $c_1,\ldots,c_n$.
	By Lemma~\ref{lem:PBW},
$\mathcal{P}_n$ can be identified with the superalgebra generated by even generators $x_1,\ldots,x_n$ and odd generators $c_1,\ldots,c_n$ subject to relations \eqref{poly}, \eqref{clifford}, \eqref{xc}. 	Clifford algebra $\mathcal{C}_n$ can be identified with the superalgebra generated by odd generators $c_1,\ldots,c_n$ with relations \eqref{clifford}.

	\subsubsection{Intertwining elements for $\mhcn$}
	Following~\cite{Na2}, we define the intertwining elements as
	\label{pag:dege BK intertwining elements}
	\begin{align}
		\tilde{\phi}_i:=s_i(x_i^2-x^2_{i+1})+(x_i+x_{i+1})+c_ic_{i+1}(x_i-x_{i+1}),\quad
		1\leq i<n.\label{intertw}
	\end{align}
	It is known that
	\begin{align}
		\tilde{\phi}_i^2=2(x_i^2+x^2_{i+1})-(x_i^2-x^2_{i+1})^2\label{sqinter},\\
		\tilde{\phi}_ix_i=x_{i+1}\tilde{\phi}_i, \tilde{\phi}_ix_{i+1}=x_i\tilde{\phi}_i,
		\tilde{\phi}_ix_l=x_l\tilde{\phi}_i\label{xinter},\\
		\tilde{\phi}_ic_i=c_{i+1}\tilde{\phi}_i, \tilde{\phi}_ic_{i+1}=c_i\tilde{\phi}_i,
		\tilde{\phi}_ic_l=c_l\tilde{\phi}_i\label{cinter},\\
		\tilde{\phi}_j\tilde{\phi}_i=\tilde{\phi}_i\tilde{\phi}_j,
		\tilde{\phi}_i\tilde{\phi}_{i+1}\tilde{\phi}_i=\tilde{\phi}_{i+1}\tilde{\phi}_i\tilde{\phi}_{i+1}\label{braidinter}
	\end{align}
	for all admissible $j,i,l$ with $l\neq i, i+1$ and $|j-i|>1$.
We also have the intertwining elements of the following version
\label{pag:dege N intertwining elements}
\begin{align*}
	\phi_{i}:=\tilde{\phi}_i (x_{i}^{2}-x_{i+1}^{2})^{-1}&=s_i+\frac{1}{x_i-x_{i+1}}+c_ic_{i+1}\frac{1}{x_i+x_{i+1}}\\
                            &\in \mathbb{K}(x_1,\ldots,x_n)\otimes_{\mathbb{K}[x_1,\ldots, x_n]} \mhcn,
\end{align*}
	for $i=1,\ldots,n-1.$
	
	For any $i=1,2,\ldots,n-1$ and $x,y \in \mathbb{K}$ satisfying $y\neq \pm x,$ let (\cite{Na2})
    \label{pag:dege N phi function}
	$$\phi_i(x,y):=s_{i}+\frac{1}{x-y}+\frac{c_{i}c_{i+1}}{x+y} \in \mHcn.$$
\begin{rem}\label{diff. of phi and Phi}
It is worth noting that $x$'s are on the right-hand-side of $c$'s in the third term of $\phi_{i},$ which differs slightly from $\Phi_{i}$ \eqref{universal-Phi} in non-degenerate case.
\end{rem}

	\begin{lem}(\cite[(5.1)]{Na2}) These functions satisfy the following properties
		\begin{align}
			\phi_i(x,y)\phi_i(y,x)=1- \frac{1}{(x-y)^2} - \frac{1}{(x+y)^2}, \label{degsquare1}\\
			\phi_i(x,y)\phi_j(z,w)=\phi_j(z,w)\phi_i(x,y), \quad \text{ if } |j-i|>1, \label{degbraidrel1} \\
			\phi_i(x,y)\phi_{i+1}(z,y)\phi_i(z,x)=\phi_{i+1}(z,x)\phi_i(z,y)\phi_{i+1}(x,y) \label{degbraidrel2}
		\end{align}
		for all possible $i,j$ and $x,y,z,w.$ Furthermore, we also have the following
		\begin{align}
			\phi_i(x,y)^2= \frac{2}{x-y} \phi_i(x,y) + 1&- \frac{1}{(x-y)^2} - \frac{1}{(x+y)^2},\label{degsquare2}\\
			c_i\phi_i(x,y)=\phi_i(x,-y)c_{i+1},\quad  c_{i+1}\phi_i(x,y)&=\phi_i(-x,y)c_i,\quad c_j\phi_i(x,y)=\phi_i(x,y)c_j \label{phic2}
		\end{align}
		for $j \neq i,i+1.$
	\end{lem}
	
	For any pair of $(x,y)\in \mathbb{K}^2$ and $y\neq \pm x$, we consider the following idempotency condition on $(x,y)$
	\begin{align}\label{invertible dege}
		\frac{1}{(x-y)^2}+\frac{1}{(x+y)^2}=1.
	\end{align}
	Note that the equation \eqref{invertible dege} holds for the pair $(x,y)$ if and only if it holds for one of these four pairs $(\pm x,\pm y).$	
	Similarly, we have
	\begin{cor}\label{phi}
		(a) Suppose that the pair $(x,y)$ satisfies (\ref{invertible dege}) and $y\neq \pm x.$ Then
		$$\phi_i(x,y)^2=\frac{2}{x-y} \phi_i(x,y).$$
		(b) Suppose that the pair $(x,y)$ does not satisfy (\ref{invertible dege}) and $y\neq \pm x.$ Then $\phi_i(x,y)$ is invertible and
		$$\phi_i(x,y)^{-1}
		=\left( 1- \frac{1}{(x-y)^2} - \frac{1}{(x+y)^2} \right)^{-1}  \left( \phi_i(x,y)-\frac{2}{x-y} \right).$$
	\end{cor}
		For any  $\iota \in \mathbb{K}$, we set
    \label{dege q-function and u-function}
	\begin{align}
		\mathtt{q}(\iota):=\iota(\iota+1) \label{qi},\quad \mathtt{u}_{\pm}(\iota):=\pm \sqrt{\iota(\iota+1)}.
	\end{align}
	
	In \cite{Na2}, the condition \eqref{invertible dege} was transformed by the substitution \begin{align}\label{substitute dege}
		x^2=u(u+1)=	\mathtt{q}(u),\qquad\qquad y^2=v(v+1)=	\mathtt{q}(v)
	\end{align} to the equivalent condition which states that $u,v$ satisfy one of the following four equations \begin{align}\label{invertible dege2}
		u-v=\pm 1,\quad u+v=0,\quad u+v=-2.
	\end{align}
	\subsubsection{Cyclotomic Sergeev algebra $\mhgcn$}
    \label{pag:CSA}
	Similarly, to define the cyclotomic Sergeev algebra $\mhgcn$, we need to fix a $g=g(x_1)\in \mathbb{K}[x_1]$ satisfying  \cite[3-e]{BK}. Since we are working over algebraically closed field  $\mathbb{K}$,  one can check that $g=g(x_1)\in \mathbb{K}[x_1]$ satisfying  \cite[3-e]{BK} must be one of the following two forms:
	$$\begin{aligned}
		g=\begin{cases}
		g^{(\mathsf{0})}_{\underline{Q}}=\prod_{i=1}^m\biggl(x^2_1-\mathtt{q}(Q_i)\biggr), \\ 
		g^{\mathsf{(s)}}_{\underline{Q}}=x_1\prod_{i=1}^m\biggl(x^2_1-\mathtt{q}(Q_i)\biggr), 
			\end{cases}
	\end{aligned}
	$$ where $Q_1,\cdots, Q_m\in\mathbb{K}$. \label{pag:dege Q-parameters} In each case, the degree $r$ of the polynomial $f$ is $2m,\,2m+1$ respectively.
	
	The cyclotomic Sergeev (or degenerate cyclotomic Hecke-Clifford) algebra $\mhgcn$ is defined as $$\mhgcn:=\mhcn/\mathcal{J}_g,
	$$ where $\mathcal{J}_g$ is the two sided ideal of $\mhcn$ generated by $g(x_1)$. We call $r:=\deg(g)$ the level of $\mhgcn.$
    \label{pag:dege level}
    We shall denote the images of $x^{\alpha}, c^{\beta}, w$ in the cyclotomic quotient $\mhgcn$ by the same symbols.
	
	Then we have the following due to \cite{BK}.
	\begin{lem}\cite[Theorem 3.6]{BK}
		The set $\{x^{\alpha}c^{\beta}w~|~ \alpha\in\{0,1,\cdots,r-1\}^n,
		\beta\in\mathbb{Z}_2^n, w\in {\mathfrak{S}_n}\}$ forms a basis of $\mhgcn$, where $r=\deg(g)$.
	\end{lem}
	By \cite[(2.32)]{BK}, the algebra $\mhgcn$ admits an anti-involution $*$ defined by $s_i^*=s_i,$ $c_j^*=c_j$ and $x_j^*=x_j,$ for
$i\in [n-1],$ $j\in[n].$ \label{pag:dege anti-involution}

	We set $Q_0:=0$.
    \label{pag:dege residue}
	\begin{defn}Suppose  $\undla\in\mathscr{P}^{\bullet,m}_{n}$ with $\bullet\in\{\mathsf{0},\mathsf{s}\}$  and $(i,j,l)\in \undla$, we define the residue of box $(i,j,l)$ with respect to the parameter $\undQ=(Q_1,\ldots,Q_m)\in\mathbb{K}^m$ as follows: \begin{equation}\label{eq:deg-residue}
			\res(i,j,l):=Q_l+j-i.
		\end{equation} If $\mathfrak{t}\in \Std(\undla)$ and $\mathfrak{t}(i,j,l)=a$, we set \begin{align}\label{res-dege}
			\res_\mathfrak{t}(a)&:=Q_l+j-i,\\
			\res(\mathfrak{t})&:=(\res_\mathfrak{t}(1),\cdots,\res_\mathfrak{t}(n)),
		\end{align}
		then the (degenerate) $\mathtt{q}$-sequence of $\mt$ is
\begin{align}\label{res-dege2}
       \mathtt{q}(\res(\mathfrak{t})):=(\mathtt{q}(\res_{\mathfrak{t}}(1)), \mathtt{q}(\res_{\mathfrak{t}}(2)),\ldots, \mathtt{q}(\res_{\mathfrak{t}}(n))).
       \end{align}
	\end{defn}
	\subsubsection{Separate Condition}
	We shall recall the separate condition in \cite[Definition 5.8]{SW} on the choice of the parameters $\underline{Q}$ and $g=g^{(\bullet)}_{\underline{Q}}$ with $\bullet\in\{\mathtt{0},\mathtt{s}\}$, $r:=\deg(g)$.

	\begin{defn}\label{important condition2}\cite[Definition 5.8]{SW}
	Let $\undQ=(Q_1,\ldots,Q_m)\in\mathbb{K}^m$ and $\undla\in\mathscr{P}^{\bullet,m}_{n}$ with $\bullet\in\{\mathsf{0},\mathsf{s}\}$. The parameter $\undQ$ is said to be {\em separate} with respect to $\undla$  if for any $\mathfrak{t}\in \Std(\undla)$, the $\mathtt{q}$-sequence for $\mathfrak{t}$ defined via \eqref{res-dege2} satisfy the following condition:
    $$
	\mathtt{q}(\res_{\mathfrak{t}}(k))\neq\mathtt{q}(\res_{\mathfrak{t}}(k+1)) \text{ for any $k=1,\cdots,n-1.$ }
	$$
	\end{defn}
	
	As in non-degenerate case, the separate condition holds for any $\underline{\mu} \in \mathscr{P}^{\bullet,m}_{n+1}$ can be reformulated via the condition $P^{(\bullet)}_{n}(1,\undQ)\neq 0$ (\cite[Proposition 5.12]{SW}), \label{pag:dege Pioncare poly} where $P^{(\bullet)}_{n}(1,\undQ)$ ($\bullet\in\{\mathsf{0},\mathsf{s}\}$) is an explicit polynomial on $\undQ$ which can be regarded as the extended definition of $P^{(\bullet)}_{n}(q^2,\undQ)$ on ``$q^2=1$''. We also skip the definition of $P^{(\bullet)}_{n}(1,\undQ)$ here.
\begin{prop}\label{separate formula dege}\cite[Proposition 5.13]{SW}
	Let $n\geq 1,\,m\geq 0$,  $\undQ=(Q_1,\ldots,Q_m)\in\mathbb{K}^m$ and $\bullet\in\{\mathsf{0},\mathsf{s}\}$. Then the parameter $\undQ$ is separate with respect to $\undla$ for any $\undla\in\mathscr{P}^{\bullet,m}_{n+1}$ if and only if $P_n^{(\bullet)}(1,\undQ)\neq 0$.
\end{prop}

Again, the following is key for our constructions in the rest of this paper.	
\begin{lem}\label{important conditionequi2}\cite{SW}
Let $\undQ=(Q_1,\ldots,Q_m)\in\mathbb{K}^m$ and $\bullet\in\{\mathsf{0},\mathsf{s}\}$. Suppose $P_n^{(\bullet)}(1,\undQ)\neq 0$. Then
	\begin{enumerate}
		\item For any $\undla\in\mathscr{P}^{\bullet,m}_{n},$ $\mathfrak{t}\in\Std(\undla)$, we have  $\mathtt{u}_{\pm}(\res_{\mathfrak{t}}(k))\neq 0$ for $k\notin \mathcal{D}_{\mathfrak{t}}$;
		\item For any $\undla,\underline{\mu} \in\mathscr{P}^{\bullet,m}_{n},$ $\mt\in\Std(\undla), \mt' \in \Std(\underline{\mu}),$
if $\mt\neq \mt',$ then we have $\mathtt{q}(\res(\mt))\neq \mathtt{q}(\res(\mt'))$;
		\item For any $\undla\in\mathscr{P}^{\bullet,m}_{n},$ $\mathfrak{t}\in\Std(\undla)$ and $k\in [n-1]$, the four pairs $(\mathtt{u}_{\pm}(\res_{\mathfrak{t}}(k)),\mathtt{u}_{\pm}(\res_{\mathfrak{t}}(k+1)))$ do not satisfy \eqref{invertible dege} if $k,k+1$ are not in the adjacent diagonals of $\mathfrak{t}$.
	\end{enumerate}
\end{lem}

	\begin{thm}\cite[Theorem 5.21]{SW}
		Let $\undQ=(Q_1,Q_2,\ldots,Q_m)\in\mathbb{K}^m$.  Assume $g=g^{(\bullet)}_{\undQ}$ and   $P_n^{(\bullet)}(1,\undQ)\neq 0$, with $\bullet\in\{\mathtt{0},\mathtt{s}\}$. Then  $\mhgcn$ is a (split) semisimple algebra.
	\end{thm}
The explicit construction of simple modules will be given in Section \ref{dege}.
	\section{Simple modules of $\mathcal{A}_n$ and $\mathcal{P}_n$}\label{simplemodules}
	The purpose of this section is to write down an explicit basis for the simple modules of $\mathcal{A}_n$ and $\mathcal{P}_n$ and to describe the actions of $\mathcal{A}_n$ and $\mathcal{P}_n$ on the basis.
		\subsection{Representations of $\mathcal{A}_n$}
	

	{\textbf {In this subsection, we shall fix $\underline{\iota}=(\iota_1,\iota_2,\cdots,\iota_n )\in (\mathbb{K}^*)^n$ and define
\label{pag:nondeg.Diota}
\begin{align}\label{nondeg.Diota}
\mathcal{D}_{\underline{\iota}}=\{i_1<i_2<\cdots<i_t\}:=\{1\leq i\leq n \mid \mathtt{b}_\pm(\iota_{i})\in\{\pm 1\}\}.
\end{align}}}
	For each $\iota_i$, let $\mathbb{L}(\iota_i)$ be the $2$-dimensional
	$\mathcal{A}_1$-supermodule $\mathbb{L}(\iota_i)=\mathbb{K}v_i\oplus \mathbb{K}v'_i$  with
	$\mathbb{L}(\iota_i)_{\bar{0}}=\mathbb{K}v_i,$ $\mathbb{L}(\iota_i)_{\bar{1}}=\mathbb{K}v'_i$
	and
	$$
	X^{\pm 1}_i v_i=\mathtt{b}_{\pm}(\iota_i)v_i,\quad X^{\pm 1}_i v'_i=\mathtt{b}_{\mp}(\iota_i)v'_i, \quad
	C_iv_i=v'_i,\quad C_iv'_i=v_i.
	$$
	Then the $\mathcal{A}_n$-supermodule $\mathbb{L}(\iota_1)\otimes\mathbb{L}(\iota_2)\otimes\cdots\otimes\mathbb{L}(\iota_n)$ is a cyclic $\mathcal{A}_n$-supermodule with $\mathbb{K}$-basis $\{C^\beta v_1\otimes v_2\otimes\cdots \otimes v_n \mid \beta\in \Z_2^n\}$. It's easy to see that each $\mathcal{A}_1$-module $\mathbb{L}(\iota_i)$ is irreducible of type $\texttt{Q}$ if $\mathtt{b}_\pm(\iota_{i})\in\{\pm 1\}$,
	and irreducible of type $\texttt{M}$ if $\mathtt{b}_\pm(\iota_{i})\notin\{\pm 1\}$. Clearly, $\mathcal{A}_n\cong \mathcal{A}_1\otimes\cdots \otimes\mathcal{A}_1.$ Applying Lemma \ref{tensorsmod}, we have
	\begin{equation}\label{decom2}\mathbb{L}(\iota_1)\otimes\mathbb{L}(\iota_2)\otimes\cdots\otimes\mathbb{L}(\iota_n)\cong \biggl(\mathbb{L}(\iota_1)\circledast
		\mathbb{L}(\iota_2)\circledast\cdots\circledast
		\mathbb{L}(\iota_n)\biggr)^{\oplus 2^{\lfloor t/2 \rfloor}}\end{equation}
	as $\mathcal{A}_n$-modules. We want to write down an explicit basis for $\mathbb{L}(\underline{\iota}):=\mathbb{L}(\iota_1)\circledast
	\mathbb{L}(\iota_2)\circledast\cdots\circledast
	\mathbb{L}(\iota_n)$ and describe the actions of $\mathcal{A}_n$ on the basis. To this end, we need the following $\mathcal{A}_n$-module homomorphism.
	
	\begin{lem}\label{lem:hom1}
		Suppose $\mathtt{b}_\pm(\iota_{i})\in\{\pm 1\}$. Then the $\mathbb{K}$-linear map $$\begin{matrix}
			\rho_i:  &\mathbb{L}(\iota_1)\otimes\mathbb{L}(\iota_2)\otimes\cdots\otimes\mathbb{L}(\iota_n)&\longrightarrow  &\mathbb{L}(\iota_1)\otimes\mathbb{L}(\iota_2)\otimes\cdots\otimes\mathbb{L}(\iota_n)&\\
			&C^\beta v_1\otimes v_2\otimes\cdots \otimes v_n&\mapsto &C^\beta C_i v_1\otimes v_2\otimes\cdots \otimes v_n,&\forall\beta\in \Z_2^n,
		\end{matrix}$$ is an $\mathcal{A}_n$-module homomorphism.
	\end{lem}
	
	\begin{proof}  First, we prove that $\rho_i$ commutes with the action of $C_l$, where $1\leq l\leq n$. We fix $1\leq l\leq n$, $\beta\in \Z_2^n$. Suppose $C_lC^\beta=aC^\omega$, where $a\in\{\pm 1\}$ and $\omega\in \Z_2^n$. Then we have
		\begin{align*}C_l\rho_i(C^\beta v_1\otimes v_2\otimes\cdots \otimes v_n)&=C_l(C^\beta C_i v_1\otimes v_2\otimes\cdots \otimes v_n)\\
			&=(C_lC^\beta )C_i v_1\otimes v_2\otimes\cdots \otimes v_n\\
			&=aC^\omega C_iv_1\otimes v_2\otimes\cdots \otimes v_n,\\
			\rho_i(C_lC^\beta v_1\otimes v_2\otimes\cdots \otimes v_n)&=\rho_i(aC^\omega v_1\otimes v_2\otimes\cdots \otimes v_n)\\
			&=aC^\omega C_iv_1\otimes v_2\otimes\cdots \otimes v_n.
		\end{align*}  Hence $\rho_i$ commutes with the action of $C_l$, where $1\leq l\leq n$.
		
		Next we check that $\rho_i$ commutes with the action of $X_j$, where $1\leq j\leq n$. We fix $1\leq j\leq n$, $\beta\in \Z_2^n$. Suppose $X_jC^\beta=C^\beta X_j^{b}$ for some $b\in\{\pm 1\}$. Then we have
		\begin{align}X_j\rho_i(C^\beta v_1\otimes v_2\otimes\cdots \otimes v_n)&=X_j(C^\beta C_i v_1\otimes v_2\otimes\cdots \otimes v_n)\nonumber\\
			&=(X_jC^\beta) C_i v_1\otimes v_2\otimes\cdots \otimes v_n\nonumber\\
			&=C^\beta X_j^{b}C_iv_1\otimes v_2\otimes\cdots \otimes v_n, \label{equ1}\\
			\rho_i(X_jC^\beta v_1\otimes v_2\otimes\cdots \otimes v_n)&=\rho_i(C^\beta X_j^{b} v_1\otimes v_2\otimes\cdots \otimes v_n)\nonumber\\
			&=\mathtt{b}_b(\iota_j)\rho_i(C^\beta  v_1\otimes v_2\otimes\cdots \otimes v_n)\nonumber\\
			&=\mathtt{b}_b(\iota_j)C^\beta C_iv_1\otimes v_2\otimes\cdots \otimes v_n. \label{equ2}
		\end{align}
		If $i\neq j$, then $X_j^{b}C_i=C_iX_j^b$, hence we have \begin{align*}X_j\rho_i(C^\beta v_1\otimes v_2\otimes\cdots \otimes v_n)&=C^\beta C_iX_j^{b}v_1\otimes v_2\otimes\cdots \otimes v_n\nonumber\\
			&=\mathtt{b}_b(\iota_j)C^\beta C_i v_1\otimes v_2\otimes\cdots \otimes v_n\nonumber\\
			&=	\rho_i(X_jC^\beta v_1\otimes v_2\otimes\cdots \otimes v_n)
		\end{align*} by \eqref{equ1} and \eqref{equ2}. If $i=j$,  then $X_i^{b}C_i=C_iX_i^{-b}$. Note that $\mathtt{b}_\pm(\iota_{i})\in\{\pm 1\}$, hence $\mathtt{b}_+(\iota_{i})=\mathtt{b}_-(\iota_{i})$. We have \begin{align*}X_i\rho_i(C^\beta v_1\otimes v_2\otimes\cdots \otimes v_n)&=C^\beta C_iX_i^{-b}v_1\otimes v_2\otimes\cdots \otimes v_n\nonumber\\
			&=\mathtt{b}_{-b}(\iota_i)C^\beta C_i v_1\otimes v_2\otimes\cdots \otimes v_n\nonumber\\
			&=\mathtt{b}_{b}(\iota_i)C^\beta C_i v_1\otimes v_2\otimes\cdots \otimes v_n\nonumber\\
			&=	\rho_i(X_jC^\beta v_1\otimes v_2\otimes\cdots \otimes v_n)\nonumber
		\end{align*} by \eqref{equ1} and \eqref{equ2} again. This implies that $\rho_i$ commutes with action of $X_j$, where $1\leq j\leq n$ and completes the proof of Lemma.
	\end{proof}
	Recall that
	\begin{align}
		\mathcal{D}_{\underline{\iota}}=\{i_1<i_2<\cdots<i_t\}:=\{1\leq i\leq n\mid \mathtt{b}_\pm(\iota_{i})\in\{\pm 1\}\}.
	\end{align}
	
	\begin{lem}\label{lem:hom2}
		There is an algebra homomorphism $$\begin{matrix}
			\rho	:  & \mathcal{C}_t&\longrightarrow  &\End_{\mathcal{A}_n}\biggl(\mathbb{L}(\iota_1)\otimes\mathbb{L}(\iota_2)\otimes\cdots\otimes\mathbb{L}(\iota_n)\biggr)&\\
			&C_j&\mapsto &\rho_{i_j},&1\leq j\leq t.
		\end{matrix}$$
	\end{lem}
	
	\begin{proof}
		Let $1\leq j_1<j_2\leq t$ and $\alpha\in\Z_2^n$. We have \begin{align*}	
			\rho(C_{j_1})	\rho(C_{j_2})(C^\beta v_1\otimes v_2\otimes\cdots \otimes v_n)&=\rho(C_{j_1})(C^\beta C_{i_{j_2}}v_1\otimes v_2\otimes\cdots \otimes v_n)\\
			&=C^\beta C_{i_{j_2}} \rho(C_{j_1})(v_1\otimes v_2\otimes\cdots \otimes v_n)\\
			&=	C^\beta C_{i_{j_2}} C_{i_{j_1}}(v_1\otimes v_2\otimes\cdots \otimes v_n),
		\end{align*} where in the second equality we have used Lemma \ref{lem:hom1}, i.e. $\rho(C_{j_1})$ is an $\mathcal{A}_n$-module homomorphism.
		Similarly, we have $$	\rho(C_{j_2})	\rho(C_{j_1})(C^\beta v_1\otimes v_2\otimes\cdots \otimes v_n)=C^\beta C_{i_{j_1}} C_{i_{j_2}}(v_1\otimes v_2\otimes\cdots \otimes v_n).$$ Hence
		$$\rho(C_{j_2})	\rho(C_{j_1})(C^\beta v_1\otimes v_2\otimes\cdots \otimes v_n)=-\rho(C_{j_1})	\rho(C_{j_2})(C^\beta v_1\otimes v_2\otimes\cdots \otimes v_n)
		$$ and we deduce $\rho(C_{j_2})	\rho(C_{j_1})=-\rho(C_{j_1})	\rho(C_{j_2})$. Moreover, \begin{align*}	
			\rho(C_{j_1})	\rho(C_{j_1})(C^\beta v_1\otimes v_2\otimes\cdots \otimes v_n)&=\rho(C_{j_1})(C^\beta C_{i_{j_1}}v_1\otimes v_2\otimes\cdots \otimes v_n)\\
			&=C^\beta C_{i_{j_1}} \rho(C_{j_1})(v_1\otimes v_2\otimes\cdots \otimes v_n)\\
			&=C^\beta C_{i_{j_1}}C_{i_{j_1}} (v_1\otimes v_2\otimes\cdots \otimes v_n)\\
			&=	C^\beta (v_1\otimes v_2\otimes\cdots \otimes v_n),
		\end{align*} where in the second equality we have used Lemma \ref{lem:hom1} again. This implies that $\rho(C_{j_1})	\rho(C_{j_1})=1$. Therefore, we have checked that $\rho(C_1),\rho(C_2),\cdots,\rho(C_t)$ satisfy the relation of Clifford algebra $\mathcal{C}_t$ and proved this Lemma.
	\end{proof}
	
	Recall that $$I_t:=\Biggl\{2^{-\lfloor t/2 \rfloor}\cdot\overrightarrow{\prod_{k=1,\cdots,{\lfloor t/2 \rfloor}}}(1+(-1)^{a_k} \sqrt{-1}C_{2k-1}C_{2k})\in \mathcal{C}_t \Biggm|a_k\in\Z_2,\,1\leq k\leq {\lfloor t/2 \rfloor} \Biggr\},$$  forms a complete set of orthogonal primitive idempotents for $\mathcal{C}_t$. Combining this with Lemma \ref{lem:hom2}, we have the following $\mathcal{A}_n$-module decomposition :\begin{align}
		\mathbb{L}(\iota_1)\otimes\mathbb{L}(\iota_2)\otimes\cdots\otimes\mathbb{L}(\iota_n)&=\bigoplus_{\gamma\in I_t}	\rho(\gamma)\biggl(\mathbb{L}(\iota_1)\otimes\mathbb{L}(\iota_2)\otimes\cdots\otimes\mathbb{L}(\iota_n)\biggr)\\
		&=\bigoplus_{\gamma\in I_t}	\rho(\gamma)\biggl(\mathcal{C}_n v_1\otimes v_2\otimes\cdots \otimes v_n\biggr).\label{decom1}
	\end{align}
	
	We define
	\begin{align}\label{nondeg.ODiota}
		\mathcal{OD}_{\underline{\iota}}&:=\{i_1,i_3,\ldots,i_{2{\lceil t/2 \rceil}-1}\}\subset \mathcal{D}_{\underline{\iota}}, \\
		\mathbb{Z}_2(\mathcal{OD}_{\underline{\iota}})&:=\{\alpha \in \mathbb{Z}_2^{n} \mid \supp(\alpha)\in \mathcal{OD}_{\underline{\iota}} \}, \nonumber\\
		\mathbb{Z}_2([n]\setminus \mathcal{D}_{\underline{\iota}})&:=\{\alpha \in \mathbb{Z}_2^{n} \mid \supp(\alpha)\in [n]\setminus \mathcal{D}_{\underline{\iota}} \}.\nonumber
	\end{align}
	Then we have
	\begin{prop}\label{L basis}
		For any $\gamma_1,\gamma_2\in I_t$, we have
		\begin{enumerate}
			\item $$\rho(\gamma_1)\biggl(\mathbb{L}(\iota_1)\otimes\mathbb{L}(\iota_2)\otimes\cdots\otimes\mathbb{L}(\iota_n)\biggr)\cong\rho(\gamma_2)\biggl(\mathbb{L}(\iota_1)\otimes\mathbb{L}(\iota_2)\otimes\cdots\otimes\mathbb{L}(\iota_n)\biggr)$$ as $\mathcal{A}_n$-modules.
			\item
			$$\rho(\gamma_1)\biggl(\mathbb{L}(\iota_1)\otimes\mathbb{L}(\iota_2)\otimes\cdots\otimes\mathbb{L}(\iota_n)\biggr)\cong\mathbb{L}(\underline{\iota})$$ is an irreducible $\mathcal{A}_n$-module.
			\item $$\Biggl\{C^{\alpha_1}C^{\alpha_2}\rho(\gamma_1)(v_1\otimes v_2\otimes\cdots \otimes v_n)\biggm|\begin{matrix}\alpha_1\in \mathbb{Z}_2([n]\setminus \mathcal{D}_{\underline{\iota}}) \\
				\alpha_2\in \mathbb{Z}_2(\mathcal{OD}_{\underline{\iota}})
			\end{matrix}\Biggr\}$$
			forms a $\mathbb{K}$-basis of $\rho(\gamma_1)\biggl(\mathbb{L}(\iota_1)\otimes\mathbb{L}(\iota_2)\otimes\cdots\otimes\mathbb{L}(\iota_n)\biggr)=\rho(\gamma_1)\biggl(\mathcal{C}_n v_1\otimes v_2\otimes\cdots \otimes v_n\biggr)$.
			
		\end{enumerate}
	\end{prop}
	
	\begin{proof}
		(1). Lemma \ref{lem:hom2} implies that both $ \rho(\gamma_1)\biggl(\mathbb{L}(\iota_1)\otimes\mathbb{L}(\iota_2)\otimes\cdots\otimes\mathbb{L}(\iota_n)\biggr)$ and $\rho(\gamma_2)\biggl(\mathbb{L}(\iota_1)\otimes\mathbb{L}(\iota_2)\otimes\cdots\otimes\mathbb{L}(\iota_n)\biggr)$ are $\mathcal{A}_n$-modules. By Lemma \ref{lem:clifford rep} (2), there exists an invertible element $C\in \mathcal{C}_t$ such that $C^{-1}\gamma_2C=\gamma_1$. Since $\forall u\in \mathbb{L}(\iota_1)\otimes\mathbb{L}(\iota_2)\otimes\cdots\otimes\mathbb{L}(\iota_n)$, we have $$
		\rho(C)\rho(\gamma_1)(u)=\rho(C\gamma_1)(u)=\rho(\gamma_2C)(u)=\rho(\gamma_2)\rho(C)(u),
		$$ hence the restriction map:$$
		\rho(C): \rho(\gamma_1)\biggl(\mathbb{L}(\iota_1)\otimes\mathbb{L}(\iota_2)\otimes\cdots\otimes\mathbb{L}(\iota_n)\biggr)\longrightarrow  \rho(\gamma_2)\biggl(\mathbb{L}(\iota_1)\otimes\mathbb{L}(\iota_2)\otimes\cdots\otimes\mathbb{L}(\iota_n)\biggr)
		$$  is well-defined. It is clearly an $\mathcal{A}_n$-module isomorphism whose inverse is the restriction map $\rho(C^{-1})$. This proves (1).
		
		(2) follows from (1), \eqref{decom1} and \eqref{decom2}. It remains to prove (3).
		
		(3). By Lemma \ref{lem:clifford rep} (3), the set in (3) is a $\mathbb{K}$-linear generating set of $\rho(\gamma_1)\biggl(\mathcal{C}_n v_1\otimes v_2\otimes\cdots \otimes v_n\biggr)$ with cardinality $2^{n-\lfloor t/2 \rfloor}$. On the other hand, \eqref{decom2} implies that $\dim_\mathbb{K} \mathbb{L}(\underline{\iota})=2^{n-\lfloor t/2 \rfloor}$. Hence (3) follows from (2).
	\end{proof}
	
	For each permutation ${\tau}\in \mathfrak{S}_n$ , the twist of the action of
	$\mathcal{A}_n$ on $\mathbb{L}(\underline{\iota})$ with
	${\tau}^{-1}$ leads to a new $\mathcal{A}_n$-module denoted by
	$\mathbb{L}(\underline{\iota})^{\tau}$ with
	$$
	\mathbb{L}(\underline{\iota})^{{\tau}}=\{z^{{\tau}}~|~z\in \mathbb{L}(\underline{\iota})\} ,\quad
	fz^{{\tau}}=({\tau}^{-1}(f)z)^{{\tau}}, \text{ for any }f\in
	\mathcal{A}_n, z\in \mathbb{L}(\underline{\iota}).
	$$
	In particular we have $(X^\pm _kz)^{{\tau}}=X^\pm _{{\tau}(k)}z^{{\tau}}$ and
	$(C_kz)^{{\tau}}=C_{{\tau}(k)}z^{{\tau}}$. It is easy to see that $\mathbb{L}(\underline{\iota})^{{\tau}}\cong \mathbb{L}({\tau}\cdot \underline{\iota})$, where
	${\tau}\cdot \underline{\iota}=(\iota_{{\tau}^{-1}(1)},\ldots,\iota_{{\tau}^{-1}(n)})$
	for  ${\tau}\in \mathfrak{S}_n$. We want to study the structure of $\mathbb{L}(\underline{\iota})^{{\tau}}$ using the basis of $\mathbb{L}(\underline{\iota})$ we obtained in Proposition \ref{L basis}.
	\label{pag:nubetak}
	\begin{defn}	For each $\beta=(\beta_1,\ldots,\beta_n) \in \mathbb{Z}_{2}^{n}$ and $1\leq k \leq n,$ we define
		$$\nu_{\beta}(k):=
		\begin{cases}
			-1, & \text{ if } \beta_k=\bar{1}, \\
			1, & \text{ if } \beta_k=\bar{0}.
		\end{cases}$$
			\end{defn}
\begin{defn}	For each $\beta=(\beta_1,\ldots,\beta_n) \in \mathbb{Z}_{2}^{n}$ and $0\leq i \leq n+1,$  we define
		$$\quad |\beta|_{<i}:=\sum_{1\leq k <i}\beta_k,\quad |\beta|:=|\beta|_{<n+1}.$$
		Similarly, we can also define $|\beta|_{\leq i},$ $|\beta|_{> i}$ and $|\beta|_{\geq i}.$
	\end{defn}
	
	Recall that we have fixed $\underline{\iota}\in (\mathbb{K}^*)^n$ and also recall the definition of
$\mathcal{D}_{\underline{\iota}}$ and $\mathcal{OD}_{\underline{\iota}}.$
\begin{defn}\label{lem. Dtauiota}
For any $\tau\in \mathfrak{S}_n$ and $\underline{\iota}=(\iota_1,\iota_2,\cdots,\iota_n )\in (\mathbb{K}^*)^n,$ we define
		\begin{align*}
			\mathcal{D}^\tau_{\underline{\iota}}&:=\tau(\mathcal{D}_{\underline{\iota}}),\\
			\mathcal{OD}^\tau_{\underline{\iota}}&:=\tau (\mathcal{OD}_{\underline{\iota}}),\\
			\mathbb{Z}_2(\mathcal{OD}^\tau_{\underline{\iota}})&:=\{\alpha \in \mathbb{Z}_2^{n} \mid \supp(\alpha)\in \mathcal{OD}^\tau_{\underline{\iota}} \}, \\
			\mathbb{Z}_2([n]\setminus \mathcal{D}^\tau_{\underline{\iota}})&:=\{\alpha \in \mathbb{Z}_2^{n} \mid \supp(\alpha)\in [n]\setminus \mathcal{D}^\tau_{\underline{\iota}} \}
		\end{align*}
		and we denote by $\gamma_{\underline{\iota}}^\tau$ the following idempotent in Clifford algebra $\mathcal{C}_n$
$$\gamma_{\underline{\iota}}^\tau:=2^{-\lfloor t/2 \rfloor}\cdot\overrightarrow{\prod_{k=1,\cdots,{\lfloor t/2 \rfloor}}}\biggl(1+ \sqrt{-1}C_{\tau(i_{2k-1})}C_{\tau(i_{2k})}\biggr)\in \mathcal{C}_n.
		$$
\end{defn}

	For $k=1,2,\ldots,n$, we define $$e_k:=(\bar{0},\ldots,\bar{1},\dots,\bar{0})\in \mathbb{Z}_2^n$$ to be the $k$-th standard vector in $\Z_2^n$. The following Proposition connects the basis of $\mathbb{L}(\underline{\iota})$ and $\mathbb{L}(\underline{\iota})^\tau$ and also gives explicit formulae  for the actions of $X_i$ and $C_j$ on $\mathbb{L}(\underline{\iota})^\tau$ for any $\tau\in \mathfrak{S}_n$.
	\begin{prop}\label{actions of X,C}
	For any $\underline{\iota}=(\iota_1,\iota_2,\cdots,\iota_n )\in (\mathbb{K}^*)^n,$ the $\mathcal{A}_n$-module $\bigoplus_{\tau\in \mathfrak{S}_n}\mathbb{L}(\underline{\iota})^\tau$ has a $\mathbb{K}$-basis of form
 $$\bigsqcup_{\tau\in \mathfrak{S}_n}\Biggl\{C^{\alpha_1}C^{\alpha_2}v_{\underline{\iota}}^\tau \biggm|\begin{matrix}\alpha_1\in \mathbb{Z}_2([n]\setminus \mathcal{D}^\tau_{\underline{\iota}}) \\
 	\alpha_2\in \mathbb{Z}_2(\mathcal{OD}^\tau_{\underline{\iota}})
 \end{matrix}\Biggr\},$$
such that for any $\tau\in \mathfrak{S}_n$, the set of elements
 $$\Biggl\{C^{\alpha_1}C^{\alpha_2}v_{\underline{\iota}}^\tau \biggm|\begin{matrix}\alpha_1\in \mathbb{Z}_2([n]\setminus \mathcal{D}^\tau_{\underline{\iota}}) \\
			\alpha_2\in \mathbb{Z}_2(\mathcal{OD}^\tau_{\underline{\iota}})
		\end{matrix}\Biggr\}$$
forms a $\mathbb{K}$-linear basis of $\mathcal{A}_n$-module $\mathbb{L}(\underline{\iota})^\tau$
and the actions of $\mathcal{A}_n$ on $\mathbb{L}(\underline{\iota})^\tau$ are given as follows.
		\begin{enumerate}
			\item For any $i\in [n],$ $\alpha_1\in \mathbb{Z}_2([n]\setminus \mathcal{D}^\tau_{\underline{\iota}})$ and $\alpha_2\in \mathbb{Z}_2(\mathcal{OD}^\tau_{\underline{\iota}}),$ we have
			$$X_i\cdot C^{\alpha_1}C^{\alpha_2}v_{\underline{\iota}}^\tau
			=\mathtt{b}_{-}\biggl(\iota_{\tau^{-1}(i)}\biggr)^{-\nu_{\alpha_1}(i)}C^{\alpha_1}C^{\alpha_2}v_{\underline{\iota}}^\tau
			$$
			\item We have $\gamma_{\tau\cdot\iota} v_{\underline{\iota}}^\tau=v_{\underline{\iota}}^\tau.$
			\item  For any $i\in [n],$ $\alpha_1\in \mathbb{Z}_2([n]\setminus \mathcal{D}^\tau_{\underline{\iota}})$ and $\alpha_2\in \mathbb{Z}_2(\mathcal{OD}^\tau_{\underline{\iota}}),$ we have
			\begin{align}
				&C_i\cdot C^{\alpha_1}C^{\alpha_2}v_{\underline{\iota}}^\tau \nonumber\\
				&=\begin{cases}
					(-1)^{|\alpha_1|_{<i}} C^{\alpha_1+e_i}C^{\alpha_2}v_{\underline{\iota}}^\tau, & \text{ if } i\in [n]\setminus \mathcal{D}^\tau_{\underline{\iota}}, \\
					(-1)^{|\alpha_1|+|\alpha_2|_{<i}} C^{\alpha_1}C^{\alpha_2+e_{i}}v_{\underline{\iota}}^\tau, & \text{ if  $i=\tau(i_p) \in \mathcal{D}^\tau_{\underline{\iota}}$, $p$ is odd},\\
					(-\sqrt{-1})(-1)^{|\alpha_1|+|\alpha_2|_{\leq \tau(i_{p-1})}}&\\
					\qquad\qquad\qquad\cdot C^{\alpha_1}C^{\alpha_2+e_{\tau(i_{p-1})}}v_{\underline{\iota}}^\tau, &\text{ if  $i=\tau(i_p) \in \mathcal{D}^\tau_{\underline{\iota}}$, $p$ is even}.\nonumber
				\end{cases}
			\end{align}
		\end{enumerate}
	\end{prop}
	
	\begin{proof}We  choose $$\gamma_1:=2^{-\lfloor t/2 \rfloor}\cdot\overrightarrow{\prod_{k=1,\cdots,{\lfloor t/2 \rfloor}}}\biggl(1+ \sqrt{-1}C_{2k-1}C_{2k}\biggr)\in \mathcal{C}_n.
		$$
		and  set $v_{\underline{\iota}}:=\rho(\gamma_1)(v_1\otimes v_2\otimes\cdots \otimes v_n)$. Applying  Proposition \ref{L basis} (3), we obtain the basis for $\mathbb{L}(\underline{\iota})$.
		Now we have \begin{align*}
			\mathbb{L}(\underline{\iota})^\tau&=\bigoplus_{\substack{\begin{matrix}\alpha_1\in \mathbb{Z}_2([n]\setminus \mathcal{D}_{\underline{\iota}}) \\
						\alpha_2\in \mathbb{Z}_2(\mathcal{OD}_{\underline{\iota}})
			\end{matrix}}}\mathbb{K}(C^{\alpha_1}C^{\alpha_2}v_{\underline{\iota}})^\tau\\
			&=\bigoplus_{\substack{\begin{matrix}\alpha_1\in \mathbb{Z}_2([n]\setminus \mathcal{D}_{\underline{\iota}}) \\
						\alpha_2\in \mathbb{Z}_2(\mathcal{OD}_{\underline{\iota}})
			\end{matrix}}}\mathbb{K}C^{\tau\cdot\alpha_1}C^{\tau\cdot\alpha_2}v_{\underline{\iota}}^\tau\\	
			&=\bigoplus_{\substack{\begin{matrix}\alpha_1\in \mathbb{Z}_2([n]\setminus \mathcal{D}^\tau_{\underline{\iota}}) \\
						\alpha_2\in \mathbb{Z}_2(\gamma_{\underline{\iota}}^\tau)
			\end{matrix}}}\mathbb{K}C^{\alpha_1}C^{\alpha_2}v_{\underline{\iota}}^\tau
		\end{align*} which gives a basis of $\mathbb{L}(\underline{\iota})^\tau$. Next we check the relations.
		\begin{enumerate}
			\item For each $i\in [n],\,\alpha_1\in \mathbb{Z}_2([n]\setminus \mathcal{D}^\tau_{\underline{\iota}})$ and $\alpha_2\in \mathbb{Z}_2(\mathcal{OD}^\tau_{\underline{\iota}}),$ we have
			\begin{align*}
				X_i\cdot C^{\alpha_1}C^{\alpha_2} v_{\underline{\iota}}^\tau&=X_i\cdot C^{\alpha_1}C^{\alpha_2}\biggl(\rho(\gamma_1)(v_1\otimes v_2\otimes\cdots \otimes v_n)\biggr)^\tau\\
				&= C^{\alpha_1}C^{\alpha_2}X_i^{\nu_{\alpha_1}(i)\nu_{\alpha_2}(i)}\biggl(\rho(\gamma_1)(v_1\otimes v_2\otimes\cdots \otimes v_n)\biggr)^\tau\\
				&= C^{\alpha_1}C^{\alpha_2}\biggl(X_{\tau^{-1}(i)}^{\nu_{\alpha_1}(i)\nu_{\alpha_2}(i)}\rho(\gamma_1)(v_1\otimes v_2\otimes\cdots \otimes v_n)\biggr)^\tau\\
				&= C^{\alpha_1}C^{\alpha_2}\biggl(\rho(\gamma_1)X_{\tau^{-1}(i)}^{\nu_{\alpha_1}(i)\nu_{\alpha_2}(i)}(v_1\otimes v_2\otimes\cdots \otimes v_n)\biggr)^\tau\\
				&=\mathtt{b}_{-}\biggl(\iota_{\tau^{-1}(i)}\biggr)^{-\nu_{\alpha_1}(i)\nu_{\alpha_2}(i)}C^{\alpha_1}C^{\alpha_2}\biggl(\rho(\gamma_1)(v_1\otimes v_2\otimes\cdots \otimes v_n)\biggr)^\tau\\
				&=\mathtt{b}_{-}\biggl(\iota_{\tau^{-1}(i)}\biggr)^{-\nu_{\alpha_1}(i)\nu_{\alpha_2}(i)}C^{\alpha_1}C^{\alpha_2} v_{\underline{\iota}}^\tau
			\end{align*} where in the last third equation we have used Lemma \ref{lem:hom2}. Note that if $i\in \mathcal{OD}^\tau_{\underline{\iota}},$ we have $\mathtt{b}_{-}(\iota_{\tau^{-1}(i)})=\pm 1.$ This proves (1).
			\item We have  \begin{align*} \gamma_{\underline{\iota}}^\tau v_{\underline{\iota}}^\tau&= \gamma_{\underline{\iota}}^\tau\biggl(\rho(\gamma_1)(v_1\otimes v_2\otimes\cdots \otimes v_n)\biggr)^\tau\\
				&=\biggl( \tau^{-1}(\gamma_{\underline{\iota}}^\tau )\rho(\gamma_1)(v_1\otimes v_2\otimes\cdots \otimes v_n)\biggr)^\tau\\
				&=\biggl(\gamma_{\underline{\iota}} \rho(\gamma_1)(v_1\otimes v_2\otimes\cdots \otimes v_n)\biggr)^\tau\\
				&=\biggl(\rho(\gamma_1)\rho(\gamma_1) (v_1\otimes v_2\otimes\cdots \otimes v_n)\biggr)^\tau\\
				&=\biggl(\rho(\gamma_1)(v_1\otimes v_2\otimes\cdots \otimes v_n)\biggr)^\tau=v_{\underline{\iota}}^\tau,
			\end{align*} where in the last third equation we have used Lemma \ref{lem:hom2} and in the last second equation we have used that $\gamma_1$ is an idempotent element.
			Hence we prove (2).		 	
			\item For each $i\in [n],\,\alpha_1\in \mathbb{Z}_2([n]\setminus \mathcal{D}^\tau_{\underline{\iota}})$ and $\alpha_2\in \mathbb{Z}_2(\mathcal{OD}^\tau_{\underline{\iota}}).$ If $ i\in [n]\setminus \mathcal{D}^\tau_{\underline{\iota}}$, we have
			$$
			C_i\cdot C^{\alpha_1}C^{\alpha_2}v_{\underline{\iota}}^\tau=(-1)^{|\alpha_1|_{<i}} C^{\alpha_1+e_i}C^{\alpha_2}v_{\underline{\iota}}^\tau.
			$$
			If $i=\tau(i_p )\in \mathcal{D}^\tau_{\underline{\iota}}$ and $p$ is odd, then $$
			C_i\cdot C^{\alpha_1}C^{\alpha_2}v_{\underline{\iota}}^\tau=(-1)^{|\alpha_1|+|\alpha_2|_{<i}} C^{\alpha_1}C^{\alpha_2+e_{i}}v_{\underline{\iota}}^\tau.
			$$
			If $i=\tau(i_p) \in \mathcal{D}^\tau_{\underline{\iota}}$ and $p$ is even, then it is straightforward to check\begin{equation}\label{commutation formula1}
				C_{i_p}\gamma_{\underline{\iota}}=-\sqrt{-1}C_{i_{p-1}}\gamma_{\underline{\iota}}.\end{equation}	
We have
\begin{align*}
				&	C_i\cdot C^{\alpha_1}C^{\alpha_2}v_{\underline{\iota}}^\tau\\
				&=C_{\tau(i_p)}\cdot C^{\alpha_1}C^{\alpha_2}\biggl(\rho(\gamma_1)(v_1\otimes v_2\otimes\cdots \otimes v_n)\biggr)^\tau\\
				&=(-1)^{|\alpha_1|+|\alpha_2|} C^{\alpha_1}C^{\alpha_2}C_{\tau(i_p)}\biggl(\rho(\gamma_1)(v_1\otimes v_2\otimes\cdots \otimes v_n)\biggr)^\tau\\
				&=(-1)^{|\alpha_1|+|\alpha_2|} C^{\alpha_1}C^{\alpha_2}\biggl(C_{i_p}\rho(\gamma_1)(v_1\otimes v_2\otimes\cdots \otimes v_n)\biggr)^\tau\\
				&=(-1)^{|\alpha_1|+|\alpha_2|} C^{\alpha_1}C^{\alpha_2}\biggl(C_{i_p}\gamma_{\underline{\iota}}(v_1\otimes v_2\otimes\cdots \otimes v_n)\biggr)^\tau\\
				&=-\sqrt{-1}(-1)^{|\alpha_1|+|\alpha_2|} C^{\alpha_1}C^{\alpha_2}\biggl(C_{i_{p-1}}\gamma_{\underline{\iota}}(v_1\otimes v_2\otimes\cdots \otimes v_n)\biggr)^\tau\\
				&=-\sqrt{-1}(-1)^{|\alpha_1|+|\alpha_2|} C^{\alpha_1}C^{\alpha_2}\biggl(C_{i_{p-1}}\rho(\gamma_1)(v_1\otimes v_2\otimes\cdots \otimes v_n)\biggr)^\tau\\
				&=-\sqrt{-1}(-1)^{|\alpha_1|+|\alpha_2|} C^{\alpha_1}C^{\alpha_2}C_{\tau(i_{p-1})}\biggl(\rho(\gamma_1)(v_1\otimes v_2\otimes\cdots \otimes v_n)\biggr)^\tau\\
				&=-\sqrt{-1}(-1)^{|\alpha_1|+|\alpha_2|_{\leq\tau(i_{p-1})}} C^{\alpha_1}C^{\alpha_2+e_{\tau(i_{p-1})}}\rho(\gamma_1)(v_1\otimes v_2\otimes\cdots \otimes v_n)^\tau\\
				&=-\sqrt{-1}(-1)^{|\alpha_1|+|\alpha_2|_{\leq\tau(i_{p-1})}} C^{\alpha_1}C^{\alpha_2+e_{\tau(i_{p-1})}}v_{\underline{\iota}}^\tau,
			\end{align*}
where in the fourth and last fourth equations we have used Lemma \ref{lem:hom2} and in the fifth equation we have used \eqref{commutation formula1}.
	\end{enumerate} This completes the proof.
	\end{proof}
	\subsection{Representations of $\mathcal{P}_n$}
	
%
	
	{\textbf {In this subsection, we shall fix $\underline{\iota}=(\iota_1,\iota_2,\cdots,\iota_n )\in \mathbb{K}^n$ and define $$\mathcal{D}_{\underline{\iota}}=\{i_1<i_2<\cdots<i_t\}:=\{1\leq i\leq n \mid \mathtt{u}_\pm(\iota_{i})=0\}.
			$$} }
	For each $\iota_i$, let $L(\iota_i)$ be the $2$-dimensional
	$\mathcal{A}_1$-supermodule $L(\iota_i)=\mathbb{K}v_i\oplus \mathbb{K}v'_i$  with
	$L(\iota_i)_{\bar{0}}=\mathbb{K}v_i$ and $L(\iota_i)_{\bar{1}}=\mathbb{K}v'_i$
	and
	$$
	x_i v_i=\mathtt{u}_{+}(\iota_i)v_i,\quad x_i v'_i=\mathtt{u}_{-}(\iota_i)v'_i, \quad
	c_iv_i=v'_i,\quad c_iv'_i=v_i.
	$$
	Then the $\mathcal{P}_n$-supermodule $L(\iota_1)\otimes L(\iota_2)\otimes\cdots\otimes L(\iota_n)$ is a cyclic $\mathcal{P}_n$-supermodule with $\mathbb{K}$-basis $\{c^\beta v_1\otimes v_2\otimes\cdots \otimes v_n \mid \beta\in \Z_2^n\}$. 	It's easy to see that each $\mathcal{P}_1$-module $L(\iota_i)$  is irreducible of type $\texttt{Q}$ if $\mathtt{u}_\pm(\iota_{i})=0$,
	and irreducible of type $\texttt{M}$ if $\mathtt{u}_\pm(\iota_{i})\neq 0$. Clearly, $\mathcal{A}_n\cong \mathcal{A}_1\otimes\cdots \otimes\mathcal{A}_1.$ Applying Lemma \ref{tensorsmod}, we have
	\begin{equation}\label{decom2 dege}
		L(\iota_1)\otimes L(\iota_2)\otimes\cdots\otimes L(\iota_n)\cong \biggl(L(\iota_1)\circledast
		L(\iota_2)\circledast\cdots\circledast
		L(\iota_n)\biggr)^{\oplus 2^{\lfloor t/2 \rfloor}}
	\end{equation}
	as $\mathcal{P}_n$-modules.
The following two Lemmas are analogues of Lemma \ref{lem:hom1},\,\ref{lem:hom2}.
	\begin{lem}\label{lem:hom1 dege}
		Suppose $\mathtt{u}_\pm(\iota_{i})=0$. Then the $\mathbb{K}$-linear map $$\begin{matrix}
			\rho_i:  &L(\iota_1)\otimes L(\iota_2)\otimes\cdots\otimes L(\iota_n)&\longrightarrow  &L(\iota_1)\otimes L(\iota_2)\otimes\cdots\otimes L(\iota_n)&\\
			&c^\beta v_1\otimes v_2\otimes\cdots \otimes v_n&\mapsto &c^\beta c_i v_1\otimes v_2\otimes\cdots \otimes v_n,&\forall\beta\in \Z_2^n
		\end{matrix}$$ is a $\mathcal{P}_n$-module homomorphism.
	\end{lem}

	\begin{lem}\label{lem:hom2 dege}
		There is an algebra homomorphism $$\begin{matrix}
			\rho	:  & \mathcal{C}_t&\longrightarrow  &\End_{\mathcal{P}_n}\biggl(L(\iota_1)\otimes L(\iota_2)\otimes\cdots\otimes L (\iota_n)\biggr)&\\
			&c_j&\mapsto &\rho_{i_j},&1\leq j\leq t.
		\end{matrix}$$
	\end{lem}
	
%
	Similarly, we define
	\begin{align*}
		\mathcal{OD}_{\underline{\iota}}&:=\{i_1,i_3,\cdots,i_{2{\lceil t/2 \rceil}-1}\}\subset \mathcal{D}_{\underline{\iota}}, \\
		\mathbb{Z}_2(\mathcal{OD}_{\underline{\iota}})&:=\{\alpha \in \mathbb{Z}_2^{n} \mid \supp(\alpha)\in \mathcal{OD}_{\underline{\iota}} \}, \\
		\mathbb{Z}_2([n]\setminus \mathcal{D}_{\underline{\iota}})&:=\{\alpha \in \mathbb{Z}_2^{n} \mid \supp(\alpha)\in [n]\setminus \mathcal{D}_{\underline{\iota}} \}.
	\end{align*}
	Then we have
	\begin{prop}\label{L basis dege}
		For any $\gamma_1,\gamma_2\in I_t$, we have
		\begin{enumerate}
			\item $$\rho(\gamma_1)\biggl( L (\iota_1)\otimes L (\iota_2)\otimes\cdots\otimes L (\iota_n)\biggr)\cong\rho(\gamma_2)\biggl( L (\iota_1)\otimes L (\iota_2)\otimes\cdots\otimes L (\iota_n)\biggr)$$ as $\mathcal{P}_n$-modules.
			\item
			$$\rho(\gamma_1)\biggl( L (\iota_1)\otimes L (\iota_2)\otimes\cdots\otimes L (\iota_n)\biggr)\cong L (\underline{\iota})$$ is an irreducible $\mathcal{P}_n$-module.
			\item $$\Biggl\{c^{\alpha_1}c^{\alpha_2}\rho(\gamma_1)(v_1\otimes v_2\otimes\cdots \otimes v_n)\biggm|\begin{matrix}\alpha_1\in \mathbb{Z}_2([n]\setminus \mathcal{D}_{\underline{\iota}}) \\
				\alpha_2\in \mathbb{Z}_2(\mathcal{OD}_{\underline{\iota}})
			\end{matrix}\Biggr\}$$
			forms a $\mathbb{K}$-basis of $\rho(\gamma_1)\biggl( L (\iota_1)\otimes L (\iota_2)\otimes\cdots\otimes L (\iota_n)\biggr)=\rho(\gamma_1)\biggl(\mathcal{C}_n v_1\otimes v_2\otimes\cdots \otimes v_n\biggr)$.
			
		\end{enumerate}
	\end{prop}
	As in \cite[Remark 5.4]{SW},
	each permutation $\tau\in {\mathfrak{S}_n}$ defines a superalgebra
	isomorphism $\tau:\mathcal{P}_n\rightarrow \mathcal{P}_n$ by mapping $x_k$ to
	$x_{\tau(k)}$ and $c_k$ to  $c_{\tau(k)}$, for $1\leq k\leq n$. For
	$\underline{\iota}\in \mathbb{K}^n$, the twist action of
	$\mathcal{P}_n$ on $L(\underline{\iota})$ leads to a new $\mathcal{P}_n$-module denoted by
	$L(\underline{\iota})^{\tau}$ with
	$$
	L(\underline{\iota})^{\tau}=\{z^{\tau}~|~z\in L(\underline{\iota})\} ,\quad
	fz^{\tau}=(\tau^{-1}(f)z)^{\tau}, \text{ for any }f\in
	\mathcal{P}_n, z\in L(\underline{\iota}).
	$$
	In particular, we have $$(x_kz)^{\tau}=x_{\tau(k)}z^{\tau},\,(c_kz)^{\tau}=c_{\tau(k)}z^{\tau}$$ for each $1\leq k\leq n$. It is also easy to see that $
	L(\underline{\iota})^{\tau}\cong L(\tau\cdot \underline{\iota})$.
	We want to study the structure of $ L (\underline{\iota})^{{\tau}}$ using the basis of $ L (\underline{\iota})$ we obtained in Proposition \ref{L basis dege}.

	\begin{defn}
For any $\tau\in \mathfrak{S}_n$ and $\underline{\iota}=(\iota_1,\iota_2,\cdots,\iota_n )\in \mathbb{K}^n$, we define
		\begin{align*}
			\mathcal{D}^\tau_{\underline{\iota}}&:=\tau(\mathcal{D}_{\underline{\iota}}),\\
			\mathcal{OD}^\tau_{\underline{\iota}}&:=\tau (\mathcal{OD}_{\underline{\iota}}),\\
			\mathbb{Z}_2(\mathcal{OD}^\tau_{\underline{\iota}})&:=\{\alpha \in \mathbb{Z}_2^{n} \mid \supp(\alpha)\in \mathcal{OD}^\tau_{\underline{\iota}} \}, \\
			\mathbb{Z}_2([n]\setminus \mathcal{D}^\tau_{\underline{\iota}})&:=\{\alpha \in \mathbb{Z}_2^{n} \mid \supp(\alpha)\in [n]\setminus \mathcal{D}^\tau_{\underline{\iota}} \}
		\end{align*}
		and denote by $\gamma_{\underline{\iota}}^\tau$ the following idempotent in Clifford algebra $\mathcal{C}_n$
 $$\gamma_{\underline{\iota}}^\tau:=2^{-\lfloor t/2 \rfloor}\cdot\overrightarrow{\prod_{k=1,\cdots,{\lfloor t/2 \rfloor}}}\biggl(1+ \sqrt{-1}c_{\tau(i_{2k-1})}c_{\tau(i_{2k})}\biggr)\in \mathcal{C}_n.
		$$
	\end{defn}

	For $k=1,2,\ldots,n$, Recall that $$e_k:=(\bar{0},\ldots,\bar{1},\dots,\bar{0})\in \mathbb{Z}_2^n$$ to be the $k$-th standard vector in $\Z_2^n$. The following Proposition connects the basis of $ L (\underline{\iota})$ and $ L (\underline{\iota})^\tau$ and also gives explicit formulae  for the actions of $x_i$ and $c_j$ on $ L (\underline{\iota})^\tau$ for any $\tau\in \mathfrak{S}_n$.
	\begin{prop}\label{actions of x,c}
	For any	$\underline{\iota}=(\iota_1,\iota_2,\cdots,\iota_n )\in \mathbb{K}^n,$ the $\mathcal{P}_n$-module $\bigoplus_{\tau\in \mathfrak{S}_n} L (\underline{\iota})^\tau$ has a $\mathbb{K}$-basis of form $$\bigsqcup_{\tau\in \mathfrak{S}_n}\Biggl\{c^{\alpha_1}c^{\alpha_2}v_{\underline{\iota}}^\tau\biggm|\begin{matrix}\alpha_1\in \mathbb{Z}_2([n]\setminus \mathcal{D}^\tau_{\underline{\iota}}) \\
		\alpha_2\in \mathbb{Z}_2(\mathcal{OD}^\tau_{\underline{\iota}})
	\end{matrix}\Biggr\}$$
such that for any $\tau\in \mathfrak{S}_n$, the set of elements
$$\Biggl\{c^{\alpha_1}c^{\alpha_2}v_{\underline{\iota}}^\tau\biggm|\begin{matrix}\alpha_1\in \mathbb{Z}_2([n]\setminus \mathcal{D}^\tau_{\underline{\iota}}) \\
			\alpha_2\in \mathbb{Z}_2(\mathcal{OD}^\tau_{\underline{\iota}})
		\end{matrix}\Biggr\}$$
forms a $\mathbb{K}$-linear basis of the $\mathcal{P}_n$-module $L(\underline{\iota})^\tau$ and the actions of $\mathcal{P}_n$ on
$L(\underline{\iota})^\tau$ are given as follows.
		\begin{enumerate}
			\item For any $i\in [n],$ $\alpha_1\in \mathbb{Z}_2([n]\setminus \mathcal{D}^\tau_{\underline{\iota}})$ and $\alpha_2\in \mathbb{Z}_2(\mathcal{OD}^\tau_{\underline{\iota}}),$ we have
			$$x_i\cdot c^{\alpha_1}c^{\alpha_2}v_{\underline{\iota}}^{\tau}
			=\nu_{\alpha_1}(i) \mathtt{u}_{+}\biggl(\iota_{\tau^{-1}(i)}\biggr) c^{\alpha_1}c^{\alpha_2}v_{\underline{\iota}}^{\tau}
			$$
			\item We have $\gamma_{\tau\cdot\iota} v_{\underline{\iota}}^\tau=v_{\underline{\iota}}^\tau.$
			\item  For any $i\in [n],$ $\alpha_1\in \mathbb{Z}_2([n]\setminus \mathcal{D}^\tau_{\underline{\iota}})$ and $\alpha_2\in \mathbb{Z}_2(\mathcal{OD}^\tau_{\underline{\iota}}),$ we have
			\begin{align}
				&c_i\cdot c^{\alpha_1}c^{\alpha_2}v_{\underline{\iota}}^\tau \nonumber\\
				&=\begin{cases}
					(-1)^{|\alpha_1|_{<i}} c^{\alpha_1+e_i}c^{\alpha_2}v_{\underline{\iota}}^\tau, & \text{ if } i\in [n]\setminus \mathcal{D}^\tau_{\underline{\iota}}, \\
					(-1)^{|\alpha_1|+|\alpha_2|_{<i}} c^{\alpha_1}c^{\alpha_2+e_{i}}v_{\underline{\iota}}^\tau, & \text{ if  $i=\tau(i_p) \in \mathcal{D}^\tau_{\underline{\iota}}$, $p$ is odd},\\
					(-\sqrt{-1})(-1)^{|\alpha_1|+|\alpha_2|_{\leq \tau(i_{p-1})}}&\\
					\qquad\qquad\qquad\cdot c^{\alpha_1}c^{\alpha_2+e_{\tau(i_{p-1})}}v_{\underline{\iota}}^\tau, &\text{ if  $i=\tau(i_p) \in \mathcal{D}^\tau_{\underline{\iota}}$, $p$ is even}.\nonumber
				\end{cases}
			\end{align}
		\end{enumerate}
	\end{prop}

	\section{Primitive idempotents and seminormal basis for cyclotomic Hecke-Clifford algebra $\mHfcn$}\label{Nondeg}
	{\bf In this section, we shall fix the parameter $\undQ=(Q_1,Q_2,\ldots,Q_m)\in(\mathbb{K}^*)^m$ and $f=f^{(\bullet)}_{\undQ}$ with $P^{(\bullet)}_{n}(q^2,\undQ)\neq 0$ for $\bullet\in\{\mathtt{0},\mathtt{s},\mathtt{ss}\}.$ Accordingly, we define the residue of boxes in the young diagram $\undla$ via \eqref{eq:residue} as well as $\res(\mathfrak{t})$ for each $\mathfrak{t}\in\Std(\undla)$ with $\undla\in\mathscr{P}^{\bullet,m}_{n}$ with $m\geq 0.$}
\subsection{Simple modules} \label{Non-dege-simplemodule}{\bf In this subsection, we fix $\undla\in\mathscr{P}^{\bullet,m}_{n}$.} We first recall the construction of simple $\mHfcn$-module $\mathbb{D}(\undla)$ in \cite{SW} for $\undla\in\mathscr{P}^{\bullet,m}_{n}.$ Then we shall give an explicit basis of $\mathbb{D}(\undla)$ and write down the action of generators of $\mHcn$ on this basis.

	Recall that for any $\ms,\mt \in \Std(\undla),$ there exists a unique $d(\ms,\mt)\in \mathfrak{S}_n$ such that
$\ms=d(\ms,\mt)\mt.$ For each $\mt\in\Std(\undla),$ we have ${d(\mt,\mt^{\undla})}\cdot \res(\mathfrak{t}^{\undla})=\res(\mt)$ since
$\mt^{-1}(k)=({d(\mt,\mt^{\undla})}\mathfrak{t}^{\undla})^{-1}(k)=(\mathfrak{t}^{\undla})^{-1}({d(\mt,\mt^{\undla})}^{-1} (k))$ for all $k\in[n].$

For $\undla\in\mathscr{P}^{\bullet,m}_{n},$ define the $\mathcal{A}_n$-module
\label{pag:nondege simple module}
$$
\mathbb{D}(\undla):=\oplus_{\mt\in \Std(\undla)}\mathbb{L}(\res(\mathfrak{t}^{\undla}))^{{d(\mt,\mt^{\undla})}}.
$$
To define a $\mHfcn$-module structure on $\mathbb{D}(\undla)$, we recall two operators on $\mathbb{D}(\undla)$  as in \cite[(4.7), (4.8)]{SW}, which can be viewed as some analogues of the operators in \cite{Wa}:
\begin{align}
\widetilde{\Xi}_i z&:=\left(-\epsilon\frac{1}{X_i X^{-1}_{i+1}-1}+\epsilon\frac{1}{X_i X_{i+1}-1}C_i C_{i+1}\right)z \label{Operater1Non-dege}\\
\widetilde{\Omega}_i z&:=\left(\sqrt{1-\epsilon^2 \biggl(\frac{X_iX^{-1}_{i+1}}{(X_iX^{-1}_{i+1}-1)^2}
+\frac{X^{-1}_iX^{-1}_{i+1}}{(X^{-1}_iX^{-1}_{i+1}-1)^2}\biggr)}\right)z\label{Operater2Non-dege},
\end{align} where $z\in \mathbb{L}(\res(\mathfrak{t}^{\undla}))^{{d(\mt,\mt^{\undla})}}\simeq \mathbb{L}(\res(\mt)).$
By Definition \ref{defn:separate} and Proposition \ref{separate formula}, the eigenvalues of $X_k+X^{-1}_k$ and $X_{k+1}+X^{-1}_{k+1}$ on $\mathbb{L}(\res(\mathfrak{t}^{\undla}))^{{d(\mt,\mt^{\undla})}}$ are different, hence the operators $\widetilde{\Xi}_i$ and $\widetilde{\Omega}_i$ are well-defined on $\mathbb{L}(\res(\mathfrak{t}^{\undla}))^{{d(\mt,\mt^{\undla})}}$ for each $\mt \in \Std(\undla)$.

\begin{thm}\label{Construction}(\cite[Theorem 4.5]{SW})
Let $\undQ=(Q_1,\ldots,Q_m)$. Suppose $f=f(X_1)=f^{(\bullet)}_{\undQ}(X_1)$ and  $P^{(\bullet)}_{n}(q^2,\undQ)\neq 0$ with $\bullet\in\{\mathsf{0},\mathsf{s},\mathsf{ss}\}$.
Then $\mathbb{D}(\undla)$ becomes an $\mHfcn$-module via
\begin{align}
T_iz^{d(\mt,\mt^{\undla})}= \left \{
 \begin{array}{ll}
 \widetilde{\Xi}_i z^{d(\mt,\mt^{\undla})}
 +\widetilde{\Omega}_i z^{s_id(\mt,\mt^{\undla})},
 &  \text{ if } s_i \mt\in \Std(\undla), \\
 \widetilde{\Xi}_i z^{d(\mt,\mt^{\undla})}
 , & \text{ otherwise},
 \end{array}
 \right.\label{actionformulaNon-dege}
\end{align}
 for any $1\leq i\leq n-1,\,z\in \mathbb{L}(\res(\mathfrak{t}^{\undla}))$ and $\mt\in \Std(\undla)$.
\end{thm}

		Recall that we have fixed  $\undla\in\mathscr{P}^{\bullet,m}_{n}$. Let $t:=\sharp \mathcal{D}_{\undla}.$
		\begin{defn}
We denote \begin{align}
				\mathcal{D}_{\mt^{\undla}}&:=\{\mt^{\undla}(a,a,l)|(a,a,l)\in\mathcal{D}_{\undla}\}=\{i_1<i_2<\cdots<i_t\}\label{stanard D},\\
				\mathcal{OD}_{\mt^{\undla}}&:=\{i_1,i_3,\cdots,i_{2{\lceil t/2 \rceil}-1}\}\subset \mathcal{D}_{\mt^{\undla}}\label{standard OD}
				\end{align} and \label{pag:dundla}
		\begin{align}
			d_{\undla}:= \left \{
			\begin{array}{ll}
				1, & \text{ if $t$ is odd}, \\
				0, & \text{ if $t$ is even or $\mathcal{D}_{\mt^{\undla}}=\emptyset$}.
			\end{array}
			\right. \nonumber
		\end{align}
\end{defn}
\begin{rem}
	By Lemma \ref{important condition1} (1), it's easy to check that $\mathcal{D}_{\mt^{\undla}}=\mathcal{D}_{\res(\mt^{\undla})}$ if we set $\underline{\iota}=\res(\mt^{\undla})$ in \eqref{nondeg.Diota}. Hence we keep the same notation $i_1<i_2<\cdots<i_t$ here.
\end{rem}
Recall the notations in Definition \ref{lem. Dtauiota}.
\label{pag:Dt,ODt,Z2ODt}		
		\begin{defn}For each $\mt\in \Std(\undla),$ we define
		\begin{align*}
			\mathcal{D}_{\mt}&:=\mathcal{D}_{\res(\mt^{\undla})}^{d(\mt,\mt^{\undla}) }=d(\mt,\mt^{\undla}) (\mathcal{D}_{\mt^{\undla}}),\nonumber\\
            \mathcal{OD}_{\mt}&:=\mathcal{oD}_{\res(\mt^{\undla})}^{d(\mt,\mt^{\undla}) }=d(\mt,\mt^{\undla}) (\mathcal{OD}_{\mt^{\undla}}),\nonumber\\
		\Z_2(\mathcal{OD}_{\mt})&:=\Z_2(\mathcal{OD}_{\res(\mt^{\undla})}^{d(\mt,\mt^{\undla}) })=\{\alpha_{\mt} \in \mathbb{Z}_{2}^{n} \mid \supp(\alpha_{\mt}) \subseteq \mathcal{OD}_{\mt}\},\\
        \Z_2([n]\setminus \mathcal{D}_{\mt})&:=\Z_2([n]\setminus \mathcal{D}_{\res(\mt^{\undla})}^{d(\mt,\mt^{\undla}) })=\{\beta_{\mt} \in \mathbb{Z}_{2}^{n} \mid \supp(\beta_{\mt}) \subseteq [n]\setminus \mathcal{D}_{\mt}\},\nonumber
	\end{align*} 	and
$$\gamma_{\mt}:=\gamma_{\res(\mt^{\undla})}^{d(\mt,\mt^{\undla})}=2^{-\lfloor t/2 \rfloor}\cdot\overrightarrow{\prod_{k=1,\cdots,{\lfloor t/2 \rfloor}}}\biggl(1+ \sqrt{-1}C_{{d(\mt,\mt^{\undla})}(i_{2k-1})}C_{{d(\mt,\mt^{\undla})}(i_{2k})}\biggr)\in\mathcal{C}_n.
	$$
			\end{defn}

		We are almost ready to apply Proposition \ref{actions of X,C} to give a basis of  $\mathbb{D}(\undla)$ and describe the action of $\mHfcn$ on the basis.
		\begin{defn}
        \label{pag:deltabetak}
For  any $\beta\in \Z_2^n,\,k\in [n]$,	we define $$\delta_{\beta}(k):=\frac{1-\nu_{\beta}(k)}{2}=
			\begin{cases}
				1, & \text{ if } \beta_k=\bar{1}, \\
				0, & \text{ if } \beta_k=\bar{0},
			\end{cases}$$
		
			\end{defn}
			
		\begin{defn}	For any $i\in [n], \mt\in \Std(\undla),$ we define
             \label{pag:nondege eigenvalues}
			 $$\mathtt{b}_{\mt,i}:=\mathtt{b}_{-}(\res_{\mt}(i)).$$  If $i\in [n-1]$, we define
			 \begin{align}
			 	\delta(s_i\mt)
			 	:=\begin{cases}
			 		1, & \text{ if } s_i\mt \in \Std(\undla), \\
			 		0, & \text{ otherwise}.\nonumber
			 	\end{cases}
			 \end{align} and
            \label{pag:nondege coeffi cti}
			\begin{align}
				\mathtt{c}_{\mt}(i):=1-\epsilon^2 \biggl(\frac{\mathtt{b}_{\mt,i}^{-1}\mathtt{b}_{\mt,i+1}}{(\mathtt{b}_{\mt,i}^{-1}\mathtt{b}_{\mt,i+1}-1)^2}
				+\frac{\mathtt{b}_{\mt,i}\mathtt{b}_{\mt,i+1}}{(\mathtt{b}_{\mt,i}\mathtt{b}_{\mt,i+1}-1)^2}\biggr)\in \mathbb{K}.
			\end{align}
			\end{defn}
			Since $\mt \in \Std(\undla),$ $\mathtt{b}_{\mt,i}\neq \mathtt{b}_{\mt,i+1}^{\pm 1}$ by Definition \ref{defn:separate} and Proposition \ref{separate formula}, which immediately implies that $\mathtt{c}_{\mt}(i)$ is well-defined. If $s_i$ is admissible with respect to $\mt$, i.e., $\delta(s_i\mt)=1$, then $\mathtt{c}_{\mt}(i)\in \mathbb{K}^{*}$ by
		the third part of Lemma \ref{important condition1}. It is clear that $\mathtt{c}_{\mt}(i)=\mathtt{c}_{s_i\mt}(i).$	
	\begin{defn}
					Let $\mt\in\Std(\undla),\,\beta_{\mt}\in\Z_2([n]\setminus \mathcal{D}_{\mt})$ and $\alpha_{\mt}\in \Z_2(\mathcal{OD}_{\mt})$. For each $i\in [n-1],$ we define the element
$R(i,\beta_\mt, \alpha_\mt, \mt)\in\mathbb{L}(\res(\mathfrak{t}^{\undla}))^{{d(\mt,\mt^{\undla})}}\subseteq \mathbb{D}(\undla)$ as follows.
		
		\begin{enumerate}
			\item if $i,i+1\in [n]\setminus \mathcal{D}_{\mt},$
			$$R(i,\beta_\mt, \alpha_\mt, \mt):=(-1)^{\delta_{\beta_{\mt}}(i)}C^{\beta_{\mt}+e_i+e_{i+1}}C^{\alpha_{\mt}}v_{\mt};
			$$
			\item
			if $i={d(\mt,\mt^{\undla})}(i_p)\in \mathcal{D}_{\mt}, i+1\in [n]\setminus \mathcal{D}_{\mt},$
			$$R(i,\beta_\mt, \alpha_\mt, \mt):=\begin{cases}(-1)^{|\beta_{\mt}|_{>i}+|\alpha_{\mt}|_{<i}}C^{\beta_{\mt}+e_{i+1}}C^{\alpha_{\mt}+e_{i}}v_{\mt}, & \text{ if  $p$ is odd},\\
				(-\sqrt{-1})(-1)^{|\beta_{\mt}|_{>i}+|\alpha_{\mt}|_{\leq {d(\mt,\mt^{\undla})}(i_{p-1})}}&\\
				\qquad\qquad\qquad\cdot C^{\beta_{\mt}+e_{i+1}}C^{\alpha_{\mt}+e_{{d(\mt,\mt^{\undla})}(i_{p-1})}}v_{\mt}, &\text{ if  $p$ is even};
			\end{cases}\nonumber
			$$
			\item
			if $i+1={d(\mt,\mt^{\undla})}(i_p)\in \mathcal{D}_{\mt}, i\in [n]\setminus \mathcal{D}_{\mt},$ $$R(i,\beta_\mt, \alpha_\mt, \mt):=\begin{cases}(-1)^{|\beta_{\mt}|_{\geq i}+|\alpha_{\mt}|_{<i+1}} C^{\beta_{\mt}+e_i}C^{\alpha_{\mt}+e_{i+1}}v_{\mt}, & \text{ if  $p$ is odd},\\
				(-\sqrt{-1})(-1)^{|\beta_{\mt}|_{\geq i}+|\alpha_{\mt}|_{\leq {d(\mt,\mt^{\undla})}(i_{p-1})}}&\\
				\qquad\qquad\qquad \cdot C^{\beta_{\mt}+e_i}C^{\alpha_{\mt}+e_{{d(\mt,\mt^{\undla})}(i_{p-1})}}v_{\mt}, &\text{ if $p$ is even};\end{cases}  \nonumber
			$$
			\item
			if $i={d(\mt,\mt^{\undla})}(i_p), i+1={d(\mt,\mt^{\undla})}(i_{p'})\in \mathcal{D}_{\mt},$ \begin{equation*}
				R(i,\beta_\mt, \alpha_\mt, \mt):=\begin{cases}(-1)^{\delta_{\alpha_{\mt}}(i)} C^{\beta_{\mt}}C^{\alpha_{\mt}+e_{i}+e_{i+1}}v_{\mt}, & \text{ if  $p,p'$ are odd},\\
					(-\sqrt{-1})(-1)^{|\alpha_{\mt}|_{<i+1}+|\alpha_{\mt}+e_{i+1}|_{\leq{d(\mt,\mt^{\undla})}(i_{p-1})}} &\\
					\qquad\qquad\qquad \cdot C^{\beta_{\mt}}C^{\alpha_{\mt}+e_{{d(\mt,\mt^{\undla})}(i_{p-1})}+e_{i+1}}v_{\mt},& \text{ if  $p$ is even, $p'$ is odd},\\
					(-\sqrt{-1})(-1)^{|\alpha_{\mt}|_{\leq {d(\mt,\mt^{\undla})}(i_{p'-1})}+|\alpha_{\mt}+e_{{d(\mt,\mt^{\undla})}(i_{p'-1})}|_{<i}}&\\
					\qquad\qquad\qquad \cdot C^{\beta_{\mt}}C^{\alpha_{\mt}+e_{i}+e_{{d(\mt,\mt^{\undla})}(i_{p'-1})}}v_{\mt}, &\text{ if  $p$ is odd, $p'$ is even},\\
					(-1)^{1+|\alpha_{\mt}|_{\leq {d(\mt,\mt^{\undla})}(i_{p'-1})}+|\alpha_{\mt}+e_{{d(\mt,\mt^{\undla})}(i_{p'-1})}|_{\leq{d(\mt,\mt^{\undla})}(i_{p-1})}}&\\
					\qquad\qquad\qquad\cdot C^{\beta_{\mt}}C^{\alpha_{\mt}+e_{{d(\mt,\mt^{\undla})}(i_{p-1})}+e_{{d(\mt,\mt^{\undla})}(i_{p'-1})}}v_{\mt}, &\text{ if  $p,p'$ are even}.\end{cases}
			\end{equation*}
		\end{enumerate}
		\end{defn}
		The following is the first main result of this paper.		
		\begin{thm}\label{actions of generators on L basis} The simple $\mHfcn$-supermodule $\mathbb{D}(\undla)$ has a $\mathbb{K}$-basis of the form
$$\bigsqcup_{\mt\in \Std(\undla)}\Biggl\{C^{\beta_{\mt}}C^{\alpha_{\mt}}v_{\mt}\biggm|\begin{matrix}\beta_{\mt} \in \Z_2([n]\setminus \mathcal{D}_{\mt})  \\
					\alpha_{\mt}\in \Z_2(\mathcal{OD}_{\mt})
					\end{matrix}\Biggr\}$$
such that
				$$ \mathbb{L}_{\mt}:=\Biggl\{C^{\beta_{\mt}}C^{\alpha_{\mt}}v_{\mt}\biggm|\begin{matrix}\beta_{\mt} \in  \Z_2([n]\setminus \mathcal{D}_{\mt} ) \\
						\alpha_{\mt}\in \Z_2(\mathcal{OD}_{\mt})
					\end{matrix}\Biggr\}$$ forms a $\mathbb{K}$-basis of $\mathbb{L}(\res(\mathfrak{t}^{\undla}))^{{d(\mt,\mt^{\undla})}}$ and the actions of generators of $\mHfcn$ are given in the following.
					
			Let $\mt\in\Std(\undla),\,\beta_{\mt}\in\Z_2([n]\setminus \mathcal{D}_{\mt})$ and $\alpha_{\mt}\in \Z_2(\mathcal{OD}_{\mt})$.
			\begin{enumerate}
				\item For each $i\in [n],$ we have
				\begin{align}\label{X eigenvalues}
					X_i\cdot C^{\beta_{\mt}}C^{\alpha_{\mt}}v_{\mt}
					=\mathtt{b}_{\mt,i}^{-\nu_{\beta_{\mt}}(i)}C^{\beta_{\mt}}C^{\alpha_{\mt}}v_{\mt}.
				\end{align}
				\item We have $\gamma_{\mt} v_{\mt}=v_{\mt}.$
				\item For each $i\in [n],$ we have
				\begin{align}
					&C_i\cdot C^{\beta_{\mt}}C^{\alpha_{\mt}}v_{\mt} \label{Caction}\\
					&=\begin{cases}
						(-1)^{|\beta_{\mt}|_{<i}} C^{\beta_{\mt}+e_i}C^{\alpha_{\mt}}v_{\mt}, & \text{ if } i\in [n]\setminus \mathcal{D}_{\mt}, \\
						(-1)^{|\beta_{\mt}|+|\alpha_{\mt}|_{<i}} C^{\beta_{\mt}}C^{\alpha_{\mt}+e_{i}}v_{\mt}, & \text{ if  $i={d(\mt,\mt^{\undla})}(i_p) \in \mathcal{D}_{\mt}$ ,where $p$ is odd},\\
								(-\sqrt{-1})(-1)^{|\beta_{\mt}|+|\alpha_{\mt}|_{\leq {d(\mt,\mt^{\undla})}(i_{p-1})}}&\\
							\qquad\qquad	\cdot C^{\beta_{\mt}}C^{\alpha_{\mt}+e_{{d(\mt,\mt^{\undla})}(i_{p-1})}}v_{\mt}, &\text{ if  $i={d(\mt,\mt^{\undla})}(i_p) \in \mathcal{D}_{\mt}$ ,where $p$ is even}.\nonumber
					\end{cases}
				\end{align}
				\item For each $i\in [n-1],$ we have
				\begin{align}
					&T_i\cdot C^{\beta_{\mt}}C^{\alpha_{\mt}}v_{\mt} \label{Taction Non-dege}\\
					=&-\frac{\epsilon}{\mathtt{b}_{\mt,i}^{-\nu_{\beta_{\mt}}(i)}\mathtt{b}_{\mt,i+1}^{\nu_{\beta_{\mt}}(i+1)}-1} C^{\beta_{\mt}}C^{\alpha_{\mt}}v_{\mt} \nonumber\\
					&	\qquad\qquad  +\frac{\epsilon}{\mathtt{b}_{\mt,i}^{\nu_{\beta_{\mt}}(i)}\mathtt{b}_{\mt,i+1}^{\nu_{\beta_{\mt}}(i+1)}-1}R(i,\beta_\mt, \alpha_\mt, \mt)\nonumber\\
					&	\qquad\qquad\qquad +\delta(s_i\mt)(-1)^{\delta_{\beta_{\mt}}(i)\delta_{\beta_{\mt}}(i+1)+\delta_{\alpha_{\mt}}(i)\delta_{\alpha_{\mt}}(i+1)}\sqrt{\mathtt{c}_{\mt}(i)}C^{s_i \cdot \beta_{\mt}}C^{s_i \cdot \alpha_{\mt}}v_{s_i\mt}.\nonumber
				\end{align}

			\end{enumerate}
		\end{thm}
		
		\begin{proof}

			By Proposition \ref{actions of X,C}, the space $\mathbb{L}(\res(\mathfrak{t}^{\undla}))$ has a $\mathbb{K}$-basis
			$$\Biggl\{C^{\beta_{\mt^{\undla}}}C^{\alpha_{\mt^{\undla}}}v_{\mt^{\undla}}\biggm|\begin{matrix}\beta_{\mt^{\undla}} \in \Z_2([n]\setminus \mathcal{D}_{\mt^{\undla}}) \\
				\alpha_{\mt^{\undla}}\in \Z_2(\mathcal{OD}_{\mt^{\undla}})
			\end{matrix}\Biggr\}$$ and the space $\mathbb{L}(\res(\mathfrak{t}^{\undla}))^{d(\mt,\mt^{\undla})}$ has a $\mathbb{K}$-basis 	$$\Biggl\{C^{\beta_{\mt}}C^{\alpha_{\mt}}v_{\mt^{\undla}}^{d(\mt,\mt^{\undla})}\biggm|\begin{matrix}\beta_{\mt} \in \Z_2([n]\setminus \mathcal{D}_{\mt}) \\
			\alpha_{\mt}\in \Z_2(\mathcal{OD}_{\mt})
			\end{matrix}\Biggr\}.$$ For $\mt\in\Std(\undla)$,  we set $$v_{\mt}:=v_{\mt^{\undla}}^{d(\mt,\mt^{\undla})}\in \mathbb{L}(\res(\mathfrak{t}^{\undla}))^{d(\mt,\mt^{\undla})}.
			$$  This gives rise to the required basis in our statement.  Now we check the action formulae.
				\begin{enumerate}
		\item	Note that for $\mt\in\Std(\undla)$, we have
			$$\mathtt{b}_-\biggl(\res_{\mt^{\undla}}({d(\mt,\mt^{\undla})}^{-1}(i))\biggr)=\mathtt{b}_-(\res_{\mt}(i))=\mathtt{b}_{\mt,i}.$$ Hence the action formulae (1) is a direct application of Proposition \ref{actions of X,C} (1).
			\item  This is Proposition \ref{actions of X,C} (2).
				\item  This is Proposition \ref{actions of X,C} (3).
			\item	From now on, we fix some $\mt\in\Std(\undla)$.
			
				By \eqref{actionformulaNon-dege}, we have
					\begin{align}\label{Ti acts on vt}
							T_iv_{\mt}= \left \{
							\begin{array}{ll}
									\widetilde{\Xi}_i v_{\mt}
									+\widetilde{\Omega}_i v_{s_i\mt},
									& \text{ if } s_i\mt \in \Std(\undla), \\
									\widetilde{\Xi}_i v_{\mt}
									, & \text{ otherwise},
								\end{array}
							\right.
						\end{align}
					Recall the definition of $\Phi_i(x,y)$ in \eqref{Phi function}. Using action formulae (1) and \eqref{Ti acts on vt}, we have
					\begin{align}\label{Phi-action}
							\Phi_i(\mathtt{b}_{\mt,i},\mathtt{b}_{\mt,i+1})v_{\mt}=	\delta(s_i\mt)\sqrt{\mathtt{c}_{\mt}(i)}v_{s_i\mt}.
						\end{align}
		For each $i\in [n-1],$ we have
			\begin{align}
				&\Phi_i(\mathtt{b}_{\mt,i}^{\nu_{\beta_{\mt}}(i)},\mathtt{b}_{\mt,i+1}^{\nu_{\beta_{\mt}}(i+1)})\cdot C^{\beta_{\mt}}C^{\alpha_{\mt}}v_{\mt} \nonumber\\
				&=(-1)^{\delta_{\beta_{\mt}}(i)\delta_{\beta_{\mt}}(i+1)+\delta_{\alpha_{\mt}}(i)\delta_{\alpha_{\mt}}(i+1)}C^{s_i \cdot \beta_{\mt}}C^{s_i \cdot \alpha_{\mt}}\Phi_i(\mathtt{b}_{\mt,i},\mathtt{b}_{\mt,i+1})v_{\mt} \nonumber\\
				&=\delta(s_i\mt)(-1)^{\delta_{\beta_{\mt}}(i)\delta_{\beta_{\mt}}(i+1)+\delta_{\alpha_{\mt}}(i)\delta_{\alpha_{\mt}}(i+1)}\sqrt{\mathtt{c}_{\mt}(i)}C^{s_i \cdot \beta_{\mt}}C^{s_i \cdot \alpha_{\mt}}v_{s_i\mt}, \nonumber
			\end{align}
			where in the first equation we have used \eqref{Phi and C} and in the last equation we have used \eqref{Phi-action}.
			So we have
		\begin{align}
				&T_i\cdot C^{\beta_{\mt}}C^{\alpha_{\mt}}v_{\mt} \label{ActionT 1}\\
				=&-\frac{\epsilon}{\mathtt{b}_{\mt,i}^{-\nu_{\beta_{\mt}}(i)}\mathtt{b}_{\mt,i+1}^{\nu_{\beta_{\mt}}(i+1)}-1} C^{\beta_{\mt}}C^{\alpha_{\mt}}v_{\mt} \nonumber\\
				&+\frac{\epsilon}{\mathtt{b}_{\mt,i}^{\nu_{\beta_{\mt}}(i)}\mathtt{b}_{\mt,i+1}^{\nu_{\beta_{\mt}}(i+1)}-1} C_iC_{i+1}C^{\beta_{\mt}}C^{\alpha_{\mt}} v_{\mt}\nonumber\\
				&+\delta(s_i\mt)(-1)^{\delta_{\beta_{\mt}}(i)\delta_{\beta_{\mt}}(i+1)+\delta_{\alpha_{\mt}}(i)\delta_{\alpha_{\mt}}(i+1)}\sqrt{\mathtt{c}_{\mt}(i)}C^{s_i \cdot \beta_{\mt}}C^{s_i \cdot \alpha_{\mt}}v_{s_i\mt}\nonumber.
				\end{align}

			  \begin{enumerate}
					\item If $i,i+1\in [n]\setminus \mathcal{D}_{\mt},$ then
			\begin{align}
				&C_iC_{i+1}\cdot C^{\beta_{\mt}}C^{\alpha_{\mt}}v_{\mt} \nonumber\\
				=&C_i \cdot (-1)^{|\beta_{\mt}|_{<i+1}}C^{\beta_{\mt}+e_{i+1}}C^{\alpha_{\mt}}v_{\mt} \nonumber\\
				=&(-1)^{|\beta_{\mt}|_{<i+1}+|\beta_{\mt}|_{<i}}C^{\beta_{\mt}+e_i+e_{i+1}}C^{\alpha_{\mt}}v_{\mt} \nonumber\\
				=&(-1)^{\delta_{\beta_{\mt}}(i)}C^{\beta_{\mt}+e_i+e_{i+1}}C^{\alpha_{\mt}}v_{\mt}.\nonumber
			\end{align}
			
			\item
			If $i={d(\mt,\mt^{\undla})}(i_p)\in \mathcal{D}_{\mt}, i+1\in [n]\setminus \mathcal{D}_{\mt},$ then
			\begin{align*}
				&C_iC_{i+1}\cdot C^{\beta_{\mt}}C^{\alpha_{\mt}}v_{\mt} \nonumber\\
				&=C_i \cdot (-1)^{|\beta_{\mt}|_{<i+1}}C^{\beta_{\mt}+e_{i+1}}C^{\alpha_{\mt}}v_{\mt} \nonumber\\
				&=\begin{cases}(-1)^{|\beta_{\mt}|_{<i+1}+|\beta_{\mt}|+1+|\alpha_{\mt}|_{<{d(\mt,\mt^{\undla})}(i_p)}}C^{\beta_{\mt}+e_{i+1}}C^{\alpha_{\mt}+e_{{d(\mt,\mt^{\undla})}(i_p)}}v_{\mt}, & \text{ if   $p$ is odd},\\
					(-\sqrt{-1})(-1)^{|\beta_{\mt}|_{<i+1}+|\beta_{\mt}|+1+|\alpha_{\mt}|_{\leq {d(\mt,\mt^{\undla})}(i_{p-1})}}&\\
					\qquad\qquad\qquad\cdot C^{\beta_{\mt}+e_{i+1}}C^{\alpha_{\mt}+e_{{d(\mt,\mt^{\undla})}(i_{p-1})}}v_{\mt}, &\text{ if   $p$ is even},
					\end{cases}\\
				&=\begin{cases}(-1)^{|\beta_{\mt}|_{>i}+|\alpha_{\mt}|_{<{d(\mt,\mt^{\undla})}(i_p)}}&\\
					\qquad\qquad\qquad\cdot C^{\beta_{\mt}+e_{i+1}}C^{\alpha_{\mt}+e_{{d(\mt,\mt^{\undla})}(i_p)}}v_{\mt}, & \text{ if  $p$ is odd},\\
					(-\sqrt{-1})(-1)^{|\beta_{\mt}|_{>i}+|\alpha_{\mt}|_{\leq {d(\mt,\mt^{\undla})}(i_{p-1})}}&\\
					\qquad\qquad\qquad\cdot C^{\beta_{\mt}+e_{i+1}}C^{\alpha_{\mt}+e_{{d(\mt,\mt^{\undla})}(i_{p-1})}}v_{\mt}, &\text{ if  $p$ is even},
				\end{cases}\\
				&=\begin{cases}(-1)^{|\beta_{\mt}|_{>i}+|\alpha_{\mt}|_{<i}}C^{\beta_{\mt}+e_{i+1}}C^{\alpha_{\mt}+e_{i}}v_{\mt}, & \text{ if  $p$ is odd},\\
					(-\sqrt{-1})(-1)^{|\beta_{\mt}|_{>i}+|\alpha_{\mt}|_{\leq {d(\mt,\mt^{\undla})}(i_{p-1})}}&\\
					\qquad\qquad\qquad\cdot C^{\beta_{\mt}+e_{i+1}}C^{\alpha_{\mt}+e_{{d(\mt,\mt^{\undla})}(i_{p-1})}}v_{\mt}, &\text{ if  $p$ is even}.
				\end{cases}\nonumber
			\end{align*}
			
		\item
			If $i+1={d(\mt,\mt^{\undla})}(i_p)\in \mathcal{D}_{\mt}, i\in [n]\setminus \mathcal{D}_{\mt},$ then
			\begin{align*}
				&C_iC_{i+1}\cdot C^{\beta_{\mt}}C^{\alpha_{\mt}}v_{\mt} \nonumber\\
				&=\begin{cases}(-1)^{|\beta_{\mt}|+|\alpha_{\mt}|_{<{d(\mt,\mt^{\undla})}(i_p)}} &\\
					\qquad\qquad\qquad\cdot C_iC^{\beta_{\mt}}C^{\alpha_{\mt}+e_{{d(\mt,\mt^{\undla})}(i_p)}}v_{\mt}, & \text{ if  $p$ is odd},\\
				(-\sqrt{-1})(-1)^{|\beta_{\mt}|+|\alpha_{\mt}|_{\leq {d(\mt,\mt^{\undla})}(i_{p-1})}}&\\
				\qquad\qquad\qquad\cdot C_iC^{\beta_{\mt}}C^{\alpha_{\mt}+e_{{d(\mt,\mt^{\undla})}(i_{p-1})}}v_{\mt}, &\text{ if  $p$ is even},\end{cases}\\
				&=\begin{cases}(-1)^{|\beta_{\mt}|+|\alpha_{\mt}|_{<{d(\mt,\mt^{\undla})}(i_p)}+|\beta_\mt|_{<i}} &\\
					\qquad\qquad\qquad\cdot C^{\beta_{\mt}+e_i}C^{\alpha_{\mt}+e_{{d(\mt,\mt^{\undla})}(i_p)}}v_{\mt}, & \text{ if $p$ is odd},\\
				(-\sqrt{-1})(-1)^{|\beta_{\mt}|+|\alpha_{\mt}|_{\leq {d(\mt,\mt^{\undla})}(i_{p-1})}+|\beta_\mt|_{<i}}&\\
				\qquad\qquad\qquad\cdot C^{\beta_{\mt}+e_i}C^{\alpha_{\mt}+e_{{d(\mt,\mt^{\undla})}(i_{p-1})}}v_{\mt}, &\text{ if  $p$ is even},\end{cases}\\
					&=\begin{cases}(-1)^{|\beta_{\mt}|_{\geq i}+|\alpha_{\mt}|_{<{d(\mt,\mt^{\undla})}(i_p)}} &\\
						\qquad\qquad\qquad\cdot C^{\beta_{\mt}+e_i}C^{\alpha_{\mt}+e_{{d(\mt,\mt^{\undla})}(i_p)}}v_{\mt}, & \text{ if  $p$ is odd},\\
					(-\sqrt{-1})(-1)^{|\beta_{\mt}|_{\geq i}+|\alpha_{\mt}|_{\leq {d(\mt,\mt^{\undla})}(i_{p-1})}}&\\
					\qquad\qquad\qquad\cdot C^{\beta_{\mt}+e_i}C^{\alpha_{\mt}+e_{{d(\mt,\mt^{\undla})}(i_{p-1})}}v_{\mt}, &\text{ if $p$ is even},\end{cases}\\
						&=\begin{cases}(-1)^{|\beta_{\mt}|_{\geq i}+|\alpha_{\mt}|_{<i+1}} &\\
							\qquad\qquad\qquad\cdot C^{\beta_{\mt}+e_i}C^{\alpha_{\mt}+e_{i+1}}v_{\mt}, & \text{ if  $p$ is odd},\\
						(-\sqrt{-1})(-1)^{|\beta_{\mt}|_{\geq i}+|\alpha_{\mt}|_{\leq {d(\mt,\mt^{\undla})}(i_{p-1})}}&\\
						\qquad\qquad\qquad\cdot C^{\beta_{\mt}+e_i}C^{\alpha_{\mt}+e_{{d(\mt,\mt^{\undla})}(i_{p-1})}}v_{\mt}, &\text{ if $p$ is even}.\end{cases}  \nonumber
			\end{align*}
\item
If $i={d(\mt,\mt^{\undla})}(i_p), i+1={d(\mt,\mt^{\undla})}(i_{p'})\in \mathcal{D}_{\mt},$ then
\begin{align*}
	&C_iC_{i+1}\cdot C^{\beta_{\mt}}C^{\alpha_{\mt}}v_{\mt} \nonumber\\
	&=\begin{cases}(-1)^{|\beta_{\mt}|+|\alpha_{\mt}|_{<{d(\mt,\mt^{\undla})}(i_{p'})}}  &\\
		\qquad\qquad\qquad\cdot C_iC^{\beta_{\mt}}C^{\alpha_{\mt}+e_{{d(\mt,\mt^{\undla})}(i_{p'})}}v_{\mt}, & \text{ if  $p'$ is odd},\\
		(-\sqrt{-1})(-1)^{|\beta_{\mt}|+|\alpha_{\mt}|_{\leq {d(\mt,\mt^{\undla})}(i_{p'-1})}}&\\
		\qquad\qquad\qquad\cdot C_iC^{\beta_{\mt}}C^{\alpha_{\mt}+e_{{d(\mt,\mt^{\undla})}(i_{p'-1})}}v_{\mt}, &\text{ if  $p'$ is even},\end{cases}\\
		&=\begin{cases}(-1)^{|\beta_{\mt}|+|\alpha_{\mt}|_{<{d(\mt,\mt^{\undla})}(i_{p'})}}&\\
			\qquad\qquad\cdot(-1)^{|\beta_{\mt}|+|\alpha_{\mt}+e_{{d(\mt,\mt^{\undla})}(i_{p'})}|_{<{d(\mt,\mt^{\undla})}(i_{p})}}&\\
			\qquad\qquad\qquad\cdot C^{\beta_{\mt}}C^{\alpha_{\mt}+e_{{d(\mt,\mt^{\undla})}(i_{p})}+e_{{d(\mt,\mt^{\undla})}(i_{p'})}}v_{\mt}, & \text{ if  $p,p'$ are odd},\\
				(-\sqrt{-1})(-1)^{|\beta_{\mt}|+|\alpha_{\mt}|_{<{d(\mt,\mt^{\undla})}(i_{p'})}}&\\
				\qquad\qquad\cdot (-1)^{|\beta_{\mt}|+|\alpha_{\mt}+e_{{d(\mt,\mt^{\undla})}(i_{p'})}|_{\leq{d(\mt,\mt^{\undla})}(i_{p-1})}}&\\
				\qquad\qquad\qquad\cdot C^{\beta_{\mt}}C^{\alpha_{\mt}+e_{{d(\mt,\mt^{\undla})}(i_{p-1})}+e_{{d(\mt,\mt^{\undla})}(i_{p'})}}v_{\mt},& \text{ if  $p$ is even, $p'$ is odd},\\
			(-\sqrt{-1})(-1)^{|\beta_{\mt}|+|\alpha_{\mt}|_{\leq {d(\mt,\mt^{\undla})}(i_{p'-1})}}&\\
			\qquad\qquad\cdot (-1)^{|\beta_{\mt}|+|\alpha_{\mt}+e_{{d(\mt,\mt^{\undla})}(i_{p'-1})}|_{<{d(\mt,\mt^{\undla})}(i_{p})}}&\\
			\qquad\qquad\qquad\cdot C^{\beta_{\mt}}C^{\alpha_{\mt}+e_{{d(\mt,\mt^{\undla})}(i_p)}+e_{{d(\mt,\mt^{\undla})}(i_{p'-1})}}v_{\mt}, &\text{ if  $p$ is odd, $p'$ is even},\\
			(-1)^{1+|\beta_{\mt}|+|\alpha_{\mt}|_{\leq {d(\mt,\mt^{\undla})}(i_{p'-1})}}&\\
			\qquad\qquad\cdot(-1)^{|\beta_{\mt}|+|\alpha_{\mt}+e_{{d(\mt,\mt^{\undla})}(i_{p'-1})}|_{\leq{d(\mt,\mt^{\undla})}(i_{p-1})}}&\\
			\qquad\qquad\qquad\cdot C^{\beta_{\mt}}C^{\alpha_{\mt}+e_{{d(\mt,\mt^{\undla})}(i_{p-1})}+e_{{d(\mt,\mt^{\undla})}(i_{p'-1})}}v_{\mt},&\text{ if  $p,p'$ are even},\end{cases}\\
				&=\begin{cases}(-1)^{|\alpha_{\mt}|_{<{d(\mt,\mt^{\undla})}(i_{p'})}+|\alpha_{\mt}+e_{{d(\mt,\mt^{\undla})}(i_{p'})}|_{<{d(\mt,\mt^{\undla})}(i_{p})}} &\\
					\qquad\qquad\qquad\cdot C^{\beta_{\mt}}C^{\alpha_{\mt}+e_{{d(\mt,\mt^{\undla})}(i_{p})}+e_{{d(\mt,\mt^{\undla})}(i_{p'})}}v_{\mt}, & \text{ if  $p,p'$ are odd},\\
				(-\sqrt{-1})(-1)^{|\alpha_{\mt}|_{<{d(\mt,\mt^{\undla})}(i_{p'})}+|\alpha_{\mt}+e_{{d(\mt,\mt^{\undla})}(i_{p'})}|_{\leq{d(\mt,\mt^{\undla})}(i_{p-1})}} &\\
				\qquad\qquad\qquad\cdot C^{\beta_{\mt}}C^{\alpha_{\mt}+e_{{d(\mt,\mt^{\undla})}(i_{p-1})}+e_{{d(\mt,\mt^{\undla})}(i_{p'})}}v_{\mt},& \text{ if  $p$ is even, $p'$ is odd},\\
				(-\sqrt{-1})(-1)^{|\alpha_{\mt}|_{\leq {d(\mt,\mt^{\undla})}(i_{p'-1})}+|\alpha_{\mt}+e_{{d(\mt,\mt^{\undla})}(i_{p'-1})}|_{<{d(\mt,\mt^{\undla})}(i_{p})}}&\\
				\qquad\qquad\qquad\cdot C^{\beta_{\mt}}C^{\alpha_{\mt}+e_{{d(\mt,\mt^{\undla})}(i_p)}+e_{{d(\mt,\mt^{\undla})}(i_{p'-1})}}v_{\mt}, &\text{ if  $p$ is odd, $p'$ is even},\\
				(-1)^{1+|\alpha_{\mt}|_{\leq {d(\mt,\mt^{\undla})}(i_{p'-1})}+|\alpha_{\mt}+e_{{d(\mt,\mt^{\undla})}(i_{p'-1})}|_{\leq{d(\mt,\mt^{\undla})}(i_{p-1})}}&\\
				\qquad\qquad\qquad\cdot C^{\beta_{\mt}}C^{\alpha_{\mt}+e_{{d(\mt,\mt^{\undla})}(i_{p-1})}+e_{{d(\mt,\mt^{\undla})}(i_{p'-1})}}v_{\mt}, &\text{ if  $p,p'$ are even},\end{cases}\\
					&=\begin{cases}(-1)^{\delta_{\alpha_{\mt}}(i)} C^{\beta_{\mt}}C^{\alpha_{\mt}+e_{i}+e_{i+1}}v_{\mt}, & \text{ if  $p,p'$ are odd},\\
					(-\sqrt{-1})(-1)^{|\alpha_{\mt}|_{<i+1}+|\alpha_{\mt}+e_{i+1}|_{\leq{d(\mt,\mt^{\undla})}(i_{p-1})}} &\\
					\qquad\qquad\qquad\cdot C^{\beta_{\mt}}C^{\alpha_{\mt}+e_{{d(\mt,\mt^{\undla})}(i_{p-1})}+e_{i+1}}v_{\mt},& \text{ if  $p$ is even, $p'$ is odd},\\
					(-\sqrt{-1})(-1)^{|\alpha_{\mt}|_{\leq {d(\mt,\mt^{\undla})}(i_{p'-1})}+|\alpha_{\mt}+e_{{d(\mt,\mt^{\undla})}(i_{p'-1})}|_{<i}}&\\
					\qquad\qquad\qquad\cdot C^{\beta_{\mt}}C^{\alpha_{\mt}+e_{i}+e_{{d(\mt,\mt^{\undla})}(i_{p'-1})}}v_{\mt}, &\text{ if  $p$ is odd, $p'$ is even},\\
					(-1)^{1+|\alpha_{\mt}|_{\leq {d(\mt,\mt^{\undla})}(i_{p'-1})}+|\alpha_{\mt}+e_{{d(\mt,\mt^{\undla})}(i_{p'-1})}|_{\leq{d(\mt,\mt^{\undla})}(i_{p-1})}}&\\
					\qquad\qquad\qquad\cdot C^{\beta_{\mt}}C^{\alpha_{\mt}+e_{{d(\mt,\mt^{\undla})}(i_{p-1})}+e_{{d(\mt,\mt^{\undla})}(i_{p'-1})}}v_{\mt}, &\text{ if  $p,p'$ are even}.\end{cases}
		\end{align*}
		\end{enumerate}		
		\end{enumerate} These, combined with \eqref{ActionT 1}, complete the proof of (4) and the whole Proposition.
			\end{proof}
			
In the end of this subsection, we introduce some notations which will be used later.
\label{pag:decomposition of OD}
\begin{defn}\label{decomposition of OD}
For any $\mt\in\Std(\undla),$ let ${d(\mt,\mt^{\undla})}\in P(\undla)$ such that $\mt={d(\mt,\mt^{\undla})}\mt^{\undla}$. We define
\begin{align}
\mathbb{Z}_2(\mathcal{OD}_{\mt})_{\bar{0}}:=\{\alpha_{\mt} \in\Z_2(\mathcal{OD}_{\mt}) \mid d(\mt,\mt^{\undla})(i_t)\notin \supp(\alpha_{\mt})\},\nonumber\\
\mathbb{Z}_2(\mathcal{OD}_{\mt})_{\bar{1}}:=\{\alpha_{\mt} \in\Z_2(\mathcal{OD}_{\mt}) \mid d(\mt,\mt^{\undla})(i_t)\in \supp(\alpha_{\mt})\}.\nonumber
\end{align}
\end{defn}

That is, if $d_{\undla}=0$ (i.e., t is even), then $\Z_2(\mathcal{OD}_{\mt})_{\bar{0}}=\Z_2(\mathcal{OD}_{\mt})$ and $\Z_2(\mathcal{OD}_{\mt})_{\bar{1}}=\emptyset;$  if $d_{\undla}=1$ (i.e., t is odd), then  $\Z_2(\mathcal{OD}_{\mt})_{\bar{0}}$ and $\Z_2(\mathcal{OD}_{\mt})_{\bar{1}}$ are both non-empty and there is a natural bijection between $\Z_2(\mathcal{OD}_{\mt})_{\bar{0}}$ and $\Z_2(\mathcal{OD}_{\mt})_{\bar{1}}$ which sends $\alpha_{\mt}\in \Z_2(\mathcal{OD}_{\mt})_{\bar{0}}$ to $\alpha_{\mt}+e_{{d(\mt,\mt^{\undla})}(i_t)}\in \Z_2(\mathcal{OD}_{\mt})_{\bar{1}}.$ In particular, we have
 $$ \Z_2(\mathcal{OD}_{\mt})=\Z_2(\mathcal{OD}_{\mt})_{\bar{0}}\sqcup \Z_2(\mathcal{OD}_{\mt})_{\bar{1}}.$$ For any $\alpha_{\mt} \in \Z_2(\mathcal{OD}_{\mt})_{\bar{0}},$ we use $\alpha_{\mt,\bar{0}}=\alpha_{\mt}$  to emphasize that $\alpha_{\mt}\in Z_2(\mathcal{OD}_{\mt})_{\bar{0}}$ and  if $d_{\undla}=1$, we define $\alpha_{\mt,\bar{1}}:=\alpha_{\mt}+e_{{d(\mt,\mt^{\undla})}(i_t)}\in \Z_2(\mathcal{OD}_{\mt})_{\bar{1}}.$
			
\begin{defn}\label{hat}
 Let $\mt\in \Std(\undla)$ and ${d(\mt,\mt^{\undla})}\in P(\undla)$ such that $\mt={d(\mt,\mt^{\undla})}\cdot \mt^{\undla}.$

\label{pag:widehat}
(1) For any $\alpha_{\mt}\in \mathbb{Z}_2 (\mathcal{OD}_{\mt})_{\bar{0}},$ we define
				$\widehat{\alpha_{\mt}}\in \mathbb{Z}_2 (\mathcal{OD}_{\mt})_{\bar{0}}$ to be the unique vector such that
				$$\supp(\alpha_{\mt})\sqcup \supp(\widehat{\alpha_{\mt}})=\{{d(\mt,\mt^{\undla})}(i_1),{d(\mt,\mt^{\undla})}(i_3),\cdots,{d(\mt,\mt^{\undla})}(i_{2\lfloor \frac{t}{2} \rfloor-1})\}.$$

(2) For any $\alpha_{\mt,\bar{1}}\in \mathbb{Z}_2 (\mathcal{OD}_{\mt})_{\bar{1}},$ we define
				$\widehat{\alpha_{\mt,\bar{1}}}\in \mathbb{Z}_2 (\mathcal{OD}_{\mt})_{\bar{1}}$ to be the unique vector such that
				$$\supp(\alpha_{\mt})\sqcup \supp(\widehat{\alpha_{\mt,\bar{1}}})=\{{d(\mt,\mt^{\undla})}(i_1),{d(\mt,\mt^{\undla})}(i_3),\cdots,{d(\mt,\mt^{\undla})}(i_{2\lceil \frac{t}{2} \rceil-1})\}.$$
\end{defn}

	\subsection{Primitive idempotents of $\mHfcn$}\label{primitiveidem}
	In this subsection, we shall use Proposition \ref{actions of generators on L basis} to give a complete set of (super) primitive idempotents of $\mHfcn$.
		
	The following definition is important in our construction of idempotents.
    \label{pag:Tri}
	\begin{defn}\label{Tri}
For $a\in \Z_2,\,\undla\in\mathscr{P}^{\bullet,m}_{n}$ with $\bullet\in\{\mathsf{0},\mathsf{s},\mathsf{ss}\},$  we define
		$${\rm Tri}_{a}(\undla):=\bigsqcup_{\mt\in \Std(\undla)}\{\mt \}\times  \Z_2(\mathcal{OD}_{\mt})_{a}\times  \Z_2([n]\setminus \mathcal{D}_{\mt}),$$
		and $${\rm Tri}(\undla):={\rm Tri}_{\bar{0}}(\undla)\sqcup {\rm Tri}_{\bar{1}}(\undla).$$
		\end{defn}

		Notice that ${\rm Tri}(\undla)={\rm Tri}_{0}(\undla)$ when $d_{\undla}=0.$ For any ${\rm T}=(\mt, \alpha_{\mt}, \beta_{\mt})\in {\rm Tri}_{\bar{0}}(\undla),$ we denote
		$${\rm T}_{a}=(\mt, \alpha_{\mt,a}, \beta_{\mt})\in {\rm Tri}_{a}(\undla),\quad a\in \mathbb{Z}_{2},$$
		when $d_{\undla}=1.$

Now we can define primitive idempotents.
\label{pag:primitive idempotents and blocks}
		\begin{defn}\label{primitive idempotents and blocks}
          For $k\in[n]$, let
         $$\mathtt{B}(k):=\{ \mathtt{b}_{\pm}(\res_{\ms}(k)) \mid \ms \in \Std(\mathscr{P}^{\bullet,m}_{n}) \}.$$
			For any ${\rm T}=(\mt, \alpha_{\mt}, \beta_{\mt})\in {\rm Tri}_{\bar{0}}(\undla),$ we define
			\begin{align}\label{definition of primitive idempotents. non-dege}
				F_{\rm T}:=\left(C^{\alpha_{\mt}}\gamma_{\mt}(C^{\alpha_{\mt}})^{-1} \right) \cdot \left(\prod_{k=1}^{n}\prod_{\mathtt{b}\in \mathtt{B}(k)\atop \mathtt{b}\neq \mathtt{b}_{+}(\res_{\mt}(k))}\frac{X_k^{\nu_{\beta_{\mt}}(k)}-\mathtt{b}}{\mathtt{b}_{+}(\res_{\mt}(k))-\mathtt{b}}\right)\in \mHfcn.
			\end{align}
We define
\begin{align}
				F_{\undla}&:=\sum_{{\rm T}\in {\rm Tri}_{\bar{0}}(\undla)} F_{\rm T},
			\end{align}
and two left ideals of $\mHfcn$ as follows
\label{pag:nondege simple block}
	\begin{align}
				\mathbb{D}_{\rm T}&:=\mHfcn F_{\rm T}\subseteq\mHfcn,\\
				B_{\undla}&:=\mHfcn  F_{\undla}\subseteq \mHfcn.
			\end{align}
		\end{defn}

\begin{defn}
For $a\in \Z_2,$ we denote
	$${\rm Tri}_{a}(\mathscr{P}^{\bullet,m}_{n}):=\bigsqcup_{\undla\in \mathscr{P}^{\bullet,m}_{n}}{\rm Tri}_{a}(\undla),$$
	and $${\rm Tri}(\mathscr{P}^{\bullet,m}_{n}):={\rm Tri}_{\bar{0}}(\mathscr{P}^{\bullet,m}_{n})\sqcup {\rm Tri}_{\bar{1}}(\mathscr{P}^{\bullet,m}_{n}).$$	
\end{defn}
	
		\begin{lem}\label{idempotent action. non-dege}	Let ${\rm T}=(\mt, \alpha_{\mt}, \beta_{\mt})\in {\rm Tri}_{\bar{0}}(\undla),$ and $
			{\rm S}=(\ms, \alpha_{\ms}', \beta_{\ms}')\in {\rm Tri}(\mathscr{P}^{\bullet,m}_{n})$. We have
			\begin{align}\label{F. eq3}
				F_{\rm T}\cdot C^{\beta_{\ms}^{'}} C^{\alpha_{\ms}^{'}} v_{\ms}=
				\begin{cases}
					C^{\beta_{\ms}^{'}} C^{\alpha_{\ms}^{'}}  v_{\ms}, & \text{ if } d_{\undla}=0 \text{ and } {\rm S}={\rm T}; \\
					C^{\beta_{\ms}^{'}} C^{\alpha_{\ms}^{'}}  v_{\ms}, & \text{ if $d_{\undla}=1$ and } {\rm S}={\rm T}_{a}\text{ for some $a\in \mathbb{Z}_2$  }; \\
					0, & \text{ otherwise.}
				\end{cases}
			\end{align}
			\end{lem}
	\begin{proof}
		
		
		If $\ms\neq \mt,$ then there exists $k_0\in [n]$ such that $\mathtt{b}_{\pm}(\res_{\ms}(k_0))\neq \mathtt{b}_{+}(\res_{\mt}(k_0))$ by the second part of Lemma \ref{important condition1}. Using \eqref{X eigenvalues}, we have
		\begin{align}
			&\prod_{\mathtt{b}\in \mathtt{B}(k_0)\atop \mathtt{b}\neq \mathtt{b}_{+}(\res_{\mt}(k_0))}\frac{X_{k_0}^{\nu_{\beta_{\mt}}(k_0)}-\mathtt{b}}{\mathtt{b}_{+}(\res_{\mt}(k_0))-\mathtt{b}}
			\cdot C^{\beta_{\ms}^{'}} C^{\alpha_{\ms}^{'}} v_{\ms} \nonumber\\
			&=\prod_{\mathtt{b}\in \mathtt{B}(k_0)\atop \mathtt{b}\neq \mathtt{b}_{+}(\res_{\mt}(k_0))}\frac{\mathtt{b}_{+}(\res_{\ms}(k_0))^{\nu_{\beta_{\mt}}(k_0)\nu_{\beta_{\ms}'}(k_0)}-\mathtt{b}}{\mathtt{b}_{+}(\res_{\mt}(k_0))-\mathtt{b}}
			\cdot C^{\beta_{\ms}^{'}} C^{\alpha_{\ms}^{'}} v_{\ms}=0.\nonumber
		\end{align}
		
		If $\ms=\mt$ and $\beta_{\mt}^{'}\neq \beta_{\mt},$ there exists $k_1\in [n]\setminus \mathcal{D}_{\mt}$ such that $\nu_{\beta_{\mt}}(k_1)\nu_{\beta_{\mt}'}(k_1)=-1.$ By the first part of Lemma \ref{important condition1}, we get
		$\mathtt{b}_{+}(\res_{\mt}(k_1))\neq \mathtt{b}_{-}(\res_{\mt}(k_1)).$
	Using \eqref{X eigenvalues}, we have
		\begin{align}
			&\prod_{\mathtt{b}\in \mathtt{B}(k_1)\atop \mathtt{b}\neq \mathtt{b}_{+}(\res_{\mt}(k_1))}\frac{X_{k_1}^{\nu_{\beta_{\mt}}(k_1)}-\mathtt{b}}{\mathtt{b}_{+}(\res_{\mt}(k_1))-\mathtt{b}}
			\cdot C^{\beta_{\mt}^{'}} C^{\alpha_{\mt}^{'}}v_{\mt} \nonumber\\
			&=\prod_{\mathtt{b}\in \mathtt{B}(k_1)\atop \mathtt{b}\neq \mathtt{b}_{+}(\res_{\mt}(k_1))}\frac{\mathtt{b}_{+}(\res_{\mt}(k_1))^{\nu_{\beta_{\mt}}(k_1)\nu_{\beta_{\mt}'}(k_1)}-\mathtt{b}}{\mathtt{b}_{+}(\res_{\mt}(k_1))-\mathtt{b}}
			\cdot C^{\beta_{\mt}^{'}} C^{\alpha_{\mt}^{'}}  v_{\mt} \nonumber\\
			&=\prod_{\mathtt{b}\in \mathtt{B}(k_1)\atop \mathtt{b}\neq \mathtt{b}_{+}(\res_{\mt}(k_1))}\frac{\mathtt{b}_{-}(\res_{\mt}(k_1))-\mathtt{b}}{\mathtt{b}_{+}(\res_{\mt}(k_1))-\mathtt{b}}
			\cdot C^{\beta_{\mt}^{'}} C^{\alpha_{\mt}^{'}}  v_{\mt}=0.\nonumber
		\end{align}
		If $\ms=\mt$ and $\beta_{\mt}^{'}= \beta_{\mt},$ using \eqref{X eigenvalues} again, we have
		\begin{align}
			&\quad \prod_{k=1}^{n}\prod_{\mathtt{b}\in \mathtt{B}(k)\atop \mathtt{b}\neq \mathtt{b}_{+}(\res_{\mt}(k))}\frac{X_k^{\nu_{\beta_{\mt}}(k)}-\mathtt{b}}{\mathtt{b}_{+}(\res_{\mt}(k))-\mathtt{b}}
			\cdot C^{\beta_{\mt}} C^{\alpha_{\mt}^{'}}  v_{\mt} \nonumber\\
			&=\quad \prod_{k=1}^{n}\prod_{\mathtt{b}\in \mathtt{B}(k)\atop \mathtt{b}\neq \mathtt{b}_{+}(\res_{\mt}(k))}\frac{\mathtt{b}_{+}(\res_{\mt}(k))-\mathtt{b}}{\mathtt{b}_{+}(\res_{\mt}(k))-\mathtt{b}}
			\cdot C^{\beta_{\mt}} C^{\alpha_{\mt}^{'}}  v_{\mt} \nonumber\\
			&=C^{\beta_{\mt}} C^{\alpha_{\mt}^{'}}  v_{\mt}.\nonumber
		\end{align}
		
		To sum up,
		\begin{align}\label{F. eq1}
			\prod_{k=1}^{n}\prod_{\mathtt{b}\in \mathtt{B}(k)\atop \mathtt{b}\neq \mathtt{b}_{+}(\res_{\mt}(k))}\frac{X_k^{\nu_{\beta_{\mt}}(k)}-\mathtt{b}}{\mathtt{b}_{+}(\res_{\mt}(k))-\mathtt{b}}\cdot C^{\beta_{\ms}^{'}} C^{\alpha_{\ms}^{'}}  v_{\ms}=
			\begin{cases}
				C^{\beta_{\ms}^{'}} C^{\alpha_{\ms}^{'}}  v_{\ms}, & \text{ if } \ms=\mt \text{ and } \beta_{\mt}^{'}=\beta_{\mt}; \\
				0, & \text{ otherwise.}
			\end{cases}
		\end{align}
		Furthermore, for any $\alpha_{\mt}^{'}\in  \Z_2([n]\setminus \mathcal{D}_{\mt}),$ we have
		\begin{align}\label{F. eq2}
			&C^{\alpha_{\mt}}\gamma_{\mt}(C^{\alpha_{\mt}})^{-1} \cdot C^{\beta_{\mt}} C^{\alpha_{\mt}^{'}}  v_{\mt} \\
			&=C^{\beta_{\mt}}C^{\alpha_{\mt}}\gamma_{\mt}(C^{\alpha_{\mt}})^{-1}  C^{\alpha_{\mt}^{'}}  v_{\mt}\nonumber\\
			&=C^{\beta_{\mt}}C^{\alpha_{\mt}^{'}}(C^{\alpha_{\mt}^{'}})^{-1}C^{\alpha_{\mt}}\gamma_{\mt}(C^{\alpha_{\mt}})^{-1}  C^{\alpha_{\mt}^{'}}  v_{\mt}\nonumber,
		\end{align} where in the first equation we used the fact that $C^{\alpha_{\mt}}\gamma_{\mt}(C^{\alpha_{\mt}})^{-1} $ commutes with $C^{\beta_{\mt}}$. Apply Lemma \ref{lem:clifford rep}(2), $$(C^{\alpha_{\mt}^{'}})^{-1}C^{\alpha_{\mt}}\gamma_{\mt}(C^{\alpha_{\mt}})^{-1}  C^{\alpha_{\mt}^{'}} $$ is an idempotent, and either $(C^{\alpha_{\mt}^{'}})^{-1}C^{\alpha_{\mt}}\gamma_{\mt}(C^{\alpha_{\mt}})^{-1}  C^{\alpha_{\mt}^{'}} $ and $\gamma_{\mt}$  are equal or they are orthogonal. By Lemma \ref{lem:clifford rep}(2) again, we have $$(C^{\alpha_{\mt}^{'}})^{-1}C^{\alpha_{\mt}}\gamma_{\mt}(C^{\alpha_{\mt}})^{-1}  C^{\alpha_{\mt}^{'}}\begin{cases}=\gamma_{\mt},&\text{if $d_{\undla}=0$ and $\alpha_{\mt}^{'}=\alpha_{\mt}$ or $d_{\undla}=1$ and $\alpha_{\mt}^{'}=\alpha_{\mt,a}$ for some $a\in \mathbb{Z}_2,$}\\
		\neq\gamma_{\mt}&\text{otherwise.}
		\end{cases}$$ On the other hand,  $v_{\mt}=\gamma_{\mt} v_{\mt}$. Combining these with \eqref{F. eq1} and \eqref{F. eq2}, we have that \begin{align*}
		F_{\rm T}\cdot C^{\beta_{\ms}^{'}} C^{\alpha_{\ms}^{'}} v_{\ms}=
		\begin{cases}
			C^{\beta_{\ms}^{'}} C^{\alpha_{\ms}^{'}}  v_{\ms}, & \text{ if } d_{\undla}=0 \text{ and } {\rm S}={\rm T}; \\
			C^{\beta_{\ms}^{'}} C^{\alpha_{\ms}^{'}}  v_{\ms}, & \text{ if $d_{\undla}=1$ and } {\rm S}={\rm T}_{a}\text{ for some $a\in \mathbb{Z}_2$  }; \\
			0, & \text{ otherwise,}
		\end{cases}
		\end{align*} which is as required.
		\end{proof}

\begin{thm}\label{semisimple:non-dege}\cite[Theorem 4.10]{SW}
Let $\undQ=(Q_1,Q_2,\ldots,Q_m)\in(\mathbb{K}^*)^m$.  Assume $f=f^{(\bullet)}_{\undQ}$ and   $P^{(\bullet)}_{n}(q^2,\undQ)\neq 0$, with $\bullet\in\{\mathtt{0},\mathtt{s},\mathtt{ss}\}$. Then $\mHfcn$ is a (split) semisimple algebra and
  $$
  \{\mathbb{D}(\undla)|~ \undla\in\mathscr{P}^{\bullet,m}_{n}\}$$ forms a complete set of pairwise non-isomorphic irreducible $\mHfcn$-modules. Moreover,  $\mathbb{D}(\undla)$ is of type  $\texttt{M}$ if and only if $\sharp \mathcal{D}_{\undla}$  is even and is of type  $\texttt{Q}$ if and only if $\sharp \mathcal{D}_{\undla}$  is odd.
	\end{thm}			
			By Theorem \ref{semisimple:non-dege}, we have the following $\mHfcn$-module isomorphism:
$$\mHfcn\cong\bigoplus_{\undla\in\mathscr{P}^{\bullet,m}_{n}}\mathbb{D}(\undla)^{\oplus 2^{n-\bigl\lceil\frac{|\mathcal{D}_{\mt^{\undla}}|}{2}\bigr\rceil}|\Std(\undla)| }\cong\bigoplus_{\undla\in\mathscr{P}^{\bullet,m}_{n}}\mathbb{D}(\undla)^{\oplus 2^{n-\bigl|\mathcal{OD}_{\mt^{\undla}}\bigr|}|\Std(\undla)| }.$$
So the block decomposition is
			$$\mHfcn=\bigoplus_{\undla \in \mathscr{P}^{\bullet,m}_{n}} B_{\undla},$$ and for each $\undla \in \mathscr{P}^{\bullet,m}_{n}$, we have $$B_{\undla}\cong \mathbb{D}(\undla)^{\oplus 2^{n-\bigl\lceil\frac{|\mathcal{D}_{\mt^{\undla}}|}{2}\bigr\rceil}|\Std(\undla)| }\cong \mathbb{D}(\undla)^{\oplus 2^{n-\bigl|\mathcal{OD}_{\mt^{\undla}}\bigr|}|\Std(\undla)| }$$ as $B_{\undla}$-modules.

			The following Theorem gives a complete set of (super) primitive orthogonal idempotents of $\mHfcn$, which is the second main result of this paper.
			
		\begin{thm}\label{primitive idempotents}
			Suppose $P^{(\bullet)}_{n}(q^2,\undQ)\neq 0$.  For $\bullet\in\{\mathtt{0},\mathtt{s},\mathtt{ss}\}$, we have the following.
			
			(a) $\{F_{\rm T} \mid \rm T\in {\rm Tri}_{\bar{0}}(\mathscr{P}^{\bullet,m}_{n})\}$ is a complete set of (super) primitive orthogonal idempotents of $\mHfcn.$
			
			(b) $\{F_{\undla} \mid \undla \in \mathscr{P}^{\bullet,m}_{n} \}$ is a complete set of (super) primitive central idempotents of $\mHfcn.$
			
			(c) For ${\rm T}=(\mt, \alpha_{\mt}, \beta_{\mt})\in {\rm Tri}_{\bar{0}}(\undla),\,
			{\rm S}=(\ms, \alpha_{\ms}', \beta_{\ms}')\in {\rm Tri}_{\bar{0}}(\underline{\mu}),$ then $\mathbb{D}_{\rm T}$ and $\mathbb{D}_{\rm S}$ belongs to the same block if and only if $\undla=\underline{\mu}$.
			
			(d) Let $\undla\in\mathscr{P}^{\bullet,m}_{n}$. If $d_{\undla}=1$, then for any ${\rm T}=(\mt, \alpha_{\mt}, \beta_{\mt}),
	{\rm S}=(\ms, \alpha_{\ms}', \beta_{\ms}')\in {\rm Tri}_{\bar{0}}(\undla),$ we have evenly isomorphic $\mHfcn$-supermodules $\mathbb{D}_{\rm T}\cong \mathbb{D}_{\rm S}$; if $d_{\undla}=0$, then for any ${\rm T}=(\mt, \alpha_{\mt}, \beta_{\mt}),
	{\rm S}=(\ms, \alpha_{\ms}', \beta_{\ms}')\in {\rm Tri}(\undla),$ we have evenly isomorphic $\mHfcn$-supermodules $\mathbb{D}_{\rm T}\cong \mathbb{D}_{\rm S}$ if and only if
	$$|\alpha_{\mt}|+|\beta_{\mt}| \equiv |\alpha_{\ms}'|+|\beta_{\ms}'| \pmod 2.$$
		\end{thm}
		\begin{proof}
			
			(a) We claim that $F_{\rm S}F_{\rm T}=\delta_{{\rm S},{\rm T}}F_{\rm S}$ for any ${\rm S,\,T} \in {\rm Tri}_{\bar{0}}(\mathscr{P}^{\bullet,m}_{n})$ and $$\sum_{\rm T\in {\rm Tri}_{\bar{0}}(\mathscr{P}^{\bullet,m}_{n})}F_{\rm T}=1.
			$$ By Theorem \ref{semisimple:non-dege}, we only need to check that for any $\undla\in\mathscr{P}^{\bullet,m}_{n}$, these equations hold on $\mathbb{D}(\undla)$,  which is a direct application of \eqref{F. eq3} and Proposition \ref{actions of generators on L basis}.  By Theorem \ref{semisimple:non-dege}, the number of primitive idempotents corresponding to  $B_{\undla}$ is $2^{n-\bigl|\mathcal{OD}_{\mt^{\undla}}\bigr|}|\Std(\undla)|$ which equals to $|{\rm Tri}_{\bar{0}}(\undla)|$. Thus we complete the proof of (a).
			
			(b) The proof of (a) implies that $\{F_{\rm T} \mid \rm T\in {\rm Tri}_{\bar{0}}(\undla)\}$ is a complete set of (super) primitive orthogonal idempotents of $B_{\undla}$. Hence $F_{\undla}$ is the identity of the block $B_{\undla}$. This proves (b).

			(c) Using \eqref{F. eq3}, we deduce that for any ${\rm T}=(\mt, \alpha_{\mt}, \beta_{\mt})\in {\rm Tri}_{\bar{0}}(\mathscr{P}^{\bullet,m}_{n})$ and $\undla\in \mathscr{P}^{\bullet,m}_{n}$, simple module $\mathbb{D}_{\rm T}$ belongs to block $B_{\undla}$ if and only if $\mt\in\Std(\undla)$. This proves (c).
			
			(d) If $d_{\undla}=1$, then $B_{\undla}$ has type $\texttt{Q}$. Hence any two simple modules are evenly-isomorphic by Example \ref{simple algebra} (2). If $d_{\undla}=0$, then $B_{\undla}$ has type $\texttt{M}$. It follows from the equation \eqref{F. eq3} and Example \ref{simple algebra} (1) that
			$\mathbb{D}_{\rm T}\cong \Pi^{a}\mathbb{D}(\undla)$ as supermodules if and only if $|\alpha_{\mt}|+|\beta_{\mt}|\equiv a \pmod 2,$ for $a\in \mathbb{Z}_2.$
		\end{proof}
		
		\begin{cor} Suppose $P^{(\bullet)}_{n}(q^2,\undQ)\neq 0$. Then
			the set of elements $$\{F_{\undla} \mid \lambda\in \mathscr{P}^{\bullet,m}_{n} \}$$ form a $\mathbb{K}$-basis of the super center ${\rm Z}(\mHfcn)_{\bar{0}}.$
		\end{cor}
		The following Lemma will be used in Proposition \ref{Non-dege BK anti-idempotent}.
		\begin{lem}	For any ${\rm T}=(\mt, \alpha_{\mt}, \beta_{\mt})\in {\rm Tri}_{\bar{0}}(\undla),$ We have
			\begin{equation}\label{comm. form of FT}
				F_{\rm T}=\prod_{k=1}^{n}\prod_{\mathtt{b}\in \mathtt{B}(k)\atop \mathtt{b}\neq \mathtt{b}_{+}(\res_{\mt}(k))}\frac{X_k^{\nu_{\beta_{\mt}}(k)}-\mathtt{b}}{\mathtt{b}_{+}(\res_{\mt}(k))-\mathtt{b}} \cdot
				\left(C^{\alpha_{\mt}}\gamma_{\mt}(C^{\alpha_{\mt}})^{-1} \right)\in \mHfcn.
			\end{equation}
		\end{lem}
		
		\begin{proof}
			Note that for any $i \in \mathcal{D}_{\mt},$ we have $\mathtt{b}_{\pm}(\res_{\mt}(i)) \in \{ 1,-1 \}.$ Following the same argument as for \eqref{F. eq1}, we deduce that for any $\underline{\mu}\in\mathscr{P}^{\bullet,m}_{n}$,
			$$
			\prod_{k=1}^{n}\prod_{\mathtt{b}\in \mathtt{B}(k)\atop \mathtt{b}\neq \mathtt{b}_{+}(\res_{\mt}(k))}\frac{X_k^{\nu_{\beta_{\mt}}(k)}-\mathtt{b}}{\mathtt{b}_{+}(\res_{\mt}(k))-\mathtt{b}} $$ and $$\prod_{k=1}^{n}\prod_{\mathtt{b}\in \mathtt{B}(k)\atop \mathtt{b}\neq \mathtt{b}_{+}(\res_{\mt}(k))}\frac{X_k^{\nu_{\beta_{\mt}}(k)\cdot (-1)^{\delta_{k,i}}}-\mathtt{b}}{\mathtt{b}_{+}(\res_{\mt}(k))-\mathtt{b}}
			$$ act as the same linear operator on $\mathbb{D}(\underline{\mu})$. Hence $$
			\prod_{k=1}^{n}\prod_{\mathtt{b}\in \mathtt{B}(k)\atop \mathtt{b}\neq \mathtt{b}_{+}(\res_{\mt}(k))}\frac{X_k^{\nu_{\beta_{\mt}}(k)}-\mathtt{b}}{\mathtt{b}_{+}(\res_{\mt}(k))-\mathtt{b}}=\prod_{k=1}^{n}\prod_{\mathtt{b}\in \mathtt{B}(k)\atop \mathtt{b}\neq \mathtt{b}_{+}(\res_{\mt}(k))}\frac{X_k^{\nu_{\beta_{\mt}}(k)\cdot (-1)^{\delta_{k,i}}}-\mathtt{b}}{\mathtt{b}_{+}(\res_{\mt}(k))-\mathtt{b}}$$ by Theorem \ref{semisimple:non-dege}.
			Now we can compute
			\begin{align*}
				&\prod_{k=1}^{n}\prod_{\mathtt{b}\in \mathtt{B}(k)\atop \mathtt{b}\neq \mathtt{b}_{+}(\res_{\mt}(k))}\frac{X_k^{\nu_{\beta_{\mt}}(k)}-\mathtt{b}}{\mathtt{b}_{+}(\res_{\mt}(k))-\mathtt{b}} \cdot C_i \\
				&=C_i \cdot \prod_{k=1}^{n}\prod_{\mathtt{b}\in \mathtt{B}(k)\atop \mathtt{b}\neq \mathtt{b}_{+}(\res_{\mt}(k))}\frac{X_k^{\nu_{\beta_{\mt}}(k)\cdot (-1)^{\delta_{k,i}}}-\mathtt{b}}{\mathtt{b}_{+}(\res_{\mt}(k))-\mathtt{b}} \\
				&=C_i \cdot \prod_{k=1}^{n}\prod_{\mathtt{b}\in \mathtt{B}(k)\atop \mathtt{b}\neq \mathtt{b}_{+}(\res_{\mt}(k))}\frac{X_k^{\nu_{\beta_{\mt}}(k)}-\mathtt{b}}{\mathtt{b}_{+}(\res_{\mt}(k))-\mathtt{b}}.
			\end{align*}
			It follows that
			$${\text{RHS of \eqref{comm. form of FT}}}=\left(C^{\alpha_{\mt}}\gamma_{\mt}(C^{\alpha_{\mt}})^{-1} \right)\cdot \prod_{k=1}^{n}\prod_{\mathtt{b}\in \mathtt{B}(k)\atop \mathtt{b}\neq \mathtt{b}_{+}(\res_{\mt}(k))}\frac{X_k^{\nu_{\beta_{\mt}}(k)}-\mathtt{b}}{\mathtt{b}_{+}(\res_{\mt}(k))-\mathtt{b}}=F_{\rm T}.
			$$
		\end{proof}
		
		\begin{defn}For any  ${\rm T}=(\mt, \alpha_{\mt}, \beta_{\mt})\in {\rm Tri}(\mathscr{P}^{\bullet,m}_{n}),$  we define $${\rm \widehat{T}}:=(\mt, \widehat{\alpha_{\mt}}, \beta_{\mt})\in {\rm Tri}(\mathscr{P}^{\bullet,m}_{n}).$$
			\end{defn}
			
		Recall the anti-involution $*$ on $\mHfcn$.
			\begin{cor}\label{Non-dege BK anti-idempotent}
			For any  ${\rm T}=(\mt, \alpha_{\mt}, \beta_{\mt})\in {\rm Tri}_{\bar{0}}(\mathscr{P}^{\bullet,m}_{n}),$ we have
			$F_{\rm T}^*=F_{\rm \widehat{T}}.$
		\end{cor}
		\begin{proof}
			Suppose 	
$$\supp(\alpha_{\mt})\sqcup \supp(\widehat{\alpha_{\mt}})=\{a_1,a_2,\cdots,a_l\}.$$
			Then we have
			\begin{align}
				F_{\rm T}^*&=\left(\prod_{k=1}^{n}\prod_{\mathtt{b}\in \mathtt{B}(k)\atop \mathtt{b}\neq \mathtt{b}_{+}(\res_{\mt}(k))}\frac{X_k^{\nu_{\beta_{\mt}}(k)}-\mathtt{b}}{\mathtt{b}_{+}(\res_{\mt}(k))-\mathtt{b}}\right)
				\cdot\left(C^{\alpha_{\mt}}\gamma_{\mt}^*(C^{\alpha_{\mt}})^{-1} \right)\nonumber\\&=\left(\prod_{k=1}^{n}\prod_{\mathtt{b}\in \mathtt{B}(k)\atop \mathtt{b}\neq \mathtt{b}_{+}(\res_{\mt}(k))}\frac{X_k^{\nu_{\beta_{\mt}}(k)}-\mathtt{b}}{\mathtt{b}_{+}(\res_{\mt}(k))-\mathtt{b}}\right)
				\cdot\left(C^{\alpha_{\mt}}\bigl(\overrightarrow{\prod_{i=a_1,\cdots,a_l}}C_i\bigr)\gamma_{\mt}\bigl(\overleftarrow{\prod_{i=a_1,\cdots,a_l}}C_i\bigr)(C^{\alpha_{\mt}})^{-1} \right)\nonumber\\
				&=\left(\prod_{k=1}^{n}\prod_{\mathtt{b}\in \mathtt{B}(k)\atop \mathtt{b}\neq \mathtt{b}_{+}(\res_{\mt}(k))}\frac{X_k^{\nu_{\beta_{\mt}}(k)}-\mathtt{b}}{\mathtt{b}_{+}(\res_{\mt}(k))-\mathtt{b}}\right)
				\cdot\left(C^{\widehat{\alpha_{\mt}}}\gamma_{\mt}(C^{\widehat{\alpha_{\mt}}})^{-1} \right)\nonumber\\
				&=\left(C^{\widehat{\alpha_{\mt}}}\gamma_{\mt}(C^{\widehat{\alpha_{\mt}}})^{-1} \right)
				\cdot \left( \prod_{k=1}^{n}\prod_{\mathtt{b}\in \mathtt{B}(k)\atop \mathtt{b}\neq \mathtt{b}_{+}(\res_{\mt}(k))}\frac{X_k^{\nu_{\beta_{\mt}}(k)}-\mathtt{b}}{\mathtt{b}_{+}(\res_{\mt}(k))-\mathtt{b}}\right)
=F_{\rm \widehat{T}},\nonumber
			\end{align}
			where in the second equation, we have used relations \eqref{Clifford} and Lemma \ref{lem:clifford rep} (2), in the fourth equation, we have used the fact $\supp(\widehat{\alpha_{\mt}}) \subset \mathcal{D}_\mt$ and \eqref{comm. form of FT}.
		\end{proof}

		\subsection{Seminormal basis of $\mHfcn$}\label{Seminormalbase}
		{\bf In this subsection, we fix $\undla\in\mathscr{P}^{\bullet,m}_{n}$ with $\bullet\in\{\mathsf{0},\mathsf{s},\mathsf{ss}\}.$} We will construct a series of seminormal bases for block $B_{\undla}$. The following definition is crucial in our construction of seminormal bases.
	\begin{defn}
    \label{pag:Phist and cst}	
For any $\ms,\mt \in \Std(\undla),$ we fix a reduced expression $d(\ms,\mt)=s_{k_p}\cdots s_{k_1}$. We define
		\begin{align}\label{Phist}
			\Phi_{\ms,\mt}:=\overleftarrow{\prod_{i=1,\ldots,p}}\Phi_{k_{i}}(\mathtt{b}_{s_{k_{i-1}}\cdots s_{k_1}\mathfrak{t},k_{i}}, \mathtt{b}_{s_{k_{i-1}}\cdots s_{k_1}\mathfrak{t},k_{i}+1})  \in \mHfcn
		\end{align}
		and the coefficient
		\begin{align}\label{c-coefficients. non-dege.}
			\mathtt{c}_{\ms,\mt}:=\prod_{i=1,\ldots,p}\sqrt{\mathtt{c}_{s_{k_{i-1}}\cdots s_{k_1}\mathfrak{t}}(k_{i})}  \in \mathbb{K}.
		\end{align}
		\end{defn}
		By Lemma \ref{admissible transposes} and the third part of Lemma \ref{important condition1}, $\mathtt{c}_{\ms,\mt}\in  \mathbb{K}^*$.  The following Lemma shows that $\Phi_{\ms,\mt}$ is independent of the reduced choice of $d(\ms,\mt)$.
		
		\begin{lem}\label{Phist. well-defi}
Let $\ms,\mt \in \Std(\undla).$
			\begin{enumerate}
				\item Suppose $d(\ms,\mt)=s_is_{i+1}s_{i}=s_{i+1}s_is_{i+1}$, we have 	\begin{align}
					\Phi_{i}(\mathtt{b}_{s_{i+1}s_i\mt,i},&\mathtt{b}_{s_{i+1}s_i\mt,i+1})\Phi_{i+1}(\mathtt{b}_{s_i\mt,i+1},\mathtt{b}_{s_i\mt,i+2})\Phi_{i}(\mathtt{b}_{\mt,i},\mathtt{b}_{\mt,i+1})\nonumber\\
					&=\Phi_{i+1}(\mathtt{b}_{s_is_{i+1}\mt,i+1},\mathtt{b}_{s_is_{i+1}\mt,i+2})\Phi_{i}(\mathtt{b}_{s_{i+1}\mt,i},\mathtt{b}_{s_{i+1}\mt,i+1})\Phi_{i+1}(\mathtt{b}_{\mt,i+1},\mathtt{b}_{\mt,i+2}).\nonumber
				\end{align}
				\item Suppose $d(\ms,\mt)=s_is_j=s_js_i$,where $|i-j|>1$, we have $$\Phi_{i}(\mathtt{b}_{s_j\mt,i},\mathtt{b}_{s_j\mt,i+1})\Phi_{j}(\mathtt{b}_{\mt,j},\mathtt{b}_{\mt,j+1})=\Phi_{j}(\mathtt{b}_{s_i\mt,j},\mathtt{b}_{s_i\mt,j+1})\Phi_{i}(\mathtt{b}_{\mt,i},\mathtt{b}_{\mt,i+1}).$$
				\item $\Phi_{\ms,\mt}$ is independent of the reduced expression of $d(\ms,\mt)$.
				\end{enumerate}
			\end{lem}
		
\begin{proof}\begin{enumerate}
		\item Let $$x:=\mathtt{b}_{-}(\res_{\mt}(i+1)),\,y:=\mathtt{b}_{-}(\res_{\mt}(i+2)),\,z:=\mathtt{b}_{-}(\res_{\mt}(i)).
		$$
One can easily check that
	\begin{align}
		x&=\mathtt{b}_{-}(\res_{s_{i+1}s_{i}\mt}(i))=\mathtt{b}_{-}(\res_{s_{i}s_{i+1}\mt}(i+2)),\nonumber\\
		y&=\mathtt{b}_{-}(\res_{s_{i+1}s_{i}\mt}(i+1))=\mathtt{b}_{-}(\res_{s_{i}\mt}(i+2))=\mathtt{b}_{-}(\res_{s_{i+1}\mt}(i+1)),\nonumber\\
		z&=\mathtt{b}_{-}(\res_{s_{i}\mt}(i+1))=\mathtt{b}_{-}(\res_{s_{i}s_{i+1}\mt}(i+1))=\mathtt{b}_{-}(\res_{s_{i+1}\mt}(i)).\nonumber
	\end{align}
By \eqref{braidrel2}, we have
	\begin{align}
		\Phi_{i}(x,y)\Phi_{i+1}(z,y)\Phi_{i}(z,x)=\Phi_{i+1}(z,x)\Phi_{i}(z,y)\Phi_{i+1}(x,y),\nonumber
	\end{align}
which is equivalent to
	\begin{align}
		\Phi_{i}(\mathtt{b}_{s_{i+1}s_i\mt,i},&\mathtt{b}_{s_{i+1}s_i\mt,i+1})\Phi_{i+1}(\mathtt{b}_{s_i\mt,i+1},\mathtt{b}_{s_i\mt,i+2})\Phi_{i}(\mathtt{b}_{\mt,i},\mathtt{b}_{\mt,i+1})\nonumber\\
		&=\Phi_{i+1}(\mathtt{b}_{s_is_{i+1}\mt,i+1},\mathtt{b}_{s_is_{i+1}\mt,i+2})\Phi_{i}(\mathtt{b}_{s_{i+1}\mt,i},\mathtt{b}_{s_{i+1}\mt,i+1})\Phi_{i+1}(\mathtt{b}_{\mt,i+1},\mathtt{b}_{\mt,i+2}).\nonumber
	\end{align}
	\item Let $$x:=\mathtt{b}_{-}(\res_{\mt}(i)),\,y:=\mathtt{b}_{-}(\res_{\mt}(i+1))\, z:=\mathtt{b}_{-}(\res_{\mt}(j)),\,w:=\mathtt{b}_{-}(\res_{\mt}(j+1)).$$
	 One can easily check that $$x=\mathtt{b}_{-}(\res_{s_{j}\mt}(i)),\,y=\mathtt{b}_{-}(\res_{s_{j}\mt}(i+1)),\,z=\mathtt{b}_{-}(\res_{s_{i}\mt}(j)),\,w=\mathtt{b}_{-}(\res_{s_{i}\mt}(j+1)).
	 $$By \eqref{braidrel1}, we have	$$\Phi_i(x,y)\Phi_j(z,w)=\Phi_j(z,w)\Phi_i(x,y),$$ which is equivalent to $$\Phi_{i}(\mathtt{b}_{s_j\mt,i},\mathtt{b}_{s_j\mt,i+1})\Phi_{j}(\mathtt{b}_{\mt,j},\mathtt{b}_{\mt,j+1})=\Phi_{j}(\mathtt{b}_{s_i\mt,j},\mathtt{b}_{s_i\mt,j+1})\Phi_{i}(\mathtt{b}_{\mt,i},\mathtt{b}_{\mt,i+1}).$$
	 \item This follows from (1) and (2).
	\end{enumerate}
	\end{proof}

	Recall the basis of simple module $\mathbb{D}(\undla)$ in Proposition \ref{actions of generators on L basis}. We summarise some properties of $\Phi_{\ms,\mt}$ which will be used later.
		
		\begin{lem}\label{Phist. lem}
			Let $\ms, \mt \in \Std(\undla).$
			\begin{enumerate}
			\item $\Phi_{\ms,\mt} \cdot v_{\mt}=\mathtt{c}_{\ms,\mt} v_{\ms}.$ Hence the coefficient $\mathtt{c}_{\ms,\mt}$ is also independent of the reduced expression of
			$d(\ms,\mt).$
			
		\item $\Phi_{\ms,\mt} C_{d(\mt,\mt^{\undla})(a)}= C_{d(\ms,\mt^{\undla})(a)}\Phi_{\ms,\mt}$ for any $a\in \mathcal{D}_{\mt^{\undla}}.$

		\item $\mathtt{c}_{\mt,\ms}=\mathtt{c}_{\ms,\mt}.$
			
		\item $\Phi_{\ms,\mt}\Phi_{\mt,\ms}=(\mathtt{c}_{\ms,\mt})^2.$
			
			\end{enumerate}
		\end{lem}

		\begin{proof} We fix a reduced expression $d(\ms,\mt)=s_{k_p}\cdots s_{k_1}$. Then
			\begin{align*}
				\Phi_{\ms,\mt}:=\overleftarrow{\prod_{i=1,\ldots,p}}\Phi_{k_{i}}(\mathtt{b}_{s_{k_{i-1}}\cdots s_{k_1}\mathfrak{t},k_{i}}, \mathtt{b}_{s_{k_{i-1}}\cdots s_{k_1}\mathfrak{t},k_{i}+1}).
			\end{align*}
			
			\begin{enumerate}
				\item This follows from \eqref{Phi-action}.
			
			\item Let $\mfku\in\Std(\undla).$  We claim that if $j\in D_{\mfku}$ and $s_i$ is admissible with respect to $\mfku$, then \begin{equation}\label{eq.commute1}\Phi_{i}(\mathtt{b}_{\mfku,i}, \mathtt{b}_{\mfku,i+1})C_j=C_{s_{i}(j)}\Phi_{i}(\mathtt{b}_{\mfku,i}, \mathtt{b}_{\mfku,i+1}).
			\end{equation} Actually, we have $\mathtt{b}_{\mfku,j}=\pm 1$. Combining this with \eqref{Phi and C}, we can compute $$
\Phi_{i}(\mathtt{b}_{\mfku,i}, \mathtt{b}_{\mfku,i+1})C_j=\begin{cases} C_{i+1}\Phi_{i}(\mathtt{b}_{\mfku,i}, \mathtt{b}_{\mfku,i+1})&\text{if $j=i\in D_{\mfku},$}\\
	C_{i}\Phi_{i}(\mathtt{b}_{\mfku,i}, \mathtt{b}_{\mfku,i+1})&\text{if $j=i+1\in D_{\mfku}$,}\\
		C_{j}\Phi_{i}(\mathtt{b}_{\mfku,i}, \mathtt{b}_{\mfku,i+1})&\text{if $i,i+1\neq j\in D_{\mfku}$}.
	\end{cases}
$$ This proves our claim.
Now we deduce that \begin{align*}
	\Phi_{\ms,\mt} C_{d(\mt,\mt^{\undla})(a)}&=\overleftarrow{\prod_{i=1,\ldots,p}}\Phi_{k_{i}}(\mathtt{b}_{s_{k_{i-1}}\cdots s_{k_1}\mathfrak{t},k_{i}}, \mathtt{b}_{s_{k_{i-1}}\cdots s_{k_1}\mathfrak{t},k_{i}+1}) C_{d(\mt,\mt^{\undla})(a)}\\
	&= C_{d(\ms,\mt)d(\mt,\mt^{\undla})(a)}\overleftarrow{\prod_{i=1,\ldots,p}}\Phi_{k_{i}}(\mathtt{b}_{s_{k_{i-1}}\cdots s_{k_1}\mathfrak{t},k_{i}}, \mathtt{b}_{s_{k_{i-1}}\cdots s_{k_1}\mathfrak{t},k_{i}+1}) \\
	&=C_{d(\ms,\mt^{\undla})(a)}\overleftarrow{\prod_{i=1,\ldots,p}}\Phi_{k_{i}}(\mathtt{b}_{s_{k_{i-1}}\cdots s_{k_1}\mathfrak{t},k_{i}}, \mathtt{b}_{s_{k_{i-1}}\cdots s_{k_1}\mathfrak{t},k_{i}+1})\\
	&=C_{d(\ms,\mt^{\undla})(a)}	\Phi_{\ms,\mt},
	\end{align*} where in the second equation we have used \eqref{eq.commute1}.

		\item We have	\begin{align}\label{cst=cts}
				\mathtt{c}_{\mt,\ms}=\prod_{i=1,\ldots,p}\sqrt{\mathtt{c}_{s_{k_{i+1}}\cdots s_{k_{p}}\ms}(k_{i})}=\prod_{i=1,\ldots,p}\sqrt{\mathtt{c}_{s_{k_{i-1}}\cdots s_{k_1}\mathfrak{t}}(k_{i})}=\mathtt{c}_{\ms,\mt}.
			\end{align}
			
		\item By definition, we have \begin{align*}
				\Phi_{\mt,\ms}&=\overrightarrow{\prod_{i=1,\ldots,p}}\Phi_{k_{i}}(\mathtt{b}_{s_{k_{i+1}}\cdots s_{k_{p}}\mathfrak{s},k_{i}}, \mathtt{b}_{s_{k_{i+1}}\cdots s_{k_{p}}\mathfrak{s},k_{i}+1})\\
				&=\overrightarrow{\prod_{i=1,\ldots,p}}\Phi_{k_{i}}(\mathtt{b}_{s_{k_{i-1}}\cdots s_{k_{1}}\mathfrak{t},k_{i}+1}, \mathtt{b}_{s_{k_{i-1}}\cdots s_{k_{1}}\mathfrak{t},k_{i}}).
			\end{align*}
			Hence \begin{align*}\Phi_{\ms,\mt}\Phi_{\mt,\ms}&=\biggl(\overleftarrow{\prod_{i=1,\ldots,p}}\Phi_{k_{i}}(\mathtt{b}_{s_{k_{i-1}}\cdots s_{k_1}\mathfrak{t},k_{i}}, \mathtt{b}_{s_{k_{i-1}}\cdots s_{k_1}\mathfrak{t},k_{i}+1})\biggr)\\
				&\qquad\qquad\qquad\cdot\biggl(\overrightarrow{\prod_{i=1,\ldots,p}}\Phi_{k_{i}}(\mathtt{b}_{s_{k_{i-1}}\cdots s_{k_{1}}\mathfrak{t},k_{i}+1}, \mathtt{b}_{s_{k_{i-1}}\cdots s_{k_{1}}\mathfrak{t},k_{i}})\biggr)\\
				&=\biggl(\overleftarrow{\prod_{i=2,\ldots,p}}\Phi_{k_{i}}(\mathtt{b}_{s_{k_{i-1}}\cdots s_{k_1}\mathfrak{t},k_{i}}, \mathtt{b}_{s_{k_{i-1}}\cdots s_{k_1}\mathfrak{t},k_{i}+1})\biggr)\\
				&\qquad\qquad\qquad\cdot\biggl(\Phi_{k_{1}}(\mathtt{b}_{\mathfrak{t},k_{1}}, \mathtt{b}_{\mathfrak{t},k_{1}+1})\Phi_{k_{1}}(\mathtt{b}_{\mathfrak{t},k_{1}+1}, \mathtt{b}_{\mathfrak{t},k_{1}})\biggr)\\
			&\qquad\qquad\qquad	\qquad\qquad\qquad\cdot\biggl(\overrightarrow{\prod_{i=2,\ldots,p}}\Phi_{k_{i}}(\mathtt{b}_{s_{k_{i-1}}\cdots s_{k_{1}}\mathfrak{t},k_{i}+1}, \mathtt{b}_{s_{k_{i-1}}\cdots s_{k_{1}}\mathfrak{t},k_{i}})\biggr)\\
				&=\mathtt{c}_{\mt}(k_1)\biggl(\overleftarrow{\prod_{i=2,\ldots,p}}\Phi_{k_{i}}(\mathtt{b}_{s_{k_{i-1}}\cdots s_{k_1}\mathfrak{t},k_{i}}, \mathtt{b}_{s_{k_{i-1}}\cdots s_{k_1}\mathfrak{t},k_{i}+1})\biggr)\\
					&\qquad\qquad\qquad\cdot\biggl(\overrightarrow{\prod_{i=2,\ldots,p}}\Phi_{k_{i}}(\mathtt{b}_{s_{k_{i-1}}\cdots s_{k_{1}}\mathfrak{t},k_{i}+1}, \mathtt{b}_{s_{k_{i-1}}\cdots s_{k_{1}}\mathfrak{t},k_{i}})\biggr),
		\end{align*} where in the third equation, we have used \eqref{square1}. Following the same computation, we deduce that $$
		\Phi_{\ms,\mt}\Phi_{\mt,\ms}=\prod_{i=1,\ldots,p}\mathtt{c}_{s_{k_{i-1}}\cdots s_{k_1}\mathfrak{t}}(k_{i})=(\mathtt{c}_{\ms,\mt})^2.
		$$
\end{enumerate}
		\end{proof}
	By definition, we have $\Phi_{\mt,\mt}=1$ and $\mathtt{c}_{\mt,\mt}=1.$  Recall in \eqref{stanard D} and \eqref{standard OD}, we have set \begin{align*}
	\mathcal{D}_{\mt^{\undla}}&=\{i_1<i_2<\cdots<i_t\},\\
	\mathcal{OD}_{\mt^{\undla}}&=\{i_1,i_3,\cdots,i_{2{\lceil t/2 \rceil}-1}\}\subset \mathcal{D}_{\mt^{\undla}}.
\end{align*}
Now we can define the seminormal basis, which is the key object of this paper.
        \label{pag:nondege seminormal basis}
		\begin{defn} Let $\mathfrak{w}\in\Std(\undla)$.
		
		(1) Supppose $d_{\undla}=0.$  For any ${\rm S}=(\ms, \alpha_{\ms}', \beta_{\ms}'), {\rm T}=(\mt, \alpha_{\mt}, \beta_{\mt})\in {\rm Tri}(\undla),$ we define
			\begin{align}\label{fst. typeM. nondege.}
				f_{{\rm S},{\rm T}}^{\mathfrak{w}}
				:=F_{\rm S}C^{\beta_{\ms}'}C^{\alpha_{\ms}'}\Phi_{\ms,\mathfrak{w}}\Phi_{\mathfrak{w},\mt}
				(C^{\alpha_{\mt}})^{-1} (C^{\beta_{\mt}})^{-1} F_{\rm T}\in F_{\rm S}\mHfcn F_{\rm T},
			\end{align} and	\begin{align}\label{re. fst. typeM. nondege.}
			f_{{\rm S},{\rm T}}
			:=F_{\rm S}C^{\beta_{\ms}'}C^{\alpha_{\ms}'}\Phi_{\ms,\mt}
			(C^{\alpha_{\mt}})^{-1} (C^{\beta_{\mt}})^{-1} F_{\rm T} \in F_{\rm S}\mHfcn F_{\rm T},
			\end{align}

			(2) Suppose $d_{\undla}=1.$ For any $a\in \mathbb{Z}_{2}$ and ${\rm S}=(\ms, \alpha_{\ms}', \beta_{\ms}')\in {\rm Tri}_{\bar{0}}(\undla), {\rm T}_{a}=(\mt, \alpha_{\mt,a}, \beta_{\mt})\in {\rm Tri}_{a}(\undla),$ we define
			\begin{align}\label{fst. typeQ. nondege.}
			&	f_{{\rm S},{\rm T}_{a}}^{\mathfrak{w}}
				:= (-1)^{|\alpha_\ms'|_{>d(\ms,\mt^{\undla})(i_t)}+a|\alpha_\mt|_{>d(\mt,\mt^{\undla})(i_t)}}\nonumber\\
			&\qquad\qquad\qquad\cdot	F_{\rm S}C^{\beta_{\ms}'}C^{\alpha_{\ms}'}\Phi_{\ms,\mathfrak{w}}\Phi_{\mathfrak{w},\mt}
				(C^{{\alpha}_{\mt,a}})^{-1} (C^{\beta_{\mt}})^{-1} F_{\rm T}\in F_{\rm S}\mHfcn F_{\rm T}
			\end{align} and 	\begin{align}\label{re. fst. typeQ. nondege.}
		&	f_{{\rm S},{\rm T}_{a}}
			:= (-1)^{|\alpha_\ms'|_{>d(\ms,\mt^{\undla})(i_t)}+a|\alpha_\mt|_{>d(\mt,\mt^{\undla})(i_t)}}\nonumber\\
			&\qquad\qquad\qquad\cdot F_{\rm S}C^{\beta_{\ms}'}C^{\alpha_{\ms}'}\Phi_{\ms,\mt}
			(C^{{\alpha}_{\mt,a}})^{-1} (C^{\beta_{\mt}})^{-1} F_{\rm T} \in F_{\rm S}\mHfcn F_{\rm T}.
			\end{align}

		\label{pag:nondege cT}	
	   (3)  For any ${\rm T}=(\mt, \alpha_{\mt}, \beta_{\mt})\in {\rm Tri}(\undla),$ we define $$\mathtt{c}_{\rm T}^{\mathfrak{w}}:=(\mathtt{c}_{\mt,\mathfrak{w}})^2\in \mathbb{K}^*.$$
		\end{defn}

		The following Lemma is crucial for our main Theorem.
		\begin{lem}We fix $\mathfrak{w}\in\Std(\undla)$.
		
\begin{enumerate}	
	\item Suppose $d_{\undla}=0$. For any ${\rm S}=(\ms, \alpha_{\ms}', \beta_{\ms}'),
	{\rm T}=(\mt, \alpha_{\mt}, \beta_{\mt}),{\rm U}=(\mfku,\alpha_{\mfku}^{''},\beta_{\mfku}^{''})\in {\rm Tri}(\undla),$ we have \begin{equation}\label{idempotent1}
		f_{{\rm T}, {\rm T}}=F_{\rm T},\,\,	f_{{\rm T}, {\rm T}}^{\mathfrak{w}}=\mathtt{c}_{\rm T}^{\mathfrak{w}}F_{\rm T},
	\end{equation}
	 \begin{equation}\label{SNB. eq1}
					f_{{\rm S},{\rm T}}^\mathfrak{w} \cdot C^{\beta_{\mfku}^{''}} C^{\alpha_{\mfku}^{''}} v_{\mfku} =\begin{cases}
						\mathtt{c}_{\ms,\mathfrak{w}}\mathtt{c}_{\mathfrak{w},\mt}C^{\beta_{\ms}^{'}} C^{\alpha_{\ms}^{'}} v_{\ms},& \text{if ${\rm T}={\rm U}$},\\
						0,&\text{otherwise,}
						\end{cases}				
				\end{equation} and  \begin{equation}\label{SNB. eq2}
				f_{{\rm S},{\rm T}} \cdot C^{\beta_{\mfku}^{''}} C^{\alpha_{\mfku}^{''}} v_{\mfku} =\begin{cases}
					\mathtt{c}_{\ms,\mt}C^{\beta_{\ms}^{'}} C^{\alpha_{\ms}^{'}} v_{\ms},& \text{if ${\rm T}={\rm U}$},\\
					0,&\text{otherwise.}
				\end{cases}				
				\end{equation}
			\item Suppose $d_{\undla}=1$.  For any $a\in \mathbb{Z}_{2}$ and ${\rm S}=(\ms, \alpha_{\ms}', \beta_{\ms}')\in {\rm Tri}_{\bar{0}}(\undla), {\rm T}_{a}=(\mt, \alpha_{\mt,a}, \beta_{\mt})\in {\rm Tri}_{a}(\undla), {\rm U}=(\mfku,\alpha_{\mfku}^{''},\beta_{\mfku}^{''})\in {\rm Tri}(\undla),$ we have \begin{equation}\label{idempotent2}
					f_{{\rm T}_{\bar{0}}, {\rm T}_{\bar{0}}}=	(-1)^{|\alpha_{\mt}|_{>d(\mt,\mt^{\undla})(i_t)}}F_{{\rm T}_{\bar{0}}},\,\,f_{{\rm T}_{\bar{0}}, {\rm T}_{\bar{0}}}^{\mathfrak{w}}=	(-1)^{|\alpha_{\mt}|_{>d(\mt,\mt^{\undla})(i_t)}}\mathtt{c}_{\rm T}^{\mathfrak{w}}F_{{\rm T}_{\bar{0}}},
			\end{equation} \begin{equation}\label{SNB. eq3}
				f_{{\rm S},{\rm T}_a}^\mathfrak{w}  \cdot C^{\beta_{\mfku}^{''}} C^{\alpha_{\mfku}^{''}} v_{\mfku} =\begin{cases}
				(-1)^{b|\alpha_{\mt}|_{>d(\mt,\mt^{\undla})(i_t)}+(a+b+1)|\alpha_{\ms}'|_{>d(\ms,\mt^{\undla})(i_t)}}&\\
				\qquad\qquad\qquad\qquad\qquad\cdot\mathtt{c}_{\mathfrak{s},\mathfrak{w} }\mathtt{c}_{\mathfrak{w},\mt }C^{\beta_{\ms}'} C^{\alpha_{\ms,a+b}'  }v_{\ms}, & \text{if  ${\rm U}={\rm T}_b$,
					for some $b\in \Z_2$,}\\
				0,&\text{otherwise,}
			\end{cases}				
			\end{equation} and  \begin{equation}\label{SNB. eq4}
				f_{{\rm S},{\rm T}_a} \cdot C^{\beta_{\mfku}^{''}} C^{\alpha_{\mfku}^{''}} v_{\mfku} =\begin{cases}
						(-1)^{b|\alpha_{\mt}|_{>d(\mt,\mt^{\undla})(i_t)}+(a+b+1)|\alpha_{\ms}'|_{>d(\ms,\mt^{\undla})(i_t)}}&\\
				\qquad\qquad\qquad\qquad\qquad\cdot\mathtt{c}_{\mathfrak{s},\mt} C^{\beta_{\ms}'} C^{\alpha_{\ms,a+b}'  }v_{\ms}, & \text{if  ${\rm U}={\rm T}_b$,
					for some $b\in \Z_2$,}\\
				0,&\text{otherwise.}
				\end{cases}				
			\end{equation}
				\end{enumerate}
			\end{lem}
			
			\begin{proof}
			\begin{enumerate}
				\item 	Assume $d_{\undla}=0.$ \eqref{idempotent1} follows from \eqref{fst. typeM. nondege.}, \eqref{re. fst. typeM. nondege.} and Lemma \ref{Phist. lem} (4). By \eqref{F. eq3}, we have
				\begin{align*}
					f_{{\rm S},{\rm T}}^{\mathfrak{w}}\cdot C^{\beta_{\mfku}^{''}} C^{\alpha_{\mfku}^{''}} v_{\mfku}
					=0, \quad \text{ if } {\rm U}=(\mfku,\alpha_{\mfku}^{''},\beta_{\mfku}^{''})\neq  {\rm T} \in {\rm Tri}(\undla).
				\end{align*}
				On the other hand, we have
				\begin{align*}
					&f_{{\rm S},{\rm T}}^{\mathfrak{w}}\cdot C^{\beta_{\mt}} C^{\alpha_{\mt}}v_{\mt} \\
					&=F_{\rm S} C^{\beta_{\ms}^{'}}C^{\alpha_{\ms}^{'}}\Phi_{\ms,\mathfrak{w}}\Phi_{\mathfrak{w},\mt}v_{\mt} \nonumber\\
					&=\mathtt{c}_{\ms,\mathfrak{w}}\mathtt{c}_{\mathfrak{w},\mt}F_{\rm S} C^{\beta_{\ms}^{'}}C^{\alpha_{\ms}^{'}}v_{\mt} \nonumber\\
			&=\mathtt{c}_{\ms,\mathfrak{w}}\mathtt{c}_{\mathfrak{w},\mt}C^{\beta_{\ms}^{'}}C^{\alpha_{\ms}^{'}}\Phi_{\ms,\mt}v_{\mt}, \nonumber
				\end{align*}
				where in the second equation, we have used Lemma \ref{Phist. lem} (1), and in the third equation, we have used \eqref{F. eq3}. This completes the proof of \eqref{SNB. eq1}. The proof of \eqref{SNB. eq2} is similar.
		
	\item	Assume $d_{\undla}=1$.  \eqref{idempotent2} follows from \eqref{fst. typeQ. nondege.}, \eqref{re. fst. typeQ. nondege.} and Lemma \ref{Phist. lem} (4).  By \eqref{F. eq3}, we have
		\begin{align*}
			f_{{\rm S},{\rm T}_{a}}^\mathfrak{w}\cdot C^{\beta_{\mfku}^{''}} C^{\alpha_{\mfku}^{''}}  v_{\mfku}
			=0, \quad \text{ if }  {\rm U}=(\mfku,\alpha_{\mfku}^{''},\beta_{\mfku}^{''})\neq {\rm T}_{b} \in {\rm Tri}(\undla) \text{ for any }b\in \mathbb{Z}_2.
		\end{align*}
		On the other hand,  for $b\in \Z_2$, we have
		\begin{align*}
			&f_{{\rm S},{\rm T}_{a}}^\mathfrak{w}\cdot C^{\beta_{\mt}} C^{\alpha_{\mt,b}} v_{\mfkv}\nonumber \\
			&= (-1)^{|\alpha_\ms'|_{>d(\ms,\mt^{\undla})(i_t)}+a|\alpha_\mt|_{>d(\mt,\mt^{\undla})(i_t)}}\\
			&\qquad\qquad\qquad\cdot F_{\rm S} C^{\beta_{\ms}'} C^{\alpha_{\ms}'} \Phi_{\ms,\mathfrak{w}} \Phi_{\mathfrak{w},\mt} (C^{{\alpha_{\mt,a}}})^{-1}
			(C^{\beta_{\mt}})^{-1}C^{\beta_{\mt}}C^{{{\alpha_{\mt,b}}}}v_{\mt} \nonumber\\
			&=\begin{cases} (-1)^{|\alpha_\ms'|_{>d(\ms,\mt^{\undla})(i_t)}+a|\alpha_\mt|_{>d(\mt,\mt^{\undla})(i_t)}}&\\
				\qquad\qquad\qquad\cdot F_{\rm S} C^{\beta_{\ms}'} C^{\alpha_{\ms}'} \Phi_{\ms,\mathfrak{w}} \Phi_{\mathfrak{w},\mt}  v_{\mt}, &\text{if $a=b$,}\\
				(-1)^{|\alpha_\ms'|_{>d(\ms,\mt^{\undla})(i_t)}+a|\alpha_\mt|_{>d(\mt,\mt^{\undla})(i_t)}}&\\
				\qquad\qquad\cdot
					(-1)^{|\alpha_{\mt}|_{>d(\mt,\mt^{\undla})(i_t)}}&\\
					\qquad\qquad\qquad\cdot F_{\rm S} C^{\beta_{\ms}'} C^{\alpha_{\ms}'} \Phi_{\ms,\mathfrak{w}} \Phi_{\mathfrak{w},\mt}  C_{d(\mt,\mt^{\undla})(i_t)} v_{\mt},  &\text{if $a\neq b,$}
				\end{cases}\nonumber\\
					&=\begin{cases}(-1)^{|\alpha_\ms'|_{>d(\ms,\mt^{\undla})(i_t)}+a|\alpha_\mt|_{>d(\mt,\mt^{\undla})(i_t)}}&\\
						\qquad\qquad\qquad\cdot \mathtt{c}_{\ms,\mathfrak{w}}\mathtt{c}_{\mathfrak{w},\mt}F_{\rm S} C^{\beta_{\ms}'} C^{\alpha_{\ms}'} v_{\ms}, &\text{if $a=b$,}\\
				(-1)^{|\alpha_\ms'|_{>d(\ms,\mt^{\undla})(i_t)}+a|\alpha_\mt|_{>d(\mt,\mt^{\undla})(i_t)}}&\\
				\qquad\qquad\cdot (-1)^{|\alpha_{\mt}|_{>d(\mt,\mt^{\undla})(i_t)}}\mathtt{c}_{\ms,\mathfrak{w}}\mathtt{c}_{\mathfrak{w},\mt}&\\
				\qquad\qquad\qquad\cdot F_{\rm S} C^{\beta_{\ms}'} C^{\alpha_{\ms}'}  C_{d(\ms,\mt^{\undla})(i_t)} v_{\ms},  &\text{if $a\neq b,$}
				\end{cases}\nonumber\\
			&=\begin{cases}(-1)^{|\alpha_\ms'|_{>d(\ms,\mt^{\undla})(i_t)}+a|\alpha_\mt|_{>d(\mt,\mt^{\undla})(i_t)}}&\\
				\qquad\qquad\qquad\cdot \mathtt{c}_{\ms,\mathfrak{w}}\mathtt{c}_{\mathfrak{w},\mt}F_{\rm S} C^{\beta_{\ms}'} C^{\alpha_{\ms}'} v_{\ms}, &\text{if $a=b$,}\\
				(-1)^{|\alpha_\ms'|_{>d(\ms,\mt^{\undla})(i_t)}+a|\alpha_\mt|_{>d(\mt,\mt^{\undla})(i_t)}}&\\
				\qquad\qquad\cdot 	(-1)^{|\alpha_{\ms}'|_{>d(\ms,\mt^{\undla})(i_t)}+|\alpha_{\mt}|_{>d(\mt,\mt^{\undla})(i_t)}}\mathtt{c}_{\ms,\mathfrak{w}}\mathtt{c}_{\mathfrak{w},\mt}&\\
				\qquad\qquad\qquad\cdot F_{\rm S} C^{\beta_{\ms}'} C^{{\alpha_{\ms,\bar{1}}'} }v_{\ms},  &\text{if $a\neq b,$}
			\end{cases}\nonumber\\
								&=\begin{cases}(-1)^{|\alpha_\ms'|_{>d(\ms,\mt^{\undla})(i_t)}+a|\alpha_\mt|_{>d(\mt,\mt^{\undla})(i_t)}}&\\
									\qquad\qquad\qquad\cdot \mathtt{c}_{\ms,\mathfrak{w}}\mathtt{c}_{\mathfrak{w},\mt} C^{\beta_{\ms}'} C^{\alpha_{\ms}'} v_{\ms}, &\text{if $a=b$,}\\
							(-1)^{(a+1)|\alpha_{\mt}|_{>d(\mt,\mt^{\undla})(i_t)}}&\\
							\qquad\qquad\qquad\cdot\mathtt{c}_{\ms,\mathfrak{w}}\mathtt{c}_{\mathfrak{w},\mt} C^{\beta_{\ms}'} C^{{\alpha_{\ms,\bar{1}}'} }v_{\ms}, &\text{if $a\neq b,$}
							\end{cases}\nonumber\\
				&=		(-1)^{b|\alpha_{\mt}|_{>d(\mt,\mt^{\undla})(i_t)}+(a+b+1)|\alpha_{\ms}'|_{>d(\ms,\mt^{\undla})(i_t)}}\\
				&\qquad\qquad\qquad\cdot\mathtt{c}_{\mathfrak{s},\mathfrak{w} }\mathtt{c}_{\mathfrak{w},\mt }C^{\beta_{\ms}'} C^{\alpha_{\ms,a+b}'  }v_{\ms}
		\end{align*}
		where in the third equation, we have used Lemma \ref{Phist. lem} (1), (2), and in the last second equation, we have used \eqref{F. eq3}. This completes the proof of \eqref{SNB. eq3}. The proof of \eqref{SNB. eq4} is similar.
		\end{enumerate}
				\end{proof}
		The following Theorem is the third main result of this paper.
		\begin{thm}\label{seminormal basis}
			Suppose $P^{(\bullet)}_{n}(q^2,\undQ)\neq 0$. We fix $\mathfrak{w}\in\Std(\undla)$. Then the following two sets
			\begin{align}\label{Non-deg seminormal1}
			\left\{ f_{{\rm S},{\rm T}}^\mathfrak{w} \Biggm|
			{\rm S}=(\ms, \alpha_{\ms}', \beta_{\ms}')\in {\rm Tri}_{\bar{0}}(\undla),
			{\rm T}=(\mt, \alpha_{\mt}, \beta_{\mt})\in {\rm Tri}(\undla)
			\right\}
		\end{align} and \begin{align}\label{Non-deg seminormal2}
			\left\{ f_{{\rm S},{\rm T}} \Biggm|
		{\rm S}=(\ms, \alpha_{\ms}', \beta_{\ms}')\in {\rm Tri}_{\bar{0}}(\undla),
		{\rm T}=(\mt, \alpha_{\mt}, \beta_{\mt})\in {\rm Tri}(\undla)
		\right\}
			\end{align}  form two $\mathbb{K}$-bases of the block $B_{\undla}$ of $\mHfcn$.
			
			Moreover, for ${\rm S}=(\ms, \alpha_{\ms}', \beta_{\ms}')\in {\rm Tri}_{\bar{0}}(\undla),
			{\rm T}=(\mt, \alpha_{\mt}, \beta_{\mt})\in {\rm Tri}(\undla),$ we have \begin{equation}\label{fst and re. fst}
			f_{{\rm S},{\rm T}}
			=\frac{\mathtt{c}_{\ms,\mt}}{\mathtt{c}_{\ms,\mathfrak{w} }\mathtt{c}_{\mathfrak{w},\mt }} f_{{\rm S},{\rm T}}^\mathfrak{w}\in F_{\rm S}\mHfcn F_{\rm T}.
			\end{equation} The multiplications of basis elements in \eqref{Non-deg seminormal1} are given as follows.
			
			(1) Suppose $d_{\undla}=0.$  Then for any
			${\rm S}=(\ms, \alpha_{\ms}', \beta_{\ms}'),
			{\rm T}=(\mt, \alpha_{\mt}, \beta_{\mt}),
			{\rm U}=(\mfku,\alpha_{\mfku}^{''},\beta_{\mfku}^{''}),
			{\rm V}=(\mfkv,\alpha_{\mfkv}^{'''},\beta_{\mfkv}^{'''})\in {\rm Tri}(\undla),$ we have
			\begin{align}\label{Non-deg multiplication1}
				f_{{\rm S},{\rm T}}^\mathfrak{w} f_{{\rm U},{\rm V}}^\mathfrak{w}
				=\delta_{{\rm T},{\rm U}} \mathtt{c}_{\rm T}^\mathfrak{w} f_{{\rm S},{\rm V}}^\mathfrak{w}.
			\end{align}

			(2) Suppose $d_{\undla}=1.$ Then for any $a,b\in \mathbb{Z}_2$ and
			\begin{align*}
				{\rm S}&=(\ms, \alpha_{\ms}', \beta_{\ms}')\in {\rm Tri}_{\bar{0}}(\undla), \quad
				{\rm T}_{a}=(\mt, \alpha_{\mt,a}, \beta_{\mt})\in {\rm Tri}_{a}(\undla),\nonumber\\
				{\rm U}&=(\mfku,\alpha_{\mfku}^{''},\beta_{\mfku}^{''})\in {\rm Tri}_{\bar{0}}(\undla), \quad
				{\rm V}_{b}=(\mfkv,{\alpha_{\mfkv,b}^{'''}},\beta_{\mfkv}^{'''})\in {\rm Tri}_{b}(\undla),\nonumber
			\end{align*} we have
			\begin{align}\label{Non-deg multiplication2}
				f_{{\rm S},{\rm T}_{a}}^\mathfrak{w} f_{{\rm U},{\rm V}_{b}}^\mathfrak{w}
				=\delta_{{\rm T}_{\bar{0}},{\rm U}}(-1)^{\left(|\alpha_{\mt}|_{>d(\mt,\mt^{\undla})(i_t)}\right)}\mathtt{c}_{\rm T}^\mathfrak{w} f_{{\rm S},{\rm V}_{a+b}}^\mathfrak{w}.
			\end{align}
		
		\end{thm}
		
		\begin{proof}
			(1) Supppose $d_{\undla}=0,$ then the block $B_{\lambda}$ is of type $\texttt{M}$ by Theorem \ref{semisimple:non-dege}. Using \eqref{SNB. eq1}, \eqref{SNB. eq2}, we deduce that  $f_{{\rm S},{\rm T}} ^{\mathfrak{w}},\,f_{{\rm S},{\rm T}} \neq 0$ in the $1$-dimensional space
			$$F_{\rm S} \mHfcn F_{\rm T}\cong {\rm Hom}_{\mHfcn}(\mathbb{D}_{\rm S},\mathbb{D}_{\rm T}).$$ Hence \eqref{Non-deg seminormal1} and \eqref{Non-deg seminormal2} form two bases of the block $B_{\undla}$. For equations \eqref{fst and re. fst} and \eqref{Non-deg multiplication1}, we only need to check they act on $\mathbb{D}(\undla)$ as the same operator by Theorem \ref{semisimple:non-dege}. These are direct consequences of \eqref{SNB. eq1}, \eqref{SNB. eq2}.

			(2) Suppose $d_{\undla}=1$, then the block $B_{\lambda}$ is of type $\texttt{Q}$ by Theorem \ref{semisimple:non-dege}.  Using \eqref{SNB. eq1}, \eqref{SNB. eq2}, we deduce that  $\{f_{{\rm S},{\rm T}_{a}}^\mathfrak{w}\mid a\in \mathbb{Z}_2\}$ and $\{f_{{\rm S},{\rm T}_{a}}\mid a\in \mathbb{Z}_2\}$ form two $\mathbb{K}$-bases of the $2$-dimensional space
			$$F_{\rm S} \mHfcn F_{\rm T}\cong {\rm Hom}_{\mHfcn}(\mathbb{D}_{\rm S},\mathbb{D}_{\rm T}).$$
		Hence \eqref{Non-deg seminormal1} and \eqref{Non-deg seminormal2} form two bases of the block $B_{\undla}$. Similarly, equations \eqref{fst and re. fst} and \eqref{Non-deg multiplication2} are direct consequences of \eqref{SNB. eq1}, \eqref{SNB. eq2} and Theorem \ref{semisimple:non-dege}.
		\end{proof}
		
		\begin{rem}
			Assume $d_{\undla}=0,$ then for any
			${\rm S}=(\ms, \alpha_{\ms}', \beta_{\ms}'),
			{\rm T}=(\mt, \alpha_{\mt}, \beta_{\mt}),
			{\rm U}=(\mfku,\alpha_{\mfku}^{''},\beta_{\mfku}^{''}),
			{\rm V}=(\mfkv,\alpha_{\mfkv}^{'''},\beta_{\mfkv}^{'''})\in {\rm Tri}(\undla),$ we have
			\begin{align*}
				f_{{\rm S},{\rm T}}f_{{\rm U},{\rm V}}
				=\delta_{{\rm T},{\rm U}} \frac{\mathtt{c}_{\ms,\mt}\mathtt{c}_{\mt,\mathfrak{v}} }{\mathtt{c}_{\ms,\mfkv}}f_{{\rm S},{\rm V}}.
			\end{align*}
			\end{rem}
	The following example shows that the coefficient $$\frac{\mathtt{c}_{\ms,\mt}\mathtt{c}_{\mt,\mathfrak{v}} }{\mathtt{c}_{\ms,\mfkv}}$$ may depend on $\ms$ and $\mfkv$ in general, where $\ms, \mt,\mfkv \in \Std(\undla).$
		\begin{example}
			Let
			$$\undla=\young(\,\,\, ,:\,)\in \mathscr{P}^{\mathsf{s},0}_{4},$$
			and $$\ms:=\mt^{\undla}=\young(123,:4),\quad \mt:=s_3\ms=\young(124,:3).$$
			Then one can check
			$$\mathtt{b}_{\mt,3}=\mathtt{b}_{\ms,4}=1,$$
			$$\mathtt{b}_{\mt,4}=\mathtt{b}_{\ms,3}=[3]_{q^2}-[2]_{q^2}-\epsilon\sqrt{[3]_{q^2} [2]_{q^2}}\neq 1.$$	
	Hence
			$$\frac{\mathtt{c}_{\mt,\mt}\mathtt{c}_{\mt,\mt}}{\mathtt{c}_{\mt,\mt}}=\frac{1\cdot1}{1}=1,$$
			$$\frac{\mathtt{c}_{\ms,\mt}\mathtt{c}_{\mt,\ms}}{\mathtt{c}_{\ms,\ms}}
			=\frac{\mathtt{c}_{\ms,\mt}\mathtt{c}_{\mt,\ms}}{1}=1-\epsilon^2 \frac{2\mathtt{b}_{\mt,4}}{\mathtt{b}_{\mt,4}-1}\neq 1.$$
		\end{example}
	\subsection{Further properties}	\label{action}{\bf In this subsection, we fix $\undla\in\mathscr{P}^{\bullet,m}_{n}$ with $\bullet\in\{\mathsf{0},\mathsf{s},\mathsf{ss}\}.$} We shall consider the actions of anti-involution $*$ and generators of $\mHfcn$ on our seminormal bases.
			\begin{defn}\label{sgn}
            \label{pag:sgn}
			Let $\mt\in \Std(\undla),\,\alpha_{\mt}\in \mathbb{Z}_2 (\mathcal{OD}_{\mt})$.
			\begin{enumerate}
				\item  Suppose $d_{\undla}=0,$  i.e. $t$ is even, we define
				$$\sgn(\mt,\alpha_{\mt}):=(-1)^{\sum\limits_{j=1}^{\frac{t}{2} }\bigl|\alpha_{\mt}\bigr|_{>d(\mt,\mt^{\undla})(i_{2j-1})}+\sum\limits_{1\leq j'<j\leq  \frac{t}{2} }\bigl|e_{d(\mt,\mt^{\undla})(i_{2j'-1})}\bigr|_{>d(\mt,\mt^{\undla})(i_{2j-1})}}.$$
				\item  Suppose $d_{\undla}=1,$ i.e. $t$ is odd, for $a\in \Z_2,$ we define
				$$\sgn(\mt,\alpha_{\mt})_{a}:=\begin{cases}(-1)^{\sum\limits_{j=1}^{\lceil \frac{t}{2} \rceil}\bigl|\alpha_{\mt}\bigr|_{>d(\mt,\mt^{\undla})(i_{2j-1})}+\sum\limits_{1\leq j'<j\leq \lfloor \frac{t}{2} \rfloor}\bigl|e_{d(\mt,\mt^{\undla})(i_{2j'-1})}\bigr|_{>d(\mt,\mt^{\undla})(i_{2j-1})}}, \text{if $a=\bar{0},$}\\
					(-1)^{\sum\limits_{j=1}^{\lceil \frac{t}{2} \rceil}\bigl|\alpha_{\mt}\bigr|_{>d(\mt,\mt^{\undla})(i_{2j-1})}+\sum\limits_{1\leq j'<j\leq \lceil \frac{t}{2} \rceil}\bigl|e_{d(\mt,\mt^{\undla})(i_{2j'-1})}\bigr|_{>d(\mt,\mt^{\undla})(i_{2j-1})}}, \text{if $a=\bar{1}.$}
					\end{cases}$$
				
				\end{enumerate}
		\end{defn}
		
		It's easy to check the following.
		
		\begin{lem}Let $\mt\in \Std(\undla),\,\alpha_{\mt}\in \mathbb{Z}_2 (\mathcal{OD}_{\mt})_{\bar{0}}$.
			\begin{enumerate}
				\item Suppose $d_{\undla}=0,$ then we have
				\begin{align}\label{sgn: type M}
					(C^{\widehat{\alpha_{\mt}}})^{-1}C^{\alpha_{\mt}}=\sgn(\mt,\widehat{\alpha_{\mt}})\overrightarrow{\prod_{j=1,2,\cdots, \frac{t}{2}}}C_{d(\mt,\mt^{\undla})(2i_j-1)}.
				\end{align}
				\item 	Suppose $d_{\undla}=1,$ then for any $a,b\in \Z_2$, we have	\begin{align}\label{sgn: type Q}
					(C^{\widehat{\alpha_{\mt,a}}})^{-1}C^{\alpha_{\mt,b}}=
						(-1)^{(a+b+\bar{1})\bigl(\lfloor \frac{t}{2}\rfloor+|\alpha_{\mt}|_{>d(\mt,\mt^{\undla})(i_t)}\bigr)}\sgn(\mt,\widehat{\alpha_{\mt,b+\overline{1}}})_{a+b}\overrightarrow{\prod\limits_{j=1,2,\cdots, \lfloor \frac{t}{2}\rfloor}}C_{d(\mt,\mt^{\undla})(2i_j-1)} \cdot C_{d(\mt,\mt^{\undla})(i_t)}^{a+b}
				\end{align}
			\end{enumerate}
		\end{lem}

		The following Lemma will be used in the discussion of anti-involution $*$.
		\begin{lem}\label{commute odd}
			(1) If $d_{\undla}=0.$ For any $\ms,\mt\in \Std(\undla), \alpha_{\ms}'\in \Z(\mathcal{OD}_{\ms}), \alpha_{\mt}\in\Z(\mathcal{OD}_{\mt}),$ we have
$$		(C^{\widehat{\alpha_{\ms}'}})^{-1}C^{\alpha_{\ms}'}\Phi_{\ms,\mt}=\sgn(\ms,\widehat{\alpha_{\ms}'})\sgn(\mt,\widehat{\alpha_{\mt}})\Phi_{\ms,\mt}(C^{\widehat{\alpha_{\mt}}})^{-1}C^{\alpha_{\mt}}.
$$
			(2) If $d_{\undla}=1.$ For any $\ms,\mt\in \Std(\undla), \alpha_{\ms}'\in \Z(\mathcal{OD}_{\ms})_{\bar{0}}, \alpha_{\mt}\in\Z(\mathcal{OD}_{\mt})_{\bar{0}}$ and $a,b\in \Z_2$, we have \begin{align*}
			(C^{\widehat{\alpha_{\ms,b}'}})^{-1}C^{\alpha_{\ms}'}\Phi_{\ms,\mt}=&(-1)^{(b+1)(|\alpha'_{\ms}|_{>d(\ms,\mt^{\undla})(i_t)}+|\alpha_{\mt}|_{>d(\mt,\mt^{\undla})(i_t)})}\\
			&\qquad\qquad\qquad\cdot\sgn(\ms,\widehat{\alpha'_{\ms,\overline{1}}})_{b}\sgn(\mt,\widehat{\alpha_{\mt,a+\overline{1}}})_{b}\Phi_{\ms,\mt}(C^{\widehat{\alpha_{\mt,a+b}}})^{-1}C^{\alpha_{\mt,a}}.
\end{align*}
			\end{lem}
		
		\begin{proof}
		This follows from \eqref{sgn: type M}, \eqref{sgn: type Q} and Lemma \ref{Phist. lem} (2).
			\end{proof}
			Now we can describe the image of seminormal basis under anti-involution $*$.
		\begin{prop}\label{re. fst and BK involution}We fix $\mathfrak{w}\in\Std(\undla).$
			
			(1) If $d_{\undla}=0.$ For any ${\rm S}=(\ms, \alpha_{\ms}', \beta_{\ms}'), {\rm T}=(\mt, \alpha_{\mt}, \beta_{\mt})\in {\rm Tri}(\undla),$ we have
			\begin{align}\label{Non-dege dual 1}
				(f_{{\rm S},{\rm T}})^*
				=\sgn(\ms,\widehat{\alpha_{\ms}'})\sgn(\mt,\widehat{\alpha_{\mt}})f_{{\rm \widehat{T}},{\rm \widehat{S}}}
			\end{align} and
				\begin{align}\label{Non-dege dual 2}
			(f_{{\rm S},{\rm T}}^\mathfrak{w})^*
				=\sgn(\ms,\widehat{\alpha_{\ms}'})\sgn(\mt,\widehat{\alpha_{\mt}})f_{{\rm \widehat{T}},{\rm \widehat{S}}}^\mathfrak{w}.
			\end{align}
%
			(2) If $d_{\undla}=1.$ For any $a\in \mathbb{Z}_{2}$ and ${\rm S}=(\ms, \alpha_{\ms}', \beta_{\ms}')\in {\rm Tri}_{\bar{0}}(\undla), {\rm T}_{a}=(\mt, \alpha_{\mt,a}, \beta_{\mt})\in {\rm Tri}_{a}(\undla),$ we have
		$$
			(f_{{\rm S},{\rm T}_{a}})^*
		=\sgn(\ms,\widehat{\alpha'_{\ms,\overline{1}}})_{\overline{0}}\sgn(\mt,\widehat{\alpha_{\mt,a+\overline{1}}})_{a+\overline{1}}f_{{\rm \widehat{T}},{\rm \widehat{S_a} } }
		$$ and
		$$
				(f_{{\rm S},{\rm T}_{a}}^\mathfrak{w})^*
		=\sgn(\ms,\widehat{\alpha'_{\ms,\overline{1}}})_{\overline{0}}\sgn(\mt,\widehat{\alpha_{\mt,a+\overline{1}}})_{a+\overline{1}}f_{{\rm \widehat{T}},{\rm \widehat{S_a} } }^\mathfrak{w}.$$
			
		\end{prop}
		\begin{proof}
			Let's prove (2). We fix a reduced expression $d(\ms,\mt)=s_{k_p}\cdots s_{k_1}.$  Note that $$
			\Phi_{i}(x,y)^*=\Phi_i(y^{-1},x^{-1})$$ for $i=1,2,\ldots,n-1$ and $x,y \in \mathbb{K}^*$ satisfying $y\neq x^{\pm 1}.$ It follows that
			\begin{align}\label{Phist*}
				(\Phi_{\ms,\mt})^*
				&=\overrightarrow{\prod_{i=1,\ldots,p}}\Phi_{k_{i}}(\mathtt{b}^{-1}_{s_{k_{i-1}}\cdots s_{k_1}\mathfrak{t},k_{i}+1},\mathtt{b}^{-1}_{s_{k_{i-1}}\cdots s_{k_1}\mathfrak{t},k_{i}})\\
				&=\overrightarrow{\prod_{i=1,\ldots,p}}\left(\Phi_{k_{i}}(\mathtt{b}_{s_{k_{i+1}}\cdots s_{k_{p}}\mathfrak{s},k_{i}}, \mathtt{b}_{s_{k_{i+1}}\cdots s_{k_{p}}\mathfrak{s},k_{i}+1})
				+a_i
			+d_iC_{k_i}C_{k_{i+1}}\right)\nonumber\\
				&=\Phi_{\mt,\ms}+\sum_{\beta\in \mathbb{Z}_2^n,d(\mt,\ms)>u\in \mathfrak{S}_n}a_{\beta,u}C_{\beta}T_u, \nonumber
			\end{align} where $a_1,\cdots,a_p,d_1,\cdots,d_p$ and $a_{\beta,u}\in \mathbb{K}.$
			
			On the other hand, using \eqref{Taction Non-dege}, we can deduce that for any $d(\mt,\ms)>u\in \mathfrak{S}_n$,
\begin{align}\label{cancel}
T_u v_{\ms}=\sum_{\substack{\beta\subset \Z_2^n\\ \mfkv\neq \mt}}g_{\beta,\mfkv} C^{\beta} v_{\mfkv}
\end{align}
for some $g_{\beta,\mfkv}\in \mathbb{K}.$

		This implies that for any $c\in \Z_2$, \begin{align}\label{tiny2}
			F_{\rm \widehat{T}}C^{\beta_{\mt}}C^{\widehat{\alpha_{\mt,c}}}(\Phi_{\ms,\mt})^*v_{\ms}&=F_{\rm \widehat{T}}C^{\beta_{\mt}}C^{\widehat{\alpha_{\mt,c}}}
			\left( \Phi_{\mt,\ms}+\sum_{\beta\in \mathbb{Z}_2^n,d(\mt,\ms)>u\in \mathfrak{S}_n}a_{\beta,u}C_{\beta}T_u \right)v_{\ms}\nonumber\\
			&=F_{\rm \widehat{T}}C^{\beta_{\mt}}C^{\widehat{\alpha_{\mt,c}}}
			\Phi_{\mt,\ms}v_{\ms}\nonumber\\
			&=\mathtt{c}_{\mt,\ms}C^{\beta_{\mt}}C^{\widehat{\alpha_{\mt,c}}}v_{\mt}
			\end{align} where in the first equation, we have used \eqref{Phist*}; in the second equation, we have used  \eqref{F. eq3} and \eqref{cancel}; in the last equation, we have used Lemma \ref{Phist. lem} (3) and \eqref{F. eq3}.
		Now we can compute
			\begin{align*}
				&(f_{{\rm S},{\rm T}_a})^* \cdot C^{\beta_{\ms}'}C^{\widehat{\alpha_{\ms,b}'}}v_{\ms}\nonumber\\
				&=(-1)^{|\alpha_\ms'|_{>d(\ms,\mt^{\undla})(i_t)}+a|\alpha_\mt|_{>d(\mt,\mt^{\undla})(i_t)}}\\
				&\qquad\cdot F_{\rm \widehat{T}}C^{\beta_{\mt}}C^{\alpha_{\mt,a}}
				(\Phi_{\ms,\mt})^*(C^{\alpha_{\ms}'})^{-1}(C^{\beta_{\ms}'})^{-1}
				F_{\rm \widehat{S}} \\
				&\qquad\qquad \cdot C^{\beta_{\ms}'}C^{\widehat{\alpha_{\ms,b}'}}v_{\ms}\nonumber\\
				&=(-1)^{|\alpha_\ms'|_{>d(\ms,\mt^{\undla})(i_t)}+a|\alpha_\mt|_{>d(\mt,\mt^{\undla})(i_t)}}\\
				&\qquad\cdot F_{\rm \widehat{T}}C^{\beta_{\mt}}C^{\alpha_{\mt,a}}
				(\Phi_{\ms,\mt})^*(C^{\alpha_{\ms}'})^{-1}C^{\widehat{\alpha_{\ms,b}'}}v_{\ms}\nonumber\\
				&=(-1)^{|\alpha_\ms'|_{>d(\ms,\mt^{\undla})(i_t)}+a|\alpha_\mt|_{>d(\mt,\mt^{\undla})(i_t)}}\\
				&\qquad\cdot F_{\rm \widehat{T}}C^{\beta_{\mt}}C^{\alpha_{\mt,a}}
				\biggl((C^{\widehat{\alpha_{\ms,b}'}})^{-1}C^{\alpha_{\ms}'}\Phi_{\ms,\mt}\biggr)^*v_{\ms}\nonumber\\
				&=(-1)^{|\alpha_\ms'|_{>d(\ms,\mt^{\undla})(i_t)}+a|\alpha_\mt|_{>d(\mt,\mt^{\undla})(i_t)}+(b+1)(|\alpha'_{\ms}|_{>d(\ms,\mt^{\undla})(i_t)}+|\alpha_{\mt}|_{>d(\mt,\mt^{\undla})(i_t)})}\\
				&\qquad\cdot\sgn(\ms,\widehat{\alpha'_{\ms,\overline{1}}})_{b}\sgn(\mt,\widehat{\alpha_{\mt,a+\overline{1}}})_{b}\\
								&\qquad\qquad\cdot F_{\rm \widehat{T}}C^{\beta_{\mt}}C^{\alpha_{\mt,a}}
				\biggl(\Phi_{\ms,\mt}(C^{\widehat{\alpha_{\mt,a+b}}})^{-1}C^{\alpha_{\mt,a}}\biggr)^*v_{\ms}\nonumber	\\
			&=(-1)^{b|\alpha_\ms'|_{>d(\ms,\mt^{\undla})(i_t)}+(a+b+1)|\alpha_\mt|_{>d(\mt,\mt^{\undla})(i_t)}}\\
			&\qquad\cdot\sgn(\ms,\widehat{\alpha'_{\ms,\overline{1}}})_{b}\sgn(\mt,\widehat{\alpha_{\mt,a+\overline{1}}})_{b}F_{\rm \widehat{T}}C^{\beta_{\mt}}C^{\widehat{\alpha_{\mt,a+b}}}
			(\Phi_{\ms,\mt})^*v_{\ms}\nonumber	\\		
			&=(-1)^{b|\alpha_\ms'|_{>d(\ms,\mt^{\undla})(i_t)}+(a+b+1)|\alpha_\mt|_{>d(\mt,\mt^{\undla})(i_t)}}\\
			&\qquad\cdot\sgn(\ms,\widehat{\alpha'_{\ms,\overline{1}}})_{b}\sgn(\mt,\widehat{\alpha_{\mt,a+\overline{1}}})_{b} \mathtt{c}_{\mt,\ms}C^{\beta_{\mt}}C^{\widehat{\alpha_{\mt,a+b}}}v_{\mt}\\
			&=(-1)^{b|\alpha_\ms'|_{>d(\ms,\mt^{\undla})(i_t)}+(a+b+1)|\alpha_\mt|_{>d(\mt,\mt^{\undla})(i_t)}+b|\widehat{\alpha'_\ms}|_{>d(\ms,\mt^{\undla})(i_t)}+(a+b+1)|\widehat{\alpha_\mt}|_{>d(\mt,\mt^{\undla})(i_t)}}\\
			&\qquad\cdot\sgn(\ms,\widehat{\alpha'_{\ms,\overline{1}}})_{b}\sgn(\mt,\widehat{\alpha_{\mt,a+\overline{1}}})_{b}\cdot 	f_{{\rm \widehat{T}},{\rm \widehat{S_a}}} \cdot C^{\beta_{\ms}'}C^{\widehat{\alpha_{\ms,b}'}}v_{\ms}\\
			&=(-1)^{b\sum\limits_{1\leq j'\leq \lfloor \frac{t}{2} \rfloor}\bigl|e_{d(\ms,\mt^{\undla})(i_{2j'-1})}\bigr|_{>d(\ms,\mt^{\undla})(i_{t})}
				+(a+b+1)\sum\limits_{1\leq j'\leq \lfloor \frac{t}{2} \rfloor}\bigl|e_{d(\mt,\mt^{\undla})(i_{2j'-1})}\bigr|_{>d(\mt,\mt^{\undla})(i_{t})}}\\
			&\qquad\cdot\sgn(\ms,\widehat{\alpha'_{\ms,\overline{1}}})_{b}\sgn(\mt,\widehat{\alpha_{\mt,a+\overline{1}}})_{b}\cdot 	f_{{\rm \widehat{T}},{\rm \widehat{S_a}}} \cdot C^{\beta_{\ms}'}C^{\widehat{\alpha_{\ms,b}'}}v_{\ms}\\
				&=\sgn(\ms,\widehat{\alpha'_{\ms,\overline{1}}})_{\overline{0}}\sgn(\mt,\widehat{\alpha_{\mt,a+\overline{1}}})_{a+\overline{1}}\cdot 	f_{{\rm \widehat{T}},{\rm \widehat{S_a}}} \cdot C^{\beta_{\ms}'}C^{\widehat{\alpha_{\ms,b}'}}v_{\ms}
				\end{align*} where in the first equation, we have used Corollary \ref{Non-dege BK anti-idempotent}; in the fourth equation, we have used Lemma \ref{commute odd} and in the last fourth equation, we have used \eqref{tiny2}.
				 Using  \eqref{F. eq3}, for any ${\rm U}=(\mfkv,\alpha_{\mfkv}^{'''},\beta_{\mfkv}^{'''})\neq {\rm \widehat{T}_b} \in {\rm Tri}(\undla)$ for any $b\in \Z_2$, we have $$(f_{{\rm S},{\rm T}})^* \cdot C^{\beta_{\mfkv}^{'''}}C^{\alpha_{\mfkv}^{'''}}v_{\mfkv}=0.$$
				 Putting these together,
	we deduce that $ (f_{{\rm S},{\rm T}_a})^* $ and $$\sgn(\ms,\widehat{\alpha'_{\ms,\overline{1}}})_{\overline{0}}\sgn(\mt,\widehat{\alpha_{\mt,a+\overline{1}}})_{a+\overline{1}}f_{{\rm \widehat{T}},{\rm \widehat{S_a}}}$$ act on $\mathbb{D}(\undla)$ as the same operator. Hence \eqref{Non-dege dual 1} follows from Theorem \ref{semisimple:non-dege}. Note that \eqref{Non-dege dual 2} is the Corollary of \eqref{fst and re. fst} and \eqref{Non-dege dual 1}. This proves (2). The proof of (1) is similar as (2).
		\end{proof}
\begin{defn}
Let $\undla\in\mathscr{P}^{\bullet,m}_{n}$ with $\bullet\in\{\mathsf{0},\mathsf{s},\mathsf{ss}\}.$
\begin{enumerate}
	\item
	Suppose $d_{\undla}=0.$ Let ${\rm S}=(\ms, \alpha_{\ms}', \beta_{\ms}'),
	{\rm T}=(\mt, \alpha_{\mt}, \beta_{\mt})\in {\rm Tri}(\undla).$ For each $i\in [n-1],$ we define the element
$F(i,{\rm T},{\rm S})\in\mHfcn$ as follows.

\begin{enumerate}
	\item If $i,i+1\in [n]\setminus \mathcal{D}_{\mt},$
	$$F(i,{\rm T},{\rm S}):=(-1)^{\delta_{\beta_{\mt}}(i)}f_{(\mt,\alpha_{\mt},\beta_{\mt}+e_{i}+e_{i+1}),{\rm S}};
	$$
	\item
	if $i={d(\mt,\mt^{\undla})}(i_p)\in \mathcal{D}_{\mt}, i+1\in [n]\setminus \mathcal{D}_{\mt},$
	$$F(i,{\rm T},{\rm S}):=\begin{cases}(-1)^{|\beta_{\mt}|_{>i}+|\alpha_{\mt}|_{<i}}f_{(\mt,\alpha_{\mt}+e_{i},\beta_{\mt}+e_{i+1}),{\rm S}}, & \text{ if  $p$ is odd},\\
		(-\sqrt{-1})(-1)^{|\beta_{\mt}|_{>i}+|\alpha_{\mt}|_{\leq {d(\mt,\mt^{\undla})}(i_{p-1})}}f_{(\mt,\alpha_{\mt}+e_{{d(\mt,\mt^{\undla})}(i_{p-1})},\beta_{\mt}+e_{i+1}),{\rm S}}, &\text{ if  $p$ is even};
	\end{cases}\nonumber
	$$
	\item
	if $i+1={d(\mt,\mt^{\undla})}(i_p)\in \mathcal{D}_{\mt}, i\in [n]\setminus \mathcal{D}_{\mt},$ $$F(i,{\rm T},{\rm S}):=\begin{cases}(-1)^{|\beta_{\mt}|_{\geq i}+|\alpha_{\mt}|_{<i+1}} f_{(\mt,\alpha_{\mt}+e_{i+1},\beta_{\mt}+e_{i}),{\rm S}}, & \text{ if  $p$ is odd},\\
		(-\sqrt{-1})(-1)^{|\beta_{\mt}|_{\geq i}+|\alpha_{\mt}|_{\leq {d(\mt,\mt^{\undla})}(i_{p-1})}}f_{(\mt,\alpha_{\mt}+e_{{d(\mt,\mt^{\undla})}(i_{p-1})},\beta_{\mt}+e_{i}),{\rm S}}, &\text{ if $p$ is even};\end{cases}  \nonumber
	$$
	\item
	if $i={d(\mt,\mt^{\undla})}(i_p), i+1={d(\mt,\mt^{\undla})}(i_{p'})\in \mathcal{D}_{\mt},$ \begin{equation*}
		F(i,{\rm T},{\rm S}):=\begin{cases}(-1)^{\delta_{\alpha_{\mt}}(i)} f_{(\mt,\alpha_{\mt}+e_i+e_{i+1},\beta_{\mt}),{\rm S}} & \text{ if  $p,p'$ are odd},\\
			(-\sqrt{-1})(-1)^{|\alpha_{\mt}|_{<i+1}+|\alpha_{\mt}+e_{i+1}|_{\leq{d(\mt,\mt^{\undla})}(i_{p-1})}} &\\
			\qquad\qquad\qquad \cdot f_{(\mt,\alpha_{\mt}+e_{{d(\mt,\mt^{\undla})}(i_{p-1})}+e_{i+1},\beta_{\mt}),{\rm S}} ,& \text{ if  $p$ is even, $p'$ is odd},\\
			(-\sqrt{-1})(-1)^{|\alpha_{\mt}|_{\leq {d(\mt,\mt^{\undla})}(i_{p'-1})}+|\alpha_{\mt}+e_{{d(\mt,\mt^{\undla})}(i_{p'-1})}|_{<i}}&\\
			\qquad\qquad\qquad \cdot f_{(\mt,\alpha_{\mt}+e_i+e_{{d(\mt,\mt^{\undla})}(i_{p'-1})},\beta_{\mt}),{\rm S}} &\text{ if  $p$ is odd, $p'$ is even},\\
			(-1)^{1+|\alpha_{\mt}|_{\leq {d(\mt,\mt^{\undla})}(i_{p'-1})}+|\alpha_{\mt}+e_{{d(\mt,\mt^{\undla})}(i_{p'-1})}|_{\leq{d(\mt,\mt^{\undla})}(i_{p-1})}}&\\
			\qquad\qquad\qquad\cdot f_{(\mt,\alpha_{\mt}+e_{{d(\mt,\mt^{\undla})}(i_{p-1})}+e_{{d(\mt,\mt^{\undla})}(i_{p'-1})},\beta_{\mt}),{\rm S}} &\text{ if  $p,p'$ are even}.
		\end{cases}
	\end{equation*}
	\end{enumerate}
		\item	Suppose $d_{\undla}=1.$
	Let $a\in \Z_2$, ${\rm T}=(\mt, \alpha_{\mt}, \beta_{\mt})\in {\rm Tri}_{\bar{0}}(\undla),$
	${\rm S}_{a}=(\ms, {\alpha_{\ms,a}'}, \beta_{\ms}')\in {\rm Tri}_{a}(\undla). $ For each $i\in [n-1],$ we define $F(i,{\rm T},{\rm S}_a)$ as follows.
	
	\begin{enumerate}
		\item If $i,i+1\in [n]\setminus \mathcal{D}_{\mt},$
		$$F(i,{\rm T},{\rm S}_a):=(-1)^{\delta_{\beta_{\mt}}(i)}f_{(\mt,\alpha_{\mt},\beta_{\mt}+e_{i}+e_{i+1}),{\rm S}_a};
		$$
		\item
		if $i={d(\mt,\mt^{\undla})}(i_p)\in \mathcal{D}_{\mt}, i+1\in [n]\setminus \mathcal{D}_{\mt},$
		$$F(i,{\rm T},{\rm S}_a):=\begin{cases}(-1)^{|\beta_{\mt}|_{>i}+|\alpha_{\mt,\bar{1}}|_{<i}}	f_{(\mt,\alpha_{\mt}+e_i,\beta_{\mt}+e_{i+1}),{\rm S}_a}, & \text{ if  $p\neq t$ is odd},\\
			(-1)^{|\beta_{\mt}|_{>i}+|\alpha_{\mt}|}	f_{(\mt,\alpha_{\mt},\beta_{\mt}+e_{i+1}),{\rm S}_{a+\bar{1}}}, & \text{ if  $p=t$},\\
			(-\sqrt{-1})(-1)^{|\beta_{\mt}|_{>i}+|\alpha_{\mt,\bar{1}}|_{\leq {d(\mt,\mt^{\undla})}(i_{p-1})}}&\\
			\qquad\qquad	\qquad\qquad \cdot f_{(\mt,\alpha_{\mt}+e_{{d(\mt,\mt^{\undla})}(i_{p-1})},\beta_{\mt}+e_{i+1}),{\rm S}_a}, &\text{ if  $p$ is even};
		\end{cases}
		$$
		\item
		if $i+1={d(\mt,\mt^{\undla})}(i_p)\in \mathcal{D}_{\mt}, i\in [n]\setminus \mathcal{D}_{\mt},$ $$F(i,{\rm T},{\rm S}_a):=\begin{cases}(-1)^{|\beta_{\mt}|_{\geq i}+|\alpha_{\mt,\bar{1}}|_{<i+1}}	f_{(\mt,\alpha_{\mt}+e_{i+1},\beta_{\mt}+e_{i}),{\rm S}_a},& \text{ if  $p\neq t$ is odd},\\
			(-1)^{|\beta_{\mt}|_{\geq i}+|\alpha_{\mt}|}	f_{(\mt,\alpha_{\mt},\beta_{\mt}+e_{i}),{\rm S}_{a+\bar{1}}},& \text{ if  $p=t$},\\
			(-\sqrt{-1})(-1)^{|\beta_{\mt}|_{\geq i}+|\alpha_{\mt,\bar{1}}|_{\leq {d(\mt,\mt^{\undla})}(i_{p-1})}}&\\
			\qquad\qquad\qquad\qquad\cdot	f_{(\mt,\alpha_{\mt}+e_{{d(\mt,\mt^{\undla})}(i_{p-1})},\beta_{\mt}+e_{i}),{\rm S}_a},&\text{ if $p$ is even};\end{cases}
		$$
		\item
		if $i={d(\mt,\mt^{\undla})}(i_p), i+1={d(\mt,\mt^{\undla})}(i_{p'})\in \mathcal{D}_{\mt},$ \begin{equation*}
			F(i,{\rm T},{\rm S}_a):=\begin{cases}(-1)^{\delta_{\alpha_{\mt}}(i)}f_{(\mt,\alpha_{\mt}+e_i+e_{i+1},\beta_{\mt}),{\rm S}_a}, & \text{ if  $p\neq t\neq p'$ are odd},\\
				(-1)^{|\alpha_\mt|_{\geq i}}f_{(\mt,\alpha_{\mt}+e_{i+1},\beta_{\mt}),{\rm S}_{a+\bar{1}}}, & \text{ if  $p=t$ and $ p'$ is odd},\\
				(-1)^{|\alpha_\mt|_{\geq i}}f_{(\mt,\alpha_{\mt}+e_{i},\beta_{\mt}),{\rm S}_{a+\bar{1}}}, & \text{ if  $p$ is odd, $ t=p'$},\\								
				(-\sqrt{-1})(-1)^{|\alpha_{\mt,\bar{1}}|_{<i+1}}&\\
				\qquad\cdot(-1)^{|\alpha_{\mt,\bar{1}}+e_{i+1}|_{\leq{d(\mt,\mt^{\undla})}(i_{p-1})}}&\\
				\qquad\qquad\cdot	 f_{(\mt,\alpha_{\mt}+e_{{d(\mt,\mt^{\undla})}(i_{p-1})}+e_{i+1}, \beta_{\mt}),{\rm S}_{a}},& \text{ if  $p$ is even, $p'\neq t$ is odd},\\
				(-\sqrt{-1})(-1)^{|\alpha_{\mt}|+|\alpha_{\mt,\bar{1}}|_{\leq{d(\mt,\mt^{\undla})}(i_{p-1})}}  &\\
				\qquad\qquad\cdot f_{(\mt,\alpha_{\mt}+e_{{d(\mt,\mt^{\undla})}(i_{p-1})}, \beta_{\mt}),{\rm S}_{a+\bar{1}}},& \text{ if  $p$ is even, $p'=t$},\\
				(-\sqrt{-1})(-1)^{|\alpha_{\mt,\bar{1}}|_{\leq {d(\mt,\mt^{\undla})}(i_{p'-1})}}&\\
				\qquad\cdot	(-1)^{|\alpha_{\mt,\bar{1}}+e_{{d(\mt,\mt^{\undla})}(i_{p'-1})}|_{<i}}&\\
				\qquad\qquad\cdot  f_{(\mt,\alpha_{\mt}+e_{{d(\mt,\mt^{\undla})}(i_{p'-1})}+e_{i}, \beta_{\mt}),{\rm S}_{a}} , &\text{ if  $p\neq t$ is odd, $p'$ is even},\\
				(-\sqrt{-1})(-1)^{|\alpha_{\mt,\bar{1}}|_{> {d(\mt,\mt^{\undla})}(i_{p'-1})}}&\\
				\qquad\qquad\cdot 	f_{(\mt,\alpha_{\mt}+e_{{d(\mt,\mt^{\undla})}(i_{p'-1})}, \beta_{\mt}),{\rm S}_{a+\bar{1}}} ,&\text{ if  $p=t$ and $p'$ is even},\\
				=(-1)^{1+|\alpha_{\mt,\bar{1}}|_{\leq {d(\mt,\mt^{\undla})}(i_{p'-1})}}&\\
				\qquad\cdot		(-1)^{|\alpha_{\mt,\bar{1}}+e_{{d(\mt,\mt^{\undla})}(i_{p'-1})}|_{\leq{d(\mt,\mt^{\undla})}(i_{p-1})}}&\\
				\qquad\qquad\cdot f_{(\mt,\alpha_{\mt}+e_{{d(\mt,\mt^{\undla})}(i_{p-1})}+e_{{d(\mt,\mt^{\undla})}(i_{p'-1})}, \beta_{\mt}),{\rm S}_{a}}&\text{ if  $p,p'$ are even}.
			\end{cases}
		\end{equation*}
	\end{enumerate}
\end{enumerate}
\end{defn}
			The action of the generators on the seminormal basis is given by the following result.
		\begin{prop}\label{generators action on seminormal basis}
			Let $\undla\in\mathscr{P}^{\bullet,m}_{n}$ with $\bullet\in\{\mathsf{0},\mathsf{s},\mathsf{ss}\}.$
			
			\begin{enumerate}
				\item		 Suppose $d_{\undla}=0.$ Let ${\rm S}=(\ms, \alpha_{\ms}', \beta_{\ms}'),
				{\rm T}=(\mt, \alpha_{\mt}, \beta_{\mt})\in {\rm Tri}(\undla)$ and ${d(\mt,\mt^{\undla})} \in P(\undla)$.
				\begin{enumerate}
					\item For each $i\in [n],$ we have
					\begin{align}
						X_i \cdot f_{{\rm T},{\rm S}}
						=\mathtt{b}_{\mt,i}^{-\nu_{\beta_{\mt}}(i)} f_{{\rm T},{\rm S}}.\nonumber
					\end{align}
					\item For each $i\in [n],$ we have
					\begin{align*}\label{C acts on f}
						C_i \cdot f_{{\rm T},{\rm S}}
						&=\begin{cases}
							(-1)^{|\beta_{\mt}|_{<i}}  f_{(\mt,\alpha_{\mt},\beta_{\mt}+e_i),{\rm S}}, & \text{ if } i\in [n]\setminus \mathcal{D}_{\mt}, \\
							(-1)^{|\beta_{\mt}|+|\alpha_{\mt}|_{<i}} f_{(\mt,\alpha_{\mt}+e_{i},\beta_{\mt}),{\rm S}}, & \text{ if  $i={d(\mt,\mt^{\undla})}(i_p) \in \mathcal{D}_{\mt}$ ,where $p$ is odd},\\
							(-\sqrt{-1})(-1)^{|\beta_{\mt}|+|\alpha_{\mt}|_{\leq {d(\mt,\mt^{\undla})}(i_{p-1})}}&\\
							\qquad\qquad \cdot f_{(\mt,\alpha_{\mt}+e_{{d(\mt,\mt^{\undla})}(i_{p-1})},\beta_{\mt}),{\rm S}}, & \text{ if  $i={d(\mt,\mt^{\undla})}(i_p) \in \mathcal{D}_{\mt}$ ,where $p$ is even},
						\end{cases}
					\end{align*}

					\item For each $i\in [n-1],$ we have
					\begin{align*}
						&T_i\cdot f_{{\rm T},{\rm S}}\\
						=&-\frac{\epsilon}{\mathtt{b}_{\mt,i}^{-\nu_{\beta_{\mt}}(i)}\mathtt{b}_{\mt,i+1}^{\nu_{\beta_{\mt}}(i+1)}-1} f_{{\rm T},{\rm S}} \nonumber\\
						&\qquad+\frac{\epsilon}{\mathtt{b}_{\mt,i}^{\nu_{\beta_{\mt}}(i)}\mathtt{b}_{\mt,i+1}^{\nu_{\beta_{\mt}}(i+1)}-1}F(i,{\rm T},{\rm S})\nonumber\\
						&\qquad\qquad+\delta(s_i\mt)(-1)^{\delta_{\beta_{\mt}}(i)\delta_{\beta_{\mt}}(i+1)+\delta_{\alpha_{\mt}}(i)\delta_{\alpha_{\mt}}(i+1)}\sqrt{\mathtt{c}_{\mt}(i)}\frac{\mathtt{c}_{\mt,\ms}}{\mathtt{c}_{s_i\cdot\mt,\ms}}f_{s_i\cdot{\rm T},{\rm S}}
					\end{align*}
					
				\end{enumerate}
				
				\item	Suppose $d_{\undla}=1.$
				Let $a\in \Z_2$, ${\rm T}=(\mt, \alpha_{\mt}, \beta_{\mt})\in {\rm Tri}_{\bar{0}}(\undla),$
				${\rm S}_{a}=(\ms, {\alpha_{\ms,a}'}, \beta_{\ms}')\in {\rm Tri}_{a}(\undla)$ and ${d(\mt,\mt^{\undla})} \in P(\undla)$.
				\begin{enumerate}
					\item For each $i\in [n],$ we have
					\begin{align}
						X_i \cdot f_{{\rm T},{\rm S}_{a}}
						=\mathtt{b}_{\mt,i}^{-\nu_{\beta_{\mt}}(i)} f_{{\rm T},{\rm S}_{a}}.\nonumber
					\end{align}
					\item For each $i\in [n],$ we have
					\begin{align*}
						C_i\cdot f_{{\rm T},{\rm S}_a}	=\begin{cases} 	
					(-1)^{|\beta_{\mt}|_{<i}}f_{(\mt,\alpha_{\mt},\beta_{\mt}+e_i ),{\rm S}_{a}},&\text{if $ i\in [n]\setminus \mathcal{D}_{\mt}$,}\\
					(-1)^{|\beta_{\mt}|+|\alpha_{\mt,\bar{1}}|_{<i}} 	&\\
					\qquad\qquad\qquad\cdot	f_{(\mt,\alpha_{\mt}+e_i ,\beta_{\mt}),{\rm S}_{a}}, &\text{if $i={d(\mt,\mt^{\undla})}(i_p) \in \mathcal{D}_{\mt},\,p\neq t$ is odd,}\\
							(-1)^{|\beta_{\mt}|+|\alpha_{\mt}|} 	f_{{\rm T},{\rm S}_{a+\bar{1}}}, &\text{if $ i={d(\mt,\mt^{\undla})}(i_p) \in \mathcal{D}_{\mt},\,p=t$,}\\
							(-\sqrt{-1})	(-1)^{|\beta_{\mt}|+|\alpha_{\mt,\bar{1}}|_{\leq {d(\mt,\mt^{\undla})}(i_{p-1})}}&\\
								\qquad\cdot	 f_{(\mt,\alpha_{\mt}+e_{{d(\mt,\mt^{\undla})}(i_{p-1})},\beta_{\mt}),{\rm S}_{a}}, &\text{if $ i={d(\mt,\mt^{\undla})}(i_p) \in \mathcal{D}_{\mt},\,p$ is even.}\\
						\end{cases}
					\end{align*}
					
					\item For each $i\in [n-1],$ we have
					\begin{align*}
						&T_i\cdot f_{{\rm T},{\rm S}_a}\\
						=&-\frac{\epsilon}{\mathtt{b}_{\mt,i}^{-\nu_{\beta_{\mt}}(i)}\mathtt{b}_{\mt,i+1}^{\nu_{\beta_{\mt}}(i+1)}-1} f_{{\rm T},{\rm S}_a} \nonumber\\
						&\qquad+\frac{\epsilon}{\mathtt{b}_{\mt,i}^{\nu_{\beta_{\mt}}(i)}\mathtt{b}_{\mt,i+1}^{\nu_{\beta_{\mt}}(i+1)}-1}F(i,{\rm T},{\rm S}_a)\nonumber\\
						&\qquad\qquad+\delta(s_i\mt)\cdot \begin{cases} (-1)^{\delta_{\beta_{\mt}}(i)\delta_{\beta_{\mt}}(i+1)+\delta_{\alpha_{\mt}}(i)\delta_{\alpha_{\mt}}(i+1)}\sqrt{\mathtt{c}_{\mt}(i)}\frac{\mathtt{c}_{\mt,\ms}}{\mathtt{c}_{s_i\cdot\mt,\ms}}&\\
							\qquad\qquad \cdot f_{s_i\cdot {\rm T}, {\rm S}_{a}},&\text{if $i\neq d(\mt,\mt^{\undla})(i_t)\neq i+1,$}\\
							(-1)^{\delta_{\alpha_{\mt}}(i+1)}\sqrt{\mathtt{c}_{\mt}(i)}\frac{\mathtt{c}_{\mt,\ms}}{\mathtt{c}_{s_i\cdot\mt,\ms}}&\\
							\qquad\qquad \cdot f_{s_i\cdot {\rm T}, {\rm S}_{a}}, &\text{if $i= d(\mt,\mt^{\undla})(i_t),$}\\
							(-1)^{\delta_{\alpha_{\mt}}(i)}\sqrt{\mathtt{c}_{\mt}(i)}\frac{\mathtt{c}_{\mt,\ms}}{\mathtt{c}_{s_i\cdot\mt,\ms}}&\\
							\qquad\qquad \cdot f_{s_i\cdot {\rm T}, {\rm S}_{a}},&\text{if $i+1= d(\mt,\mt^{\undla})(i_t).$}
							\end{cases}	
						\end{align*}
				\end{enumerate}
			\end{enumerate}
		\end{prop}
		
		\begin{proof}
			Using Proposition \ref{actions of generators on L basis}, \eqref{SNB. eq1}, \eqref{SNB. eq2},\,\eqref{SNB. eq3} and \eqref{SNB. eq4}, we can compute that these equations hold as operators on $\mathbb{D}(\undla)$ and the Proposition follows from Theorem \ref{semisimple:non-dege}. We shall only give details of (2) which are tedious. The details of (1) is similar as (2) and easier, hence we shall omit them.
			
			From now on, we assume $d_{\undla}=1.$ Let $a,b\in \Z_2$, ${\rm T}=(\mt, \alpha_{\mt}, \beta_{\mt})\in {\rm Tri}_{\bar{0}}(\undla),$
			${\rm S}_{a}=(\ms, {\alpha_{\ms, a}'}, \beta_{\ms}')\in {\rm Tri}_{a}(\undla)$. By \eqref{SNB. eq4}, we have
				\begin{align}&f_{{\rm T},{\rm S}_a} \cdot C^{\beta_{\ms}^{'}} C^{\alpha_{\ms,b}^{'}} v_{\ms} \nonumber\\
				&\qquad\qquad	=(-1)^{b|\alpha_{\ms}'|_{>d(\ms,\mt^{\undla})(i_t)}+(a+b+1)|\alpha_{\mt}|_{>d(\mt,\mt^{\undla})(i_t)}}\mathtt{c}_{\mt,\ms}\nonumber\\
					&\qquad\qquad\qquad\qquad\qquad\cdot C^{\beta_{\mt}} C^{\alpha_{\mt, a+b}}v_{\mt}.
					\end{align}\label{basic action}
\begin{enumerate}
	\item Let $i\in [n].$  (2,a) follows from  Proposition \ref{actions of generators on L basis} (1), \eqref{basic action} and Theorem \ref{semisimple:non-dege}.
		\item Let $i\in [n].$
		\begin{enumerate}
			\item If $ i\in [n]\setminus \mathcal{D}_{\mt}$, using Proposition \ref{actions of generators on L basis} (3), we have
			\begin{align*}&(-1)^{b|\alpha_{\ms}'|_{>d(\ms,\mt^{\undla})(i_t)}+(a+b+1)|\alpha_{\mt}|_{>d(\mt,\mt^{\undla})(i_t)}}\mathtt{c}_{\mt,\ms}\\
				&\qquad\qquad	\qquad\qquad\qquad\cdot	C_i\cdot C^{\beta_{\mt}}C^{\alpha_{\mt,a+b}}v_{\mt}\\
				&=(-1)^{b|\alpha_{\ms}'|_{>d(\ms,\mt^{\undla})(i_t)}+(a+b+1)|\alpha_{\mt}|_{>d(\mt,\mt^{\undla})(i_t)}}\mathtt{c}_{\mt,\ms}\\
		&\qquad\qquad	\qquad\qquad\qquad\cdot	(-1)^{|\beta_{\mt}|_{<i}}  C^{\beta_{\mt}+e_i}C^{\alpha_{\mt,a+b}}v_{\mt}\\
			&=	(-1)^{|\beta_{\mt}|_{<i}}f_{(\mt,\alpha_{\mt},\beta_{\mt}+e_i ),{\rm S}_{a}}C^{\beta_{\ms}^{'}} C^{\alpha_{\ms,b}^{'}} v_{\ms}.
				\end{align*}
				\item If $i={d(\mt,\mt^{\undla})}(i_p) \in \mathcal{D}_{\mt},$ using Proposition \ref{actions of generators on L basis} (3), we have 	\begin{align*}&(-1)^{b|\alpha_{\ms}'|_{>d(\ms,\mt^{\undla})(i_t)}+(a+b+1)|\alpha_{\mt}|_{>d(\mt,\mt^{\undla})(i_t)}}\mathtt{c}_{\mt,\ms}\\
					&\qquad\qquad	\qquad\qquad\qquad\cdot	C_i\cdot C^{\beta_{\mt}}C^{\alpha_{\mt,a+b}}v_{\mt}\\
							&=\begin{cases} (-1)^{b|\alpha_{\ms}'|_{>d(\ms,\mt^{\undla})(i_t)}+(a+b+1)|\alpha_{\mt}|_{>d(\mt,\mt^{\undla})(i_t)}}\mathtt{c}_{\mt,\ms}&\\
								\qquad\qquad\qquad\cdot	(-1)^{|\beta_{\mt}|+|\alpha_{\mt, a+b}|_{<i}} &\\
								\qquad\qquad\	\qquad\qquad\qquad\cdot		C^{\beta_{\mt}}C^{\alpha_{\mt,a+b}+e_{i}}v_{\mt}, &\text{if $p\neq t$ is odd,}\\
								(-1)^{b|\alpha_{\ms}'|_{>d(\ms,\mt^{\undla})(i_t)}+(a+b+1)|\alpha_{\mt}|_{>i}}\mathtt{c}_{\mt,\ms}&\\
								\qquad\qquad\qquad\cdot	(-1)^{|\beta_{\mt}|+|\alpha_{\mt, a+b}|_{<i}} &\\
								\qquad\qquad\	\qquad\qquad\qquad\cdot		C^{\beta_{\mt}}C^{\alpha_{\mt,a+b+\bar{1}}}v_{\mt}, &\text{if $p=t$,}\\
							(-1)^{b|\alpha_{\ms}'|_{>d(\ms,\mt^{\undla})(i_t)}+(a+b+1)|\alpha_{\mt}|_{>d(\mt,\mt^{\undla})(i_t)}}\mathtt{c}_{\mt,\ms}&\\
								\qquad\qquad\qquad\cdot		(-\sqrt{-1})(-1)^{|\beta_{\mt}|+|\alpha_{\mt,a+b}|_{\leq {d(\mt,\mt^{\undla})}(i_{p-1})}}&\\
								\qquad\qquad\	\qquad\qquad\qquad\cdot	C^{\beta_{\mt}}C^{\alpha_{\mt,a+b}+e_{{d(\mt,\mt^{\undla})}(i_{p-1})}}v_{\mt}, &\text{if $p$ is even}\\
							\end{cases}\\
							&=\begin{cases} (-1)^{b|\alpha_{\ms}'|_{>d(\ms,\mt^{\undla})(i_t)}+(a+b+1)|\alpha_{\mt}+e_i|_{>d(\mt,\mt^{\undla})(i_t)}}\mathtt{c}_{\mt,\ms}&\\
								\qquad\qquad\qquad\cdot	(-1)^{(a+b+1)|e_i|_{>d(\mt,\mt^{\undla})(i_t)}+|\beta_{\mt}|+|\alpha_{\mt, a+b}|_{<i}} &\\
								\qquad\qquad\	\qquad\qquad\qquad\cdot		C^{\beta_{\mt}}C^{\alpha_{\mt,a+b}+e_{i}}v_{\mt}, &\text{if $p\neq t$ is odd,}\\
								(-1)^{b|\alpha_{\ms}'|_{>d(\ms,\mt^{\undla})(i_t)}+(a+1+b+1)|\alpha_{\mt}|_{>i}}\mathtt{c}_{\mt,\ms}&\\
								\qquad\qquad\qquad\cdot	(-1)^{|\alpha_{\mt}|_{>i}+|\beta_{\mt}|+|\alpha_{\mt, a+b}|_{<i}} &\\
								\qquad\qquad\	\qquad\qquad\qquad\cdot		C^{\beta_{\mt}}C^{\alpha_{\mt,a+b+\bar{1}}}v_{\mt}, &\text{if $p=t$,}\\
								(-1)^{b|\alpha_{\ms}'|_{>d(\ms,\mt^{\undla})(i_t)}+(a+b+1)|\alpha_{\mt}+e_{{d(\mt,\mt^{\undla})}}(i_{p-1})|_{>d(\mt,\mt^{\undla})(i_t)}}\mathtt{c}_{\mt,\ms}&\\
								\qquad\qquad\cdot		(-\sqrt{-1})(-1)^{(a+b+1)|e_{{d(\mt,\mt^{\undla})(i_{p-1})}}|_{>d(\mt,\mt^{\undla})(i_t)}}&\\
							\qquad\qquad\qquad\cdot		(-1)^{|\beta_{\mt}|+|\alpha_{\mt,a+b}|_{\leq {d(\mt,\mt^{\undla})}(i_{p-1})}}&\\
								\qquad\qquad\	\qquad\qquad\qquad\cdot	C^{\beta_{\mt}}C^{\alpha_{\mt,a+b}+e_{{d(\mt,\mt^{\undla})}(i_{p-1})}}v_{\mt}, &\text{if $p$ is even}\\
							\end{cases}\\
								&=\begin{cases} 	(-1)^{|\beta_{\mt}|+|\alpha_{\mt}+e_{d(\mt,\mt^{\undla})(i_t)}|_{<i}} 	f_{(\mt,\alpha_{\mt}+e_i ,\beta_{\mt}),{\rm S}_{a}}C^{\beta_{\ms}^{'}} C^{\alpha_{\ms,b}^{'}} v_{\ms}, &\text{if $p\neq t$ is odd,}\\
							(-1)^{|\beta_{\mt}|+|\alpha_{\mt}|} 	f_{{\rm T},{\rm S}_{a+\bar{1}}}C^{\beta_{\ms}^{'}} C^{\alpha_{\ms,b}^{'}} v_{\ms}, &\text{if $p=t$,}\\
								(-\sqrt{-1})	(-1)^{|\beta_{\mt}|+|\alpha_{\mt,\bar{1}}|_{\leq {d(\mt,\mt^{\undla})}(i_{p-1})}}&\\
								\qquad\qquad	\qquad\qquad\qquad\cdot	 f_{(\mt,\alpha_{\mt}+e_{{d(\mt,\mt^{\undla})}(i_{p-1})},\beta_{\mt}),{\rm S}_{a}}C^{\beta_{\ms}^{'}} C^{\alpha_{\ms,b}^{'}} v_{\ms}, &\text{if $p$ is even}\\
							\end{cases}
						\end{align*}
		\end{enumerate}
		These together with \eqref{basic action} and Theorem \ref{semisimple:non-dege} prove (2,b).
		\item
			Let $i\in [n-1].$ 			
	\begin{enumerate}	\item
				
			We first compute
			\begin{align*}
				(-1)^{b|\alpha_{\ms}'|_{>d(\ms,\mt^{\undla})(i_t)}+(a+b+1)|\alpha_{\mt}|_{>d(\mt, \mt^{\undla})(i_t)}}\mathtt{c}_{\mt,\ms}R(i,\beta_\mt,\alpha_{\mt,a+b},\mt).
		\end{align*}
		
			 \begin{enumerate}
				\item If $i,i+1\in [n]\setminus \mathcal{D}_{\mt},$  we have
				\begin{align*}	&(-1)^{b|\alpha_{\ms}'|_{>d(\ms,\mt^{\undla})(i_t)}+(a+b+1)|\alpha_{\mt}|_{>d(\mt,\mt^{\undla})(i_t)}}\mathtt{c}_{\mt,\ms}R(i,\beta_\mt,\alpha_{\mt,a+b},\mt)\\
				&	=(-1)^{b|\alpha_{\ms}'|_{>d(\ms,\mt^{\undla})(i_t)}+(a+b+1)|\alpha_{\mt}|_{>d(\mt,\mt^{\undla})(i_t)}}\mathtt{c}_{\mt,\ms}(-1)^{\delta_{\beta_{\mt}}(i)}C^{\beta_{\mt}+e_i+e_{i+1}}C^{\alpha_{\mt,a+b}}v_{\mt};			\\
				&=(-1)^{\delta_{\beta_{\mt}}(i)}f_{(\mt,\alpha_{\mt},\beta_{\mt}+e_i+e_{i+1}),S_a}C^{\beta_\ms'}C^{\alpha_{\ms,b}'}v_{\ms}.
				\end{align*}
				\item
				if $i={d(\mt,\mt^{\undla})}(i_p)\in \mathcal{D}_{\mt}, i+1\in [n]\setminus \mathcal{D}_{\mt},$ we have
				\begin{align*}&(-1)^{b|\alpha_{\ms}'|_{>d(\ms,\mt^{\undla})(i_t)}+(a+b+1)|\alpha_{\mt}|_{>d(\mt,\mt^{\undla})(i_t)}}\mathtt{c}_{\mt,\ms}R(i,\beta_\mt,\alpha_{\mt,a+b},\mt)\\
					&=\begin{cases}(-1)^{b|\alpha_{\ms}'|_{>d(\ms,\mt^{\undla})(i_t)}+(a+b+1)|\alpha_{\mt}|_{>d(\mt,\mt^{\undla})(i_t)}}\mathtt{c}_{\mt,\ms}&\\
					\qquad\qquad\qquad\qquad	(-1)^{|\beta_{\mt}|_{>i}+|\alpha_{\mt,a+b}|_{<i}}C^{\beta_{\mt}+e_{i+1}}C^{\alpha_{\mt,a+b}+e_{i}}v_{\mt}, & \text{ if  $p\neq t$ is odd},\\
					(-1)^{b|\alpha_{\ms}'|_{>d(\ms,\mt^{\undla})(i_t)}+(a+b+1)|\alpha_{\mt}|_{>i}}\mathtt{c}_{\mt,\ms}&\\
					\qquad\qquad\qquad\qquad	(-1)^{|\beta_{\mt}|_{>i}+|\alpha_{\mt}|_{<i}}C^{\beta_{\mt}+e_{i+1}}C^{\alpha_{\mt,a+b+1}}v_{\mt}, & \text{ if  $p= t$},\\
				(-1)^{b|\alpha_{\ms}'|_{>d(\ms,\mt^{\undla})(i_t)}+(a+b+1)|\alpha_{\mt}|_{>d(\mt,\mt^{\undla})(i_t)}}\mathtt{c}_{\mt,\ms}&\\	\qquad\qquad\cdot(-\sqrt{-1})(-1)^{|\beta_{\mt}|_{>i}+|\alpha_{\mt,a+b}|_{\leq {d(\mt,\mt^{\undla})}(i_{p-1})}}&
			\\	\qquad\qquad\qquad\qquad \cdot C^{\beta_{\mt}+e_{i+1}}C^{\alpha_{\mt,a+b}+e_{{d(\mt,\mt^{\undla})}(i_{p-1})}}v_{\mt}, &\text{ if  $p$ is even},
				\end{cases}\\
				&=\begin{cases}(-1)^{b|\alpha_{\ms}'|_{>d(\ms,\mt^{\undla})(i_t)}+(a+b+1)|\alpha_{\mt}+e_i|_{>d(\mt,\mt^{\undla})(i_t)}+|e_i|_{>d(\mt,\mt^{\undla})(i_t)}}\mathtt{c}_{\mt,\ms}&\\
					\qquad\qquad\qquad\qquad	(-1)^{|\beta_{\mt}|_{>i}+|\alpha_{\mt}|_{<i}}C^{\beta_{\mt}+e_{i+1}}C^{\alpha_{\mt,a+b}+e_{i}}v_{\mt}, & \text{ if  $p\neq t$ is odd},\\
						(-1)^{b|\alpha_{\ms}'|_{>d(\ms,\mt^{\undla})(i_t)}+(a+1+b+1)|\alpha_{\mt}|_{>i}+|\alpha_{\mt}|_{>i}}\mathtt{c}_{\mt,\ms}&\\
					\qquad\qquad\qquad\qquad	(-1)^{|\beta_{\mt}|_{>i}+|\alpha_{\mt}|_{<i}}C^{\beta_{\mt}+e_{i+1}}C^{\alpha_{\mt,a+b+1}}v_{\mt}, & \text{ if  $p= t$},\\
					(-1)^{b|\alpha_{\ms}'|_{>d(\ms,\mt^{\undla})(i_t)}+(a+b+1)|\alpha_{\mt}+e_{{d(\mt,\mt^{\undla})}(i_{p-1})}|_{>d(\mt,\mt^{\undla})(i_t)}}\mathtt{c}_{\mt,\ms}&\\
					\qquad\qquad\cdot(-\sqrt{-1})(-1)^{|e_{{d(\mt,\mt^{\undla})}(i_{p-1})}|_{>d(\mt,\mt^{\undla})(i_t)}+|\beta_{\mt}|_{>i}+|\alpha_{\mt}|_{\leq {d(\mt,\mt^{\undla})}(i_{p-1})}}&\\
						\qquad\qquad	\qquad\qquad \cdot C^{\beta_{\mt}+e_{i+1}}C^{\alpha_{\mt,a+b}+e_{{d(\mt,\mt^{\undla})}(i_{p-1})}}v_{\mt}, &\text{ if  $p$ is even},
				\end{cases}\\
					&=\begin{cases}(-1)^{|\beta_{\mt}|_{>i}+|\alpha_{\mt}+e_{d(\mt,\mt^{\undla})(i_t)}|_{<i}}	f_{(\mt,\alpha_{\mt}+e_i,\beta_{\mt}+e_{i+1}),{\rm S}_a}C^{\beta_{\ms}'}C^{\alpha_{\mt,b}}v_{\ms}, & \text{ if  $p\neq t$ is odd},\\
						(-1)^{|\beta_{\mt}|_{>i}+|\alpha_{\mt}|}	f_{(\mt,\alpha_{\mt},\beta_{\mt}+e_{i+1}),{\rm S}_{a+\bar{1}}}C^{\beta_{\ms}'}C^{\alpha_{\mt,b}}v_{\ms}, & \text{ if  $p=t$},\\
			(-\sqrt{-1})(-1)^{|\beta_{\mt}|_{>i}+|\alpha_{\mt}+e_{d(\mt,\mt^{\undla})(i_t)}|_{\leq {d(\mt,\mt^{\undla})}(i_{p-1})}}&\\
				\qquad\qquad	\qquad\qquad \cdot f_{(\mt,\alpha_{\mt}+e_{{d(\mt,\mt^{\undla})}(i_{p-1})},\beta_{\mt}+e_{i+1}),{\rm S}_a}C^{\beta_{\ms}'}C^{\alpha_{\mt,b}}v_{\ms}, &\text{ if  $p$ is even}.
				\end{cases}
			\end{align*}
				\item
				If $i+1={d(\mt,\mt^{\undla})}(i_p)\in \mathcal{D}_{\mt}, i\in [n]\setminus \mathcal{D}_{\mt},$ we have \begin{align*}&(-1)^{b|\alpha_{\ms}'|_{>d(\ms,\mt^{\undla})(i_t)}+(a+b+1)|\alpha_{\mt}|_{>d(\mt,\mt^{\undla})(i_t)}}\mathtt{c}_{\mt,\ms}R(i,\beta_\mt,\alpha_{\mt,a+b},\mt)\\
					&=\begin{cases}(-1)^{b|\alpha_{\ms}'|_{>d(\ms,\mt^{\undla})(i_t)}+(a+b+1)|\alpha_{\mt}|_{>d(\mt,\mt^{\undla})(i_t)}}\mathtt{c}_{\mt,\ms}&\\
						\qquad\qquad\qquad\qquad\cdot(-1)^{|\beta_{\mt}|_{\geq i}+|\alpha_{\mt,a+b}|_{<i+1}} C^{\beta_{\mt}+e_i}C^{\alpha_{\mt,a+b}+e_{i+1}}v_{\mt}, & \text{ if  $p\neq t$ is odd},\\
						(-1)^{b|\alpha_{\ms}'|_{>d(\ms,\mt^{\undla})(i_t)}+(a+b+1)|\alpha_{\mt}|_{>i+1}}\mathtt{c}_{\mt,\ms}&\\
						\qquad\qquad\qquad\qquad\cdot(-1)^{|\beta_{\mt}|_{\geq i}+|\alpha_{\mt}|_{<i+1}} C^{\beta_{\mt}+e_i}C^{\alpha_{\mt,a+b+\bar{1}}}v_{\mt}, & \text{ if  $p=t$ },\\
					(-1)^{b|\alpha_{\ms}'|_{>d(\ms,\mt^{\undla})(i_t)}+(a+b+1)|\alpha_{\mt}|_{>d(\mt,\mt^{\undla})(i_t)}}\mathtt{c}_{\mt,\ms}&\\
				\qquad\qquad\cdot	(-\sqrt{-1})(-1)^{|\beta_{\mt}|_{\geq i}+|\alpha_{\mt,a+b}|_{\leq {d(\mt,\mt^{\undla})}(i_{p-1})}}&\\
				\qquad\qquad\qquad\qquad\cdot C^{\beta_{\mt}+e_i}C^{\alpha_{\mt,a+b}+e_{{d(\mt,\mt^{\undla})}(i_{p-1})}}v_{\mt}, &\text{ if $p$ is even},\end{cases}  \nonumber\\
					&=\begin{cases}(-1)^{b|\alpha_{\ms}'|_{>d(\ms,\mt^{\undla})(i_t)}+(a+b+1)|\alpha_{\mt}+e_{i+1}|_{>d(\mt,\mt^{\undla})(i_t)}+|e_{i+1}|_{>d(\mt,\mt^{\undla})(i_t)}}\mathtt{c}_{\mt,\ms}&\\
					\qquad\qquad\qquad\qquad\cdot (-1)^{|\beta_{\mt}|_{\geq i}+|\alpha_{\mt}|_{<i+1}} C^{\beta_{\mt}+e_i}C^{\alpha_{\mt,a+b}+e_{i+1}}v_{\mt}, & \text{ if  $p\neq t$ is odd},\\
					(-1)^{b|\alpha_{\ms}'|_{>d(\ms,\mt^{\undla})(i_t)}+(a+1+b+1)|\alpha_{\mt}|_{>i+1}+|\alpha_{\mt}|_{>i+1}}\mathtt{c}_{\mt,\ms}&\\
					\qquad\qquad\qquad\qquad\cdot (-1)^{|\beta_{\mt}|_{\geq i}+|\alpha_{\mt}|_{<i+1}} C^{\beta_{\mt}+e_i}C^{\alpha_{\mt,a+b+\bar{1}}}v_{\mt}, & \text{ if  $p=t$ },\\
					(-1)^{b|\alpha_{\ms}'|_{>d(\ms,\mt^{\undla})(i_t)}+(a+b+1)|\alpha_{\mt}+e_{{d(\mt,\mt^{\undla})}(i_{p-1})}|_{>d(\mt,\mt^{\undla})(i_t)}}\mathtt{c}_{\mt,\ms}&\\
					\qquad\qquad	\cdot (-\sqrt{-1})(-1)^{|e_{{d(\mt,\mt^{\undla})}(i_{p-1})}|_{>d(\mt,\mt^{\undla})(i_t)}+|\beta_{\mt}|_{\geq i}+|\alpha_{\mt}|_{\leq {d(\mt,\mt^{\undla})}(i_{p-1})}}&\\
				\qquad\qquad\qquad\qquad\cdot	C^{\beta_{\mt}+e_i}C^{\alpha_{\mt,a+b}+e_{{d(\mt,\mt^{\undla})}(i_{p-1})}}v_{\mt}, &\text{ if $p$ is even},\end{cases}  \nonumber\\
				&=\begin{cases}(-1)^{|\beta_{\mt}|_{\geq i}+|\alpha_{\mt,\bar{1}}|_{<i+1}}	f_{(\mt,\alpha_{\mt}+e_{i+1},\beta_{\mt}+e_{i}),{\rm S}_a}C^{\beta_{\ms}'}C^{\alpha_{\mt,b}}v_{\ms}, & \text{ if  $p\neq t$ is odd},\\
			(-1)^{|\beta_{\mt}|_{\geq i}+|\alpha_{\mt}|}	f_{(\mt,\alpha_{\mt},\beta_{\mt}+e_{i}),{\rm S}_{a+\bar{1}}}C^{\beta_{\ms}'}C^{\alpha_{\mt,b}}v_{\ms} & \text{ if  $p=t$ },\\
			(-\sqrt{-1})(-1)^{|\beta_{\mt}|_{\geq i}+|\alpha_{\mt,\bar{1}}|_{\leq {d(\mt,\mt^{\undla})}(i_{p-1})}}&\\
		\qquad\qquad\qquad\qquad\cdot	f_{(\mt,\alpha_{\mt}+e_{{d(\mt,\mt^{\undla})}(i_{p-1})},\beta_{\mt}+e_{i}),{\rm S}_a}C^{\beta_{\ms}'}C^{\alpha_{\mt,b}}v_{\ms}, &\text{ if $p$ is even}.\end{cases}  \nonumber
			\end{align*}
				\item
				If $i={d(\mt,\mt^{\undla})}(i_p), i+1={d(\mt,\mt^{\undla})}(i_{p'})\in \mathcal{D}_{\mt}.$
				\begin{enumerate}
					\item If  $p$ and $p'$ are odd, we have \begin{align*}&(-1)^{b|\alpha_{\ms}'|_{>d(\ms,\mt^{\undla})(i_t)}+(a+b+1)|\alpha_{\mt}|_{>d(\mt,\mt^{\undla})(i_t)}}\mathtt{c}_{\mt,\ms}R(i,\beta_\mt,\alpha_{\mt,a+b},\mt)\\
						&=\begin{cases}(-1)^{b|\alpha_{\ms}'|_{>d(\ms,\mt^{\undla})(i_t)}+(a+b+1)|\alpha_{\mt}|_{>d(\mt,\mt^{\undla})(i_t)}}\mathtt{c}_{\mt,\ms}&\\
							\qquad\qquad\qquad\qquad	\cdot(-1)^{\delta_{\alpha_{\mt,a+b}}(i)} C^{\beta_{\mt}}C^{\alpha_{\mt,a+b}+e_{i}+e_{i+1}}v_{\mt}, & \text{ if  $p\neq t\neq p'$ are odd},\\
							(-1)^{b|\alpha_{\ms}'|_{>d(\ms,\mt^{\undla})(i_t)}+(a+b+1)|\alpha_{\mt}|_{>d(\mt,\mt^{\undla})(i_t)}}\mathtt{c}_{\mt,\ms}&\\
							\qquad\qquad\qquad\qquad	\cdot(-1)^{\delta_{\alpha_{\mt,a+b}}(i)} C^{\beta_{\mt}}C^{\alpha_{\mt,a+b+\bar{1}}+e_{i+1}}v_{\mt}, & \text{ if  $p=t$ and $ p'$ is odd},\\
							(-1)^{b|\alpha_{\ms}'|_{>d(\ms,\mt^{\undla})(i_t)}+(a+b+1)|\alpha_{\mt}|_{>d(\mt,\mt^{\undla})(i_t)}}\mathtt{c}_{\mt,\ms}&\\
							\qquad\qquad\qquad\qquad\cdot	(-1)^{\delta_{\alpha_{\mt,a+b}}(i)} C^{\beta_{\mt}}C^{\alpha_{\mt,a+b+\bar{1}}+e_{i}}v_{\mt}, & \text{ if  $p$ is odd and $ t=p'$, }
				\end{cases}\\
				&=\begin{cases}(-1)^{b|\alpha_{\ms}'|_{>d(\ms,\mt^{\undla})(i_t)}+(a+b+1)|\alpha_{\mt}+e_i+e_{i+1}|_{>d(\mt,\mt^{\undla})(i_t)}}\mathtt{c}_{\mt,\ms}&\\
				\qquad\qquad \cdot(-1)^{(a+b+1)|e_i+e_{i+1}|_{>d(\mt,\mt^{\undla})(i_t)}}&\\
				\qquad\qquad\qquad\qquad\cdot	(-1)^{\delta_{\alpha_{\mt}}(i)} C^{\beta_{\mt}}C^{\alpha_{\mt,a+b}+e_i+e_{i+1}}v_{\mt}, & \text{ if  $p\neq t\neq p'$ are odd},\\
				(-1)^{b|\alpha_{\ms}'|_{>d(\ms,\mt^{\undla})(i_t)}+(a+1+b+1)|\alpha_{\mt}+e_{i+1}|_{>i}}\mathtt{c}_{\mt,\ms}&\\
				\qquad\qquad\cdot(-1)^{(a+b)|e_{i+1}|_{>i}+|\alpha_\mt|_{>i}}&\\
				\qquad\qquad\qquad\qquad\cdot	(-1)^{\delta_{\alpha_{\mt,a+b}}(i)} C^{\beta_{\mt}}C^{\alpha_{\mt,a+b+\bar{1}}+e_{i+1}}v_{\mt}, & \text{ if  $p=t$ and $ p'$ is odd},\\
				(-1)^{b|\alpha_{\ms}'|_{>d(\ms,\mt^{\undla})(i_t)}+(a+1+b+1)|\alpha_{\mt}+e_i|_{>i+1}}\mathtt{c}_{\mt,\ms}&\\
				\qquad\qquad \cdot (-1)^{(a+b)|e_{i}|_{>i+1}+|\alpha_\mt|_{>i+1}}&\\
				\qquad\qquad\qquad\qquad\cdot	(-1)^{\delta_{\alpha_{\mt,a+b}}(i)} C^{\beta_{\mt}}C^{\alpha_{\mt,a+b+\bar{1}}+e_{i}}v_{\mt}, & \text{ if  $p$ is odd and $ t=p'$. }
	\end{cases}\\
&=\begin{cases}	(-1)^{\delta_{\alpha_{\mt}}(i)}f_{(\mt,\alpha_{\mt}+e_i+e_{i+1},\beta_{\mt}),{\rm S}_a}C^{\beta_{\ms}'}C^{\alpha_{\mt,b}}v_{\ms}, & \text{ if  $p\neq t\neq p'$ are odd},\\
(-1)^{|\alpha_\mt|_{\geq i}}f_{(\mt,\alpha_{\mt}+e_{i+1},\beta_{\mt}),{\rm S}_{a+\bar{1}}}C^{\beta_{\ms}'}C^{\alpha_{\mt,b}}v_{\ms}, & \text{ if  $p=t$ and $ p'$ is odd},\\
(-1)^{|\alpha_\mt|_{\geq i}}f_{(\mt,\alpha_{\mt}+e_{i},\beta_{\mt}),{\rm S}_{a+\bar{1}}}C^{\beta_{\ms}'}C^{\alpha_{\mt,b}}v_{\ms}, & \text{ if  $p$ is odd and $ t=p'$. }
\end{cases}
\end{align*}					
					\item If just one of $p$ and $p'$ is odd, we have \begin{align*}&(-1)^{b|\alpha_{\ms}'|_{>d(\ms,\mt^{\undla})(i_t)}+(a+b+1)|\alpha_{\mt}|_{>d(\mt,\mt^{\undla})(i_t)}}\mathtt{c}_{\mt,\ms}R(i,\beta_\mt,\alpha_{\mt,a+b},\mt)\\
						&=\begin{cases}
							(-1)^{b|\alpha_{\ms}'|_{>d(\ms,\mt^{\undla})(i_t)}+(a+b+1)|\alpha_{\mt}|_{>d(\mt,\mt^{\undla})(i_t)}}\mathtt{c}_{\mt,\ms}&\\
							\qquad\qquad	\cdot(-\sqrt{-1})(-1)^{|\alpha_{\mt,a+b}|_{<i+1}+|\alpha_{\mt,a+b}+e_{i+1}|_{\leq{d(\mt,\mt^{\undla})}(i_{p-1})}} &\\
							\qquad\qquad\qquad \cdot C^{\beta_{\mt}}C^{\alpha_{\mt,a+b}+e_{{d(\mt,\mt^{\undla})}(i_{p-1})}+e_{i+1}}v_{\mt},& \text{ if  $p$ is even, $p'\neq t$ is odd},\\
							(-1)^{b|\alpha_{\ms}'|_{>d(\ms,\mt^{\undla})(i_t)}+(a+b+1)|\alpha_{\mt}|_{>d(\mt,\mt^{\undla})(i_t)}}\mathtt{c}_{\mt,\ms}&\\
							\qquad\qquad	\cdot (-\sqrt{-1})(-1)^{|\alpha_{\mt,a+b}|_{<i+1}+|\alpha_{\mt,a+b+\bar{1}}|_{\leq{d(\mt,\mt^{\undla})}(i_{p-1})}} &\\
							\qquad\qquad\qquad \cdot C^{\beta_{\mt}}C^{\alpha_{\mt,a+b+\bar{1}}+e_{{d(\mt,\mt^{\undla})}(i_{p-1})}}v_{\mt},& \text{ if  $p$ is even and $p'=t$ },\\
							(-1)^{b|\alpha_{\ms}'|_{>d(\ms,\mt^{\undla})(i_t)}+(a+b+1)|\alpha_{\mt}|_{>d(\mt,\mt^{\undla})(i_t)}}\mathtt{c}_{\mt,\ms}&\\
							\qquad\qquad\cdot	(-\sqrt{-1})(-1)^{|\alpha_{\mt,a+b}|_{\leq {d(\mt,\mt^{\undla})}(i_{p'-1})}+|\alpha_{\mt,a+b}+e_{{d(\mt,\mt^{\undla})}(i_{p'-1})}|_{<i}}&\\
							\qquad\qquad\qquad \cdot C^{\beta_{\mt}}C^{\alpha_{\mt,a+b}+e_{i}+e_{{d(\mt,\mt^{\undla})}(i_{p'-1})}}v_{\mt}, &\text{ if  $p\neq t$ is odd and $p'$ is even},\\
							(-1)^{b|\alpha_{\ms}'|_{>d(\ms,\mt^{\undla})(i_t)}+(a+b+1)|\alpha_{\mt}|_{>d(\mt,\mt^{\undla})(i_t)}}\mathtt{c}_{\mt,\ms}&\\
							\qquad\qquad\cdot	(-\sqrt{-1})(-1)^{|\alpha_{\mt,a+b}|_{\leq {d(\mt,\mt^{\undla})}(i_{p'-1})}+|\alpha_{\mt,a+b}+e_{{d(\mt,\mt^{\undla})}(i_{p'-1})}|_{<i}}&\\
							\qquad\qquad\qquad \cdot C^{\beta_{\mt}}C^{\alpha_{\mt,a+b+\bar{1}}+e_{{d(\mt,\mt^{\undla})}(i_{p'-1})}}v_{\mt}, &\text{ if  $p=t$ and $p'$ is even},\\
						\end{cases}\\
						&=\begin{cases}	(-1)^{b|\alpha_{\ms}'|_{>d(\ms,\mt^{\undla})(i_t)}+(a+b+1)|\alpha_{\mt}+e_{{d(\mt,\mt^{\undla})}(i_{p-1})}+e_{i+1}|_{>d(\mt,\mt^{\undla})(i_t)}}\mathtt{c}_{\mt,\ms}&\\
							\qquad \cdot(-1)^{(a+b+1)|e_{{d(\mt,\mt^{\undla})}(i_{p-1})}+e_{i+1}|_{>d(\mt,\mt^{\undla})(i_t)}}&\\
							\qquad\qquad	\cdot(-\sqrt{-1})(-1)^{|\alpha_{\mt,a+b}|_{<i+1}+|\alpha_{\mt,a+b}+e_{i+1}|_{\leq{d(\mt,\mt^{\undla})}(i_{p-1})}} &\\
							\qquad\qquad\qquad \cdot C^{\beta_{\mt}}C^{\alpha_{\mt,a+b}+e_{{d(\mt,\mt^{\undla})}(i_{p-1})}+e_{i+1}}v_{\mt},& \text{ if  $p$ is even, $p'\neq t$ is odd},\\
							(-1)^{b|\alpha_{\ms}'|_{>d(\ms,\mt^{\undla})(i_t)}+(a+1+b+1)|\alpha_{\mt}+e_{{d(\mt,\mt^{\undla})}(i_{p-1})}|_{>i+1}}\mathtt{c}_{\mt,\ms}&\\
							\qquad\cdot(-1)^{|\alpha_{\mt}|_{>i+1}+(a+b)|e_{{d(\mt,\mt^{\undla})}(i_{p-1})}|_{>i+1}}&\\
							\qquad\qquad	\cdot (-\sqrt{-1})(-1)^{|\alpha_{\mt,a+b}|_{<i+1}+|\alpha_{\mt,a+b+\bar{1}}|_{\leq{d(\mt,\mt^{\undla})}(i_{p-1})}} &\\
							\qquad\qquad\qquad \cdot C^{\beta_{\mt}}C^{\alpha_{\mt,a+b+\bar{1}}+e_{{d(\mt,\mt^{\undla})}(i_{p-1})}}v_{\mt},& \text{ if  $p$ is even and $p'=t$ },\\
							(-1)^{b|\alpha_{\ms}'|_{>d(\ms,\mt^{\undla})(i_t)}+(a+b+1)|\alpha_{\mt}+e_{i}+e_{{d(\mt,\mt^{\undla})}(i_{p'-1})}|_{>d(\mt,\mt^{\undla})(i_t)}}\mathtt{c}_{\mt,\ms}&\\
							\qquad \cdot(-1)^{(a+b+1)|e_{i}+e_{{d(\mt,\mt^{\undla})}(i_{p'-1})}|_{>d(\mt,\mt^{\undla})(i_t)}}&\\
							\qquad\qquad\cdot	(-\sqrt{-1})(-1)^{|\alpha_{\mt,a+b}|_{\leq {d(\mt,\mt^{\undla})}(i_{p'-1})}+|\alpha_{\mt,a+b}+e_{{d(\mt,\mt^{\undla})}(i_{p'-1})}|_{<i}}&\\
							\qquad\qquad\qquad \cdot C^{\beta_{\mt}}C^{\alpha_{\mt,a+b}+e_{i}+e_{{d(\mt,\mt^{\undla})}(i_{p'-1})}}v_{\mt}, &\text{ if  $p\neq t$ is odd and $p'$ is even},\\
							(-1)^{b|\alpha_{\ms}'|_{>d(\ms,\mt^{\undla})(i_t)}+(a+1+b+1)|\alpha_{\mt}+e_{{d(\mt,\mt^{\undla})}(i_{p'-1})}|_{>i}}\mathtt{c}_{\mt,\ms}&\\
							\qquad\cdot(-1)^{|\alpha_{\mt}|_{>i}+(a+b)|e_{{d(\mt,\mt^{\undla})}(i_{p'-1})}|_{>i}}	&\\
							\qquad\qquad\cdot	(-\sqrt{-1})(-1)^{|\alpha_{\mt,a+b}|_{\leq {d(\mt,\mt^{\undla})}(i_{p'-1})}+|\alpha_{\mt,a+b}+e_{{d(\mt,\mt^{\undla})}(i_{p'-1})}|_{<i}}&\\
							\qquad\qquad\qquad \cdot C^{\beta_{\mt}}C^{\alpha_{\mt,a+b+\bar{1}}+e_{{d(\mt,\mt^{\undla})}(i_{p'-1})}}v_{\mt}, &\text{ if  $p=t$ and $p'$ is even},\\
						\end{cases}\\
						&=\begin{cases}	
							(-\sqrt{-1})(-1)^{|\alpha_{\mt,\bar{1}}|_{<i+1}+|\alpha_{\mt,\bar{1}}+e_{i+1}|_{\leq{d(\mt,\mt^{\undla})}(i_{p-1})}}&\\
							\qquad\qquad\cdot	 f_{(\mt,\alpha_{\mt}+e_{{d(\mt,\mt^{\undla})}(i_{p-1})}+e_{i+1}, \beta_{\mt}),{\rm S}_{a}} C^{\beta_{\ms}'}C^{\alpha_{\mt,b}}v_{\ms},& \text{ if  $p$ is even, $p'\neq t$ is odd},\\
							(-\sqrt{-1})(-1)^{|\alpha_{\mt}|+|\alpha_{\mt,\bar{1}}|_{\leq{d(\mt,\mt^{\undla})}(i_{p-1})}}  &\\
							\qquad\qquad\cdot f_{(\mt,\alpha_{\mt}+e_{{d(\mt,\mt^{\undla})}(i_{p-1})}, \beta_{\mt}),{\rm S}_{a+\bar{1}}} C^{\beta_{\ms}'}C^{\alpha_{\mt,b}}v_{\ms},& \text{ if  $p$ is even and $p'=t$ },\\
							(-\sqrt{-1})(-1)^{|\alpha_{\mt,\bar{1}}|_{\leq {d(\mt,\mt^{\undla})}(i_{p'-1})}+|\alpha_{\mt,\bar{1}}+e_{{d(\mt,\mt^{\undla})}(i_{p'-1})}|_{<i}}&\\
							\qquad\qquad\cdot  f_{(\mt,\alpha_{\mt}+e_{{d(\mt,\mt^{\undla})}(i_{p'-1})}+e_{i}, \beta_{\mt}),{\rm S}_{a}} C^{\beta_{\ms}'}C^{\alpha_{\mt,b}}v_{\ms}, &\text{ if  $p\neq t$ is odd and $p'$ is even},\\
							(-\sqrt{-1})(-1)^{|\alpha_{\mt,\bar{1}}|_{> {d(\mt,\mt^{\undla})}(i_{p'-1})}}f_{(\mt,\alpha_{\mt}+e_{{d(\mt,\mt^{\undla})}(i_{p'-1})}, \beta_{\mt}),{\rm S}_{a+\bar{1}}} C^{\beta_{\ms}'}C^{\alpha_{\mt,b}}v_{\ms}, &\text{ if  $p=t$ and $p'$ is even}.
						\end{cases}
					\end{align*}
					
					\item If  $p$ and $p'$ are even, we have
						 \begin{align*}&(-1)^{b|\alpha_{\ms}'|_{>d(\ms,\mt^{\undla})(i_t)}+(a+b+1)|\alpha_{\mt}|_{>d(\mt,\mt^{\undla})(i_t)}}\mathtt{c}_{\mt,\ms}R(i,\beta_\mt,\alpha_{\mt,a+b},\mt)\\
						 		&=	(-1)^{b|\alpha_{\ms}'|_{>d(\ms,\mt^{\undla})(i_t)}+(a+b+1)|\alpha_{\mt}|_{>d(\mt,\mt^{\undla})(i_t)}}\mathtt{c}_{\mt,\ms}\\
						 &	\qquad\qquad\cdot(-1)^{1+|\alpha_{\mt,a+b}|_{\leq {d(\mt,\mt^{\undla})}(i_{p'-1})}+|\alpha_{\mt,a+b}+e_{{d(\mt,\mt^{\undla})}(i_{p'-1})}|_{\leq{d(\mt,\mt^{\undla})}(i_{p-1})}}\\
						 &	\qquad\qquad\qquad\cdot C^{\beta_{\mt}}C^{\alpha_{\mt,a+b}+e_{{d(\mt,\mt^{\undla})}(i_{p-1})}+e_{{d(\mt,\mt^{\undla})}(i_{p'-1})}}v_{\mt}\\
						 &=(-1)^{b|\alpha_{\ms}'|_{>d(\ms,\mt^{\undla})(i_t)}+(a+b+1)|\alpha_{\mt}+e_{{d(\mt,\mt^{\undla})}(i_{p-1})}+e_{{d(\mt,\mt^{\undla})}(i_{p'-1})}|_{>d(\mt,\mt^{\undla})(i_t)}}\mathtt{c}_{\mt,\ms}\\
						 &	\qquad \cdot(-1)^{	(a+b+1)|e_{{d(\mt,\mt^{\undla})}(i_{p-1})}+e_{{d(\mt,\mt^{\undla})}(i_{p'-1})}|_{>d(\mt,\mt^{\undla})(i_t)}}\\
						 	&\qquad\qquad\cdot (-1)^{1+|\alpha_{\mt,a+b}|_{\leq {d(\mt,\mt^{\undla})}(i_{p'-1})}+|\alpha_{\mt,a+b}+e_{{d(\mt,\mt^{\undla})}(i_{p'-1})}|_{\leq{d(\mt,\mt^{\undla})}(i_{p-1})}}\\
						 &	\qquad\qquad\qquad\cdot C^{\beta_{\mt}}C^{\alpha_{\mt,a+b}+e_{{d(\mt,\mt^{\undla})}(i_{p-1})}+e_{{d(\mt,\mt^{\undla})}(i_{p'-1})}}v_{\mt}\\
						 &	=(-1)^{1+|\alpha_{\mt,\bar{1}}|_{\leq {d(\mt,\mt^{\undla})}(i_{p'-1})}+|\alpha_{\mt,\bar{1}}+e_{{d(\mt,\mt^{\undla})}(i_{p'-1})}|_{\leq{d(\mt,\mt^{\undla})}(i_{p-1})}}\\
						& \qquad\qquad\cdot f_{(\mt,\alpha_{\mt}+e_{{d(\mt,\mt^{\undla})}(i_{p-1})}+e_{{d(\mt,\mt^{\undla})}(i_{p'-1})}, \beta_{\mt}),{\rm S}_{a}} C^{\beta_{\ms}'}C^{\alpha_{\mt,b}}v_{\ms}.
						 	\end{align*}

									\end{enumerate}
			\end{enumerate}
			
		\item	Now suppose $s_i\mt\in\Std(\undla)$, we have
		\begin{align*}&(-1)^{b|\alpha_{\ms}'|_{>d(\ms,\mt^{\undla})(i_t)}+(a+b+1)|\alpha_{\mt}|_{>d(\mt,\mt^{\undla})(i_t)}}\mathtt{c}_{\mt,\ms}\\
				&\qquad\qquad \cdot(-1)^{\delta_{\beta_{\mt}}(i)\delta_{\beta_{\mt}}(i+1)+\delta_{\alpha_{\mt,a+b}}(i)\delta_{\alpha_{\mt,a+b}}(i+1)}\sqrt{\mathtt{c}_{\mt}(i)}\\
				&\qquad\qquad\qquad\qquad \cdot C^{s_i \cdot \beta_{\mt}}C^{s_i \cdot \alpha_{\mt,a+b}}v_{s_i\mt}\\
				&=\begin{cases} (-1)^{b|\alpha_{\ms}'|_{>d(\ms,\mt^{\undla})(i_t)}+(a+b+1)|s_i\cdot\alpha_{\mt}|_{>d(s_i\cdot\mt,\mt^{\undla})(i_t)}}\mathtt{c}_{\mt,\ms}&\\
					\qquad\qquad \cdot(-1)^{\delta_{\beta_{\mt}}(i)\delta_{\beta_{\mt}}(i+1)+\delta_{\alpha_{\mt}}(i)\delta_{\alpha_{\mt}}(i+1)}\sqrt{\mathtt{c}_{\mt}(i)}&\\
					\qquad\qquad\qquad\qquad \cdot C^{s_i \cdot \beta_{\mt}}C^{s_i \cdot \alpha_{\mt,a+b}}v_{s_i\mt}&\text{if $i\neq d(\mt,\mt^{\undla})(i_t)\neq i+1,$}\\
					(-1)^{b|\alpha_{\ms}'|_{>d(\ms,\mt^{\undla})(i_t)}+(a+b+1)|\alpha_{\mt}|_{>i}}\mathtt{c}_{\mt,\ms}&\\
					\qquad\qquad \cdot(-1)^{\delta_{\alpha_{\mt,a+b}}(i)\delta_{\alpha_{\mt,a+b}}(i+1)}\sqrt{\mathtt{c}_{\mt}(i)}&\\
					\qquad\qquad\qquad\qquad \cdot C^{s_i \cdot \beta_{\mt}}C^{s_i \cdot \alpha_{\mt,a+b}}v_{s_i\mt}, &\text{if $i= d(\mt,\mt^{\undla})(i_t),$}\\
					(-1)^{b|\alpha_{\ms}'|_{>d(\ms,\mt^{\undla})(i_t)}+(a+b+1)|\alpha_{\mt}|_{>i+1}}\mathtt{c}_{\mt,\ms}&\\
					\qquad\qquad \cdot(-1)^{\delta_{\alpha_{\mt,a+b}}(i)\delta_{\alpha_{\mt,a+b}}(i+1)}\sqrt{\mathtt{c}_{\mt}(i)}&\\
					\qquad\qquad\qquad\qquad \cdot C^{s_i \cdot \beta_{\mt}}C^{s_i \cdot \alpha_{\mt,a+b}}v_{s_i\mt}, &\text{if $i+1= d(\mt,\mt^{\undla})(i_t),$}
					\end{cases}\\
					&=\begin{cases} (-1)^{b|\alpha_{\ms}'|_{>d(\ms,\mt^{\undla})(i_t)}+(a+b+1)|s_i\cdot\alpha_{\mt}|_{>d(s_i\cdot\mt,\mt^{\undla})(i_t)}}\mathtt{c}_{\mt,\ms}&\\
						\qquad\qquad \cdot(-1)^{\delta_{\beta_{\mt}}(i)\delta_{\beta_{\mt}}(i+1)+\delta_{\alpha_{\mt}}(i)\delta_{\alpha_{\mt}}(i+1)}\sqrt{\mathtt{c}_{\mt}(i)}&\\
						\qquad\qquad\qquad\qquad \cdot C^{s_i \cdot \beta_{\mt}}C^{s_i \cdot \alpha_{\mt,a+b}}v_{s_i\mt}&\text{if $i\neq d(\mt,\mt^{\undla})(i_t)\neq i+1,$}\\
						(-1)^{b|\alpha_{\ms}'|_{>d(\ms,\mt^{\undla})(i_t)}+(a+b+1)|s_i\cdot \alpha_{\mt}|_{>i+1}+(a+b+1)\delta_{\alpha_\mt}(i+1)}\mathtt{c}_{\mt,\ms}&\\
						\qquad\qquad \cdot(-1)^{\delta_{\alpha_{\mt,a+b}}(i)\delta_{\alpha_{\mt,a+b}}(i+1)}\sqrt{\mathtt{c}_{\mt}(i)}&\\
						\qquad\qquad\qquad\qquad \cdot C^{s_i \cdot \beta_{\mt}}C^{s_i \cdot \alpha_{\mt,a+b}}v_{s_i\mt}, &\text{if $i= d(\mt,\mt^{\undla})(i_t),$}\\
						(-1)^{b|\alpha_{\ms}'|_{>d(\ms,\mt^{\undla})(i_t)}+(a+b+1)|s_i\cdot\alpha_{\mt}|_{>i}+(a+b+1)\delta_{\alpha_\mt}(i)}\mathtt{c}_{\mt,\ms}&\\
						\qquad\qquad \cdot(-1)^{\delta_{\alpha_{\mt,a+b}}(i)\delta_{\alpha_{\mt,a+b}}(i+1)}\sqrt{\mathtt{c}_{\mt}(i)}&\\
						\qquad\qquad\qquad\qquad \cdot C^{s_i \cdot \beta_{\mt}}C^{s_i \cdot \alpha_{\mt,a+b}}v_{s_i\mt}, &\text{if $i+1= d(\mt,\mt^{\undla})(i_t),$}
					\end{cases}\\
					&=\begin{cases} (-1)^{b|\alpha_{\ms}'|_{>d(\ms,\mt^{\undla})(i_t)}+(a+b+1)|s_i\cdot\alpha_{\mt}|_{>d(s_i\cdot\mt,\mt^{\undla})(i_t)}}\mathtt{c}_{\mt,\ms}&\\
						\qquad\qquad \cdot(-1)^{\delta_{\beta_{\mt}}(i)\delta_{\beta_{\mt}}(i+1)+\delta_{\alpha_{\mt}}(i)\delta_{\alpha_{\mt}}(i+1)}\sqrt{\mathtt{c}_{\mt}(i)}&\\
						\qquad\qquad\qquad\qquad \cdot C^{s_i \cdot \beta_{\mt}}C^{s_i \cdot \alpha_{\mt,a+b}}v_{s_i\mt}&\text{if $i\neq d(\mt,\mt^{\undla})(i_t)\neq i+1,$}\\
						(-1)^{b|\alpha_{\ms}'|_{>d(\ms,\mt^{\undla})(i_t)}+(a+b+1)|s_i\cdot \alpha_{\mt}|_{>i+1}+\delta_{\alpha_\mt}(i+1)}\mathtt{c}_{\mt,\ms}&\\
						\qquad\qquad \cdot \sqrt{\mathtt{c}_{\mt}(i)} C^{s_i \cdot \beta_{\mt}}C^{s_i \cdot \alpha_{\mt,a+b}}v_{s_i\mt}, &\text{if $i= d(\mt,\mt^{\undla})(i_t),$}\\
						(-1)^{b|\alpha_{\ms}'|_{>d(\ms,\mt^{\undla})(i_t)}+(a+b+1)|s_i\cdot\alpha_{\mt}|_{>i}+\delta_{\alpha_\mt}(i)}\mathtt{c}_{\mt,\ms}&\\
						\qquad\qquad \cdot \sqrt{\mathtt{c}_{\mt}(i)} C^{s_i \cdot \beta_{\mt}}C^{s_i \cdot \alpha_{\mt,a+b}}v_{s_i\mt}, &\text{if $i+1= d(\mt,\mt^{\undla})(i_t),$}
					\end{cases}\\
						&=\begin{cases} (-1)^{\delta_{\beta_{\mt}}(i)\delta_{\beta_{\mt}}(i+1)+\delta_{\alpha_{\mt}}(i)\delta_{\alpha_{\mt}}(i+1)}\sqrt{\mathtt{c}_{\mt}(i)}\frac{\mathtt{c}_{\mt,\ms}}{\mathtt{c}_{s_i\cdot\mt,\ms}}&\\
						\qquad\qquad \cdot f_{s_i\cdot {\rm T}, {\rm S}_{a}}C^{\beta_{\ms}'}C^{\alpha_{\ms,b}'}v_{\ms}&\text{if $i\neq d(\mt,\mt^{\undla})(i_t)\neq i+1,$}\\
						(-1)^{\delta_{\alpha_{\mt}}(i+1)}\sqrt{\mathtt{c}_{\mt}(i)}\frac{\mathtt{c}_{\mt,\ms}}{\mathtt{c}_{s_i\cdot\mt,\ms}}&\\
						\qquad\qquad \cdot f_{s_i\cdot {\rm T}, {\rm S}_{a}}C^{\beta_{\ms}'}C^{\alpha_{\ms,b}'}v_{\ms}, &\text{if $i= d(\mt,\mt^{\undla})(i_t),$}\\
						(-1)^{\delta_{\alpha_{\mt}}(i)}\sqrt{\mathtt{c}_{\mt}(i)}\frac{\mathtt{c}_{\mt,\ms}}{\mathtt{c}_{s_i\cdot\mt,\ms}}&\\
						\qquad\qquad \cdot f_{s_i\cdot {\rm T}, {\rm S}_{a}}C^{\beta_{\ms}'}C^{\alpha_{\ms,b}'}v_{\ms}, &\text{if $i+1= d(\mt,\mt^{\undla})(i_t).$}
					\end{cases}
				\end{align*}
				\end{enumerate}
					These together with \eqref{basic action} and Theorem \ref{semisimple:non-dege} prove (2,c).
			\end{enumerate}
			\end{proof}		
\begin{rem}
By Proposition \ref{re. fst and BK involution} and Proposition \ref{generators action on seminormal basis}, one can deduce the right action of generators on semimormal bases.
\end{rem}		
	
\subsection{Some subalgebras of $\mHfcn$}\label{some subalgebras}
In this subsection, we shall apply our seminormal bases theory to study some subalgebras inside $\mHfcn$.
\label{pag:subalgebras of mHfcn}
 \begin{defn}
(1) Let $\mathcal{G}_n^f$ be the subalgebra of $\mHfcn$ generated by $X_1+X_1^{-1},\ldots, X_n+X_n^{-1}.$

(2) Let $\mathbb{P}_n^f$ be the subalgebra of $\mHfcn$ generated by $X_1^{\pm 1},\ldots, X_n^{\pm 1}.$

(3) Let $\mathcal{A}_n^f$ be the subalgebra of $\mHfcn$ generated by $X_1^{\pm 1},\ldots, X_n^{\pm 1}$ and $C_1,\ldots, C_n.$
 \end{defn}

 Next, we shall explain those subalgebras $\mathcal{G}_n^f,$ $\mathcal{P}_n^f$ and $\mathcal{A}_n^f$ via above seminormal basis.

 \begin{defn}	Let $\undla\in\mathscr{P}^{\bullet,m}_{n}$ with $\bullet\in\{\mathsf{0},\mathsf{s},\mathsf{ss}\}$ and $\mt\in \Std(\undla).$

(1) We define
\label{pag:some idempotents of mHfcn}
 \begin{align}
 F_{\mt}:=\sum_{(\alpha_{\mt},\beta_{\mt})\in \mathbb{Z}_2(\mathcal{OD}_{\mt})_{\bar{0}}\times \mathbb{Z}_2([n]\setminus \mathcal{D}_{\mt})}F_{(\mt,\alpha_{\mt},\beta_{\mt})}.
\end{align}

(2) We define
 \begin{align}
 F_{(\mt,\beta_{\mt})}:=\sum_{\alpha_{\mt}\in \mathbb{Z}_2(\mathcal{OD}_{\mt})_{\bar{0}}}F_{(\mt,\alpha_{\mt},\beta_{\mt})}.
\end{align}
 \end{defn}

\label{pag:nondege q-values}
\begin{defn}	
For any $i\in [n], \mt\in \Std(\undla),$ we define
$$\mathtt{q}_{\mt,i}:=\mathtt{q}(\res_{\mt}(i)).$$
	\end{defn}

\begin{prop}\label{subalgebras}
Suppose $P^{(\bullet)}_{n}(q^2,\undQ)\neq 0$. Then

 (1) The following set of elements \begin{align}\label{G_n^f via seminormal basis}
 \biggl\{F_{\mt}~\biggm|\mt\in \Std(\mathscr{P}^{\bullet,m}_{n})\biggr\}
 \end{align} forms a $\mathbb{K}$-basis of $\mathcal{G}_n^f$.
 In particular, $$\dim_{\mathbb{K}}\mathcal{G}_n^f=|\Std(\mathscr{P}^{\bullet,m}_{n})|.$$

 (2) The following set of elements
   \begin{align}
\biggl\{F_{(\mt,\beta_{\mt})}~\biggm|\mt\in \Std(\mathscr{P}^{\bullet,m}_{n}), \beta_{\mt}\in \mathbb{Z}_2([n]\setminus \mathcal{D}_{\mt})\biggr\}
 \end{align} forms a $\mathbb{K}$-basis of $\mathcal{P}_n^f$.
 In particular, $$\dim_{\mathbb{K}}\mathcal{P}_n^f=\sum_{\undla \in \mathscr{P}^{\bullet,m}_{n}}2^{n-\sharp \mathcal{D}_{\undla}}|\Std(\undla)|.$$

 (3) The following set of elements
  \begin{align}\label{A_n^f via seminormal basis}
 \biggl\{{f_{{\rm T},{\rm S}}}~\biggm|\begin{matrix}{\rm T}=(\mt,\alpha_{\mt},\beta_{\mt})\in {\rm Tri}_{\bar{0}}(\mathscr{P}^{\bullet,m}_{n})\\
 {\rm S}=(\ms,\alpha'_{\ms},\beta'_{\ms})\in {\rm Tri}(\mathscr{P}^{\bullet,m}_{n}),\,	\mt=\ms\end{matrix}\biggr\}
 \end{align} forms a $\mathbb{K}$-basis of $\mathcal{A}_n^f.$ In particular, $$\dim_{\mathbb{K}}\mathcal{A}_n^f=\sum_{\undla \in \mathscr{P}^{\bullet,m}_{n}}2^{2n-\sharp \mathcal{D}_{\undla}}|\Std(\undla)|.$$
\end{prop}
\begin{proof}
 For $k=1,2,\ldots,n,$ we have
 \begin{align*}
 X_k^{\pm 1}=&X_k^{\pm 1}\cdot 1=X_k^{\pm 1}\left(\sum_{{\rm T}=(\mt,\alpha_{\mt},\beta_{\mt})\in {\rm Tri}_{\bar{0}}(\mathscr{P}^{\bullet,m}_{n})} F_{\rm T}
 \right)\\
 =&\sum_{{\rm T}=(\mt,\alpha_{\mt},\beta_{\mt})\in {\rm Tri}_{\bar{0}}(\mathscr{P}^{\bullet,m}_{n})} \mathtt{b}_{\mt,k}^{\mp\nu_{\beta_{\mt}}(k)}F_{\rm T}\\
 =&\sum_{\mt\in \Std(\mathscr{P}^{\bullet,m}_{n}), \beta_{\mt}\in \mathbb{Z}_2([n]\setminus \mathcal{D}_{\mt})} \mathtt{b}_{\mt,k}^{\mp\nu_{\beta_{\mt}}(k)}F_{(\mt,\beta_{\mt})}, \nonumber
 \end{align*} where the last second equation is due to Proposition \ref{generators action on seminormal basis}.

We deduce that, for any Laurent ploynomial $f(Y_1,\ldots,Y_n)\in \mathbb{K}[Y_1^{\pm1},\ldots,Y_n^{\pm1}],$ the following holds
  \begin{align}\label{laurent poly}
 f(X_1,\ldots,X_n)=\sum_{\mt\in \Std(\mathscr{P}^{\bullet,m}_{n}), \beta_{\mt}\in \mathbb{Z}_2([n]\setminus \mathcal{D}_{\mt})}
                   f\left(\mathtt{b}_{\mt,1}^{-\nu_{\beta_{\mt}}(1)},\ldots,\mathtt{b}_{\mt,1}^{-\nu_{\beta_{\mt}}(n)}\right)
                   F_{(\mt,\beta_{\mt})} .
 \end{align}

 (1). By \eqref{laurent poly},  for $k=1,2,\ldots,n,$ we have \begin{align}
 	X_k+X_k^{-1}
 	=\sum_{\mt\in \Std(\mathscr{P}^{\bullet,m}_{n}), \beta_{\mt}\in \mathbb{Z}_2([n]\setminus \mathcal{D}_{\mt})} (\mathtt{b}_{\mt,k}+\mathtt{b}_{\mt,k}^{-1})F_{(\mt,\beta_{\mt})} =\sum_{\mt\in \Std(\mathscr{P}^{\bullet,m}_{n})} \mathtt{q}_{\mt,k}F_{\mt}. \nonumber
 \end{align}
This implies that, for any ploynomial $g(Y_1,\ldots,Y_n)\in \mathbb{K}[Y_1,\ldots,Y_n],$
 \begin{align}\label{X+X^{-1} and Ft}
 	g(X_1+X_1^{-1},\ldots,X_n+X_n^{-1})=\sum_{\mt\in \Std(\mathscr{P}^{\bullet,m}_{n})} g(\mathtt{q}_{\mt,1},\ldots,\mathtt{q}_{\mt,n})
 	F_{\mt}.
 \end{align}
  Conversely, by the second part of Lemma \ref{important condition1}, if $\ms\neq\mt$, then there exists some $k\in [n]$ such that $\mathtt{q}_{\mt,k}\neq\mathtt{q}_{\ms,k}$. Combing this fact with
 \eqref{X+X^{-1} and Ft}, we have
 \begin{align}
 	F_{\mt}=\prod_{\ms\in \Std(\mathscr{P}^{\bullet,m}_{n}) \atop \ms\neq \mt}\prod_{k=1 \atop \mathtt{q}_{\ms,k}\neq \mathtt{q}_{\mt,k}}^{n} \frac{X_k+X_k^{-1}-\mathtt{q}_{\ms,k}}{\mathtt{q}_{\mt,k}-\mathtt{q}_{\ms,k}}\in \mathcal{G}_n^f \nonumber
 \end{align}
 for any $\mt\in \Std(\mathscr{P}^{\bullet,m}_{n}).$ This proves (1).

  (2). Since
\begin{align}
	\sum_{\alpha_{\mt}\in \mathbb{Z}_2(\mathcal{OD}_{\mt})_{\bar{0}}} C^{\alpha_{\mt}}\gamma_{\mt}(C^{\alpha_{\mt}})^{-1} =1,
\end{align}
we deduce that
\begin{align}
	F_{(\mt,\beta_{\mt})}=\prod_{k=1}^{n}\prod_{\mathtt{b}\in \mathtt{B}(k)\atop \mathtt{b}\neq \mathtt{b}_{+}(\res_{\mt}(k))}\frac{X_k^{\nu_{\beta_{\mt}}(k)}-\mathtt{b}}{\mathtt{b}_{+}(\res_{\mt}(k))-\mathtt{b}}\in \mathcal{P}_n^f.
\end{align}
 By \eqref{laurent poly}, we prove (2).

 (3). We denote   \begin{align*}
 (\mathcal{A}_n^f)':=	\mathbb{K}-{\text {span}}\biggl\{F_{{\rm T},{\rm S}}~\biggm|\begin{matrix}{\rm T}=(\mt,\alpha_{\mt},\beta_{\mt})\in {\rm Tri}_{\bar{0}}(\mathscr{P}^{\bullet,m}_{n})\\
 		{\rm S}=(\ms,\alpha'_{\ms},\beta'_{\ms})\in {\rm Tri}(\mathscr{P}^{\bullet,m}_{n}),\,	\mt=\ms\end{matrix}\biggr\}.
 \end{align*} Then
 \begin{align}
 C_i=C_i \cdot 1=\sum_{\substack{\undla \in \mathscr{P}^{\bullet,m}_{n},d_{\undla}=0 \\ {\rm T}\in {\rm Tri}(\undla)}}C_if_{{\rm T},{\rm T}}
                 +\sum_{\substack{\undla \in \mathscr{P}^{\bullet,m}_{n},d_{\undla}=1 \\ {\rm T}\in {\rm Tri}_{\bar{0}}(\undla)}}C_if_{{\rm T},{\rm T}_{\bar{0}}} \in (\mathcal{A}_n^f)'\nonumber
 \end{align}
 by (1.b) and (2.b) of Proposition \ref{generators action on seminormal basis}. Combining with \eqref{laurent poly}, we have $$
\mathcal{A}_n^f\subset (\mathcal{A}_n^f)' .$$ Conversely, recall the definition \eqref{re. fst. typeM. nondege.} and \eqref{re. fst. typeQ. nondege.} and note that $\Phi_{\mt,\mt}\equiv 1$ for any $\mt \in \mathscr{P}^{\bullet,m}_{n}$, we deduce that $$
(\mathcal{A}_n^f)'\subset \mathcal{A}_n^f .$$
\end{proof}

Suppose $B$ is a ring and $A\subseteq B$ is a subring, recall that the centralizer of $A$ in $B$ is defined by $C_{B}(A):=\{b\in B \mid ab=ba,\forall a\in A\}.$
\begin{prop}
Suppose $P^{(\bullet)}_{n}(q^2,\undQ)\neq 0$. Then we have
 \begin{align}
\mathcal{A}_n^f&=C_{\mHfcn}(\mathcal{G}_n^f),\\
\mathcal{G}_n^f&=(C_{\mHfcn}(\mathcal{A}_n^f))_{\bar{0}}, \label{even centralizer of A_n^f}\\
	C_{\mHfcn}(\mathbb{P}_n^f)=\mathbb{K}-{\text {span}}\biggl\{F_{{\rm T},{\rm S}}~&\biggm|\begin{matrix}{\rm T}=(\mt,\alpha_{\mt},\beta_{\mt})\in {\rm Tri}_{\bar{0}}(\mathscr{P}^{\bullet,m}_{n})\\
		{\rm S}=(\ms,\alpha'_{\ms},\beta'_{\ms})\in {\rm Tri}(\mathscr{P}^{\bullet,m}_{n}),\,	\mt=\ms, \beta_{\mt}=\beta'_{\ms}\end{matrix}\biggr\}
\end{align}
 and
 \begin{align}
 (C_{\mHfcn}C_{\mHfcn}(\mathbb{P}_n^f))_{\bar{0}}=\mathbb{P}_n^f.
 \end{align}
 In particular,
 $$\dim_{\mathbb{K}}C_{\mHfcn}(\mathbb{P}_n^f)=\sum_{\undla \in \mathscr{P}^{\bullet,m}_{n}}2^n|\Std(\undla)|.$$
\end{prop}
\begin{proof}
The idea of the proof is to apply the multiplication formulae in Theorem \ref{seminormal basis} and use the basis given in Proposition \ref{subalgebras}. We only check \eqref{even centralizer of A_n^f} here since the proof for other equalities is similar. It is clearly that $\mathcal{G}_n^f \subseteq (C_{\mHfcn}(\mathcal{A}_n^f))_{\bar{0}}$ by relation \eqref{XC} or Proposition \ref{subalgebras}. Conversely, for any $h\in (C_{\mHfcn}(\mathcal{A}_n^f))_{\bar{0}},$ we can write
\begin{align}
 h=\sum_{\substack{\undla \in \mathscr{P}^{\bullet,m}_{n},d_{\undla}=0 \\ {\rm S}, {\rm T}\in {\rm Tri}(\undla)}}
a_{{\rm S},{\rm T}} f^\mathfrak{w}_{{\rm S},{\rm T}}
 +
 \sum_{\substack{a\in \mathbb{Z}_2 \\ \undla \in \mathscr{P}^{\bullet,m}_{n},d_{\undla}=1 \\ {\rm U}, {\rm V}\in {\rm Tri}_{\bar{0}}(\undla)}}
 b_{{\rm U},{\rm V}_a} f^\mathfrak{w}_{{\rm U},{\rm V}_a}\nonumber
 \end{align}by Theorem \ref{seminormal basis}.

(i) For any ${\rm T}=(\mt,\alpha_{\mt},\beta_{\mt}), {\rm T}'=(\mt,\alpha_{\mt}',\beta_{\mt}') \in {\rm Tri}(\undla)$ with $d_{\undla}=0,$
we have $ f^{\mathfrak{w}}_{{\rm T},{\rm T}'} h=h f^{\mathfrak{w}}_{{\rm T},{\rm T}'} $ by \eqref{A_n^f via seminormal basis}. Combing with \eqref{Non-deg multiplication1}, it follows that
\begin{align}
 \sum_{\substack{\undla \in \mathscr{P}^{\bullet,m}_{n},d_{\undla}=0 \\ {\rm S} \in {\rm Tri}(\undla)}}
a_{{\rm T}',{\rm S}} \mathtt{c}_{{\rm T}'}^{\mathfrak{w}} f^\mathfrak{w}_{{\rm T},{\rm S}}
= \sum_{\substack{\undla \in \mathscr{P}^{\bullet,m}_{n},d_{\undla}=0 \\ {\rm S} \in {\rm Tri}(\undla)}}
a_{{\rm S},{\rm T}} \mathtt{c}_{\rm T}^{\mathfrak{w}} f^\mathfrak{w}_{{\rm S},{\rm T}'}. \nonumber
 \end{align}
Since by definition $\mathtt{c}_{\rm T}^{\mathfrak{w}}=\mathtt{c}_{{\rm T}'}^{\mathfrak{w}}=(\mathtt{c}_{\mt,\mathfrak{w}})^2\in \mathbb{K}^*$, we deduce that $a_{{\rm S},{\rm T}}=0$ unless ${\rm S}={\rm T}$ and $a_{{\rm T},{\rm T}}=a_{{\rm T}',{\rm T}'}.$

(ii) For any ${\rm T}=(\mt,\alpha_{\mt},\beta_{\mt}),{\rm T}'=(\mt,\alpha_{\mt}',\beta_{\mt}') \in {\rm Tri}_{\bar{0}}(\undla)$ with $d_{\undla}=1,$
we have $ f^{\mathfrak{w}}_{{\rm T},{\rm T}'_a} h=h f^{\mathfrak{w}}_{{\rm T},{\rm T}'_a}, a\in \mathbb{Z}_2$ by \eqref{A_n^f via seminormal basis}. Similarly, it follows from \eqref{Non-deg multiplication2} that $b_{{\rm S},{\rm T}_b}=0$ unless ${\rm S}={\rm T}$ and $b_{{\rm T},{\rm T}_b}=b_{{\rm T}',{\rm T}'_b},$ for any $b\in \mathbb{Z}_2.$ Since $h$ is even, we must have $b_{{\rm T},{\rm T}_{\bar{1}}}=0.$

Combining (i) and (ii) above, we deduce that $h\in \mathcal{G}_n^f$ by \eqref{G_n^f via seminormal basis}.
\end{proof}

The above Proposition shows that subalgebra $\mathbb{P}_n^f$ is not maximal commutative in general. To sum up, suppose $P^{(\bullet)}_{n}(q^2,\undQ)\neq 0$, we have described bases and dimensions for the following subalgebras
$${\rm Z}(\mHfcn)_{\bar{0}}\subset \mathcal{G}_n^f \subset \mathbb{P}_n^f \subset C_{\mHfcn}(\mathbb{P}_n^f) \subset \mathcal{A}_n^f \subset \mHfcn.$$

\begin{conjecture}
For each subalgebra $A\in \{ {\rm Z}(\mHfcn)_{\bar{0}}, \mathcal{G}_n^f, \mathbb{P}_n^f , C_{\mHfcn}(\mathbb{P}_n^f), \mathcal{A}_n^f \},$ the
$\dim_{\mathbb{K}}A$ is independent of ground field $\mathbb{K}$ and parameters $(q,\undQ).$
\end{conjecture}

\section{Primitive idempotents and seminormal basis for cyclotomic Sergeev algebra $\mhgcn$}\label{dege}
{\bf	In this section, we fix the parameter $\undQ=(Q_1,Q_2,\ldots,Q_m)\in \mathbb{K}^m$ and $g=g^{(\bullet)}_{\undQ}$ with $P^{\bullet}_{n}(1,\undQ)\neq 0$ for $\bullet\in\{\mathtt{0},\mathtt{s}\}.$  Accordingly, we define the residue of boxes in the young diagram $\undla$ via \eqref{eq:deg-residue} as well as $\res(\mathfrak{t})$ for each $\mathfrak{t}\in\Std(\undla)$ with $\undla\in\mathscr{P}^{\bullet,m}_{n}$ with $m\geq 0.$}
\subsection{Simple modules}		

As in nondegenerate case, we first recall the construction of $\mhgcn$-simple module $D(\undla)$ in \cite{SW} for $\undla \in \mathscr{P}^{\bullet,m}_{n}.$ Then we shall give an explicit basis of $D(\undla)$ and write down the action of generators of $\mhcn$ on this basis.

\label{pag:dege simple module}
For $\undla\in\mathscr{P}^{\bullet,m}_{n}$ with $\bullet\in\{\mathsf{0},\mathsf{s}\}$, we define the $\mathcal{A}_n$-module
	$$
D(\undla):=\oplus_{\mt \in \Std(\undla)}L(\res(\mathfrak{t}^{\undla}))^{d(\mt,\mt^{\undla})}.
	$$

To define a $\mhgcn$-module structure  on $D(\undla)$, we recall two operators as in \cite[(5.22),(5.23)]{SW} which are similar with the operators in \cite{Wa}:
\begin{align}
	\Xi_i u&:=-\Big(\frac{1}{x_i-x_{i+1}}+c_ic_{i+1}\frac{1}{x_i+x_{i+1}}\Big)u,\label{Operater1-dege}\\
	\Omega_i u&:=\Bigg(\sqrt{1-\frac{1}{(x_i-x_{i+1})^2}-\frac{1}{(x_i+x_{i+1})^2}}\Bigg)u,\label{Operater2-dege}
\end{align} where $u\in L(\res(\mathfrak{t}^{\undla}))^{d(\mt,\mt^{\undla})}\simeq L(\res(\mt))$.
By Definition \ref{important condition2} and Proposition \ref{separate formula dege}, the eigenvalues of $x^2_i$ and $x^2_{i+1}$ on $L(\res(\mathfrak{t}^{\undla}))^{d(\mt,\mt^{\undla})}$ are different, hence the operators $\Xi_i$ and $\Omega_i$  are well-defined on $L(\res(\mathfrak{t}^{\undla}))^{d(\mt,\mt^{\undla})}$ for each $\mt \in \Std(\undla)$.

\begin{thm}(\cite[Theorem 5.20]{SW})
Let $\undQ=(Q_1,\ldots,Q_m)$. Suppose $g=g(x_1)=g^{(\bullet)}_{\undQ}(x_1)$ and  $P_n^{(\bullet)}(1,\undQ)\neq 0$ with $\bullet\in\{\mathsf{0},\mathsf{s}\}$.
  $D(\undla)$ affords a $\mhgcn$-module via
\begin{align}
s_i z^{\tau}= \left \{
 \begin{array}{ll}
 \Xi_i z^{d(\mt,\mt^{\undla})}
 +\Omega_i z^{s_id(\mt,\mt^{\undla})},
 & \text{ if } s_i \mt\in \Std(\undla), \\
 \Xi_iz^{d(\mt,\mt^{\undla})}
 , & \text{ otherwise},
 \end{array}
 \right.\label{actionformula}
\end{align}
 for  $1\leq i\leq n-1, z\in L(\res(\mathfrak{t}^{\undla}))$ and $\mt\in
\Std(\undla)$.
\end{thm}

Recall that we have fixed  $\undla\in\mathscr{P}^{\bullet,m}_{n}$. Let $t:=\sharp \mathcal{D}_{\undla}.$
		\begin{defn}
We denote \begin{align}
				\mathcal{D}_{\mt^{\undla}}&:=\{i_1<i_2<\cdots<i_t\}=\{\mt^{\undla}(a,a,l)|(a,a,l)\in\mathcal{D}_{\undla}\}\label{dege.stanard D},\\
				\mathcal{OD}_{\mt^{\undla}}&:=\{i_1,i_3,\cdots,i_{2{\lceil t/2 \rceil}-1}\}\subset \mathcal{D}_{\mt^{\undla}}\label{dege.standard OD}.
	\end{align}
\end{defn}
For each $\mt\in \Std(\undla),$ recall the definition of $\mathcal{D}_{\mt},\,\mathcal{OD}_{\mt},\,\Z_2(\mathcal{OD}_{\mt}),\,\Z_2([n]\setminus \mathcal{D}_{\mt})$.
		

%
		\begin{defn}	
For any $i\in [n], \mt\in \Std(\undla),$ we denote
			 $$\mathtt{u}_{\mt,i}:=\mathtt{u}_{+}(\res_{\mt}(i)).$$
If $i\in [n-1]$, we define
\label{pag:dege coeffi cti}	
			\begin{align}
				\mathfrak{c}_{\mt}(i):=1-\frac{1}{(\mathtt{u}_{\mt,i}-\mathtt{u}_{\mt,i+1})^2}
				 -\frac{1}{(\mathtt{u}_{\mt,i}+\mathtt{u}_{\mt,i+1})^2}\in \mathbb{K}.
			\end{align}
			\end{defn}
\begin{rem}
It is notable that the choice of sign of $\mathtt{u}_{\mt,i}$ is different from that of $\mathtt{b}_{\mt,i}$ in non-degenerate case, which aims to be compatible with the form of intertwining element, see Remark \ref{diff. of phi and Phi}.
\end{rem}
			Since $\mt \in \Std(\undla),$ $\mathtt{u}_{\mt,i}\neq \pm\mathtt{u}_{\mt,i+1}$ by Definition \ref{important condition2} and Proposition \ref{separate formula dege}, which immediately implies that $\mathfrak{c}_{\mt}(i)$ is well-defined. If $s_i$ is admissible with respect to $\mt$, i.e., $\delta(s_i\mt)=1$, then $\mathfrak{c}_{\mt}(i)\in \mathbb{K}^{*}$ by
		the third part of Lemma \ref{important conditionequi2}. It is clear that $\mathfrak{c}_{\mt}(i)=\mathfrak{c}_{s_i\mt}(i).$	
		
		\begin{defn}
			Let $\mt\in\Std(\undla),\,\beta_{\mt}\in\Z_2([n]\setminus \mathcal{D}_{\mt})$ and $\alpha_{\mt}\in \Z_2(\mathcal{OD}_{\mt})$.	For each $i\in [n-1],$ we define $\mathcal{R}(i,\beta_\mt, \alpha_\mt, \mt)\in D(\undla)$ as follows.
		
		\begin{enumerate}
			\item If $i,i+1\in [n]\setminus \mathcal{D}_{\mt},$
			$$\mathcal{R}(i,\beta_\mt, \alpha_\mt, \mt):=(-1)^{\delta_{\beta_{\mt}}(i)}c^{\beta_{\mt}+e_i+e_{i+1}}c^{\alpha_{\mt}}v_{\mt};
			$$
			\item
			if $i={d(\mt,\mt^{\undla})}(i_p)\in \mathcal{D}_{\mt}, i+1\in [n]\setminus \mathcal{D}_{\mt},$
			$$\mathcal{R}(i,\beta_\mt, \alpha_\mt, \mt):=\begin{cases}(-1)^{|\beta_{\mt}|_{>i}+|\alpha_{\mt}|_{<i}}c^{\beta_{\mt}+e_{i+1}}c^{\alpha_{\mt}+e_{i}}v_{\mt}, & \text{ if  $p$ is odd},\\
				(-\sqrt{-1})(-1)^{|\beta_{\mt}|_{>i}+|\alpha_{\mt}|_{\leq {d(\mt,\mt^{\undla})}(i_{p-1})}}&\\
				\qquad\qquad\qquad\cdot c^{\beta_{\mt}+e_{i+1}}c^{\alpha_{\mt}+e_{{d(\mt,\mt^{\undla})}(i_{p-1})}}v_{\mt}, &\text{ if  $p$ is even};
			\end{cases}\nonumber
			$$
			\item
			if $i+1={d(\mt,\mt^{\undla})}(i_p)\in \mathcal{D}_{\mt}, i\in [n]\setminus \mathcal{D}_{\mt},$ $$\mathcal{R}(i,\beta_\mt, \alpha_\mt, \mt):=\begin{cases}(-1)^{|\beta_{\mt}|_{\geq i}+|\alpha_{\mt}|_{<i+1}} c^{\beta_{\mt}+e_i}c^{\alpha_{\mt}+e_{i+1}}v_{\mt}, & \text{ if  $p$ is odd},\\
				(-\sqrt{-1})(-1)^{|\beta_{\mt}|_{\geq i}+|\alpha_{\mt}|_{\leq {d(\mt,\mt^{\undla})}(i_{p-1})}}&\\
				\qquad\qquad\qquad \cdot c^{\beta_{\mt}+e_i}c^{\alpha_{\mt}+e_{{d(\mt,\mt^{\undla})}(i_{p-1})}}v_{\mt}, &\text{ if $p$ is even}.\end{cases}  \nonumber
			$$
		\end{enumerate}
\end{defn}

		\begin{prop}\label{dege-actions of generators on L basis} $D(\undla)$ has a $\mathbb{K}$-basis of the form
$$\bigsqcup_{\mt\in \Std(\undla)}\Biggl\{c^{\beta_{\mt}}c^{\alpha_{\mt}}v_{\mt}\biggm|\begin{matrix}\beta_{\mt} \in \Z_2([n]\setminus \mathcal{D}_{\mt})  \\
					\alpha_{\mt}\in \Z_2(\mathcal{OD}_{\mt})
					\end{matrix}\Biggr\}.$$ such that
				$$  L_{\mt}:=\Biggl\{c^{\beta_{\mt}}c^{\alpha_{\mt}}v_{\mt}\biggm|\begin{matrix}\beta_{\mt} \in  \Z_2([n]\setminus \mathcal{D}_{\mt} ) \\
						\alpha_{\mt}\in \Z_2(\mathcal{OD}_{\mt})
					\end{matrix}\Biggr\}$$ forms a $\mathbb{K}$-basis of $ L(\res(\mathfrak{t}^{\undla}))^{{d(\mt,\mt^{\undla})}}$ and the actions of generators of $\mhgcn$ are given in the following.
					
			Let $\mt\in\Std(\undla),\,\beta_{\mt}\in\Z_2([n]\setminus \mathcal{D}_{\mt})$ and $\alpha_{\mt}\in \Z_2(\mathcal{OD}_{\mt})$.
			\begin{enumerate}
				\item For each $i\in [n],$ we have
				\begin{align}\label{x eigenvalues}
					x_i\cdot c^{\beta_{\mt}}c^{\alpha_{\mt}}v_{\mt}
					=\nu_{\beta_{\mt}}(i)\mathtt{u}_{\mt,i}c^{\beta_{\mt}}c^{\alpha_{\mt}}v_{\mt}.
				\end{align}
				\item We have $\gamma_{\mt} v_{\mt}=v_{\mt}.$
				\item For each $i\in [n],$ we have
				\begin{align}
					&c_i\cdot c^{\beta_{\mt}}c^{\alpha_{\mt}}v_{\mt} \nonumber\\
					&=\begin{cases}
						(-1)^{|\beta_{\mt}|_{<i}} c^{\beta_{\mt}+e_i}c^{\alpha_{\mt}}v_{\mt}, & \text{ if } i\in [n]\setminus \mathcal{D}_{\mt}, \\
						(-1)^{|\beta_{\mt}|+|\alpha_{\mt}|_{<i}} c^{\beta_{\mt}}c^{\alpha_{\mt}+e_{i}}v_{\mt}, & \text{ if  $i={d(\mt,\mt^{\undla})}(i_p) \in \mathcal{D}_{\mt}$ ,where $p$ is odd},\\
								(-\sqrt{-1})(-1)^{|\beta_{\mt}|+|\alpha_{\mt}|_{\leq {d(\mt,\mt^{\undla})}(i_{p-1})}}&\\
							\qquad\qquad	\cdot c^{\beta_{\mt}}c^{\alpha_{\mt}+e_{{d(\mt,\mt^{\undla})}(i_{p-1})}}v_{\mt}, &\text{ if  $i={d(\mt,\mt^{\undla})}(i_p) \in \mathcal{D}_{\mt}$ ,where $p$ is even}.\nonumber
					\end{cases}
				\end{align}
				\item	For each $i\in [n-1],$ we have
				\begin{align}
					&s_i\cdot c^{\beta_{\mt}}c^{\alpha_{\mt}}v_{\mt} \label{Taction dege}\\
					=&-\frac{1}{\nu_{\beta_{\mt}}(i)\mathtt{u}_{\mt,i}-\nu_{\beta_{\mt}}(i+1)\mathtt{u}_{\mt,i+1}}  c^{\beta_{\mt}}c^{\alpha_{\mt}}v_{\mt} \nonumber\\
					&	\qquad\qquad  -\frac{1}{\nu_{\beta_{\mt}}(i)\mathtt{u}_{\mt,i}+\nu_{\beta_{\mt}}(i+1)\mathtt{u}_{\mt,i+1}} \mathcal{R}(i,\beta_\mt, \alpha_\mt, \mt)\nonumber\\
					&	\qquad\qquad\qquad +\delta(s_i\mt)(-1)^{\delta_{\beta_{\mt}}(i)\delta_{\beta_{\mt}}(i+1)}\sqrt{\mathfrak{c}_{\mt}(i)}c^{s_i \cdot \beta_{\mt}}c^{s_i \cdot \alpha_{\mt}}v_{s_i\mt}.\nonumber
				\end{align}

			\end{enumerate}
		\end{prop}
\subsection{Primitive idempotents of $\mhgcn$}

Recall the definition of $\Z_2(\mathcal{OD}_{\mt})_{a}$ (\ref{decomposition of OD}) and ${\rm Tri}_{a}(\undla)$ (\ref{Tri}) for $a\in \Z_2,\,\undla\in\mathscr{P}^{\bullet,m}_{n}$ with $\bullet\in\{\mathsf{0},\mathsf{s}\}$. We shall define the primitive idempotents in degenerate case.
\begin{defn}
           For $k\in[n]$, let
           $$\mathtt{U}(k):=\{ \mathtt{u}_{\pm}(\res_{\ms}(k)) \mid \ms \in \Std(\mathscr{P}^{\bullet,m}_{n}) \}.$$
			For any ${\rm T}=(\mt, \alpha_{\mt}, \beta_{\mt})\in {\rm Tri}_{\bar{0}}(\undla),$ we define
            \label{pag:dege primitive idempotents}
			\begin{align}\label{definition of primitive idempotents. dege}
				 \mathcal{F}_{\rm T}:=\left(c^{\alpha_{\mt}}\gamma_{\mt}(c^{\alpha_{\mt}})^{-1} \right) \cdot \left(\prod_{k=1}^{n}\prod_{\mathtt{u}\in \mathtt{U}(k)\atop \mathtt{u}\neq \mathtt{u}_{+}(\res_{\mt}(k))}\frac{\nu_{\beta_{\mt}}(k)x_k-\mathtt{u}}{\mathtt{u}_{+}(\res_{\mt}(k))-\mathtt{u}}\right)\in \mhgcn.
\end{align}
We define
\begin{align}
				 \mathcal{F}_{\undla}&:=\sum_{{\rm T}\in {\rm Tri}_{\bar{0}}(\undla)}  \mathcal{F}_{\rm T},
			\end{align}
and two left ideals of $\mhgcn$ as follows
    \label{pag:dege simple blocks}
	\begin{align}
				\mathfrak{D}_{\rm T}&:=\mhgcn  \mathcal{F}_{\rm T}\subseteq\mhgcn,\\
				 \mathcal{B}_{\undla}&:=\mhgcn   \mathcal{F}_{\undla}\subseteq \mhgcn.
			\end{align}
		\end{defn}
The following is the analogue of Theorem \ref{primitive idempotents}.
\begin{thm}\label{primitive idempotents. dege}
			Suppose $P_n^{(\bullet)}(1,\undQ)\neq 0$. For $\bullet\in\{\mathsf{0},\mathsf{s}\}$, we have the followin.
			
			(a) $\{\mathcal{F}_{\rm T} \mid \rm T\in {\rm Tri}_{\bar{0}}(\mathscr{P}^{\bullet,m}_{n})\}$ is a complete set of (super) primitive orthogonal idempotents of $\mhgcn.$
			
			(b) $\{\mathcal{F}_{\undla} \mid \undla \in \mathscr{P}^{\bullet,m}_{n} \}$ is a complete set of (super) primitive central idempotents of $\mhgcn.$
			
			(c) For ${\rm T}=(\mt, \alpha_{\mt}, \beta_{\mt})\in {\rm Tri}_{\bar{0}}(\undla),\,
			{\rm S}=(\ms, \alpha_{\ms}', \beta_{\ms}')\in {\rm Tri}_{\bar{0}}(\underline{\mu}),$ then $\mathfrak{D}_{\rm T}$ and $\mathfrak{D}_{\rm S}$ belongs to the same block if and only if $\undla=\underline{\mu}$.
			
			(d) Let $\undla\in\mathscr{P}^{\bullet,m}_{n}$. If $d_{\undla}=1$, then for any ${\rm T}=(\mt, \alpha_{\mt}, \beta_{\mt}),
			{\rm S}=(\ms, \alpha_{\ms}', \beta_{\ms}')\in {\rm Tri}_{\bar{0}}(\undla),$ we have evenly isomorphic $\mhgcn$-supermodules $\mathfrak{D}_{\rm T}\cong \mathfrak{D}_{\rm S}$; if $d_{\undla}=0$, then for any ${\rm T}=(\mt, \alpha_{\mt}, \beta_{\mt}),
			{\rm S}=(\ms, \alpha_{\ms}', \beta_{\ms}')\in {\rm Tri}(\undla),$ we have evenly isomorphic $\mhgcn$-supermodules $\mathfrak{D}_{\rm T}\cong \mathfrak{D}_{\rm S}$ if and only if
			$$|\alpha_{\mt}|+|\beta_{\mt}| \equiv |\alpha_{\ms}'|+|\beta_{\ms}'| \pmod 2.$$
		\end{thm}

\begin{cor} Suppose $P_n^{(\bullet)}(1,\undQ)\neq 0$. Then
			the set of elements $$\{\mathcal{F}_{\undla} \mid \lambda\in \mathscr{P}^{\bullet,m}_{n} \}$$ form a $\mathbb{K}$-basis of the super center ${\rm Z}(\mhgcn)_{\bar{0}}.$
		\end{cor}

\begin{rem}\label{KMS idempotent}
	For any $\mt \in \Std(\undla), \beta_{\mt} \in \Z_2([n]\setminus \mathcal{D}_{\mt}),$ we denote that
	$$\mathcal{F}_{(\mt, \beta_{\mt})}:=\sum_{\alpha_{\mt}\in \Z(\mathcal{OD}_{\mt})_{\bar{0}}}\mathcal{F}_{(\mt, \alpha_{\mt},\beta_{\mt})}=\prod_{k=1}^{n}\prod_{\mathtt{u}\in \mathtt{U}(k)\atop \mathtt{u}\neq \mathtt{u}_{+}(\res_{\mt}(k))}\frac{\nu_{\beta_{\mt}}(k)x_k-\mathtt{u}}{\mathtt{u}_{+}(\res_{\mt}(k))-\mathtt{u}} \in \mhgcn.$$
	For any standard tableau $\ms$ and $1\leq k \leq n,$ we denote by ${\rm Add}(\ms)$ the set of all addable boxes of ${\rm shape}(\ms)$ and $\ms\downarrow_k$ the subtableau of $\ms$ labeled by $1,2,\ldots,k$ in $\ms.$ Then
	I. Kashuba, A. Molev and V. Serganova defined another elements in Sergeev algebra \cite{KMS}, i.e., $g=x_1$:
	$$\mathcal{E}_{(\mt, \beta_{\mt})}:= \prod_{k=1}^{n}\prod_{\mathtt{u}\in \{\mathtt{u}_{\pm}(\res(\gamma))\mid \gamma \in {\rm Add}(\mt\downarrow_{k-1})\}\backslash\{\mathtt{u}_{+}(\res_{\mt}(k))\}} \frac{\nu_{\beta_{\mt}}(k)x_k-\mathtt{u}}{\mathtt{u}_{+}(\res_{\mt}(k))-\mathtt{u}} \in \mhgcn.$$
	We claim that $\mathcal{F}_{(\mt, \beta_{\mt})}=\mathcal{E}_{(\mt, \beta_{\mt})}.$ In fact, we can check this equality similarly to the proof of Lemma \ref{idempotent action. non-dege}, i.e., $\mathcal{F}_{(\mt, \beta_{\mt})}$ and $\mathcal{E}_{(\mt, \beta_{\mt})}$ act as the same linear operator on each $\mathbb{D}(\underline{\mu})$. For example, for any standard tableau $\ms \neq \mt,$ there exists a minimal number $k_0$ such that $\ms\downarrow_{k_0} \neq \mt\downarrow_{k_0}$ but $\ms\downarrow_{k_0-1} = \mt\downarrow_{k_0-1},$ i.e., $\ms^{-1}(k_0),\mt^{-1}(k_0)\in {\rm Add}(\mt\downarrow_{k_0-1})$ and $\ms^{-1}(k_0) \neq \mt^{-1}(k_0).$ So we have $\mathtt{u}_{\pm}(\res_{\ms}(k_0))\neq \mathtt{u}_{+}(\res_{\mt}(k_0))$ by the second part of Lemma \ref{important conditionequi2}.
	Thus, it follows that $\mathcal{E}_{(\mt, \beta_{\mt})} \cdot C^{\beta_{\ms}^{'}} C^{\alpha_{\ms}^{'}} v_{\ms} =0$ by \eqref{x eigenvalues} for any $\alpha_{\ms}^{'}\in \Z_2(\mathcal{OD}_{\ms}), \beta_{\ms}^{'}\in \Z_2([n]\setminus \mathcal{D}_{\ms}).$
\end{rem}

\subsection{Seminormal basis of $\mhgcn$}
In this subsection, we also fix $\undla\in\mathscr{P}^{\bullet,m}_{n}$ with $\bullet\in\{\mathsf{0},\mathsf{s}\}.$ We will construct a series of seminormal bases for block $\mathcal{B}_{\undla}$ of $\mhgcn$.
	\begin{defn}
    \label{pag:phist and cst}	
For any $\ms,\mt \in \Std(\undla),$ we fix a reduced expression $d(\ms,\mt)=s_{k_p}\cdots s_{k_1}$. We define
		\begin{align}\label{phist}
			\phi_{\ms,\mt}:=\overleftarrow{\prod_{i=1,\ldots,p}}\phi_{k_{i}}(\mathtt{u}_{s_{k_{i-1}}\cdots s_{k_1}\mathfrak{t},k_{i}}, \mathtt{u}_{s_{k_{i-1}}\cdots s_{k_1}\mathfrak{t},k_{i}+1})  \in \mhgcn
		\end{align}
		and the coefficient
		\begin{align}\label{c-coefficients. dege}
			\mathfrak{c}_{\ms,\mt}:=\prod_{i=1,\ldots,p}\sqrt{\mathfrak{c}_{s_{k_{i-1}}\cdots s_{k_1}\mathfrak{t}}(k_{i})}  \in \mathbb{K}.
		\end{align}
		\end{defn}
By Lemma \ref{admissible transposes} and the third part of Lemma \ref{important conditionequi2}, we have $\mathfrak{c}_{\ms,\mt}\in \mathbb{K}^*$.  We can similarly prove that $\phi_{\ms,\mt}$ is independent of the reduced choice of $d(\ms,\mt)$ and has similar properties as $\Phi_{\ms,\mt}$ in Lemma \ref{Phist. well-defi} and Lemma \ref{Phist. lem}.

Now we define the seminormal basis in degenerate case.
\begin{defn}
Let $\mathfrak{w}\in\Std(\undla).$

\label{pag:dege seminormal basis}		
		(1) Supppose $d_{\undla}=0.$  For any ${\rm S}=(\ms, \alpha_{\ms}', \beta_{\ms}'), {\rm T}=(\mt, \alpha_{\mt}, \beta_{\mt})\in {\rm Tri}(\undla),$ we define
			\begin{align}\label{fst. typeM. dege.}
				\mathfrak{f}_{{\rm S},{\rm T}}^{\mathfrak{w}}
				:=\mathcal{F}_{\rm S}c^{\beta_{\ms}'}c^{\alpha_{\ms}'}\phi_{\ms,\mathfrak{w}}\phi_{\mathfrak{w},\mt}
				(c^{\alpha_{\mt}})^{-1} (c^{\beta_{\mt}})^{-1} \mathcal{F}_{\rm T}\in \mathcal{F}_{\rm S}\mhgcn \mathcal{F}_{\rm T},
			\end{align} and	\begin{align}\label{re. fst. typeM. dege.}
			\mathfrak{f}_{{\rm S},{\rm T}}
			:=\mathcal{F}_{\rm S}c^{\beta_{\ms}'}c^{\alpha_{\ms}'}\phi_{\ms,\mt}
			(c^{\alpha_{\mt}})^{-1} (c^{\beta_{\mt}})^{-1}\mathcal{F}_{\rm T} \in \mathcal{F}_{\rm S}\mhgcn \mathcal{F}_{\rm T},
			\end{align}

			(2) Suppose $d_{\undla}=1.$ For any $a\in \mathbb{Z}_{2}$ and ${\rm S}=(\ms, \alpha_{\ms}', \beta_{\ms}')\in {\rm Tri}_{\bar{0}}(\undla), {\rm T}_{a}=(\mt, \alpha_{\mt,a}, \beta_{\mt})\in {\rm Tri}_{a}(\undla),$ we define
			\begin{align}\label{fst. typeQ. dege.}
			\mathfrak{f}_{{\rm S},{\rm T}_{a}}^{\mathfrak{w}}
            :=\mathcal{F}_{\rm S}c^{\beta_{\ms}'}c^{\alpha_{\ms}'}\phi_{\ms,\mathfrak{w}}\phi_{\mathfrak{w},\mt}
				(c^{{\alpha}_{\mt,a}})^{-1} (c^{\beta_{\mt}})^{-1} \mathcal{F}_{\rm T}\in \mathcal{F}_{\rm S}\mhgcn \mathcal{F}_{\rm T},
			\end{align} and 	\begin{align}\label{re. fst. typeQ. dege.}
		    \mathfrak{f}_{{\rm S},{\rm T}_{a}}
			:= \mathcal{F}_{\rm S}c^{\beta_{\ms}'}c^{\alpha_{\ms}'}\phi_{\ms,\mt}
			(c^{{\alpha}_{\mt,a}})^{-1} (c^{\beta_{\mt}})^{-1} \mathcal{F}_{\rm T} \in \mathcal{F}_{\rm S}\mhgcn \mathcal{F}_{\rm T}.
			\end{align}
			
	   (3)  For any ${\rm T}=(\mt, \alpha_{\mt}, \beta_{\mt})\in {\rm Tri}(\undla),$ we define
\label{pag:dege cT}
$$\mathfrak{c}_{\rm T}^{\mathfrak{w}}:=(\mathfrak{c}_{\mt,\mathfrak{w}})^2\in \mathbb{K}^*.$$
		\end{defn}

We have the following result as a degenerate version of Theorem \ref{seminormal basis}.
\begin{thm}\label{seminormal basis. dege}
			Suppose $P_n^{(\bullet)}(1,\undQ)\neq 0$.  We fix $\mathfrak{w}\in\Std(\undla)$. Then the following two sets
			\begin{align}\label{deg seminormal1}
			\left\{ \mathfrak{f}_{{\rm S},{\rm T}}^\mathfrak{w} \Biggm|
			{\rm S}=(\ms, \alpha_{\ms}', \beta_{\ms}')\in {\rm Tri}_{\bar{0}}(\undla),
			{\rm T}=(\mt, \alpha_{\mt}, \beta_{\mt})\in {\rm Tri}(\undla)
			\right\}
		\end{align} and \begin{align}\label{deg seminormal2}
			\left\{ \mathfrak{f}_{{\rm S},{\rm T}} \Biggm|
		{\rm S}=(\ms, \alpha_{\ms}', \beta_{\ms}')\in {\rm Tri}_{\bar{0}}(\undla),
		{\rm T}=(\mt, \alpha_{\mt}, \beta_{\mt})\in {\rm Tri}(\undla)
		\right\}
			\end{align}  form two $\mathbb{K}$-bases of the block $B_{\undla}$ of $\mhgcn$.
			
			Moreover, for ${\rm S}=(\ms, \alpha_{\ms}', \beta_{\ms}')\in {\rm Tri}_{\bar{0}}(\undla),
			{\rm T}=(\mt, \alpha_{\mt}, \beta_{\mt})\in {\rm Tri}(\undla),$ we have \begin{equation}\label{dege-fst and re. fst}
			\mathfrak{f}_{{\rm S},{\rm T}}
			=\frac{\mathfrak{c}_{\ms,\mt}}{\mathfrak{c}_{\ms,\mathfrak{w} }\mathfrak{c}_{\mathfrak{w},\mt }} \mathfrak{f}_{{\rm S},{\rm T}}^\mathfrak{w}\in \mathcal{F}_{\rm S}\mhgcn \mathcal{F}_{\rm T}.
			\end{equation} The multiplications of basis elements in \eqref{deg seminormal1} are given as follows.
			
			(1) Supppose $d_{\undla}=0.$  Then for any
			${\rm S}=(\ms, \alpha_{\ms}', \beta_{\ms}'),
			{\rm T}=(\mt, \alpha_{\mt}, \beta_{\mt}),
			{\rm U}=(\mfku,\alpha_{\mfku}^{''},\beta_{\mfku}^{''}),
			{\rm V}=(\mfkv,\alpha_{\mfkv}^{'''},\beta_{\mfkv}^{'''})\in {\rm Tri}(\undla),$ we have
			\begin{align}\label{deg multiplication1}
				\mathfrak{f}_{{\rm S},{\rm T}}^\mathfrak{w} f_{{\rm U},{\rm V}}^\mathfrak{w}
				=\delta_{{\rm T},{\rm U}} \mathfrak{c}_{\rm T}^\mathfrak{w} \mathfrak{f}_{{\rm S},{\rm V}}^\mathfrak{w}.
			\end{align}

			(2) Suppose $d_{\undla}=1.$ Then for any $a,b\in \mathbb{Z}_2$ and
			\begin{align*}
				{\rm S}&=(\ms, \alpha_{\ms}', \beta_{\ms}')\in {\rm Tri}_{\bar{0}}(\undla), \quad
				{\rm T}_{a}=(\mt, \alpha_{\mt,a}, \beta_{\mt})\in {\rm Tri}_{a}(\undla),\nonumber\\
				{\rm U}&=(\mfku,\alpha_{\mfku}^{''},\beta_{\mfku}^{''})\in {\rm Tri}_{\bar{0}}(\undla), \quad
				{\rm V}_{b}=(\mfkv,{\alpha_{\mfkv,b}^{'''}},\beta_{\mfkv}^{'''})\in {\rm Tri}_{b}(\undla),\nonumber
			\end{align*} we have
			\begin{align}\label{deg multiplication2}
				\mathfrak{f}_{{\rm S},{\rm T}_{a}}^\mathfrak{w} f_{{\rm U},{\rm V}_{b}}^\mathfrak{w}
				=\delta_{{\rm T}_{\bar{0}},{\rm U}} \mathfrak{c}_{\rm T}^\mathfrak{w} \mathfrak{f}_{{\rm S},{\rm V}_{a+b}}^\mathfrak{w}.
			\end{align}
		\end{thm}
		
\subsection{Further properties}
Recall \cite[(2.32)]{BK}, the algebra $\mhgcn$ admits an anti-involution $*$ defined by $s_i^*=s_i,$ $c_j^*=c_j$ and $x_j^*=x_j,$ for
$i\in [n-1],$ $j\in[n].$
Recall the Definition \ref{sgn}. Then we have
\begin{prop}\label{dege-re. fst and BK involution} We fix $\mathfrak{w}\in\Std(\undla).$
			
			(1) If $d_{\undla}=0.$ For any ${\rm S}=(\ms, \alpha_{\ms}', \beta_{\ms}'), {\rm T}=(\mt, \alpha_{\mt}, \beta_{\mt})\in {\rm Tri}(\undla),$ we have
			\begin{align}\label{dege dual 1}
				(\mathfrak{f}_{{\rm S},{\rm T}})^*
				=\sgn(\ms,\widehat{\alpha_{\ms}'})\sgn(\mt,\widehat{\alpha_{\mt}})\mathfrak{f}_{{\rm \widehat{T}},{\rm \widehat{S}}}
			\end{align} and
				\begin{align}\label{dege dual 2}
			(\mathfrak{f}_{{\rm S},{\rm T}}^\mathfrak{w})^*
				=\sgn(\ms,\widehat{\alpha_{\ms}'})\sgn(\mt,\widehat{\alpha_{\mt}})\mathfrak{f}_{{\rm \widehat{T}},{\rm \widehat{S}}}^\mathfrak{w}.
			\end{align}
		
			(2) If $d_{\undla}=1.$ For any $a\in \mathbb{Z}_{2}$ and ${\rm S}=(\ms, \alpha_{\ms}', \beta_{\ms}')\in {\rm Tri}_{\bar{0}}(\undla), {\rm T}_{a}=(\mt, \alpha_{\mt,a}, \beta_{\mt})\in {\rm Tri}_{a}(\undla),$ we have
		$$
			(\mathfrak{f}_{{\rm S},{\rm T}_{a}})^*
		= \sgn(\ms,\widehat{\alpha'_{\ms,\overline{1}}})_{\overline{0}}\sgn(\mt,\widehat{\alpha_{\mt,a+\overline{1}}})_{a+\overline{1}}\mathfrak{f}_{{\rm \widehat{T}},{\rm \widehat{S_a} } }
		$$ and
		$$
				(\mathfrak{f}_{{\rm S},{\rm T}_{a}}^\mathfrak{w})^*
		= \sgn(\ms,\widehat{\alpha'_{\ms,\overline{1}}})_{\overline{0}}\sgn(\mt,\widehat{\alpha_{\mt,a+\overline{1}}})_{a+\overline{1}}\mathfrak{f}_{{\rm \widehat{T}},{\rm \widehat{S_a} } }^\mathfrak{w}.$$
		\end{prop}
\begin{defn}	Let $\undla\in\mathscr{P}^{\bullet,m}_{n}$ with $\bullet\in\{\mathsf{0},\mathsf{s}\}.$
	\begin{enumerate}
			\item		 Suppose $d_{\undla}=0.$ Let ${\rm S}=(\ms, \alpha_{\ms}', \beta_{\ms}'),
{\rm T}=(\mt, \alpha_{\mt}, \beta_{\mt})\in {\rm Tri}(\undla)$.	For each $i\in [n-1],$ we define $f(i,{\rm T},{\rm S})\in \mhgcn$ as follows.

\begin{enumerate}
	\item if $i,i+1\in [n]\setminus \mathcal{D}_{\mt},$
	$$f(i,{\rm T},{\rm S}):=(-1)^{\delta_{\beta_{\mt}}(i)}\mathfrak{f}_{(\mt,\alpha_{\mt},\beta_{\mt}+e_{i}+e_{i+1}),{\rm S}};
	$$
	\item
	if $i={d(\mt,\mt^{\undla})}(i_p)\in \mathcal{D}_{\mt}, i+1\in [n]\setminus \mathcal{D}_{\mt},$
	$$f(i,{\rm T},{\rm S}):=\begin{cases}(-1)^{|\beta_{\mt}|_{>i}+|\alpha_{\mt}|_{<i}}\mathfrak{f}_{(\mt,\alpha_{\mt}+e_{i},\beta_{\mt}+e_{i+1}),{\rm S}}, & \text{ if  $p$ is odd},\\
		(-\sqrt{-1})(-1)^{|\beta_{\mt}|_{>i}+|\alpha_{\mt}|_{\leq {d(\mt,\mt^{\undla})}(i_{p-1})}}\mathfrak{f}_{(\mt,\alpha_{\mt}+e_{{d(\mt,\mt^{\undla})}(i_{p-1})},\beta_{\mt}+e_{i+1}),{\rm S}}, &\text{ if  $p$ is even};
	\end{cases}\nonumber
	$$
	\item
	if $i+1={d(\mt,\mt^{\undla})}(i_p)\in \mathcal{D}_{\mt}, i\in [n]\setminus \mathcal{D}_{\mt},$ $$f(i,{\rm T},{\rm S}):=\begin{cases}(-1)^{|\beta_{\mt}|_{\geq i}+|\alpha_{\mt}|_{<i+1}} \mathfrak{f}_{(\mt,\alpha_{\mt}+e_{i+1},\beta_{\mt}+e_{i}),{\rm S}}, & \text{ if  $p$ is odd},\\
		(-\sqrt{-1})(-1)^{|\beta_{\mt}|_{\geq i}+|\alpha_{\mt}|_{\leq {d(\mt,\mt^{\undla})}(i_{p-1})}}\mathfrak{f}_{(\mt,\alpha_{\mt}+e_{{d(\mt,\mt^{\undla})}(i_{p-1})},\beta_{\mt}+e_{i}),{\rm S}}, &\text{ if $p$ is even};\end{cases}  \nonumber
	$$
\end{enumerate}
	\item	Suppose $d_{\undla}=1.$
	Let $a\in \Z_2$, ${\rm T}=(\mt, \alpha_{\mt}, \beta_{\mt})\in {\rm Tri}_{\bar{0}}(\undla),$
	${\rm S}_{a}=(\ms, {\alpha_{\ms,a}'}, \beta_{\ms}')\in {\rm Tri}_{a}(\undla)$.  For each $i\in [n-1],$ we define $f(i,{\rm T},{\rm S}_a)\in\mhgcn$ as follows.

\begin{enumerate}
	\item if $i,i+1\in [n]\setminus \mathcal{D}_{\mt},$
	$$f(i,{\rm T},{\rm S}_a):=(-1)^{\delta_{\beta_{\mt}}(i)}\mathfrak{f}_{(\mt,\alpha_{\mt},\beta_{\mt}+e_{i}+e_{i+1}),{\rm S}_a};
	$$
	\item
	if $i={d(\mt,\mt^{\undla})}(i_p)\in \mathcal{D}_{\mt}, i+1\in [n]\setminus \mathcal{D}_{\mt},$
	$$f(i,{\rm T},{\rm S}_a):=\begin{cases}(-1)^{|\beta_{\mt}|_{>i}+|\alpha_{\mt,\bar{1}}|_{<i}}	\mathfrak{f}_{(\mt,\alpha_{\mt}+e_i,\beta_{\mt}+e_{i+1}),{\rm S}_a}, & \text{ if  $p\neq t$ is odd},\\
		(-1)^{|\beta_{\mt}|_{>i}+|\alpha_{\mt}|}	\mathfrak{f}_{(\mt,\alpha_{\mt},\beta_{\mt}+e_{i+1}),{\rm S}_{a+\bar{1}}}, & \text{ if  $p=t$},\\
		(-\sqrt{-1})(-1)^{|\beta_{\mt}|_{>i}+|\alpha_{\mt,\bar{1}}|_{\leq {d(\mt,\mt^{\undla})}(i_{p-1})}}&\\
		\qquad\qquad	\qquad\qquad \cdot \mathfrak{f}_{(\mt,\alpha_{\mt}+e_{{d(\mt,\mt^{\undla})}(i_{p-1})},\beta_{\mt}+e_{i+1}),{\rm S}_a}, &\text{ if  $p$ is even};
	\end{cases}
	$$
	\item
	if $i+1={d(\mt,\mt^{\undla})}(i_p)\in \mathcal{D}_{\mt}, i\in [n]\setminus \mathcal{D}_{\mt},$ $$f(i,{\rm T},{\rm S}_a):=\begin{cases}(-1)^{|\beta_{\mt}|_{\geq i}+|\alpha_{\mt,\bar{1}}|_{<i+1}}	\mathfrak{f}_{(\mt,\alpha_{\mt}+e_{i+1},\beta_{\mt}+e_{i}),{\rm S}_a},& \text{ if  $p\neq t$ is odd},\\
		(-1)^{|\beta_{\mt}|_{\geq i}+|\alpha_{\mt}|}	\mathfrak{f}_{(\mt,\alpha_{\mt},\beta_{\mt}+e_{i}),{\rm S}_{a+\bar{1}}},& \text{ if  $p=t$ },\\
		(-\sqrt{-1})(-1)^{|\beta_{\mt}|_{\geq i}+|\alpha_{\mt,\bar{1}}|_{\leq {d(\mt,\mt^{\undla})}(i_{p-1})}}&\\
		\qquad\qquad\qquad\qquad\cdot	\mathfrak{f}_{(\mt,\alpha_{\mt}+e_{{d(\mt,\mt^{\undla})}(i_{p-1})},\beta_{\mt}+e_{i}),{\rm S}_a},&\text{ if $p$ is even};\end{cases}
	$$
\end{enumerate}
	\end{enumerate}
	\end{defn}
The action formulae of generators on seminormal basis are similar as Proposition \ref{generators action on seminormal basis} with some modifications.
\begin{prop}
			Let $\undla\in\mathscr{P}^{\bullet,m}_{n}$ with $\bullet\in\{\mathsf{0},\mathsf{s}\}.$
			
			\begin{enumerate}
				\item		 Suppose $d_{\undla}=0.$ Let ${\rm S}=(\ms, \alpha_{\ms}', \beta_{\ms}'),
				{\rm T}=(\mt, \alpha_{\mt}, \beta_{\mt})\in {\rm Tri}(\undla)$.
				\begin{enumerate}
					\item For each $i\in [n],$ we have
					\begin{align}
						x_i \cdot \mathfrak{f}_{{\rm T},{\rm S}}
						=\nu_{\beta_{\mt}}(i) \mathtt{u}_{\mt,i} \mathfrak{f}_{{\rm T},{\rm S}}.\nonumber
					\end{align}
					\item For each $i\in [n],$ we have
					\begin{align*}
						c_i \cdot \mathfrak{f}_{{\rm T},{\rm S}}
						&=\begin{cases}
							(-1)^{|\beta_{\mt}|_{<i}}  \mathfrak{f}_{(\mt,\alpha_{\mt},\beta_{\mt}+e_i),{\rm S}}, & \text{ if } i\in [n]\setminus \mathcal{D}_{\mt}, \\
							(-1)^{|\beta_{\mt}|+|\alpha_{\mt}|_{<i}} \mathfrak{f}_{(\mt,\alpha_{\mt}+e_{i},\beta_{\mt}),{\rm S}}, & \text{ if  $i={d(\mt,\mt^{\undla})}(i_p) \in \mathcal{D}_{\mt}$ ,where $p$ is odd},\\
							(-\sqrt{-1})(-1)^{|\beta_{\mt}|+|\alpha_{\mt}|_{\leq {d(\mt,\mt^{\undla})}(i_{p-1})}}&\\
							\qquad\qquad \cdot \mathfrak{f}_{(\mt,\alpha_{\mt}+e_{{d(\mt,\mt^{\undla})}(i_{p-1})},\beta_{\mt}),{\rm S}}, & \text{ if  $i={d(\mt,\mt^{\undla})}(i_p) \in \mathcal{D}_{\mt}$ ,where $p$ is even},
						\end{cases}
					\end{align*}

					\item
					For each $i\in [n-1],$
%
				we have
					\begin{align*}
						&s_i\cdot \mathfrak{f}_{{\rm T},{\rm S}}\\
						=&-\frac{1}{\nu_{\beta_{\mt}}(i)\mathtt{u}_{\mt,i}-\nu_{\beta_{\mt}}(i+1)\mathtt{u}_{\mt,i+1}} \mathfrak{f}_{{\rm T},{\rm S}} \nonumber\\
						&\qquad -\frac{1}{\nu_{\beta_{\mt}}(i)\mathtt{u}_{\mt,i}+\nu_{\beta_{\mt}}(i+1)\mathtt{u}_{\mt,i+1}}f(i,{\rm T},{\rm S})\nonumber\\
						&\qquad\qquad+\delta(s_i\mt)(-1)^{\delta_{\beta_{\mt}(i)}\delta_{\beta_{\mt}(i+1)}}\sqrt{\mathfrak{c}_{\mt}(i)}\frac{\mathfrak{c}_{\mt,\ms}}{\mathfrak{c}_{s_i\cdot\mt,\ms}}\mathfrak{f}_{s_i\cdot{\rm T},{\rm S}}
					\end{align*}
					
				\end{enumerate}
				
				\item	Suppose $d_{\undla}=1.$
				Let $a\in \Z_2$, ${\rm T}=(\mt, \alpha_{\mt}, \beta_{\mt})\in {\rm Tri}_{\bar{0}}(\undla),$
				${\rm S}_{a}=(\ms, {\alpha_{\ms,a}'}, \beta_{\ms}')\in {\rm Tri}_{a}(\undla)$.
				\begin{enumerate}
					\item For each $i\in [n],$ we have
					\begin{align}
						x_i \cdot \mathfrak{f}_{{\rm T},{\rm S}_{a}}
						=\nu_{\beta_{\mt}}(i)\mathtt{u}_{\mt,i} \mathfrak{f}_{{\rm T},{\rm S}_{a}}.\nonumber
					\end{align}
					\item For each $i\in [n],$ we have
					\begin{align*}
						c_i\cdot \mathfrak{f}_{{\rm T},{\rm S}_a}	=\begin{cases} 	
					(-1)^{|\beta_{\mt}|_{<i}}\mathfrak{f}_{(\mt,\alpha_{\mt},\beta_{\mt}+e_i ),{\rm S}_{a}},&\text{if $ i\in [n]\setminus \mathcal{D}_{\mt}$,}\\
					(-1)^{|\beta_{\mt}|+|\alpha_{\mt,\bar{1}}|_{<i}} 	&\\
					\qquad\qquad\qquad\cdot	\mathfrak{f}_{(\mt,\alpha_{\mt}+e_i ,\beta_{\mt}),{\rm S}_{a}}, &\text{if $i={d(\mt,\mt^{\undla})}(i_p) \in \mathcal{D}_{\mt},\,p\neq t$ is odd,}\\
							(-1)^{|\beta_{\mt}|+|\alpha_{\mt}|} 	\mathfrak{f}_{{\rm T},{\rm S}_{a+\bar{1}}}, &\text{if $ i={d(\mt,\mt^{\undla})}(i_p) \in \mathcal{D}_{\mt},\,p=t$,}\\
							(-\sqrt{-1})	(-1)^{|\beta_{\mt}|+|\alpha_{\mt,\bar{1}}|_{\leq {d(\mt,\mt^{\undla})}(i_{p-1})}}&\\
								\qquad\cdot	 \mathfrak{f}_{(\mt,\alpha_{\mt}+e_{{d(\mt,\mt^{\undla})}(i_{p-1})},\beta_{\mt}),{\rm S}_{a}}, &\text{if $ i={d(\mt,\mt^{\undla})}(i_p) \in \mathcal{D}_{\mt},\,p$ is even.}\\
						\end{cases}
					\end{align*}
					
					\item For each $i\in [n-1],$
					we have
					\begin{align*}
						s_i\cdot \mathfrak{f}_{{\rm T},{\rm S}_a}
						=&-\frac{1}{\nu_{\beta_{\mt}}(i)\mathtt{u}_{\mt,i}-\nu_{\beta_{\mt}}(i+1)\mathtt{u}_{\mt,i+1}}  \mathfrak{f}_{{\rm T},{\rm S}_a} \nonumber\\
						&\qquad-\frac{1}{\nu_{\beta_{\mt}}(i)\mathtt{u}_{\mt,i}+\nu_{\beta_{\mt}}(i+1)\mathtt{u}_{\mt,i+1}} f(i,{\rm T},{\rm S}_a)\nonumber\\
						&\qquad\qquad+ \delta(s_i\mt)(-1)^{\delta_{\beta_{\mt}}(i)\delta_{\beta_{\mt}}(i+1)}\sqrt{\mathfrak{c}_{\mt}(i)}\frac{\mathfrak{c}_{\mt,\ms}}{\mathfrak{c}_{s_i\cdot\mt,\ms}}
							 \mathfrak{f}_{s_i\cdot {\rm T}, {\rm S}_{a}}.
							\end{align*}
			
				\end{enumerate}
				
			\end{enumerate}
		\end{prop}
\label{pag:subalgebras of mhgcn}
\begin{defn}
(1) Let $G_n^f$ be the subalgebra of $\mhgcn$ generated by $x_1^2,\ldots, x_n^2.$

(2) Let $P_n^f$ be the subalgebra of $\mhgcn$ generated by $x_1,\ldots, x_n.$

(3) Let $A_n^f$ be the subalgebra of $\mhgcn$ generated by $x_1,\ldots, x_n$ and $c_1,\ldots, c_n.$
 \end{defn}

We can similarly obtain parallel results on these subalgebras of $\mhgcn$ as in subsection \ref{some subalgebras}. We skip the details.

\section*{Index of notation}
The Index of notation summarizes the main symbols used in this paper, including their brief descriptions and the page numbers of their first formal introduction.

\begin{center}
\textbf{Generalities}
\vspace{0.8em}
\begin{tabular}{|l|l|l|l|}
\hline  Symbol  & Description & Page \\
\hline  $\N$ & The set of positive integers \{1,2,\ldots\} & ~\pageref{pag:N}~ \\
        $\mathfrak{S}_n$ & The symmetric group of $n$ letters & ~\pageref{pag:N}~ \\
        $\mathbb{K}$ & An algebraically closed field of characteristic different from $2$ & ~\pageref{pag:K}~ \\
        $\mathbb{K}^*$ & The set $\mathbb{K}\setminus\{0\}$ & ~\pageref{pag:K}~ \\
        $|v|$ & The parity (or superdegree) of vecter $v$ in some super vertor space & ~\pageref{pag:||}~ \\
        $\Pi V$ & The parity shift of supermodule $V$ & ~\pageref{pag:parity shift}~ \\
        $V\circledast W$ & The irreducible component of $V\boxtimes W$ for irreducible supermodules $V,$ $W$ & ~\pageref{pag:irrtensor}~ \\
        $\mathcal{C}_n$ & The Clifford algebra & ~\pageref{pag:Clifford algebra}~ \\
        $\overrightarrow{\prod}$ & The ordered product & ~\pageref{pag:ordered product}~ \\
        $\lfloor x \rfloor$ & The greatest integer less than or equal to the real number $x$ & ~\pageref{pag:round down}~ \\
        $\lceil x \rceil$ & The smallest integer greater than or equal to the real number $x$ & ~\pageref{pag:round up}~ \\
        $\supp(\beta)$ &  The supporting set $\{1 \leq k \leq n:\beta_{k}=\bar{1}\}$ for $\beta=(\beta_1,\ldots,\beta_n)\in\mathbb{Z}_2^n$  & ~\pageref{pag:suppot and sum}~ \\
        $|\beta|$ & $\Sigma_{i=1}^{n}\beta_i$ for $\beta=(\beta_1,\ldots,\beta_n)\in\mathbb{Z}_2^n$  & ~\pageref{pag:suppot and sum}~ \\
        $\mathsf{0},\mathsf{s},\mathsf{ss}$ & The types of combinatorics & ~\pageref{pag:The types of combinatorics}~ \\
        $\mathscr{P}^{\bullet,m}_{n}$ & The set of mixed ($\bullet+m$)-multipartitions of $n$ for $\bullet\in\{\mathsf{0},\mathsf{s},\mathsf{ss}\}$ & ~\pageref{pag:The types of combinatorics}~ \\
        $\undla$ & An element in $\mathscr{P}^{\bullet,m}_{n}$ &  ~\pageref{pag:multipartition}~ \\
        $\alpha\in\undla$ & A box (or node) of $\undla$ & ~\pageref{pag:multipartition}~ \\
        $\Std(\undla)$ & The set of standard tableaux of shape $\undla$ & ~\pageref{pag:standard tableaux}~ \\
        $\mt$ & An element in $\Std(\undla)$ & ~\pageref{pag:standard tableaux}~ \\
        $\mt^{\undla},$ $\mt_{\undla}$ & Initial row tableau of shape $\undla,$ Initial column tableau of shape $\undla$ & ~\pageref{pag:standard tableaux}~ \\
        $\mathcal{D}_{\undla}$ & The set of boxes in the first diagnoals of strict partiton components of $\undla$ & ~\pageref{pag:diag of undlam}~ \\
        $\mathcal{D}_{\mt}$ & The set of numbers in the first diagnoals of strict partiton components of $\mt$ & ~\pageref{pag:diag of undlam}~ \\
        $P(\undla)$ & The set of admissible permutations indexed by $\undla$ & ~\pageref{pag:adimssible}~ \\
        $d(\ms,\mt)$ & The unique element in $\mathfrak{S}_n$ such that $\ms=d(\ms,\mt)\mt$ for $\ms,\mt\in\Std(\undla)$ & ~\pageref{pag:adimssible}~ \\
        $[n]$ & The set of positive integers $\{1,2,\ldots,n\}$ & ~\pageref{pag:[n]}~ \\
        $\nu_{\beta}(k)$ & It equals $-1$ if $\beta_{k}=\bar{1}$ and equals $1$ if $\beta_{k}=\bar{0},$ for $\beta\in\mathbb{Z}_2^n$ & ~\pageref{pag:nubetak}~ \\
        $|\beta|_{<i}$ & $\Sigma_{k=1}^{i-1}\beta_k$ for $\beta\in\mathbb{Z}_2^n$ & ~\pageref{pag:nubetak}~ \\
        $d_{\undla}$ & It equals $1$ if $\sharp \mathcal{D}_{\undla}$ is odd, otherwise $0$ & ~\pageref{pag:dundla}~ \\
        $\mathcal{OD}_{\mt}$ & Some key subset of $\mathcal{D}_{\mt}$ &  ~\pageref{pag:Dt,ODt,Z2ODt}~ \\
        $\mathbb{Z}_2(\mathcal{OD}_{\mt})$ & The subset of $\mathbb{Z}_2^n$ supported on $\mathcal{OD}_{\mt}$ &  ~\pageref{pag:Dt,ODt,Z2ODt}~ \\
        $\mathbb{Z}_2([n]\setminus\mathcal{D}_{\mt})$ & The subset of $\mathbb{Z}_2^n$ supported on $[n]\setminus\mathcal{D}_{\mt}$ &  ~\pageref{pag:Dt,ODt,Z2ODt}~ \\
        $\gamma_{\mt}$ & A certain idempotent of $\mathcal{C}_n$ related to $\mt$ & ~\pageref{pag:Dt,ODt,Z2ODt}~ \\
        $\delta_{\beta}(k)$ & It equals 1 if $\beta_k=\bar{1}$ and equals $0$ if $\beta_k=\bar{0},$ for $\beta\in\mathbb{Z}_2^n$  & ~\pageref{pag:deltabetak}~ \\
        $\delta(s_i\mt)$ & It equals 1 if $s_i\mt\in\Std(\undla)$ for $\mt\in\Std(\undla),$ otherwise $0$ & ~\pageref{pag:nondege coeffi cti}~ \\
        $\mathbb{Z}_2(\mathcal{OD}_{\mt})_{a}$  & There is a certain decomposition $\mathbb{Z}_2(\mathcal{OD}_{\mt})=\sqcup_{a\in \mathbb{Z}_2}\mathbb{Z}_2(\mathcal{OD}_{\mt})_{a}$ & ~\pageref{pag:decomposition of OD}~ \\
        $\widehat{\ }$ & A certain involution on $\mathbb{Z}_2(\mathcal{OD}_{\mt})_{\bar{0}}$ or $\mathbb{Z}_2(\mathcal{OD}_{\mt})_{\bar{1}}$ & ~\pageref{pag:widehat}~ \\
        ${\rm Tri}(\undla)$ & The set of triples associated with standard tableaux of shape $\undla$ & ~\pageref{pag:Tri}~ \\
        ${\rm Tri}_{a}(\undla)$ & There is a certain decomposition ${\rm Tri}(\undla)=\sqcup_{a\in \mathbb{Z}_2}{\rm Tri}_{a}(\undla)$ & ~\pageref{pag:Tri}~ \\
        $\sgn(\mt,\alpha_{\mt})$ & A sign associated to $\mt\in\Std(\undla)$ and $\alpha_{\mt}\in\mathbb{Z}_2(\mathcal{OD}_{\mt})$ when $d_{\undla}=0$ & ~\pageref{pag:sgn}~ \\
        $\sgn(\mt,\alpha_{\mt})_{a}$ & A sign associated to $\mt\in\Std(\undla),$ $\alpha_{\mt}\in\mathbb{Z}_2(\mathcal{OD}_{\mt})$ and $a\in\mathbb{Z}_2$ when $d_{\undla}=1$ & ~\pageref{pag:sgn}~ \\
\hline
\end{tabular}
\end{center}

\begin{center}
\textbf{Non-degenerate case}
\vspace{0.8em}
\begin{tabular}{|l|l|l|l|}
\hline  Symbol  & Description & Page \\
\hline  $\mHcn$ & The affine Hecke-Clifford algebra & ~\pageref{pag:AHCA}~ \\
        $q$ & The Hecke parameter in $\mathbb{K}\setminus\{0,\pm 1,\pm \sqrt{-1}\}$ & ~\pageref{pag:AHCA}~ \\
        $\epsilon$ & $q-q^{-1}$ & ~\pageref{pag:AHCA}~ \\
        $\mathcal{A}_n$ & A certain subalgebra of $\mHcn$ & ~\pageref{pag:subalg An}~ \\
        $\tilde{\Phi}_i$ & The intertwining element & ~\pageref{pag:BK intertwining element}~ \\
        $\Phi_i$ & Jone-Nazarov's intertwining element & ~\pageref{pag:JN intertwining element}~ \\
        $\Phi_i(x,y)$ & An element in $\mHcn$ related to $\Phi_i$ & ~\pageref{pag:Phi function}~ \\
        $\mathtt{q}(\iota)$ & $2(q\iota+(q\iota)^{-1})/(q+q^{-1})$ for $\iota\in\mathbb{K}^*$ & ~\pageref{pag:q-function and b-function}~ \\
        $\mathtt{b}_{\pm}(\iota)$ & The solutions of equation $x+x^{-1}=\mathtt{q}(\iota)$ & ~\pageref{pag:q-function and b-function}~ \\
        $\mHfcn$ & The cyclotomic Hecke-Clifford algebra & ~\pageref{pag:CHCA}~ \\
        $\underline{Q}$ & The cyclotomic parameters $(Q_1,Q_2,\ldots,Q_m)\in(\mathbb{K}^*)^m$ & ~\pageref{pag:Q-parameters}~\\
        $r$ & The level of $\mHfcn$ & ~\pageref{pag:nondege level}~\\
        $*$ & The anti-involution on $\mHfcn$ & ~\pageref{pag:nondege anti-involution}~ \\
        $\res(\alpha)$ & The residue $Q_lq^{2(j-i)}$ of box $\alpha=(i,j,l)$ & ~\pageref{pag:nondeg residue}~ \\
        $\res_{\mt}(k)$ & The residue of box $\mt^{-1}(k)$ for $\mt\in\Std(\undla)$ & ~\pageref{pag:nondeg residue}~ \\
        $\res(\mt)$ & The residue sequence $(\res_{\mt}(1),\ldots,\res_{\mt}(n))$ of $\mt\in\Std(\undla)$ & ~\pageref{pag:nondeg residue}~ \\
        $\mathtt{q}(\res(\mt))$ & The $\mathtt{q}$-sequence $(\mathtt{q}(\res_{\mt}(1)),\ldots,\mathtt{q}(\res_{\mt}(n)))$ of $\mt\in\Std(\undla)$ & ~\pageref{pag:nondeg residue}~ \\
        $P^{(\bullet)}_{n}(q^2,\undQ)$ & The Poincar\'e polynomial of type $\bullet\in\{\mathsf{0},\mathsf{s},\mathsf{ss}\}$ &  ~\pageref{pag:nondege Pioncare poly}~ \\
        $\mathbb{D}(\undla)$ & The simple module of $\mHfcn$ indexed by $\undla$ & ~\pageref{pag:nondege simple module}~ \\
        $\mathtt{b}_{\mt,i}$ & $\mathtt{b}_{-}(\res_{\mt}(i))$ &  ~\pageref{pag:nondege coeffi cti}~ \\
        $\mathtt{c}_{\mt}(i)$ & Some structure coefficient appeared in module $\mathbb{D}(\undla)$ &  ~\pageref{pag:nondege coeffi cti}~ \\
        $F_{\rm T}$ & The primitive idempotent indexed by ${\rm T}\in{\rm Tri}_{\bar{0}}(\undla)$ &  ~\pageref{pag:primitive idempotents and blocks}~ \\
        $F_{\undla}$ & The primitive central idempotent indexed by $\undla$ &  ~\pageref{pag:primitive idempotents and blocks}~ \\
        $B_{\undla}$ & The simple block of  $\mHfcn$ indexed by $\undla$ & ~\pageref{pag:primitive idempotents and blocks}~ \\
        $\Phi_{\ms,\mt}$ & A key element in $\mHfcn$ indexed by $\ms,\mt\in\Std(\undla)$ &  ~\pageref{pag:Phist and cst}~ \\
        $\mathtt{c}_{\ms,\mt}$ & A key coefficient indexed by $\ms,\mt\in\Std(\undla)$ &  ~\pageref{pag:Phist and cst}~ \\
        $f_{{\rm S},{\rm T}}^{\mathfrak{w}},$ $f_{{\rm S},{\rm T}_a}^{\mathfrak{w}}$ & The seminormal basis factoring though a fixed standard tableau $\mathfrak{w}$ &  ~\pageref{pag:nondege seminormal basis}~ \\
        $f_{{\rm S},{\rm T}},$ $f_{{\rm S},{\rm T}_a}$ & The (reduced) seminormal basis  & ~\pageref{pag:nondege seminormal basis}~ \\
        $\mathtt{c}_{\rm T}^{\mathfrak{w}}$ & $(\mathtt{c}_{\mt,\mathfrak{w}})^2$ for ${\rm T}=(\mt, \alpha_{\mt}, \beta_{\mt})\in {\rm Tri}(\undla)$ and a fixed standard tableau $\mathfrak{w}$  & ~\pageref{pag:nondege cT}~ \\
        $\mathcal{G}_n^f,$ $\mathbb{P}_n^f,$ $\mathcal{A}_n^f$ & Some certain subalgebras of $\mHfcn$ & ~\pageref{pag:subalgebras of mHfcn}~ \\
        $F_{\mt},$ $F_{(\mt,\beta_{\mt})}$ & Some key idempotents of $\mHfcn$ & ~\pageref{pag:some idempotents of mHfcn}~ \\
        $\mathtt{q}_{\mt,i}$ & $\mathtt{q}(\res_{\mt}(i))$ & ~\pageref{pag:nondege q-values}~ \\
\hline
\end{tabular}
\end{center}

\begin{center}
\textbf{Degenerate case}
\vspace{0.8em}
\begin{tabular}{|l|l|l|l|}
\hline  Symbol  & Description & Page \\
\hline  $\mhcn$ & The affine Sergeev algebra & ~\pageref{pag:ASA}~ \\
        $\mathcal{P}_n$ & A certain subalgebra of $\mhcn$ & ~\pageref{pag:subalg Pn}~ \\
		$\tilde{\phi}_i$ & The intertwining element & ~\pageref{pag:dege BK intertwining elements}~ \\
        $\phi_i$ & Nazarov's intertwining element & ~\pageref{pag:dege N intertwining elements}~ \\
        $\phi_i(x,y)$ & An element in $\mhcn$ related to $\phi_i$ & ~\pageref{pag:dege N phi function}~ \\
        $\mathtt{q}(\iota)$ & $\iota(\iota+1)$ for $\iota\in\mathbb{K}$ & ~\pageref{dege q-function and u-function}~ \\
        $\mathtt{u}_{\pm}(\iota)$ & $\pm \sqrt{\iota(\iota+1)}$ for $\iota\in\mathbb{K}$ & ~\pageref{dege q-function and u-function}~ \\
        $\mhgcn$ & The cyclotomic Sergeev algebra & ~\pageref{pag:CSA}~ \\
        $\underline{Q}$ & The cyclotomic parameters $(Q_1,Q_2,\ldots,Q_m)\in\mathbb{K}^m$ & ~\pageref{pag:dege Q-parameters}~\\
        $r$ & The level of $\mhgcn$ & ~\pageref{pag:dege level}~\\
        $*$ & The anti-involution on $\mhgcn$ & ~\pageref{pag:dege anti-involution}~ \\
        $\res(\alpha)$ & The residue $Q_l+j-i$ of box $\alpha=(i,j,l)$ & ~\pageref{pag:dege residue}~ \\
        $\res_{\mt}(k)$ & The residue of box $\mt^{-1}(k)$ for $\mt\in\Std(\undla)$ & ~\pageref{pag:dege residue}~ \\
        $\res(\mt)$ & The residue sequence $(\res_{\mt}(1),\ldots,\res_{\mt}(n))$ of $\mt\in\Std(\undla)$ & ~\pageref{pag:dege residue}~ \\
        $\mathtt{q}(\res(\mt))$ & The $\mathtt{q}$-sequence $(\mathtt{q}(\res_{\mt}(1)),\ldots,\mathtt{q}(\res_{\mt}(n)))$ of
        $\mt\in\Std(\undla)$ & ~\pageref{pag:dege residue}~ \\
        $P^{(\bullet)}_{n}(1,\undQ)$ & The Poincar\'e polynomial of type $\bullet\in\{\mathsf{0},\mathsf{s}\}$ &  ~\pageref{pag:dege Pioncare poly}~ \\
        $D(\undla)$ & The simple module of $\mhgcn$ indexed by $\undla$ & ~\pageref{pag:dege simple module}~ \\
        $\mathtt{u}_{\mt,i}$ & $\mathtt{u}_{+}(\res_{\mt}(i))$ &  ~\pageref{pag:dege coeffi cti}~ \\
        $\mathfrak{c}_{\mt}(i)$ & Some structure coefficient appeared in module $D(\undla)$ &  ~\pageref{pag:dege coeffi cti}~ \\
        $\mathcal{F}_{\rm T}$ & The primitive idempotent indexed by ${\rm T}\in{\rm Tri}_{\bar{0}}(\undla)$ &  ~\pageref{pag:dege primitive idempotents}~ \\
        $\mathcal{F}_{\undla}$ & The primitive central idempotent indexed by $\undla$ &  ~\pageref{pag:dege primitive idempotents}~ \\
        $\mathcal{B}_{\undla}$ & The simple block of  $\mhgcn$ of indexed by $\undla$ & ~\pageref{pag:dege simple blocks}~ \\
        $\phi_{\ms,\mt}$ & A key element in $\mhgcn$ indexed by $\ms,\mt\in\Std(\undla)$ &  ~\pageref{pag:phist and cst}~ \\
        $\mathfrak{c}_{\ms,\mt}$ & A key coefficient indexed by $\ms,\mt\in\Std(\undla)$ &  ~\pageref{pag:phist and cst}~ \\
        $\mathfrak{f}_{{\rm S},{\rm T}}^{\mathfrak{w}},$ $\mathfrak{f}_{{\rm S},{\rm T}_a}^{\mathfrak{w}}$ & The seminormal basis factoring though a fixed standard tableau $\mathfrak{w}$ &  ~\pageref{pag:dege seminormal basis}~ \\
        $\mathfrak{f}_{{\rm S},{\rm T}},$ $\mathfrak{f}_{{\rm S},{\rm T}_a}$ & The (reduced) seminormal basis  & ~\pageref{pag:dege seminormal basis}~ \\
        $\mathfrak{c}_{\rm T}^{\mathfrak{w}}$ & $(\mathfrak{c}_{\mt,\mathfrak{w}})^2$ for ${\rm T}=(\mt, \alpha_{\mt}, \beta_{\mt})\in {\rm Tri}(\undla)$ and a fixed standard tableau $\mathfrak{w}$  & ~\pageref{pag:dege cT}~ \\
        $G_n^f,$ $P_n^f,$ $A_n^f$ & Some certain subalgebras of $\mhgcn$ & ~\pageref{pag:subalgebras of mhgcn}~ \\
\hline
\end{tabular}
\end{center}

	\end{document}